\documentclass{report}
\usepackage{pstricks}
\usepackage{amscd}
\usepackage{stmaryrd}
\usepackage{amssymb}
\usepackage{enumerate}
\usepackage{amsmath}
\usepackage{amsthm}
\usepackage[all,cmtip]{xy}
\usepackage{mathrsfs}
\usepackage{centernot}
\usepackage{hyperref}
\usepackage[a4paper,
]{geometry}
\usepackage{pdfpages}

\usepackage{verbatim}

\usepackage[ansinew]{inputenc} 

\numberwithin{equation}{subsection}

\newtheorem{theorem}{Theorem}[section]
\newtheorem{keytheorem}{Theorem}
\newtheorem{lemma}[theorem]{Lemma}
\newtheorem{proposition}[theorem]{Proposition}
\newtheorem{corollary}[theorem]{Corollary}
\theoremstyle{definition}
\newtheorem*{example}{Example}
\newtheorem*{examples}{Examples}
\newtheorem*{definition}{Definition}
\newtheorem*{remark}{Remark}
\newtheorem*{remarks}{Remarks}

\newtheorem*{convention}{Convention}

\DeclareMathOperator{\init}{in}
\DeclareMathOperator{\minit}{wk-in}
\DeclareMathOperator{\inv}{inv}
\DeclareMathOperator{\ord}{ord}

\DeclareMathOperator{\resord}{res.ord}
\DeclareMathOperator{\coeff}{coeff}
\DeclareMathOperator{\topp}{top}
\DeclareMathOperator{\chara}{char}
\DeclareMathOperator{\Sing}{Sing}
\DeclareMathOperator{\Dir}{Dir}
\DeclareMathOperator{\Spec}{Spec}

\DeclareMathOperator{\Diff}{Diff}

\newcommand{\N}{\mathbb{N}}

\newcommand{\F}{\mathcal{F}}
\newcommand{\G}{\mathcal{G}}

\newcommand{\OO}{\mathcal{O}}
\newcommand{\OWa}{\mathcal{O}_{W,a}}
\newcommand{\OWap}{\mathcal{O}_{W',a'}}
\newcommand{\OXa}{\mathcal{O}_{X,a}}
\newcommand{\hOWa}{\widehat{\mathcal{O}}_{W,a}}
\newcommand{\hOWb}{\widehat{\mathcal{O}}_{W,b}}
\newcommand{\hOWap}{\widehat{\mathcal{O}}_{W',a'}}
\newcommand{\mWa}{m_{W,a}}

\newcommand{\IXa}{I_{X,a}}
\newcommand{\hIXa}{\widehat{I}_{X,a}}
\newcommand{\hIXap}{\widehat{I}_{X',a'}}

\newcommand{\w}{\omega}

\newcommand{\wh}{\widehat}
\newcommand{\wt}{\widetilde}
\newcommand{\ol}{\overline}

\newcommand{\mon}{_{\textnormal{mon}}}

\newcommand{\X}{\mathcal{X}}

\newcommand{\W}{\mathcal{W}}
\newcommand{\st}{^{\textnormal{st}}}
\newcommand{\x}{\textnormal{\bf x}}
\newcommand{\wk}{^{\textnormal{wk}}}

\newcommand{\Eg}{E_a}
\newcommand{\Egold}{\Eg^{\textnormal{old}}}
\newcommand{\Egp}{E'_{a'}}
\newcommand{\Ebo}{E_{>o}}
\newcommand{\Ego}{E_{o}}
\newcommand{\HH}{\mathcal{H}}

\newcommand{\iv}{i}
\newcommand{\ul}{\lceil}
\newcommand{\ur}{\rceil}
\newcommand{\jz}{b}
\newcommand{\jzz}{a}

\newcommand{\ii}{\textnormal{\textbf{i}}}
\newcommand{\Dnew}{D_{\textnormal{new}}}
\newcommand{\ivX}{\iv_{\X}}
\newcommand{\ivXa}{\iv_{\X}(a)}
\newcommand{\ivXap}{\iv_{\X'}(a')}

\newcommand{\Xa}{_\X(a)}
\newcommand{\Xap}{_{\X'}(a')}
\newcommand{\eps}{\varepsilon}
\newcommand{\Comp}{\textnormal{Comp}}

\newcommand{\gp}{_{\textnormal{sr}}}
\newcommand{\Xmax}{X_{\max}}
\newcommand{\vp}{\varphi}
\newcommand{\codim}{\textnormal{codim}}

\newcommand{\sd}{\frac{s}{d!}}

\newcommand{\age}{\textnormal{age}}
\newcommand{\lab}{\textnormal{lab}}

\newcommand{\corner}{_{\ord}}

\newcommand{\Hom}{\textnormal{Hom}}
\newcommand{\wsd}{\frac{\wt s}{d!}}

\newcommand{\Ioc}{I_{\geq(o,c)}}
\newcommand{\Ni}{\N_\infty}

\newcommand{\ry}{r_y}
\newcommand{\Jc}{J_{\geq c}}
\newcommand{\Jcc}{(J_{-1})_{\geq c!}}
\newcommand{\FF}{\mathscr{F}}

\newcommand{\m}{\mathfrak{m}}
\newcommand{\Ea}{\Eg}
\newcommand{\Eap}{E'_{a'}}

\newcommand{\mult}{\textnormal{mult}}

\newcommand{\cid}{\frac{c-i}{c!}d}
\newcommand{\qd}{\frac{q}{c!}d}
\newcommand{\dijf}{\partial_{y^j}\partial_{z^i}(f)}

\newcommand{\IX}{\mathcal{I}_X}
\newcommand{\rad}{\textnormal{rad}}
\newcommand{\y}{\upsilon}
\newcommand{\B}{\Big}
\newcommand{\D}{\delta}
\newcommand{\wc}{\widehat}
\newcommand{\td}{\frac{\wc s}{d!}}

\definecolor{darkblue}{rgb}{0,0,0.4}

\begin{document}

\title{A New Proof for the Embedded Resolution of Surface Singularities in Arbitrary Characteristic}
\author{Stefan Perlega}
\date{}
\maketitle


\tableofcontents

\chapter{Introduction} \label{chapter_introduction}

The goal of this thesis is to exhibit a new proof for the resolution of singular algebraic surfaces $X$ embedded in a regular $3$-dimensional ambient variety $W$ over an algebraically closed field of arbitrary characteristic. The resolution process that we devise makes use of an upper semicontinuous resolution invariant which prescribes the centers of blowup and strictly decreases during each iteration of the process.

While the core of the proof makes essential use of the low dimension of $W$, several new techniques that are introduced in this thesis work both independently of dimension and characteristic. Most techniques were developed in an attempt to generalize fundamental concepts that were first introduced by Hironaka \cite{Hironaka_Annals} in his seminal proof of embedded resolution of singularities over fields of characteristic zero and its more recent revisions and simplifications \cite{Villamayor_Constructiveness}, \cite{BM_Canonical_Desing}, \cite{EV_Good_Points}, \cite{EH}. The resolution invariant which is used in this thesis can thus be seen as an effort to adapt the invariant which is successfully used in characteristic zero to the setting of arbitrary characteristic. 

In this introduction, resolutions of singularities will be defined and the modern philosophy of proving resolution by use of an upper semicontinuous invariant will be explained in some detail. Further, it will be explained how the results established in this thesis relate to the proof of resolution of singularities over fields of characteristic zero and an overview of the contents of the remaining chapters will be given. The final section of this chapter will provide a list of notational conventions used throughout the thesis and of results that the reader is assumed to be familiar with.

\section{Resolution of singularities} \label{section_resolution_of_singularities}

 Resolution of singularities is a much-studied problem in algebraic geometry and a lot of excellent literature has been published on it. Introductions to the subject can be found in \cite{Lipman_Introduction}, \cite{Ha_BAMS_1}, \cite{Cutkosky_Book}, \cite{Kollar_Book}, \cite{Ha_Obergurgl_2014}.

\begin{definition}
 Let $X$ be a variety. A \emph{resolution of singularities} of $X$ is a regular variety $X'$ together with a proper birational morphism $\pi:X'\to X$. 
\end{definition}

 The setting that we are mostly interested in is that of a regular variety $W$ (often called the ambient variety or ambient space) and a closed subvariety $X\subseteq W$ whose singularities we want to resolve. In this case, we cannot expect to find a proper birational morphism $\pi:W'\to W$ such that both $W'$ and the pre-image $\pi^{-1}(X)$ are regular. Instead, we only ask for the strict transform of $X$ to be regular and that the singularities of $\pi^{-1}(X)$ are of a particularly simple form, as will be explained in the following two definitions.

\begin{definition}
 Let $W$ be a regular variety and $X\subseteq W$ a closed subset. Denote for a point $a\in W$ by $\IXa$ the stalk of the ideal sheaf $\IX$ defining $X$ at $a$.
 
 We say that $X$ does have \emph{simple normal crossings} at a closed point $a\in X$ if there is a regular system of parameters $x_1,\ldots,x_n$ for the local ring $\OWa$ such that the ideal $\IXa$ is generated by monomials in $x_1,\ldots,x_n$. In particular, if $X$ has simple normal crossings at $a$, then every irreducible component of $X$ that contains $a$ is regular at this point.
 
 If $X$ has simple normal crossings at all of its closed points, we say that $X$ has simple normal crossings. This implies that all components of $X$ are regular. If all components of $X$ have codimension $1$ in $W$, we call $X$ a \emph{simple normal crossings divisor}. 
 
 We use the same terminology for subsets $X\subseteq\Spec(\hOWa)$ where $\hOWa$ denotes the completion of the local ring $\OWa$ with respect to its maximal ideal.
\end{definition}

\begin{definition}
 Let $W$ be a regular variety and $X\subseteq W$ a closed subvariety. An \emph{embedded resolution of singularities} of $X$ is given by a regular variety $W'$ and a proper birational morphism $\pi:W'\to W$ that have the following properties:
 \begin{enumerate}[(1)]
  \item The strict transform $X'$ of $X$ is regular.
  \item The total transform $X^*=\pi^{-1}(X)$ of $X$ has simple normal crossings.
  \item The morphism $\pi:W'\to W$ is given as a composition of blowups at regular centers.
  \item The morphism $\pi$ is an isomorphism outside the singular locus $\Sing(X)$ of $X$. Thus, it induces an isomorphism $W'\setminus\pi^{-1}(\Sing(X))\cong W\setminus\Sing(X)$.
 \end{enumerate}
\end{definition}

\section{State of the art in resolution of singularities} \label{section_state_of_the_art}

Hironaka proved the embedded resolution of singularities of algebraic varieties of arbitrary dimension over fields of characteristic zero \cite{Hironaka_Annals}. Since then, many researchers have contributed towards strengthening Hironaka's result and simplifying its proof \cite{Villamayor_Constructiveness}, \cite{Villamayor_Patching}, \cite{BM_A_Simple_Proof}, \cite{BM_Canonical_Desing}, \cite{EV_Good_Points}, \cite{BV_Strengthening}, \cite{EH}, \cite{Wlodarczyk}. 

Over fields of positive characteristic, Abhyankar was the first to prove the resolution of surface singularities \cite{Abhyankar_56}. Later, he also gave proofs for the embedded resolution of surface singularities over fields of positive characteristic and non-embedded resolution of three-folds over fields of characteristic $p>5$ \cite{Abhyankar_3_Folds}. His result was improved only recently by Cossart and Piltant who proved the non-embedded resolution of three-folds over fields of arbitrary characteristic \cite{Cossart_Piltant_1}, \cite{Cossart_Piltant_2}.

Whether an embedded resolution of the singularities is possible for threefolds and all higher dimensional varieties over fields positive characteristic is so far not known.

In the last decade, several groups of researchers have devised programs whose ultimate goal it is to prove the embedded resolution of singularities in arbitrary dimension and characteristic.

Villamayor formulated the approach to replace the concept of restriction to hypersurfaces of maximal contact by projections and the use of elimination algebras in \cite{Villamayor_Hypersurface}. In \cite{BV_Simplification}, Bravo and Villamayor used these techniques to define an invariant independently of dimension and characteristic which enabled them to prove a simplification of singularities. This simplification is commonly referred to as a reduction to the \emph{monomial case}. (Although the same terminology is used, it has to be pointed out that this monomial case is different from the monomial case that is studied in this thesis.) In characteristic zero, this monomial case can be resolved, leading to a proof of embedded resolution of singularities. Benito and Villamayor showed in \cite{Benito_Villamayor} how to resolve this monomial case in the surface case in arbitrary characteristic, thus establishing a new proof for the embedded resolution of surfaces.

Kawanoue and Matsuki have developed a new approach towards resolution in arbitrary dimension and characteristic that is called the Idealistic Filtration Program \cite{Kawanoue_IF_1}, \cite{Kawanoue_Matsuki}, \cite{Kawanoue_Obergurgl}. In this program, hypersurfaces of maximal contact are replaced by so-called leading generator systems. The biggest difference between these is that hypersurfaces of maximal contact are always regular, but the elements of a leading generator system in positive characteristic may define singular hypersurfaces. Kawanoue and Matsuki were able to show that the resulting invariants lead to embedded resolution of singularities in characteristic zero and to a reduction to a monomial case (which is again different from the two monomial cases mentioned so far) in positive characteristic. In ambient dimension $3$, they were able to resolve this monomial case \cite{Kawanoue_Matsuki_Surfaces}.

\section{Resolution via an upper semicontinuous invariant} \label{section_resolution_via_usc_inv}

With the definition of embedded resolution of singularities in mind, our philosophy is to see resolution as an \emph{iterative process}. Roughly speaking, this means that we repeatedly blow up regular centers which are contained in the singular locus of our variety and which have simple normal crossings with the exceptional components produced by previous blowups until embedded resolution is achieved.

To make this more precise, let us consider triples $\X=(W,X,E)$ of the kind where $W$ is a regular variety, $X\subseteq W$ a closed subvariety and $E$ a simple normal crossings divisor on $W$. We say that a closed subvariety $Z\subseteq W$ is a \emph{permissible center} of blowup for the triple $\X$ if the following properties hold:
\begin{itemize}
 \item $Z$ is contained in $\Sing(X)\cup E$.
 \item The union $Z\cup E$ has simple normal crossings. In particular, $Z$ is regular.
\end{itemize}
The blowup $\pi:W'\to W$ of $W$ along a permissible center $Z$ induces a new triple $\X'=(W',X',E')$ where $X'=\ol{\pi^{-1}(X\setminus Z)}$ is the strict transform of $X$ and $E'=E\st\cup \Dnew$ where $E\st$ is the strict transform of $E$ and $\Dnew=\pi^{-1}(Z)$ is the exceptional divisor of the blowup $\pi$. It follows from basic properties of blowups that the triple $\X'$ is of the same kind as $\X$. Thus, $W'$ is a regular variety of the same dimension as $W$ and $E'$ is a simple normal crossings divisor on $W'$. (More information on the properties of blowups can be found in Lectures IV and V of \cite{Ha_Obergurgl_2014}.)

To prove embedded resolution of the singularities of an embedded variety $X\subseteq W$, we start with the triple $(W,X,\emptyset)$ and have to find a sequence of blowups in permissible centers so that we eventually obtain a triple $(W',X',E')$ with the properties that $X'$ is regular and $X'\cup E'$ forms a simple normal crossings divisor on $W'$. We call such a triple $(W',X',E')$ \emph{resolved}.

To show that embedded resolution of resolution by this approach is possible, we have to solve two problems: The first problem is the \emph{choice of center}. For each triple $\X$ we have to prescribe a permissible center $Z$ that will be blown up in this iteration of the resolution process. The second problem is to show that this process \emph{terminates}. That is, proving that we always obtain a resolved triple after finitely many iterations of the resolution process. 

Both of these problems can be solved simultaneously by finding a suitable \emph{resolution invariant}. The idea is to define for every triple $\X=(W,X,E)$ a map $\mu_{\X}:X\to\Gamma$ that measures for each point $a\in X$ how far $\X$ is away from being resolved locally at $a$. We require the map $\mu_\X$ to be a local geometric invariant.
Since we want to be able to compare the singularities of $X$ in different points and measure their improvement under blowup, we consider the set $\Gamma$ with a well-ordering $\leq$. Further, we require that the invariant $\mu_\X$ is \emph{upper semicontinuous}. This means that for each value $\gamma\in\Gamma$, the set 
\[X_{\geq\gamma}=\{a\in X:\mu_\X(a)\geq \gamma\}\]
is closed. Thus, $\mu_\X$ defines a stratification of $X$ into locally closed strata. If $\mu_{\X}(a)=\min\Gamma$ for a point $a\in X$, then we require the triple $\X$ to be locally resolved at $a$.

The problem of choosing the center is solved by always blowing up the locus of points at which $\X$ is farthest from being resolved. Thus, we choose as a center the set of points $a\in X$ at which the invariant $\mu_{\X}$ is maximal in the sense that $\mu_{\X}(b)\leq\mu_{\X}(a)$ holds for all points $b\in X$. This set is called the \emph{top locus} of $X$ with respect to the invariant $\mu_{\X}$. It is closed by upper semicontinuity of $\mu_{\X}$. We require that the top locus always constitutes a permissible center of blowup for the triple $\X$. The problem of showing that the process terminates reduces to the problem of verifying that the maximal value of the invariant strictly decreases during each iteration of the resolution process. Since $\Gamma$ is well-ordered, the invariant can only decrease finitely many times, thus leading to a resolved triple in finitely many steps.

We sum up this discussion in the following definitions and proposition:

\begin{definition}
 Fix a positive integer $n>0$ and an algebraically closed field $K$. Let $\W$ be the class of triples $\X=(W,X,E)$ such that $W$ is an $n$-dimensional regular variety over $K$, $X\subseteq W$ is a closed subvariety and $E$ is a simple normal crossings divisor on $W$.

Let $(\Gamma,\leq)$ be a well-ordered set and $\mu$ a collection of maps $\mu_{\X}:X\to\Gamma$ for each triple $\X\in \W$. We say that $\mu$ is a \emph{resolution invariant} for $n$-dimensional ambient varieties over the field $K$ if it fulfills the following five properties:
\begin{enumerate}[(1)]
 \item \emph{The map $\mu$ is a local geometric invariant.}
 
 Let $\X=(W,X,E)$ and $\X_1=(W_1,X_1,E_1)$ be two triples in $\W$ which are are locally isomorphic in points $a\in W$ and $a_1\in W_1$ in the following sense: There is an isomorphism of the local rings $\OO_{W,a}$ and $\OO_{W_1,a_1}$ that locally maps $X$ and $X_1$, as well as $E$ and $E_1$ into each other. Then $\mu_{\X}(a)=\mu_{\X_1}(a_1)$.
 
 \item \emph{If $\mu_{\X}$ is minimal, the triple $\X$ is locally resolved.}
 
 Let $\X=(W,X,E)\in\W$ be a triple and $a\in X$. If $\mu_{\X}(a)=\min\Gamma$, then $X$ is regular at $a$ and the union $X\cup E$ has simple normal crossings at $a$.
 
 \item \emph{Each map $\mu_{\X}$ is upper semicontinuous.}
 
 For each triple $\X=(W,X,E)\in\W$ and each value $\gamma\in\Gamma$ the set 
 \[X_{\geq\gamma}=\{a\in X:\mu_{\X}(a)\geq\gamma\}\]
 is closed.
 
 \item \emph{The top locus of $\mu_\X$ is a permissible center of blowup for $\X$.}
 
 Let $\X=(W,X,E)\in\W$ be a triple. By property (3), there is an element $\gamma_{\max}\in\Gamma$ such that $\gamma_{\max}=\max\{\mu_\X(a):a\in X\}$. Set 
 \[X_{\max}=\{a\in X:\mu_{\X}(a)=\gamma_{\max}\}.\]
 Then $X_{\max}$ is a permissible center of blowup for the triple $\X$.
 
 \item \emph{Blowing up the top locus makes $\mu$ decrease.}
 
 Let $\X=(W,X,E)\in\W$ be a non-resolved triple. Let $\pi:W'\to W$ be the blowup of $W$ with center $X_{\max}$. Let $\X'=(W',X',E')\in\W$ be the triple which is induced by the blowup $\pi$. Let $a\in X_{\max}$ and $a'\in \pi^{-1}(a)\cap X'$. Then $\mu_{\X'}(a')<\mu_{\X}(a)$.
 
 Since the blowup $\pi$ is an isomorphism outside the center, for all points $a\notin\Xmax$ and $a'\in\pi^{-1}(a)$ the equality $\mu_{\X'}(a')=\mu_{\X}(a)$ holds by property (1). Thus,
 \[\max\{\mu_{\X'}(a'):a'\in X'\}<\max\{\mu_\X(a):a\in X\}.\]
\end{enumerate}
\end{definition}

With this definition of a resolution invariant, the following proposition is straightforward to prove:

\begin{proposition}
 Let $n>0$ be a positive integer and $K$ an algebraically closed field. If a resolution invariant $\mu$ for $n$-dimensional ambient varieties over $K$ exists, then each singular variety $X$ which is embedded in an $n$-dimensional regular variety $W$ over $K$ has an embedded resolution of singularities that is given by repeatedly blowing up the top locus of $\mu$.
\end{proposition}


\section{Results of the thesis} \label{section_content_of_thesis}

 As mentioned before, the resolution of surfaces over fields of arbitrary characteristic is not a new result. Apart from the ones already mentioned, there are several other noteworthy proofs. Lipman was able to prove (non-embedded) resolution for all excellent surfaces in \cite{Lipman_Surfaces} by using normalization and point-blowups. Hironaka developed a particularly short and simple proof for the resolution of excellent surfaces embedded in a regular $3$-dimensional ambient space in \cite{Hironaka_Bowdoin}. His resolution invariant is defined via the Newton polygon of the defining equation. The most general result on surfaces is due to Cossart, Jannsen and Saito \cite{CJS}, who proved canonical resolution of singularities with boundaries for all excellent surfaces. Their result implies both the embedded and the non-embedded case.
 
 All of these proofs at some point employ techniques which essentially make use of the low dimension. It was impossible so far to generalize any of the proofs in such a way that they can be applied to higher-dimensional varieties. On the other hand, Hironaka's proof for embedded resolution of singularities in characteristic zero has so far resisted all attempts to generalize it to the setting of positive characteristic. The main problem that prevents this is the absence of hypersurfaces of maximal contact over fields of positive characteristic.
 
 Over fields of characteristic zero, hypersurfaces of maximal contact are used to apply induction over the dimension of the ambient variety when proving the embedded resolution of singularities. In other words, they provide a way of (locally) reducing the resolution problem in an $n$-dimensional ambient space to a resolution problem in an $(n-1)$-dimensional ambient space. This technique is often called \emph{descent in dimension}. The specific construction that will be used for the descent in dimension in this thesis is the \emph{coefficient ideal} which will be introduced in Section \ref{section_coeff_ideals}. 
 
 Due to the absence of hypersurfaces of maximal contact, descent in dimension via restriction to a regular hypersurface is usually not used in proofs of resolution over fields of arbitrary characteristic. It has been suggested though (\cite{Abhyankar_Plane_Curves}, \cite{EH}, \cite{Ha_BAMS_2}) to perform the descent in dimension independently of characteristic in the following way: Instead of only using hypersurfaces of maximal contact, the set of all regular local (or formal) hypersurfaces is considered. We then assign to each hypersurface an invariant which measures how well this hypersurface is suited for the descent in dimension, that is to say, how well the singularity of the variety is preserved when restricting it to this hypersurface via the coefficient ideal construction. The hypersurface which maximizes this invariant is then used to define the resolution invariant. Over fields of characteristic zero, this maximum is realized by hypersurfaces of maximal contact. Hence, this technique can be seen as a generalization of maximal contact.
 
 There are several serious problems which appear when using this approach to generalize the usual resolution invariant from characteristic zero to the setting of positive characteristic. The resulting invariant is not upper semicontinuous and can increase under blowup, even if the center consists only of a closed point (\cite{Moh}, \cite{Ha_BAMS_2}). This and related problems will be discussed in detail in Chapter \ref{chapter_pathologies}.
 
 The main goal of this thesis is to show how these problems can be overcome under the restriction that $X$ is a surface which is embedded in a regular $3$-dimensional ambient space. Instead of considering at a point the set of all regular formal hypersurfaces, we will consider the set of all formal flags $\F$, consisting of a regular curve $\F_1$ and a regular surface $\F_2$ that contains $\F_1$. The invariant that is assigned to each flag $\F$ to evaluate how well it is suited for the descent in dimension makes use of the geometric configuration of the union $\F_1\cup(\F_2\cap E)$ where $E$ denotes the exceptional locus produced by previous blowups. The details of this construction will be described in Section \ref{section_modifying_the_residual_order}. The rigorous definition of the resolution invariant $\ivX$ will be given in Chapter \ref{chapter_invariant}. 
 
 Many of the results that we will prove along the way are not specific to the surface case and hold true independently of dimension and characteristic. Most importantly, this includes the \emph{cleaning} techniques that we will develop in Chapter \ref{chapter_cleaning}. These can be understood as a characteristic-free generalization of both the Tschirnhausen transformation that is used over fields of characteristic zero to construct hypersurfaces of maximal contact and the often-studied cleaning of purely inseparable equations of the form
 \[z^{p^e}+F(x_1,\ldots,x_n)=0\]
 over a field of characteristic $p>0$ where all $p^e$-th powers in the expansion of $F$ can be eliminated by a coordinate change $z\mapsto z+g(x_1,\ldots,x_n)$. The purpose of these cleaning techniques is to construct hypersurfaces and flags which maximize the invariants that are associated to them via the coefficient ideal construction. Hence, these techniques will form an integral part of our proof.
 
 Unfortunately, it is not clear how to generalize the resolution invariant that is presented in this thesis to higher-dimensional varieties. It is possible that our construction of the invariant via flags is ultimately restricted to the surface case or would have to undergo a major modification before it can be applied to higher dimensions. Still, we hope that the techniques that are developed in this thesis will be useful for researchers attempting to prove the resolution of singularities in higher dimensions and that the idea behind the construction of the presented invariant will inspire new resolution invariants that make use of the descent in dimension via formal flags and the coefficient ideal construction.
 
 As a final note it has to be mentioned that the proof presented in this thesis is neither the shortest, easiest, nor the most elegant proof for the resolution of surface singularities. Some of the involved results are quite technical and have lengthy proofs. The purpose of this thesis was not to develop the best or the most general proof for the resolution of surface singularities, but to seriously put the aforementioned idea of using local flags which maximize certain associated invariants to the test by using it to prove resolution of surfaces in arbitrary characteristic.

\section{Structure of the thesis} \label{section_structure_of_thesis}

In Chapter 2, several basic techniques and methods for the resolution of singularities are introduced. Most of these will be well-known to the expert reader. An object which is not commonly found in proofs of resolution of singularities are the multi-valued weighted order functions which are introduced in Section \ref{section_weighted_order}. The definition of the coefficient ideal in Section \ref{section_coeff_ideals} and the ensuing discussion of its properties are fundamental to all of the latter chapters.

Chapter 3 gives an exposition of the known pathologies which appear when trying to generalize the usual resolution invariant from characteristic zero to the setting of positive characteristic. It is then explained how these problems can be overcome in the surface case. The definition of the flag invariant $\inv(\F)$ that is used to determine which flags are used for the descent in dimension is described in detail and motivated by various examples. Further, the problem of upper semicontinuity is discussed.

Chapter 4 and 5 concern themselves with invariants associated to coefficient ideals and their behavior under coordinate changes. The results of these chapters are independent of both dimension and characteristic.

In Chapter 4, we will show that under certain conditions, the associated invariants of coefficient ideals are unchanged under all coordinate changes that preserve the geometric objects which are involved in their definition. This is then applied to show that the flag invariant $\inv(\F)$ is well-defined and does not depend on a particular choice of a regular system of parameters as long as the parameters are chosen subordinately to the flag $\F$.

In Chapter 5, it is described how the associated invariants can be maximized by changing the regular hypersurface with respect to which the coefficient ideal is defined. The techniques developed in this chapter will be referred to as cleaning techniques. They will be applied to construct flags which maximize the flag invariant.

Chapter 6 contains a number of rather technical results which are needed for our proof of resolution of surfaces. Most of these involve invariants associated to coefficient ideals and the cleaning techniques that were introduced before.

In the remaining three chapters, the new proof for the embedded resolution of surface singularities embedded in a $3$-dimensional ambient space is presented.

In Chapter 7, the resolution invariant $\iv_\X$ is defined. The biggest technical difficulty to overcome here is to show that there always exists a flag which maximizes the flag invariant $\inv(\F)$. It is also shown that $\iv_\X$ is a local geometric invariant and that if it attains its minimal value only at points at which $\X$ is locally resolved.

Chapter 8 and 9 form the core of the proof. In Chapter 8 it is proved that the invariant $\iv_\X$ is upper semicontinuous and that its top locus always constitutes a permissible center of blowup. Finally, in Chapter 9 it is shown that the invariant $\iv_\X$ strictly decreases whenever its top locus is blown up. Thus, $\iv_\X$ fulfills all properties that we required of a resolution invariant in Section \ref{section_resolution_via_usc_inv}.

\section{Acknowledgements}

First and foremost, I want to thank my advisor Herwig Hauser, without whom this thesis would have never seen the light of day. He initially awoke my interest in the fascinating and beautiful problem that is the resolution of singularities. Since then, he has helped me to deepen my understanding and gradually develop my own view on this intricate subject in countless private seminars, exchanges of e-mails and enthusiastic discussions both at the university and in coffee shops all over Vienna. He has also shown unending patience and good will in supporting me throughout this long project. During the many times that I got stuck and felt like this thesis would never reach completion, he has always been a great source of encouragement. 

Another big thanks goes out to my second mentor and good friend Hiraku Kawanoue. In numerous discussions both in Vienna and in Kyoto, he opened my eyes for many of the fine subtleties of the resolution problem. He was always eager to hear about any progress I made in my work and to discuss all promising new ideas down to the last detail. It was in discussions with him that the proof eventually assumed its final form.

I am also indebted to several other great mathematicians whom I met and who were so kind to share their insights on the resolution problem with me. In particular, I want to thank Orlando Villamayor, Santiago Encinas, Josef Schicho, Bernd Schober, Ana Bravo, Dale Cutkosky, Daniel Panazzolo, Heisuke Hironaka, Anne Frühbis-Krüger, Edward Bierstone and Kenji Matsuki for everything they taught me.

I am also grateful to have had a number of fantastic office colleagues during my time as a PhD student who were always willing to lend a sympathetic ear. Thank you, Markus, Christopher, Alberto and Hana, for the many helpful discussions and all the fun we had. 

\section{Notation, conventions and prerequisites} \label{section_conventions}

 In this section we will list all notational conventions that will be used throughout this thesis. At the end of the section, we will further provide a list of concepts and results that the reader is assumed to be familiar with, along with references.\\

\emph{Geometric objects}

\begin{itemize}
 \item The letter $K$ will always denote a field. Unless stated otherwise, $K$ is always assumed to be algebraically closed and of arbitrary characteristic.
 \item A \emph{variety} is an integral separated scheme of finite type over an algebraically closed field $K$.
 \item Whenever we consider a closed subset $X$ of a scheme $W$, we consider it to be endowed with the \emph{reduced induced closed subscheme structure}. (See \cite{Hartshorne} Ex. 3.2.6.) The induced ideal sheaf of $X$ will be denoted by $\IX$.
 \item In this sense, a \emph{subvariety} $X$ of a variety $W$ is an irreducible closed subset of $W$.
 \item A \emph{curve} is an irreducible closed subset of dimension $1$. Similarly, a \emph{hypersurface} is an irreducible closed subset of codimension $1$.
 \item Consider a variety $W$ and the blowup $\pi:W'\to W$ along a closed subset $Z$. We will call $\pi^{-1}(Z)$ the \emph{exceptional divisor} of the blowup $\pi$. For a closed subset $X\subseteq W$, we call $X^*=\pi^{-1}(X)$ the \emph{total transform} and $X'=\ol{\pi^{-1}(X\setminus Z)}$ the \emph{strict transform} of $X$. The \emph{weak transform} will be defined in Section \ref{section_order_function}.
\end{itemize}

\emph{Rings and ideals}

\begin{itemize}
 \item All rings that we consider will be commutative and unitary.
 \item For a $K$-algebra $R$, we denote by $\Omega_{R/K}$ its \emph{module of relative differentials}.
 \item For elements $x_1,\ldots,x_n$ of a ring $R$, we denote the ideal generated by these elements by $(x_1,\ldots,x_n)$.
 \item Often, the expression $(x_1,\ldots,x_n)$ will instead denote the ordered tuple in $R^n$. Whenever there is a possible risk of confusion, we specify in the text which of the two is meant.
 \item For an ideal $I$, we denote by $\rad(I)$ its radical.
\end{itemize}

\emph{Local and formal objects}

\begin{itemize}
 \item Let $W$ be a variety and $a\in W$ a (not necessarily closed) point. We denote by $\OWa$ the local ring of $W$ at $a$.
 \item We denote by $\mWa$ the maximal ideal of $\OWa$. For an affine open neighborhood $U=\Spec(R)$ of $a$, we denote the prime ideal of $R$ that corresponds to the point $a$ also by $\mWa$.
 \item We denote by $\hOWa$ the completion of the local ring $\OWa$ with respect to its maximal ideal.
 \item For a variety $W$, a closed subset $X\subseteq W$ and a point $a\in W$ we denote by $\IXa$ the stalk of the ideal sheaf $\IX$ at $a$. We set $\hIXa=\hOWa\cdot \IXa$.
 \item Curves and hypersurfaces in $\Spec(\OWa)$ will be referred to as \emph{local} curves and hypersurfaces.
 \item Similarly, curves and hypersurfaces in $\Spec(\hOWa)$ will be referred to as \emph{formal} curves and hypersurfaces.
\end{itemize}

\emph{Multi-indices}

\begin{itemize}
 \item The symbol $\N$ denotes the non-negative integers. We set $\Ni=\N\cup\{\infty\}$.
 \item For a multi-index $\alpha=(\alpha_1,\ldots,\alpha_n)\in\N^n$, we denote $|\alpha|=\sum_{i=1}^n\alpha_i$.
 \item Let $\alpha,\beta$ be multi-indices in $\N^n$ or $\Ni^n$. The expressions $\alpha<\beta$ and $\alpha\leq\beta$ are always understood with respect to the lexicographic order.
 \item For multi-indices $\alpha=(\alpha_1,\ldots,\alpha_n)$ and $\beta=(\beta_1,\ldots,\beta_n)$, we denote $\binom{\alpha}{\beta}=\prod_{i=1}^n\binom{\alpha_i}{\beta_i}$.
 \item An ordered tuple of elements $(x_1,\ldots,x_n)$ of a ring will usually be denoted by $\x$. For a multi-index $\alpha\in\N^n$, we denote $\x^\alpha=\prod_{i=1}^n x_i^{\alpha_i}$.
\end{itemize}

\emph{Prerequisites}

\begin{itemize}
 \item Particular concepts of commutative algebra that will be used throughout the thesis are completions of local rings, regular local rings, regular systems of parameters, the Cohen structure theorem for regular local rings which contain a field and the module of relative differentials. All of these can be found in \cite{Eisenbud_Com_Alg} and \cite{Hartshorne}.
 \item Concepts of algebraic geometry that will be used extensively are blowups and regularity. Introductions to blowups can be found in \cite{Hartshorne} II.7 and \cite{Ha_Obergurgl_2014} Lecture IV and V. Regularity (non-singularity) and its relation with the module of differentials is explained in \cite{Hartshorne} II.8. 
 \item Results about the power series ring $R=K[[x_1,\ldots,x_n]]$ in $n$ variables over a field $K$ that will be assumed are the Weierstrass preparation theorem and the inverse function theorem which characterizes all regular systems of parameters for $R$. 
 \item To be able to understand the proof presented in this thesis, it is not necessary for the reader to have a profound knowledge of any of the proofs for embedded resolution of singularities over fields of characteristic zero. It has to be mentioned though that an understanding of one of these proofs (in particular, how hypersurfaces of maximal contact are used) is certainly useful to understand how the techniques presented in this thesis relate to the known results in characteristic zero.
\end{itemize}

\chapter{Basic concepts and techniques for the resolution of singularities} \label{chapter_basics}

\section{The order function} \label{section_order_function}

 Since we are following the approach of proving resolution of singularities via an upper semicontinuous invariant, it is natural to look for geometric invariants which measure as precisely as possible how complex the singularity of a variety $X$ in a point $a\in X$ is. The most basic such invariant is the order function. It generalizes the idea of measuring up to which degree a polynomial's Taylor expansion vanishes at a closed point. The order function is so fundamental to the embedded resolution of singularities that it appears in virtually all proofs in some form.
 
 In this section we are going to define the order function and state several of its most important properties. In particular, we will see that the order function itself is not sufficiently fine to constitute a suitable resolution invariant, but it is natural to look for resolution invariants which are refinements the order function. Since all stated results are well-known, we are going to give references for most proofs.

\begin{definition}
 Let $R$ be a ring, $J\subseteq R$ an ideal and $p\subseteq R$ a prime ideal. We define the \emph{order of $J$ at $p$} as
 \[\ord_pJ=\sup\{n\in\N: R_pJ\subseteq R_pp^n\}\in\N_\infty.\]
 If $R$ is a local ring with maximal ideal $m$, we will usually just write $\ord J$ for $\ord_m J$. For an element $f\in R$ we will denote by $\ord_pf$ the order of the principal ideal $(f)$ at $p$.
 
 Let $W$ be a scheme, $X\subseteq W$ a closed subset and $a\in W$ a (not necessarily closed) point. Then we define the \emph{order of $X$ at $a$} as the order of the ideal $I_{X,a}$ at the maximal ideal of the local ring $\OO_{W,a}$. We denote it by $\ord_aX$. The ambient scheme $W$ is suppressed in this notation.
\end{definition}

\begin{examples}
 \begin{enumerate}[(1)]
  \item Consider the $n$-dimensional affine space $W=\Spec(K[x_1,\ldots,x_n])$ over a field $K$ and a hypersurface $X=V(f)$ defined by a polynomial $f\in K[x_1,\ldots,x_n]$. Then the order $\ord_aX$ at a closed point $a\in X$ coincides with the highest degree up to which the Taylor expansion of $f$ vanishes at $a$.
 
 Notice that $\ord_aX\geq 2$ holds if and only if both $f$ and all of its partial derivatives $\frac{\partial f}{\partial x_i}$ vanish at $a$. Thus, $a$ is a singular point of $X$ if and only if $\ord_aX\geq2$ holds. We will see later that this is false in general for subvarieties of codimension bigger than $1$.
 \item Let $c\in\mathbb Z$ be an integer $c\neq0$ and $p$ a prime number. Then $\ord_pc$ is the biggest natural number $n\in\N$ such that $p^n$ divides $c$.
 \end{enumerate}
\end{examples}

 The order function exhibits a number of good properties which make it useful for the embedded resolution of singularities. Most importantly, it is upper semicontinuous (Proposition \ref{order_is_usc}) and it does not increase under blowup as long as the center is regular and contained in its top locus (Proposition \ref{ord_under_blowup}). On the downside, the top locus defined by the order function is in general not regular and thus, does not constitute a permissible center of blowup. Also, even if it is regular, blowing up the top locus does not necessarily make the order function decrease. Hence, the order function itself does not satisfy the properties that we required for a resolution invariant in Section \ref{section_resolution_via_usc_inv}. In other words, the order function's measure of singularity is not sufficiently precise for our purposes.
 
 This suggests to consider as a resolution invariant a suitable \emph{refinement} of the order function. That is, a map $\mu_X:X\to \N_\infty\times\Gamma$ of the form 
 \[\mu_X(a)=(\ord_a X,\wt\mu_X(a))\]
 for some well-ordered set $(\Gamma,\leq)$ where $\N_\infty\times\Gamma$ is considered with the induced lexicographic order. Notice that the second component $\wt\mu_X$ of the invariant $\mu_X$ is only needed to measure improvement of a singularity under blowup if the order function remains constant.
 

\begin{proposition} \label{order_basic_properties}
 Let $W$ be a regular variety and $X\subseteq W$ a closed proper subset. Let $a\in W$ be any point.
 \begin{enumerate}[(1)]
  \item $\ord_aX<\infty$.
  \item If $\ord_aX\geq2$, then $a$ is a singular point of $X$.
  \item If $X$ is a hypersurface and $\ord_aX=1$, then $a$ is a regular point of $X$.
 \end{enumerate}
\end{proposition}
\begin{proof}
 (1): Assume that $\ord_aX=\infty$. This implies that $I_{X,a}\subseteq \bigcap_{n=1}^\infty m_{W,a}^n$. By Krull's intersection theorem, this implies that $I_{X,a}=0$ which is a contradiction to $X$ being a proper subset of $W$.
 
 (2): Assume that $a$ is a regular point of $X$. Then there exists a regular system of parameters $x_1,\ldots,x_n$ for $\OWa$ such that $\IXa=(x_1,\ldots,x_k)$ for some integer $1\leq k\leq n$. Since $x_i\in \mWa\setminus\mWa^2$, it is clear that $\IXa\not\subseteq\mWa^2$. Thus, $\ord_aX=\ord\IXa=1$.
 
 (3): Since $\OWa$ is a regular local ring and the ideal $\IXa$ has height $1$ by assumption, we know that it is generated by a single element $x_1$. Since $\ord_aX=1$, it is clear that $x_1\in\mWa\setminus\mWa^2$. Thus, there are elements $x_2,\ldots,x_n\in\mWa$ where $n=\dim(\OWa)$ such that $x_1,\ldots,x_n$ is a regular system of parameters for $\OWa$. Consequently, $\OXa=\OWa/\IXa$ is a regular local ring.
\end{proof}

\begin{remark}
 The assertion of Proposition \ref{order_basic_properties} (3) is false if the local codimension of $X$ is bigger than $1$. For example, consider the curve $X=V(y^2-x^3,z)$ embedded in the $3$-dimensional affine space $W=\Spec(K[x,y,z])$. Let $a$ be the origin of $W$. Then $\ord_aX=1$, although $a$ is a singular point of $X$. This example also makes clear that the order of $X$ greatly depends on the ambient space $W$.
 
 This is one of the reasons why many authors consider instead of the order function more sophisticated invariants to measure the complexity of the singularity of $X$ at a point, such as the multiplicity or the Hilbert-Samuel function (Cf. \cite{Bennett}, \cite{Villamayor_Constructiveness}, \cite{BM_Canonical_Desing} for the Hilbert-Samuel Function and \cite{BV_Multiplicity} for multiplicity). Both of these are similar to the order function in the sense that they have to be further refined to obtain a suitable resolution invariant.
\end{remark}

\begin{proposition} \label{order_is_usc}
 Let $W$ be a regular variety and $X\subseteq W$ a closed subset. Let $c\geq0$ be a non-negative integer. Then the set 
 \[X_{\geq c}=\{a\in X:\ord_aX\geq c\}\]
 is closed. In other words, the order function $\ord X:X\to\N_\infty$ is upper semicontinuous.
\end{proposition}
\begin{proof}
The classical proof of this theorem was given in \cite{Hironaka_Annals} Corollary 1, p. 220. A modern proof using differential operators can be found in \cite{Kawanoue_IF_1} Lemma 1.2.3.1. 
\end{proof}

\begin{corollary} \label{max_ord_exists}
 Let $W$ be a regular variety and $X\subseteq W$ a proper closed subset. There exists an integer $c\in\N$ such that 
 \[c=\max\{\ord_aX:a\in X\}.\]
\end{corollary}

\begin{definition}
 Let $W$ be a regular variety and $X\subseteq W$ a closed subset. Set $c=\max_{a\in X}\ord_aX$. Then the set
 \[\topp(X)=\{a\in X:\ord_aX=c\}\]
 is called the \emph{top locus} of $X$ with respect to the order function (or only top locus). Notice that $\topp(X)$ is closed by Proposition \ref{order_is_usc}.
\end{definition}

 In Section \ref{section_diff_ops} we will see a method how to construct the defining of ideal of $\topp(X)$ from the defining ideal of $X$ via differential operators.

\begin{definition}
 Let $W$ be a regular variety and $X\subseteq W$ a closed subset. A subvariety $Z\subseteq W$ is said to be a \emph{permissible center of blowup} for $X$ with respect to the order function if $Z$ is regular and $Z\subseteq\topp(X)$ holds. A blowup $\pi:W'\to W$ along such a center will be called \emph{order-permissible}.
\end{definition}

Consider now a regular variety $W$ and a closed subset $X\subseteq W$. Let $\pi:W'\to W$ be the blowup of $W$ along a center $Z$ that is permissible with respect to the order function. Denote by $E=\pi^{-1}(Z)$ the exceptional divisor of $\pi$. Let $a\in\topp(X)$ be a closed point and $a'\in\pi^{-1}(a)$ a closed point lying over $a$. Set $c=\max_{a\in X}\ord_aX$. 

Denote by $\vp:\OWa\to\OWap$ the induced map between the local rings. Since the exceptional divisor $E$ is a hypersurface on $W'$, the ideal $I_{E,a'}$ is generated by a single element $x_E\in\hOWap$. The local defining ideals of the total transform $X^*$ and the strict transform $X'$ of $X$ at $a'$ are given by
\[I_{X^*,a'}=(\IXa)^*:=\OWap\cdot \vp(\IXa),\]
respectively
\[I_{X',a'}=(\IXa)\st:=\Big(x_E^{-\ord_{I_{Z,a}}(f)}\cdot\vp(f):f\in\IXa\Big).\]
Since the ideal $(\IXa)\st$ is quite difficult to handle, a simpler ideal is often considered instead which is called the \emph{weak transform} of $\IXa$. It is defined as
\[(I_{X,a})\wk:=x_E^{-c}\cdot(\IXa)^*.\]
The weak transform of an ideal $J$ will also often just be denoted by $J'$. Notice that strict transform and weak transform coincide in the case that $X$ is a hypersurface. 

\begin{proposition} \label{ord_under_blowup}
 For an order-permissible blowup, the following inequalities hold:
 \[\ord (\IXa)\st\leq \ord (\IXa)\wk\leq \ord\IXa.\]
 In particular,
 \[\ord_{a'}X'\leq\ord_aX.\]
\end{proposition}
\begin{proof}
 This is shown in \cite{Ha_Obergurgl_2014} Prop. 8.13.
\end{proof}

\begin{definition}
 A closed point $a'\in\pi^{-1}(a)$ that fulfills $\ord_{a'}X'=\ord_aX$ is called an \emph{equiconstant point} over $a$.
\end{definition}

 Since the order function fails to measure improvement of the singularity under blowup at equiconstant points, we will need to develop techniques beyond the order function which enable us to measure improvement at these points. In Section \ref{section_directrix} we will take a first step towards this goal by narrowing down the locus on the pre-image $\pi^{-1}(a)$ inside which equiconstant points can appear. In Section \ref{section_coeff_ideals} we will then introduce a technique to measure improvement at equiconstant points by considering an ideal which is defined on an ambient space of lower dimension than $W$, called the coefficient ideal.
 
 In the remainder of this section we will state a result which ensures that the order along regular subvarieties can be computed without passing to the localization and use this to verify that the order along regular subvarieties is preserved when passing to the completion.

\begin{lemma}
 Let $R$ be a ring, $J\subseteq R$ an ideal and $p\subseteq R$ a prime ideal. Then
 \[\ord_pJ=\sup\{n\in\N: J\subseteq p^{(n)}\}\]
 where $p^{(n)}=R_pp^n\cap R$ denotes the $n$-th symbolic power of $p$.
\end{lemma}
\begin{proof}
 This is immediate from the definition of $\ord_pJ$.
\end{proof}

\begin{proposition} \label{hochster}
 Let $R$ be a Noetherian domain and $p\subseteq R$ a prime ideal that is generated by a regular sequence. Then $p^n=p^{(n)}$ for all $n\geq0$.
\end{proposition}
\begin{proof}
 This is proved in \cite{Hochster_73}, Application (2.1), p. 57.
\end{proof}

\begin{corollary} \label{cor_symb_powers_of_reg_ideals}
 Let $R$ be a regular local ring and $p\subseteq R$ a prime ideal such that $R/p$ is regular. Then for each ideal $J\subseteq R$ the following equality holds:
 \[\ord_pJ=\sup\{n\in\N: J\subseteq p^n\}.\]
\end{corollary}

%
%

\begin{lemma} \label{completion_preserves_order}
 Let $(R,m)$ be a regular local ring and $J\subseteq R$ an ideal. Denote by $\wh R$ the $m$-adic completion of $R$ and by $\wh J=\wh RJ$ the extension of $J$ in $\wh R$.
 
 Let $p\subseteq R$ be a prime ideal such that $R/p$ is regular. Set $\wh p=\wh Rp$. Then $\ord_pJ=\ord_{\wh p}\wh J$.
\end{lemma}
\begin{proof}
 Let $x_1,\ldots,x_m$ be a regular system of parameters for $R$ such that $p=(x_1,\ldots,x_k)$ for an index $k\leq p$.
 
 Let $n\in\N$. By Corollary \ref{cor_symb_powers_of_reg_ideals} we know that $\ord_pJ\geq n$ holds if and only if $J\subseteq (x_1,\ldots,x_k)^n$ holds. Clearly, this is equivalent to $\wh J\subseteq (x_1,\ldots,x_k)^n$. Hence, $\ord_pJ=\ord_{\wh p}\wh J$.
\end{proof}

\section{Multi-valued weighted order functions} \label{section_weighted_order}

 In this section we are going to introduce a generalization of the order function as a map from the power series ring $K[[\x]]$ to $\Ni$. Here, $K$ denotes an arbitrary field. 

 Consider the power series ring $R=K[[\x]]$ with parameters $\x=(x_1,\ldots,x_n)$. Denote for each element $f\in R$ its expansion by $f=\sum_{\alpha\in\N}c_{f,\alpha}\x^\alpha$ with $c_{f,\alpha}\in K$. Then the order function $\ord:K[[\x]]\to\N_\infty$ can be described as
 \[\ord f=\min\B\{\sum_{i=1}^n\alpha_i:\alpha\in\N^n,c_{f,\alpha}\neq0\B\}\]
 for $f\neq0$. By assigning multi-valued weights $c_1,\ldots,c_n\in\N^k$ to the variables $x_1,\ldots,x_n$, we can define more general maps $\w:K[[\x]]\to\N^k_\infty$ which are defined by 
 \[\w(f)=\min\B\{\sum_{i=1}^n\alpha_i\cdot c_i:\alpha\in\N^n,c_{f,\alpha}\neq0\B\}\]
 for $f\neq0$ where $\N^k_\infty$ is considered with the lexicographic order. We set $\w(0)=(\infty,\ldots,\infty)$.
 
 These maps are called (multi-valued) weighted order functions. Notice that, unlike the usual order function, the definition of $\w$ is dependent on the chosen parameter system for $K[[\x]]$. While this seems like a disadvantage at first, assigning weights to the variables offers a great deal of flexibility for defining new invariants that behave similarly to the order function. They will be used in situations in which some of the parameters of $K[[\x]]$ have a distinguished role. (For example, if these parameters define a geometric object with respect to which the invariant is defined.)
 
 The usage of weighted order functions for the resolution of singularities is not common. Their application to coefficient ideals is a unique feature of this thesis that plays a key role for the definition of our resolution invariant for surface singularities. It is still unexplored how weighted order functions could be used to prove resolution of higher-dimensional varieties.

\begin{definition}
 Let $R=K[[\x]]$ with $\x=(x_1,\ldots,x_n)$. A map $\w:K[[\x]]\to\N^k_\infty$ is called a \emph{weighted order function} that is defined on the parameters $\x$ if the following hold:
 \begin{itemize}
 \item Let $f\in R$ be a non-zero element with expansion $f=\sum_{\alpha\in\N}c_{\alpha}\x^\alpha$. Then
 \[\w(f)=\min\B\{\sum_{i=1}^n\alpha_i\cdot\w(x_i):\alpha\in\N^n,c_{\alpha}\neq0\B\}\]
 where $\N_\infty^k$ is considered with the lexicographic order.
 \item $\w(0)=(\infty,\ldots,\infty)$.
 \end{itemize}
 For an ideal $J\subseteq R$ we define 
 \[\w(J)=\min\{\w(f):f\in J\}.\]
 
 If it is clear from the context with respect to which parameters $\w$ is defined, we will just call $\w$ a weighted order function on $K[[\x]]$.
\end{definition}

\begin{remarks}
\begin{enumerate}[(1)]
 \item It is easy to see that a weighted order function $\w:K[[\x]]\to\N^k_\infty$ is a valuation. Hence, for elements $f,g\in K[[\x]]$ the following hold:
 \begin{itemize}
  \item $\w(f+g)\geq \min\{\w(f),\w(g)\}$ and equality holds if $\w(f)\neq\w(g)$.
  \item $\w(f\cdot g)=\w(f)+\w(g)$.
  \item $\w(f)=(\infty,\ldots,\infty) \iff f=0$.
 \end{itemize}
 \item A weighted order function $\w:K[[\x]]\to\N_\infty^k$ that is defined on the parameters $\x=(x_1,\ldots,x_n)$ is uniquely prescribed by the images $\w(x_1),\ldots,\w(x_n)$.
 \item Notice that $\w(u)=(0,\ldots,0)$ for all units $u\in K[[\x]]^*$.
 \item If an ideal $J$ is generated by elements $f_1,\ldots,f_m$, it is clear that $\w(J)=\min\{\w(f_1),\ldots,\w(f_m)\}$.
 \item Some weighted order functions can also be described as orders at certain prime ideals $p\in\Spec(K[[\x]])$. For example, the weighted order function $\w:K[[\x]]\to\N_\infty$ which is defined by $\w(x_i)=1$ for a certain index $i$ and $\w(x_j)=0$ for $j\neq i$ coincides with the order at the prime ideal $(x_i)$. This specific weighted order is called the \emph{$x_i$-order} and will be denoted by $\ord_{(x_i)}$.
 
 Most weighted order functions cannot be defined as the order at a prime ideal. For example, consider $R=K[[x,y]]$ and the weighted order function $\w:K[[x,y]]\to\N_\infty$ defined by $\w(x)=1$ and $\w(y)=2$. Then there is no prime ideal $p\in\Spec(K[[x,y]])$ such that the maps $\w$ and $\ord_p$ coincide.

\end{enumerate}

\end{remarks}

\begin{definition}
 Let $R=K[[\x]]$ and let $\w:K[[\x]]\to\N_\infty^k$ be a weighted order function that is defined on $\x$. For an element $f\neq0$ with expansion $f=\sum_{\alpha\in\N^n}c_\alpha\x^\alpha$ we define the \emph{initial form of $f$} with respect to $\w$ and $\x$ as
 \[\init_\w(f)=\sum_{\substack{\alpha\in\N^n\\ \w(\x^\alpha)=\w(f)}}c_\alpha \x^\alpha.\]
 Set $\init_\w(0)=0$. The dependence of the initial form on $\x$ will be suppressed in the notation. We say that $f$ is \emph{weighted homogeneous} with respect to $\w$ and $\x$ if $\init_\w(f)=f$ holds.
 
 For an ideal $J\subseteq R$ we define the \emph{weak initial ideal of $J$} with respect to $\w$ and $\x$ as
 \[\minit_\w(J)=(\init_\w(f):f\in J,\w(f)=\w(J)).\]
 
 We will write $\init(f)$ and $\minit(J)$ for the initial form and the weak initial ideal with respect to the usual order function. Notice that, although the order function itself is coordinate-independent, $\init(f)$ and $\minit(J)$ are dependent on the choice of parameters $\x$ for $R$.
 
 We will write $\init_{(x_i)}(f)$ and $\minit_{(x_i)}(J)$ for the initial form and weak initial ideal with respect to the $x_i$-order.
\end{definition}

\begin{remark}
 The definition of the weak initial ideal is non-standard. A similar object which appears way more often in the literature is the (strict) initial ideal of $J$ which is defined as
 \[\init(J)=(\init(f):f\in J).\]
 This is generally a much more complicated ideal than the weak initial ideal of $J$. For example, if $J$ is generated by elements $f_1,\ldots,f_k$, then the weak initial ideal of $J$ is generated by the initial forms $\init(f_i)$ of those elements $f_i$ that fulfill $\ord f_i=\ord J$. Finding generators for the usual initial ideal $J$ is much more complicated and leads to the theory of standard bases.
 
 For our purposes, the weak initial ideal with respect to weighted order functions is a very useful object and it will be used in the definition of our resolution invariant.
\end{remark}

\begin{lemma} \label{double_weighted order function}
 Let $R=K[[\x]]$ with $\x=(x_1,\ldots,x_n)$ and consider two weighted order functions $\w_1:K[[\x]]\to\Ni^k$ and $\w_2:K[[\x]]\to\Ni^l$ which are defined on the parameters $\x$. Let $\y:K[[\x]]\to\N_\infty^{k+l}$ be the weighted order function that is defined on $\x$ by setting $\y(x_i)=(\w_1(x_i),\w_2(x_i))$ for $i=1,\ldots,n$.
 
 Then 
 \[\y(f)=(\w_1(f),\w_2(\init_{\w_1}(f)))\]
 for elements $f\in R$ and 
 \[\y(J)=(\w_1(J),\w_2(\minit_{\w_1}(J)))\]
 for ideals $J\subseteq R$.
\end{lemma}
\begin{proof}
 Let $f\in R$ be an element with expansion $f=\sum_{\alpha\in\N^n}c_\alpha\x^\alpha$. Then
 \[\y(f)=\min\B\{\B(\sum_{i=1}^n\alpha_i\cdot\w_1(x_i),\sum_{i=1}^n\alpha_i\cdot\w_2(x_i)\B):\alpha\in\N^n,c_\alpha\neq0\B\}\]
 \[=\min\B\{\B(\w_1(f),\sum_{i=1}^n\alpha_i\cdot\w_2(x_i)\B):\alpha\in\N^n,c_\alpha\neq0,\w_1(\x^\alpha)=\w_1(f)\B\}\]
 \[=(\w_1,\w_2(\init_{\w_1}(f))).\]
 
 Further, let $J\subseteq R$ be an ideal. Then
 \[\y(J)=\min\{(\w_1(f),\w_2(\init_{\w_1}(f))):f\in J\}\]
 \[=\min\{(\w_1(J),\w_2(\init_{\w_1}(f)):f\in J,\w_1(f)=\w_1(J)\}\]
 \[=(\w_1(J),\w_2(\minit_{\w_1}(J))).\]
\end{proof}

 

\section{Differential operators} \label{section_diff_ops}

In the following, we will briefly discuss some important results about differential operators. Our main reference for the theory of differential operators is Chapter 1 of \cite{Kawanoue_IF_1} which gives a very accessible account on this topic and its application to the resolution of singularities over fields of arbitrary characteristic. Here, $K$ always denotes an algebraically closed field.


\begin{definition}
 Let $R$ be a $K$-algebra. Fix a non-negative integer $n\in\N$. Consider the multiplication map $\mu:R\otimes_KR\to R$ defined by $\mu(x\otimes y)=xy$. Set $I=\ker(\mu)$ and $D^n(R)=R\otimes_KR/I^{n+1}$. Let $d_n:R\to D^n(R)$ be the map $d_n(x)=\overline{1\otimes x}$. Set
 \[\Diff^n_{R/K}=\{\varphi\circ d_n:\varphi\in\Hom_R(D^n(R),R)\}.\]
 Notice that $\Diff^n_{R/K}$ is an $R$-module via the multiplication $a\cdot(\vp\circ d_n)=(a\cdot\vp)\circ d_n$.
 
 The elements $d\in \Diff^n_{R/K}$ are $K$-linear maps $d:R\to R$ which are called \emph{differential operators on $R$ of degree $\leq n$}. 
\end{definition}

\begin{lemma} \label{composition_of_diff_ops}
 Let $R$ be be a $K$-algebra and $n,m\in\N$. Further, let $d\in\Diff^n_{R/K}$ and $d'\in\Diff^m_{R/K}$ be differential operators. Then $d'\circ d\in\Diff^{n+m}_{R/K}$.
\end{lemma}
\begin{proof}
 \cite{Kawanoue_IF_1} Lemma 1.1.2.1. (6).
\end{proof}

\begin{lemma} \label{sheaf_of_diff_ops}
 Let $W$ be a variety over a field $K$ and $n\in\N$ a non-negative integer. Then there exists a coherent sheaf $\Diff_W^n$ of differential operators with the following properties:
 \begin{enumerate}[(1)]
  \item For each affine open subset $U=\Spec(R)$ of $W$, we have $\Diff_W^n(U)=\Diff^n_{R/K}$.
  \item For each point $a\in W$, the stalk of the sheaf $\Diff_W^n$ at $a$ has the form $\{\Diff_W^n\}_a=\Diff^n_{\OWa/K}$.
  \item For each closed point $a\in W$, the identity $\Diff^n_{\OWa/K}\otimes_{\OWa}\hOWa\cong\Diff^n_{\hOWa/K}$ holds.
 \end{enumerate}
\end{lemma}
\begin{proof}
 \cite{Kawanoue_IF_1} Corollary 1.1.2.2.
\end{proof}

 The following result is the main reason why we are interested in differential operators. It provides us with a method of computing the defining ideal of the top locus $\topp(X)$ of a closed set $X$ from the defining ideal of $X$.

\begin{proposition} \label{ord_via_diff_ops}
 Let $W$ be a regular variety and $X\subseteq W$ a closed subset. Set $X_{\geq c}=\{a\in W:\ord_aX\geq c\}$ for a positive integer $c>0$. Then
 \[X_{\geq c}=V(\Diff_W^{c-1}(\IX)).\]
\end{proposition}
\begin{proof}
 \cite{Kawanoue_IF_1} Lemma 1.2.3.1.
\end{proof}

\begin{remark}
 Notice that $\Diff_W^{c-1}(\IX)$ defines a not necessarily reduced closed subscheme structure on the set $X_{\geq c}$. While it is not immediately clear what purpose this additional structure serves, we will see one application in Section \ref{section_woah_maximal_contact} where it will provide us with a sufficient criterion for a local hypersurface to have maximal contact with $X$.
\end{remark}

In the following two propositions, we consider the following setting: Let $W$ be a regular variety over a field $K$. Let $R$ be either the coordinate ring of an affine open subset $U\subseteq W$, the local ring $\OWa$ at a closed point $a\in W$ or its completion $\hOWa$. In particular, the results hold true for the power series ring $R=K[[\x]]$.

\begin{proposition} \label{basis_for_diff_ops}
 Assume that there are elements $x_1,\ldots,x_n\in R$ with $n=\dim(R)$ such that the module of differentials $\Omega_{R/K}$ is freely generated by $dx_1,\ldots,dx_n$. Set $\x=(x_1,\ldots,x_n)$.
 
 Let $m>0$ be a positive integer. Then there are elements $\partial_{\x^\alpha}\in\Diff^m_{R/K}$ for all multi-indices $\alpha\in\N^n$ with $|\alpha|\leq m$ such that the following hold:
 \begin{enumerate}[(1)]
  \item The $R$-module $\Diff^m_{R/K}$ is freely generated by the elements $\partial_{\x^\alpha}\in\Diff^m_{R/K}$ with $\alpha\in\N^n$ and $|\alpha|\leq m$.
  \item $\partial_{\x^\alpha}(\x^\beta)=\binom{\beta}{\alpha}\x^{\beta-\alpha}$ for all $\beta\in\N^n$.
  \item $\partial_{\x^\alpha}(f\cdot g)=\sum_{\substack{\beta,\gamma\in\N^n\\ \beta+\gamma=\alpha}}\partial_{\x^\beta}(f)\partial_{\x^\gamma}(g)$ for elements $f,g\in R$.
 \end{enumerate}
\end{proposition}
\begin{proof}
 \cite{Kawanoue_IF_1} Lemma 1.2.1.2.
\end{proof}

\begin{remarks}
 \begin{enumerate}[(1)]
  \item Notice that the definition of each map $\partial_{\x^\alpha}:R\to R$ depends on the entire set of parameters $x_1,\ldots,x_n$. For example, let $R=K[[x,y]]$. Now consider another regular system of parameters $x_1,y_1$ for $R$ where $x_1=x$ and $y_1=y+x$. Although $x=x_1$, the maps $\partial_x:R\to R$ and $\partial_{x_1}:R\to R$ are different since
  \[\partial_{x}(y)=0,\]
  but
  \[\partial_{x_1}(y)=\partial_{x_1}(y_1-x_1)=-1.\]
  \item Over a field of characteristic zero, any differential operator $d\in\Diff^n_{R/K}$ can be expressed as an $R$-linear combination of compositions of derivations (\cite{Kawanoue_IF_1} Remark 1.2.1.3. (1)).
  
  Over a field of characteristic $p>0$, the differential operator $\partial_{x_i^p}:R\to R$ can not be expressed in such a way. This is is since every derivation $\partial:R\to R$ fulfills 
  \[\partial(x_i^p)=p\cdot x_i^{p-1}\partial(x_i)=0,\] but $\partial_{x_i^p}(x_i^p)=1$. Thus, the theory of differential operators is much better equipped to deal with algebras over fields of positive characteristic than the theory of derivations. 
 \end{enumerate}
\end{remarks}

The last result in this section is specific to the situation of positive characteristic and provides a criterion for checking whether an ideal is generated by $p$-th powers.

\begin{proposition} \label{p_th_powers_via_diff_ops}
 Assume that $\chara(K)=p>0$. Set $q=p^e$ for a positive integer $e>0$. Let $J\subseteq R$ be an ideal. Then the ideal $J$ is generated by $q$-th powers in $R$ if and only if $\Diff^{q-1}_{R/K}(J)=J$.
\end{proposition}
\begin{proof}
 \cite{Kawanoue_IF_1} Proposition 1.3.1.2.
\end{proof}

\section{Coordinate expressions of blowup maps} \label{section_coord_expression_of_blowups}

 In this section we will outline the useful concept of coordinate expressions of blowup maps between local rings. More background and proofs for the given statements can be found in Lecture IV of \cite{Ha_Obergurgl_2014}.

 Let $W$ be a regular $n$-dimensional variety over a field $K$ and $Z\subseteq W$ a regular subvariety. Let $a\in Z$ be a closed point and $\x=(x_1,\ldots,x_n)$ a regular system of parameters for $\OWa$ such that $I_{Z,a}=(x_1,\ldots,x_k)$. Let $\pi:W'\to W$ be the blowup of $W$ with center $Z$. Denote by $\Dnew=\pi^{-1}(Z)$ the exceptional divisor of the blowup.
 
 Let $i$ be an index in the range $1\leq i\leq k$. The set $\pi^{-1}(a)\setminus V(x_i)\st$ where $V(x_i)\st$ denotes the strict transform of $V(x_i)$ is called the \emph{$x_i$-chart} over $a$ with respect to the parameters $\x$. The $x_i$-charts for $i=1,\ldots,k$ form an open cover of the pre-image $\pi^{-1}(a)$.

 Further, there is an isomorphism between $\pi^{-1}(a)$ and $\mathbb{P}^{k-1}_K$ with the following properties:
 
 Let $a'\in\pi^{-1}(a)$ be a closed point with projective coordinates $(t_1:\ldots:t_k)$. Then $a'$ lies in the $x_i$-chart for an index $i\in\{1,\ldots,k\}$ if and only if $t_i\neq0$. The affine coordinates of $a'$ in the $x_i$-chart are $(\frac{t_1}{t_i},\ldots,\frac{t_k}{t_i})$. Now fix an index $i$ and assume that $a'$ lies in the $x_i$-chart. Then there is an \emph{induced regular system of parameters} $x_1',\ldots,x_n'$ for $\OWap$ such that the local ring map $\OWa\to\OWap$ which is induced by $\pi$ is given by
 \[\begin{array}{ll}
 x_i\mapsto x_i', &\\ 
 x_j\mapsto x_i'(x_j'+\frac{t_j}{t_i}) & \text{for $1\leq j\leq k$, $j\neq i$,} \\
 x_l\mapsto x_l' & \text{for $k<l\leq n$.}
\end{array}\]
 The exceptional divisor $\Dnew$ is defined in $\OWap$ by the principal ideal $(x_i')$.

 Obviously, calculations are easiest when $t_j=0$ for $j\neq i$ and hence, the images of the parameters $x_1,\ldots,x_n$ are monomials in the induced parameters $x_1',\ldots,x_n'$. In this case, we say that the local blowup map $\OWa\to\OWap$ is \emph{monomial} with respect to these parameters, or equivalently, that $a'$ is the origin of the $x_i$-chart with respect to the parameters $\x$.
 
 The $x_i$-charts and the induced projective coordinates on $\pi^{-1}(a)$ only depend on the parameters $x_i$ up to their equivalence class modulo $\mWa^2$. Thus, all of the results in this section still hold after passing to the completion. In particular, a regular system of parameters $\x=(x_1,\ldots,x_n)$ for $\hOWa$ with $\wh I_{Z,a}=(x_1,\ldots,x_k)$ induces an open cover of $\pi^{-1}(a)$ by $x_i$-charts and the ring maps $\hOWa\to\hOWap$ are of the form stated above.
 
 After fixing a point $a'\in\pi^{-1}(a)$, it is always possible to change the parameters $\x$ in such a way that the local blowup map $\OWa\to\OWap$ (or $\hOWa\to\hOWap$) becomes monomial. Set $\wt\x=(\wt x_1,\ldots,\wt x_n)$ where
 \[\begin{array}{ll}
 \wt x_i=x_i, &\\ 
 \wt x_j=x_j-\frac{t_j}{t_i}x_i & \text{for $1\leq j\leq k$, $j\neq i$,} \\
 \wt x_l=x_l & \text{for $k<l\leq n$.}
\end{array}\]
 Then $I_{Z,a}=(\wt x_1,\ldots,\wt x_k)$ and $a'$ is the origin of the $\wt x_i$-chart with respect to $\wt\x$. Thus, the induced map $\OWa\to\OWap$ is then given by
\[\begin{array}{ll}
 \wt x_j\mapsto x_i'x_j' & \text{for $1\leq j\leq k$, $j\neq i$,} \\
 \wt x_l\mapsto x_l' & \text{for $l=i$ and $k<l\leq n$.}
\end{array}\]

\begin{remark}
 Although we just stated that it is always possible to choose parameters for $\OWa$ in such a way that the local blowup map $\OWa\to\OWap$ is monomial, this technique is not always useful. In some situations we will consider invariants that are defined via a certain choice of regular parameters $\x$ for $\OWa$. A typical example for this is that some of the parameters $\x$ define a given local geometric object. Showing that such invariants behave well under blowup is usually only possible when the local blowup map $\OWa\to\OWap$ is monomial with respect to $\x$. Although the coordinate change $x_i\mapsto \wt x_i$ described above makes the map monomial, we might not be able to define our invariant with respect to the parameters $\wt\x$. Again, a typical example for this is that there is a given geometric object which is defined by some of the parameters $\x$, but not by the parameters $\wt\x$.
 
 This problem will appear again in Section \ref{section_residual_order_under_blowup} when discussing the increase of the residual order under blowup in positive characteristic.
\end{remark}

\section{Directrix, $\tau$-invariant and adjacent hypersurfaces} \label{section_directrix}

 As established in Section \ref{section_order_function}, we are looking for a refinement of the order function that enables us to measure improvement of singularities under order-permissible blowups at equiconstant points. In this section, we will develop techniques to narrow down the location of equiconstant points $a'$ lying over a closed point $a$. 
 The techniques introduced in this section go back to Zariski \cite{Zariski_1944} and Hironaka \cite{Hironaka_Annals}. They appear in various forms in most approaches to the embedded resolution of singularities.
 
 The following example conveys the main idea of the technique that we will employ.

\begin{example}
 Let $W$ be a regular variety, $X\subseteq W$ a hypersurface and $a\in X$ a closed point. Set $c=\ord_aX$. Let $\x=(x_1,\ldots,x_n)$ be a regular system of parameters for the completed local ring $\hOWa$. Let the ideal $\wh I_{X,a}$ be generated by a single element $f$. We can write $f$ as 
 \[f=F+H\]
 where $F=\init(f)$ is the initial form of $f$ and $H\in K[[\x]]$ is a power series of order $\ord H>c$.
 
 Consider now the blowup $\pi:W'\to W$ of $W$ at the point $a$ and let $a'\in\pi^{-1}(a)$ be a closed point that is contained in the $x_1$-chart. Let $a'$ have the affine coordinates $(t_2,\ldots,t_n)$ in the $x_1$-chart. Denote the induced parameters for $\hOWap$ again by $\x=(x_1,\ldots,x_n)$. Then the ideal $\hIXap$ is generated by the element
 \[f'=x_1^{-c}\cdot f(x_1,x_1(x_2+t_2),\ldots,x_1(x_n+t_n))\]
 \[=\underbrace{F(1,x_2+t_2,\ldots,x_n+t_n)}_{=:F'}+x_1\cdot \wt H\]
 for some element $\wt H\in\hOWap$. Since $x_1$ does not appear in $F'$, there can be no cancellation of terms between $F'$ and $x_1\cdot \wt H$. Consequently, $\ord_{a'}X'\leq \ord F'$.
 
 Assume now that $x_1$ appears in a monomial with non-zero coefficient in the initial form of $f$. If either $\init(f)=x_1^c$ or $t_2=\ldots=t_n=0$, it is clear that $\ord F'<c$ must hold and thus, $a'$ is not an equiconstant point. In general though, cancellation of terms can happen in the expansion of $F'$. For example, if $\init(f)=(x_2-t_2x_1)^c$, then $\ord F'=c$. Notice that in this case, $x_1$ can be eliminated from the initial form of $F$ by making a change of coordinates $\wt x_2=x_2-t_2x_1$. 
\end{example}

 In this section we will show that if $x_1,\ldots,x_k$ are parameters for $\OWa$ which are \emph{minimal} with the property that they generated the weak initial ideal of $\IXa$, then no equiconstant points are contained in the $x_i$-chart for $i=1,\ldots,k$ for all order-permissible blowups.
 
 To this end, we will first show that the residues of a set of parameters $x_1,\ldots,x_k$ which minimally generate the weak initial ideal of $\IXa$ span a well-defined linear subspace of the Zariski cotangent space $\mWa/\mWa^2$. This space is called the \emph{directrix} and its dimension is called (Hironaka's) \emph{$\tau$-invariant}. Both appear in many proofs of resolution of singularities in some form. 

 Finally, we will show in Proposition \ref{directrix_under_blowup} that the directrix prescribes a locus that contains all equiconstant points $a'$ lying over a closed point $a$.

\begin{convention}
\begin{itemize}
 \item Let $R$ be a local ring with maximal ideal $m$. The residues of elements $x_i\in m$ in the cotangent space $m/m^2$ will always be denoted by $\ol x_i$.
 \item Let $R$ be a Noetherian local ring with maximal ideal $m$. Let $\wh R$ denote the completion of $R$ with respect to $m$ and $\wh m$ the maximal ideal of $\wh R$. Then the cotangent spaces $m/m^2$ and $\wh m/\wh m^2$ will usually be identified with each other.
\end{itemize}
\end{convention}

\begin{lemma} \label{technical_directrix_lemma}
 Let $(R,m)$ be a regular local ring that contains a field $K$ and $J\subseteq R$ an ideal of order $c=\ord J$. Let $D\subseteq m/m^2$ be a $K$-linear subspace. Let $y_1,\ldots,y_k\in m$ be elements such that $D=(\ol y_1,\ldots,\ol y_k)$.
 
 Assume that the inclusion $J\subseteq (y_1,\ldots,y_k)^c+m^{c+1}$ holds. Then for any other set of elements $z_1,\ldots,z_l\in m$ with $D=(\ol z_1,\ldots,\ol z_l)$ the inclusion $J\subseteq (z_1,\ldots,z_l)^c+m^{c+1}$ also holds.
\end{lemma}
\begin{proof}
 For all indices $i=1,\ldots,k$ we can write $y_i=\sum_{j=1}^l\lambda_{i,j}z_j+Q_i$ for constants $\lambda_{i,j}\in K$ and elements $Q_i\in m^2$. By assumption, $J$ is generated by elements of the form 
 \[f=\sum_{\substack{\alpha\in\N^k\\|\alpha|=c}}c_\alpha \prod_{i=1}^k y_i^{\alpha_i}+H_{c+1}\]
 with $c_\alpha\in K$ and $H_{c+1}\in m^{c+1}$. Consequently,
 \[f=\sum_{\substack{\alpha\in\N^k\\|\alpha|=c}}c_\alpha \prod_{i=1}^k \B(\sum_{j=1}^k\lambda_{i,j}z_j+Q_i\B)^{\alpha_i}+H_{c+1}\]
 \[=\sum_{\substack{\beta\in\N^l\\|\beta|=c}}\wt c_\beta \prod_{j=1}^l z_j^{\beta_j}+\wt H_{c+1}\]
 for certain constants $\wt c_\beta\in K$ and an element $\wt H_{c+1}\in m^{c+1}$. This proves that $J\subseteq (z_1,\ldots,z_l)^c+m^{c+1}$.
\end{proof}

\begin{definition}
 Let $R$ be a regular local ring that contains a field $K$ and $J\subseteq R$ an ideal of order $c=\ord J$. Denote the maximal ideal of $R$ by $m$. A $K$-linear subspace $D$ of $m/m^2$ is said to \emph{generate $\minit(J)$} if there are elements $y_1,\ldots,y_k\in m$ such that $D=(\ol y_1,\ldots,\ol y_k)$ and $J\subseteq (y_1,\ldots,y_k)^c+m^{c+1}$. 
\end{definition}

\begin{lemma} \label{directrix_exists}
 Let $(R,m)$ be a regular local ring containing a field $K$ and $J\subseteq R$ an ideal of order $c=\ord J$. Let $D_1,D_2\subseteq m/m^2$ be two $K$-linear subspaces which both generate $\minit(J)$. Then $D_1\cap D_2$ also generates $\minit(J)$.
\end{lemma}
\begin{proof}
 We may assume that the regular system of parameters $\x=(x_1,\ldots,x_n)$ for $R$ is chosen in such a way that $D_1\cap D_2=(\ol x_1,\ldots,\ol x_r)$, $D_1=(\ol x_1,\ldots,\ol x_k)$ and $D_2=(\ol x_1,\ldots,\ol x_r,\ol x_{k+1},\ldots,\ol x_m)$ for certain integers $r\leq k\leq m\leq n$. Let $f\in J$ be an element with power series expansion $f=\sum_{\alpha\in\N^n}c_\alpha \x^\alpha$. Let $\alpha=(\alpha_1,\ldots,\alpha_n)\in\N^n$ be a multi-index with $|\alpha|=c$. Since $D_1$ generates $\minit(J)$, we know that $c_\alpha=0$ if there is an index $i>k$ such that $\alpha_i\neq0$. On the other hand, since $D_2$ also generates $\minit(J)$, we know that $c_\alpha=0$ if $\alpha_i\neq0$ for an index $i>r$ that does not fulfill $k<i\leq m$. In total, we know that $c_\alpha\neq0$ implies that $\alpha_i=0$ for all indices $i>r$. Consequently, $J\subseteq (x_1,\ldots,x_r)^c+m^{c+1}$.
\end{proof}

\begin{definition}
 Let $(R,m)$ be a regular local ring that contains a field $K$ and $J\subseteq R$ an ideal of order $c=\ord J$. By Lemma \ref{directrix_exists} there is a minimal $K$-linear subspace $D$ of $m/m^2$ with the property that $D$ generates $\minit(J)$. We call $D$ the \emph{directrix of $J$} and denote it by $\Dir(J)=D$. Further, we define $\tau(J)=\dim_K(\Dir(J))$.
 
 Let $W$ be a regular variety, $X\subseteq W$ a closed subset and $a\in W$ a closed point. Then we define $\Dir_a(X)=\Dir(\IXa)$ and $\tau_a(X)=\tau(\IXa)$.
\end{definition}

\begin{lemma} \label{directrix_and_top_locus}
 Let $W$ be a regular $n$-dimensional variety, $X\subseteq W$ a closed subset and $Z\subseteq W$ a permissible center of blowup with respect to the order function. Let $a\in Z$ be a closed point. Let $x_1,\ldots,x_n$ be a regular system of parameters for $\OWa$ (or $\hOWa$) and $m\leq n$ an index such that $I_{Z,a}=(x_1,\ldots,x_m)$ (resp. $\wh I_{Z,a}=(x_1,\ldots,x_m)$) holds.
 
 Then $\Dir_a(X)\subseteq(\ol x_1,\ldots,\ol x_m)$. In particular, $\tau_a(X)\leq\codim_a(Z)$. 
\end{lemma}
\begin{proof}
  Set $c=\ord_aX$. Since $Z$ is contained in $\topp(X)$, we know that $\ord_{I_{Z,a}}I_{X,a}=c$. Further, since $Z$ is regular, we know by Corollary \ref{cor_symb_powers_of_reg_ideals} that $I_{X,a}\subseteq (x_1,\ldots,x_m)^c$. Thus, $\Dir_a(X)\subseteq (\ol x_1,\ldots,\ol x_m)$ holds. The result for the completion follows from Lemma \ref{completion_preserves_order}.
\end{proof}

\begin{proposition} \label{directrix_under_blowup}
 Let $W$ be a regular $n$-dimensional variety, $X\subseteq W$ a closed subset and $Z\subseteq W$ a permissible center with respect to the order function. Consider the blowup $\pi:W'\to W$ of $W$ with center $Z$ and the strict transform $X'$ of $X$. Let $a\in Z$ and $a'\in\pi^{-1}(a)$ be closed points.
 
 Further, let $x_1,\ldots,x_n$ be a regular system of parameters for $\OWa$ (or $\hOWa$) and $m,k$ integers with $k\leq m\leq n$ such that $I_{Z,a}=(x_1,\ldots,x_m)$ (resp. $\wh I_{Z,a}=(x_1,\ldots,x_m)$) and $\ol x_1,\ldots,\ol x_k$ form a basis of $\Dir_a(X)$.
 
 If $a'$ is an equiconstant point, then $a'$ is not contained in the $x_i$-chart for $i=1,\ldots,k$. In other words, the strict transform of $V(x_1,\ldots,x_k)$ contains all equiconstant points $a'$ lying over $a$.
\end{proposition}
\begin{proof}
 Set $c=\ord_aX$. For the sake of simplicity, we will denote the induced parameters for $\OWap$ again by $x_1,\ldots,x_n$.
 
 We can assume without loss of generality that $a'$ lies is the origin of the $x_1$-chart and the local blowup map $\vp:\OWa\to\OWap$ is given by $\vp(x_1)=x_1$, $\vp(x_i)=x_1 x_i$ for $1<i\leq m$ and $\vp(x_j)=x_j$ for $j>m$. Let $f\in\IXa$ be an element such that $\ord f=c$. Let $f$ have the power series expansion $f=\sum_{\alpha\in\N^n}c_\alpha x^\alpha$. Then $f'=x_1^{-c}\vp(f)\in I_{X',a'}$. 
 
 For an index $\alpha\in\N^n$ set $\alpha_+=(\alpha_1,\ldots,\alpha_m)\in\N^m$ and $\alpha_-=(\alpha_2,\ldots,\alpha_n)\in\N^{n-1}$. We can write $f$ as $f=\sum_{i\geq c} f_i$ where 
 \[f_i=\sum_{\substack{\alpha\in\N^n\\|\alpha_+|=i}}c_\alpha \x^\alpha.\]
 Then $f'$ can be written as $f'=\sum_{i\geq c}x_1^{i-c}f_i'$ where $f_i'\in K[[x_2,\ldots,x_n]]$ has the form
 \[f_i'=\sum_{\substack{\alpha\in\N^n\\|\alpha_+|=i}}c_\alpha \x_-^{\alpha_-}.\]
 Here, $\x_-=(x_2,\ldots,x_n)$. Since $\ord I_{X',a'}=c$, we know that $\ord f'\geq c$. In particular, $\ord f_c'\geq c$.
 
 Since $\Dir(I_{X,a})=(\ol x_1,\ldots,\ol x_k)$, we know that 
 \[f_c=\init(f)\in K[x_1,\ldots,x_k].\]
 Assume that $x_1$ appears in a monomial with non-zero coefficient in the expansion of $\init(f)$. This implies that $\ord f_c'<c$ which is a contradiction. Hence, $\init(f)\in(x_2,\ldots,x_k)^c$. Since $f\in\IXa$ was chosen arbitrarily with the property $\ord f=c$, this implies that $I_{X,a}\subseteq (x_2,\ldots,x_k)^c+m_{W,a}^{c+1}$. Consequently, $\Dir_a(X)\subseteq (\ol x_2,\ldots,\ol x_k)$ which contradicts our assumption.
\end{proof}

\begin{remarks}
\begin{enumerate}[(1)]
 \item It is clear from Proposition \ref{directrix_under_blowup} that the bigger the number $\tau_a(X)$ is, the narrower is the locus inside of which equiconstant points can appear over $a$. In particular, if $\tau_a(X)=\dim(W)$, there can be no equiconstant points lying over $a$. In other words, this guarantees that the order function decreases under the next order-permissible blowup.
 \item The $\tau$-invariants exhibits another good property which will not be proven here: In the setting of Proposition \ref{directrix_under_blowup}, $\tau_{a'}(X')\geq \tau_a(X)$ holds for equiconstant points $a'\in\pi^{-1}(a)$. Furthermore, it is possible to choose the parameters $x_1,\ldots,x_n$ for $\OWa$ in such a way that $(\ol x_1',\ldots,\ol x_k')\subseteq\Dir_{a'}(X')$ where $\x'=(x_1',\ldots,x_n')$ denotes the induced parameters for $\OWap$.
 
 Thus, it would be feasible to consider the invariant $\mu:X\to\N^2$, $\mu_a(X)=(\ord_X(a),\dim(W)-\tau_a(X))$ as a refinement of the order function. By what we have stated so far, this invariant does not increase under blowup of regular centers contained in its top locus. Still, it will remain constant in many situations. Thus, we would have to further refine this invariant to find a resolution invariant that always decreases under blowup. 
 \item In this thesis, the $\tau$-invariant will not be used as a component of the resolution invariant, but rather as a tool that is useful when analyzing the behavior of the order function under blowup. 
\end{enumerate}

\end{remarks}


\begin{definition}
 Let $W$ be a regular variety, $X\subseteq W$ a closed subset and $a\in X$ a closed point. Let $H$ be either a regular local or a regular formal hypersurface at $a$. Then $H$ is said to be \emph{adjacent to $X$ at $a$} if $H=V(z)$ and $\ol z\in\Dir_a(X)$.
\end{definition}

\begin{corollary}
 Let $W$ be a regular variety, $X\subseteq W$ a closed subset and $Z\subseteq W$ a center that is permissible with respect to the order function. Consider the blowup $\pi:W'\to W$ of $W$ with center $Z$ and the strict transform $X'$ of $X$. Let $a\in Z$ be a closed point.
 
 Let $H$ be a regular hypersurface that is adjacent to $X$ at $a$ and locally contains $Z$. Then the strict transform $H'$ of $H$ is regular and contains all equiconstant points $a'\in\pi^{-1}(a)$.
\end{corollary}

\section{Hypersurfaces of maximal contact} \label{section_woah_maximal_contact}

 Hypersurfaces of maximal contact are one of the most important techniques that is used in the proof of embedded resolution of singularities over fields of characteristic zero. They provide a method of locally reducing the resolution problem in an $n$-dimensional ambient space to a resolution problem in an $(n-1)$-dimensional ambient space, thus allowing to use induction on dimension of the ambient space. In this section we will outline what hypersurfaces of maximal contact are and discuss their relationship with differential operators. More information on the much-studied subject of maximal contact can be found in \cite{Giraud_Maximal_Contact}, \cite{AHV_Maximal_Contact}, \cite{Wlodarczyk} and more generally, all proofs of embedded resolution of singularities in characteristic zero.
 
 This section is different from the other sections in this chapter since most of it does not feature techniques that will be used later on in our proof for the embedded resolution of surface singularities. The reason for this is that hypersurfaces of maximal contact need not exist over fields of positive characteristic. Still, we feel that hypersurfaces of maximal contact are such an important concept for the embedded resolution of singularities that it is necessary to include them in our presentation.
 
 Let $W$ be a regular variety over a field $K$ of characteristic zero, $X\subseteq W$ a closed subset and $a\in X$ a closed point of order $c=\ord_aX$. In this section, $X_{\geq c}$ will always denote the closed set 
 \[X_{\geq c}=\{b\in X:\ord_bX\geq c\}.\]
 
 

\begin{definition}
  A regular hypersurface $H$ is said to have \emph{maximal contact} with $X$ at $a$ if the following two properties hold:
 \begin{itemize}
  \item The hypersurface $H$ contains $X_{\geq c}$ locally at $a$.
  \item Consider any finite sequence of order-permissible blowups 
  \[W^{(m)}\to\cdots W^{(1)}\to W.\]
  Let $a^{(m)},\ldots,a^{(1)},a$ be a sequence of equiconstant points $a^{(i)}\in W^{(i)}$ lying over each other. Then $a^{(i)}\in H^{(i)}$ where $H^{(i)}$ denotes the $i$-th strict transform of $H$.
 \end{itemize}
\end{definition}

 The next proposition gives sufficient criteria for a regular local hypersurface to have maximal contact with $X$ at $a$ and clarifies how the property of maximal contact is related to the ideal $\Diff^{c-1}_{\OWa/K}(\IXa)$ and its radical.

\begin{proposition} \label{max_contact_prop}
 Let $H=V(z)$ with $z\in\OWa$ be a regular local hypersurface. Consider the following statements:
 \begin{enumerate}[(1)]
  \item There is an element $f\in\IXa$ which has an expansion $f=\sum_{i\geq0}f_iz^i$ with respect to a regular system of parameters $(\x,z)$ for $\OWa$ such that $\ord f_c=0$ and $f_{c-1}=0$.
  \item $z\in\Diff^{c-1}_{\OWa/K}(\IXa)$.
  \item $H$ has maximal contact with $X$ at $a$.
  \item $H$ contains $X_{\geq c}$ locally at $a$.
  \item $z\in\rad(\Diff^{c-1}_{\OWa/K}(\IXa))$.
 \end{enumerate}
 The following implications hold:
 \[(1)\underset{\centernot{\Longleftarrow}}{\implies} (2)\underset{\centernot{\Longleftarrow}}{\implies} (3)\underset{\centernot{\Longleftarrow}}{\implies} (4)\iff (5).\]
\end{proposition}
\begin{proof}
 $(1)\implies(2)$: 
 We can compute that
 \[\partial_{z^{c-1}}(f)=\sum_{i\geq c-1}\binom{i}{c-1}f_iz^{i-(c-1)}=c\cdot f_cz+z^2\cdot\sum_{i\geq c+1}\binom{i}{c-1}f_iz^{i-(c+1)}=uz\]
 for a unit $u\in\OWa^*$. 
 This proves that $z\in\Diff^{c-1}_{\OWa/K}(\IXa)$.
 
 $(2)\implies(3)$: This is proven in \cite{Kollar_Book} Theorem 3.80, p. 171.
 
 
 $(3)\implies(4)$: This is included in the definition of maximal contact.
 
 $(4)\iff(5)$: This follows from Proposition \ref{ord_via_diff_ops}.
 
 $(2)\centernot\implies(1)$: Consider the origin $a=(0,0)$ on the curve $X=V(xy)$ in $W=\Spec(K[x,y])$. Then $\Diff_{\OWa/K}^{1}(\IXa)=(x,y)$. Thus, the local hypersurface $H=V(y)$ fulfills property $(2)$, but not $(1)$.
 
 $(3)\centernot\implies(2)$: Consider the origin $a=(0,0)$ on the curve $X=V(y^2-x^3)$ in $W=\Spec(K[x,y])$. Then $\Diff_{\OWa/K}^{1}(\IXa)=(x^2,y)$. It is easy to check that blowing up the point $a$ resolves the curve $X$. Hence, there are no equiconstant points. In other words, all regular hypersurfaces have maximal contact with $X$ at $a$. Hence, the local hypersurface $H=V(x)$ fulfills property $(3)$, but not $(2)$.
 
 $(5)\centernot\implies(3)$: Consider the origin $a=(0,0)$ on the curve $X=V(y^2-x^5)$ in $W=\Spec(K[x,y])$. Then $\Diff_{\OWa/K}^{1}(\IXa)=(y,x^4)$ and $\rad(\Diff_{\OWa/K}^{1}(\IXa))=(y,x)$. When blowing up the point $a$, there is exactly one equiconstant point $a'$ lying over $a$ which is the origin of the $y$-chart. Hence, the local hypersurface $H=V(x)$ fulfills property $(5)$, but not $(3)$.
\end{proof}

\begin{remarks}
 \begin{enumerate}[(1)]
  \item If $K$ is a field of characteristic zero, it is easy to see that the ideal $\Diff_{\OWa/K}^{c-1}(\IXa)$ always contains an element $\partial_{\x^\alpha}(f)$ of order $1$. Hence, the regular local hypersurface $H=V(\partial_{\x^\alpha}(f))$ has maximal contact with $X$ at $a$. Moreover, it is clear that there is an open neighborhood $U$ of $a$ such that the hypersurface $H$ has maximal contact with $X$ at all points $b\in X_{\geq c}\cap U$.
  \item If $c$ is not divisible by the characteristic of $K$, it is always possible to construct a regular formal hypersurface $H\subseteq\Spec(\hOWa)$ fulfilling property (1) of Proposition \ref{max_contact_prop}. By Lemma \ref{z_regular_after_generic_coord_change} we can choose a regular system of parameters $(\x,z)$ for $\OWa$ such that there is an element $f\in\hIXa$ with power series expansion $f=\sum_{i\geq0}f_iz^i$ where $f_i\in K[[\x]]$ and $\ord f_c=0$. By the Weierstrass preparation theorem, we may even assume that $f$ has the form $f=z^c+\sum_{i<c}f_iz^i$.
  
  Now set $z_1=z+\frac{1}{c}f_{c-1}$. It easy to see that $f$ has a power series expansion $f=z_1^c+\sum_{i=0}^{c-2}\wt f_iz_1^i$ with $\wt f_i\in K[[\x]]$. Thus, the regular formal hypersurface $H=V(z_1)$ fulfills property (1) of Proposition \ref{max_contact_prop}. Apart from having maximal contact with $X$ at $a$, the formal hypersurface $H$ also has the strong property that it maximizes certain invariants associated to coefficient ideals, as it will be explained in Chapter \ref{chapter_cleaning}. 
  
  The coordinate change $z\mapsto z_1$ described above is called \emph{Tschirnhausen transformation}. The concept was introduced by Abhyankar and Zariski \cite{Zariski_Abhyankar_55} and was one of the key techniques used in Hironaka's proof of resolution in characteristic zero \cite{Hironaka_Annals}. Obviously, the Tschirnhausen transformation cannot be applied whenever $c$ is divisible by the characteristic of $K$.
  \item Over fields of positive characteristic, it is known that there are generally no regular local (or formal) hypersurfaces fulfilling any of the properties (1)-(5) in Proposition \ref{max_contact_prop}. Narasimhan \cite{Narasimhan} was the first to give an example of a hypersurface $X$ in a $4$-dimensional ambient space $W$ over a field of characteristic $2$ whose top locus is a curve that is locally at a closed point $a$ not contained in any regular hypersurface. In particular, there is no hypersurface which has maximal contact with $X$ at $a$.
  
  Since we want to treat the cases of characteristic zero and positive characteristic uniformly, hypersurfaces of maximal contact will not be used in our proof for the resolution of surfaces.
 \end{enumerate}
\end{remarks}

\section{Coefficient ideals} \label{section_coeff_ideals}

 In this section we introduce the construction which is most fundamental to our approach to the embedded resolution of singularities. Coefficient ideals will play an important role in all ensuing chapters of this thesis. They are our main tool to measure improvement of singularities under blowup at equiconstant points.
 
 Coefficient ideals are a well-known technique for proving embedded resolution of singularities over fields of characteristic zero where they are used with respect to hypersurfaces of maximal contact. In our approach, we use coefficient ideals independently of the characteristic of the ground field. Since hypersurfaces of maximal contact do generally not exist over fields of positive characteristic, we consider coefficient ideals with respect to all regular formal hypersurfaces. The problem of deriving significant invariants from these coefficient ideals that can be used to measure improvement at equiconstant points will be discussed in detail in Chapter \ref{chapter_pathologies}.

 \subsection{Descent in dimension}

 Let $W$ be a regular hypersurface, $X\subseteq W$ a closed subset and $a\in X$ a closed point. 
 Consider the blowup $\pi:W'\to W$ of a center $Z$ that is permissible for $X$ with respect to the order function. As we established in Section \ref{section_directrix}, there exists locally at $a$ a regular hypersurface $H$ that contains $Z$ and whose strict transform $H'$ is regular and contains all equiconstant points $a'$ lying over $a$. We can choose a regular system of parameters $(\x,z)=(x_1,\ldots,x_n,z)$ for the completed  local ring $\hOWa$ such that $H=V(z)$. The coefficient ideal will be defined as an ideal $J_{-1}$ of the ring $\hOWa/(z)$ which is constructed from $\wh I_{X,a}$. This ideal will provide us with a more sophisticated way to measure the complexity of the singularity of $X$ at $a$ than the order function alone provides. Now let $a'\in\pi^{-1}(a)$ be an equiconstant point lying over $a$. Since $a'$ lies on the strict transform $H'$ of $H$, the point $a'$ is not contained in the $z$-chart. Let $(\x',z')$ be the induced parameters for $\hOWap$ in some $x_i$-chart. Then $H'=V(z')$. To measure improvement, we will associate to the \emph{weak transform} $(\hIXa)'$ of the ideal $\hIXa$ at $a'$ again a coefficient ideal $J_{-1}'$ in the ring $\hOWap/(z')$ and compare it with $J_{-1}$.
 
 This sort of technique is commonly called \emph{descent in dimension} since it reduces the problem of measuring improvement under blowup of a singularity in an $n$-dimensional ambient space $W$ to an $(n-1)$-dimensional ambient space $H$. Notice that hypersurfaces of maximal contact are naturally perfectly suited for the descent in dimension since they allow us to use above construction to measure improvement not only for a single blowup, but for any sequence of order-permissible blowups.
 
 Recall that the weak transform $(\hIXa)'$ of $\hIXa$ does generally not coincide with the completed defining ideal $\hIXap$ of the strict transform $X'$ of $X$ at $a'$. It is an established technique in proofs of embedded resolution of singularities over fields of characteristic zero to first prove a resolution statement for the weak transform and derive from this the resolution statement for the strict transform of $X$ (Cf. \cite{EH}, \cite{EV_Char_Zero}). In the case that $X$ is a hypersurface, the weak and strict transform coincide. 

 The following example demonstrates the idea underlying the descent in dimension:
 
\begin{example}
 Let $W$ be a regular variety, $X\subseteq W$ a hypersurface and $a\in X$ a closed point. Set $c=\ord_aX$. Let $(\x,z)$ be a regular system of parameters for $\hOWa$ and assume that the ideal $\hIXa$ is generated by an element $f$ of the form
 \[f=z^c+F(\x)\]
 where $F\in K[[\x]]$ is a power series of order $\ord F>c$. Clearly, the order function only takes into account the term $z^c$ while ignoring the remainder $F(\x)$.
 
 Now consider the blowup $\pi:W'\to W$ at the point $a$ and let $a'\in\pi^{-1}(a)$ be an equiconstant closed point. Since the hypersurface $H=V(z)$ is adjacent to $X$ at $a$, the point $a'$ is not contained in the $z$-chart. Denote the induced parameters for $\hOWap$ in some $x_i$-chart again by $(\x,z)$. Then the ideal $\hIXap$ in $\hOWap$ is generated by an element $f'$ of the form
 \[f'=z^c+F'(\x)\]
 where $F'(\x)\in K[[\x]]$ is a power series with order $\ord F'\geq c$.
 
 It is clear that to measure the improvement of the singularity under this blowup, we have to compare the power series $F(\x)$ and $F'(\x)$.
\end{example}

\subsection{Definition and main properties}

\begin{definition}
 Let $R$ be the ring of formal power series in $(n+1)$ variables over an arbitrary field $K$ and $J\subseteq R$ an ideal. Let $(\x,z)$ be a regular system of parameters for $R$. Let each element $f\in J$ have the expansion $f=\sum_{i\geq0}f_iz^i$ with $f_i\in K[[\x]]$. Let $c\in\N$ be a non-negative integer.
 
 The \emph{coefficient ideal} of $J$ with respect to $c$ and the parameter system $(\x,z)$ is defined as the ideal
 \[\coeff_{(\x,z)}^c(J)=\B(f_i^{\frac{c!}{c-i}}:f\in J,i<c\B)\]
 in the ring $R/(z)\cong K[[\x]]$. Set $\coeff_{(\x,z)}^0(J)=0$.
 
 We will sometimes say that the coefficient ideal $\coeff_{(\x,z)}^c(J)$ is defined with respect to the regular hypersurface $V(z)\subseteq\Spec(K[[\x,z]])$.
 
\end{definition}

\begin{remarks}
\begin{enumerate}[(1)]
 \item There is a way to extend our definition of the coefficient ideal to more rings than just the power series ring by using differential operators. Let $W$ be a regular variety over a field $K$ and $a\in W$ a closed point. Let $R$ be either the coordinate ring of an affine open neighborhood of $a$, the local ring $\OWa$ or its completion $\hOWa$. Assume that the module of differentials $\Omega_{R/K}$ is freely generated by elements $dx_1,\ldots,dx_n,dz$. Then we can define for an ideal $J\subseteq R$ and a number $c>0$ the coefficient ideal of $J$ with respect to $c$ and $(\x,z)=(x_1,\ldots,x_n,z)$ as
 \[\coeff_{(\x,z)}^c(J)=(\partial_{z^i}(f)^{\frac{c!}{c-i}}:f\in J,i<c)/(z)\subseteq R/(z).\]
 In the case that $R=\hOWa=K[[\x,z]]$, this coincides with above definition of the coefficient ideal. Notice though, that the coordinate-dependence of the definition of the coefficient ideal prevents us from defining anything like a sheaf of coefficient ideals on $W$.
 \item Coefficient ideals have been defined by several authors in different ways. (Cf. \cite{BM_Canonical_Desing} 4.18, p. 246, \cite{EV_Good_Points} 4.14, p. 122, \cite{EH} p. 829, \cite{Wlodarczyk} 3.6.1, p. 801, \cite{Kollar_Book} 3.54, p. 156, \cite{AFK_Modified_Coeff_Ideal}) While these versions of the coefficient ideal differ in definition, they all fulfill the same purpose.
 
 One particular variant of the definition of the coefficient ideal is to set 
 \[\coeff_{(z)}^c(J)=(\partial(f)^{\frac{c!}{c-i}}:f\in J,\partial\in\Diff^i_{R/K},i<c)/(z)\subseteq R/(z).\]
 This definition has the big advantage that, unlike our definition, it only depends on the principal ideal $(z)$ and not on the entire regular parameter system for $R$. The drawback of this definition is that it does not behave as well under blowup as the version of the coefficient ideal that we use.
 \item In the setting of positive or arbitrary characteristic, other constructions than the coefficient ideal have been proposed for the descent in dimension. (Cf. \cite{Villamayor_Hypersurface}, \cite{BV_Elimination} and \cite{Kawanoue_IF_1}, \cite{Kawanoue_Matsuki}) These generally follow the approach of first \emph{enlarging} the ideal $J$ as much as possible (for example, by applying differential operators or integral closure) before reducing the dimension of the ambient space.
 
 Our definition of the coefficient ideal can be seen as the opposite approach. We enlarge the ideal by only using differential operators of the form $\partial_{z^i}$ with $i<c$ instead of all differential operators up to degree $c-1$. This yields a coefficient ideal that is in a certain sense still very closely related to the original ideal. This is helpful in particular when trying to show that improvement of the coefficient ideal under blowup ultimately leads to a decrease of the order of the ideal $J$. On the downside, the resulting coordinate-dependence of our coefficient ideal makes it a quite complicated object to handle in general.
 \item We will use the following notational conventions for coefficient ideals: A positive index (as in $J_n,J_{2,\x},\ldots$) will be used to indicate the dimension of the ring inside which the coefficient ideal is defined. On the other hand, a negative index (as in $J_{-1}$) is used to indicate the codimension. 
\end{enumerate}
\end{remarks}

 The coefficient ideal of the ideal $\wh I_{X,a}$ with respect to $c=\ord_aX$ and a regular formal hypersurface $H=V(z)\subseteq\Spec(\hOWa)$ exhibits a number of good properties that make it suited for the descent in dimension. Most importantly, it behaves well under oder-permissible blowups at equiconstant points $a'$ under the three conditions that the weak transform $(\hIXa)'$ of $\hIXa$ is considered, that the center of blowup is locally contained in the hypersurface $H$ and that the point $a'$ is contained in the strict transform $H$. This will be shown in Lemma \ref{coeff_ideal_under_blowup}. The behavior of the invariants associated to the coefficient ideal under blowup will be investigated in more detail in Section \ref{section_weighted_orders_under_blowup}.
 
 Another important property will be shown in Proposition \ref{top_locus_defined_via_coeff_ideal}. It says that under the condition that the locus $X_{\geq c}=\{a\in X:\ord_aX\geq c\}$ is at the point $a$ locally contained (on the level of the completion) in a regular hypersurface $H$, its defining ideal $\wh I_{X_{\geq c},a}$ can be computed from the coefficient ideal of $\wh I_{X,a}$ with respect to $H$.
 
 As both of these statements already suggest, the coefficient ideal of $\hIXa$ with respect to a regular hypersurface $H$ only encodes significant geometric information if the hypersurface $H$ is chosen in a particularly good way. One could also phrase this by saying that the coefficient ideal construction is only as strong as our ability to construct regular hypersurfaces which locally approximate $X$ very well. In characteristic zero, this is accomplished by hypersurfaces of maximal contact. Our goal in the subsequent chapters is to construct regular formal hypersurfaces independently of characteristic which, even if they do not have maximal contact with $X$, enable us to define coefficient ideals that carry relevant information for the resolution of singularities.
 
 As already mentioned, our definition of the coefficient ideal is highly dependent on the choice of the parameters $(\x,z)$. The only change of coordinates that trivially leaves the coefficient ideal invariant are permutations of the parameters $x_1,\ldots,x_n$. In Proposition \ref{coeff_ideal_well_def}, we will show that multiplication of $z$ with a unit also leaves the coefficient ideal invariant. The effect of coordinate changes $x_i\mapsto x_i+g(\x,z)$ and multiplication of the parameters $x_i$ with units will be investigated in detail in Chapter \ref{chapter_invariance}. In Chapter \ref{chapter_cleaning}, we will then investigate the effect of coordinate changes $z\mapsto z+g(\x)$ which move the underlying hypersurface $H$.

\subsection{Behavior of the coefficient ideal under blowup}

\begin{lemma} \label{coeff_ideal_under_blowup}
 Let $R=K[[\x,z]]$ with $\x=(x_1,\ldots,x_n)$ and $J\subseteq R$ an ideal of order $c=\ord J$. Let $k\leq n$ be such that $\ord_{(x_1,\ldots,x_k,z)}J=c$. Consider the blowup map $\pi:R\to R$ with center $(x_1,\ldots,x_k,z)$ in the $x_1$-chart that is given by
 \[\begin{array}{ll}
 \pi(x_1)=x_1, &\\
 \pi(x_i)=x_1(x_i+t_i) & \text{for $1<i\leq k$,}\\ 
 \pi(x_j)=x_j & \text{for $j>k$,} \\
 \pi(z)=x_1z. &
\end{array}\]
 for certain constants $t_2,\ldots,t_k\in K$. Let $J^*=R\pi(J)$ be the total transform and $J'=x_1^{-c}J^*$ the weak transform of $J$. Assume that $\ord J'=c$. 
 
 
 Further, set $J_{-1}=\coeff^c_{(\x,z)}(J)$ and $J_{-1}'=\coeff_{(\x,z)}^c(J')$. Denote the induced map $K[[\x]]\to K[[\x]]$ again by $\pi$. Let $J_{-1}^*=K[[\x]]\pi(J_{-1})$ be the total transform of $J_{-1}$. Then the following inclusion holds:
 \[x_1^{-c!}J_{-1}^*\subseteq J_{-1}'.\]
\end{lemma}
\begin{proof}
 
 Let each element $f\in J$ have an expansion $f=\sum_{i\geq0}f_{i}z^i$ with $f_{i}\in K[[\x]]$. Then $f'=x_1^{-c}\pi(f)\in J'$. Further, we can compute that
 \[f'=x_1^{-c}\sum_{i\geq0}\pi(f_{i})\pi(z^i)=\sum_{i\geq0}\underbrace{x_1^{i-c}\pi(f_{i})}_{=:f_{i'}}z^i\]
 where $f_{i}'\in K[[\x]]$. Notice that for $i<c$ the equality
 \[x_1^{-c!}\pi\B(f_i^{\frac{c!}{c-i}}\B)=(f_{i}')^{\frac{c!}{c-i}}\]
 holds. This proves the claimed inclusion.
\end{proof}

\begin{remarks}
\begin{enumerate}[(1)]
 \item The following example shows that in the statement of Lemma \ref{coeff_ideal_under_blowup}, $x_1^{-c!}J_{-1}^*$ is generally a proper subset of $J_{-1}'$: 
 
 Let $J$ be the ideal generated by the element
 \[f=y^2z+x^4\]
 in the ring $R=K[[w,x,y,z]]$ where $K$ is a field of characteristic zero. Then it can be computed that
 \[J_{-1}=\coeff_{(w,x,y,z)}^3(J)=(x^8,y^6,x^4y^4).\]
 Now consider the point-blowup map $\pi:R\to R$ in the origin of the $w$-chart, given by $\pi(w)=w$, $\pi(x)=wx$, $\pi(y)=wy$, $\pi(z)=wz$. The weak transform $J'$ of $J$ is generated by the element
 \[f'=y^2z+wx^4.\]
 It can be computed that
 \[J_{-1}'=\coeff_{(w,x,y,z)}^3(J')=(w^2x^8,y^6,wx^4y^4).\]
 On the other hand, the total transform $J_{-1}^*$ of $J_{-1}$ has the form
 \[J_{-1}^*=(w^8x^8,w^6y^6,w^8x^4y^4)=(w^6)\cdot (w^2x^8,y^6,w^2x^4y^4).\]
 Hence,
 \[w^{-6}J_{-1}^*\neq J_{-1}'.\]
 
 While this seems to suggest at first sight that the coefficient ideal might not be well-suited to systematically measure improvement under blowup, we will prove in Section \ref{section_weighted_orders_under_blowup} that all relevant invariants (such as weighted orders) of the ideals $J_{-1}'$ and $x_1^{-c!}J_{-1}^*$ coincide.

 \item The order of the coefficient ideal is not an invariant that behaves well under blowup. 
 
 For example, consider the ideal $J\subseteq K[[x,y,z]]$ generated by the element
 \[f=z^2+y^5+x^9.\]
 Its coefficient ideal is $J_{-1}=\coeff_{(x,y,z)}^2(J)=(y^5+x^9)$. Hence, $\ord J_{-1}=5$.
 
 Under a point-blowup that is monomial in the $x$-chart, the weak transform $J'$ of $J$ is generated by
 \[f'=z^2+x^3(y^5+x^4).\]
 Thus, $J_{-1}'=\coeff_{(x,y,z)}^2(J')=(x^3(y^5+x^4))$ and $\ord J_{-1}'=7$. The order of the coefficient ideal increases under this blowup.
 
 Consequently, the order of the coefficient ideal itself is not suited as a resolution invariant. We will discuss from Section \ref{section_residual_order} onwards how to derive invariants from the coefficient ideal which are more suited for the usage as resolution invariants.
\end{enumerate}
\end{remarks}

\subsection{Computing $X_{\geq c}$ from the coefficient ideal}

 The following proposition should be viewed in the light of the correspondence between differential operators and the locus $X_{\geq c}=\{a\in X:\ord_aX\geq c\}$ that was stated in Proposition \ref{ord_via_diff_ops}. It tells us that if a regular formal hypersurface $H$ contains $X_{\geq c}$ locally at $a$, then the completed local defining ideal of $X_{\geq c}$ can be computed from the coefficient ideal. This will enable us in certain situations to determine from the fact that the coefficient ideal is of a certain (simple) form that the locus $X_{\geq c}$ has a particularly good shape locally at $a$. These results will be proved in Section \ref{section_form_of_top_locus}.

\begin{proposition} \label{top_locus_defined_via_coeff_ideal}
 Let $R=K[[\x,z]]$ and $J\subseteq R$ an ideal. Set $J_{-1}=\coeff_{(\x,z)}^c(J)$ for a positive integer $c>0$. Further, set 
 \[J_{\geq c}=\rad(\Diff_{R/K}^{c-1}(J))\]
 and
 \[(J_{-1})_{\geq c!}=\rad(R\cdot\Diff_{K[[\x]]/K}^{c!-1}(J_{-1}))\]
 Then the following hold:
 \begin{enumerate}[(1)]
  \item $\Jc\subseteq\Jcc+(z)$.
  \item If $z\in\Jc$, then $\Jc=\Jcc+(z)$.
 \end{enumerate}
\end{proposition}
\begin{proof}
 Let each element $f\in J$ have the expansion $f=\sum_{i\geq0}f_iz^i$ with $f_i\in K[[\x]]$.
 
 (1): By Proposition \ref{basis_for_diff_ops}, we know that $\Diff_{R/K}^{c-1}(J)$ is generated by elements of the form $\partial_{\x^\alpha}\partial_{z^k}(f)$ with $|\alpha|+k<c$. Notice that
 \[\partial_{\x^\alpha}\partial_{z^k}(f)=\sum_{i\geq k}\partial_{\x^\alpha}(f_i)\binom{i}{k}z^{i-k}\]
 \[=\partial_{\x^\alpha}(f_k)+zF_{k,\alpha}\]
 for certain elements $F_{k,\alpha}\in R$. Thus, it suffices to show that
 \[\partial_{\x^\alpha}(f_k)\in \Jcc\]
 for all indices $\alpha\in\N^n$ and $k\in\N$ with $|\alpha|+k<c$.
 
 To this end, let $k$ be fixed and use induction over $|\alpha|$. For $|\alpha|=0$, this is clear by the definition of the coefficient ideal. So let $\alpha\in\N^n$ be such that $|\alpha|>0$ and assume that
 \[\partial_{\x^\beta}(f_k)\in \Jcc\]
 for all multi-indices $\beta\in\N^n$ with $|\beta|<|\alpha|$. Set $d=\frac{c!}{c-k}$ and $A=d\cdot\alpha$. Since $k+|\alpha|<c$, we know that $|A|<c!$. Consequently, 
 \[\partial_{\x^A}(f_k^{d})\in\Diff_{K[[\x]]/K}^{c!-1}(J_{-1}).\]
 But we know by Proposition \ref{basis_for_diff_ops} (3) that
 \[\partial_{\x^A}(f_k^{d})=\sum_{\substack{\beta_1,\ldots,\beta_{d}\in\N^n\\ \beta_1+\ldots+\beta_{d}=A}}\partial_{\x^{\beta_1}}(f_k)\cdots \partial_{\x^{\beta_{d}}}(f_k)\]
 Now let $\beta_1,\ldots,\beta_{d}\in\N^n$ be multi-indices such that $\sum_{i=1}^{d}\beta_i=A$. This implies that either $\beta_1=\ldots=\beta_{d}=\alpha$ or there is an index $i$ such that $|\beta_i|<|\alpha|$. By the induction hypothesis, this implies that
 \[(\partial_{\x^\alpha}(f_k))^{d}\in\Jcc.\]
 Consequently, $\partial_{\x^\alpha}(f_k)\in\Jcc$.
 
 (2): We know that $\Diff_{K[[\x]]/K}^{c!-1}(J_{-1})$ is generated by elements of the form $\partial_{\x^\alpha}\big(f_i^{\frac{c!}{c-i}}\big)$ with $i<c$ and $|\alpha|<c!$. As argued before, we can write this as
 \[\partial_{\x^\alpha}(f_i^{d})=\sum_{\substack{\beta_1,\ldots,\beta_{d}\in\N^n\\ \beta_1+\ldots+\beta_{d}=\alpha}}\partial_{\x^{\beta_1}}(f_i)\cdots \partial_{\x^{\beta_{d}}}(f_i)\]
 where $d=\frac{c!}{c-i}$. Notice that $\beta_1+\ldots+\beta_d=\alpha$ implies that there is an index $1\leq j\leq d$ such that $|\beta_j|<c-i$. Since
 \[\partial_{\x^{\beta_j}}\partial_{z^i}(f)=\partial_{\x^{\beta_j}}(f_i)+zF_{i,\beta_j}\]
 for some element $F_{i,\beta_j}$ and $z\in \Jc$ by assumption, this proves that $\Jcc\subseteq\Jc$. Hence, we have proved the assertion.
\end{proof}

\subsection{Stability under multiplication of $z$ with a unit}

 The following lemma will be useful whenever we consider how the expansion of an element $f\in K[[\x,z]]$ changes under coordinate changes in $z$.

\begin{lemma} \label{coordinate_change_g_z}
 Let $R=K[[\x,z]]$ and $f\in R$ an element. Let $f$ have the expansion $f=\sum_{i\geq0}f_iz^i$ with $f_i\in K[[\x]]$. Consider a change of coordinates $z\mapsto \wt z$ that will be specified in the following. Let $f$ have the expansion $f=\sum_{i\geq0}\wt f_i\wt z^i$ with $\wt f_i\in K[[\x]]$. The following hold:
 \begin{enumerate}[(1)]
  \item If $z=\wt z+g$ with $g\in K[[\x]]$, then
 \[\wt f_i=\sum_{k\geq i}\binom{k}{i}f_kg^{k-i}.\]
 \item If $z=u\wt z$ for a unit $u\in R^*$ with expansion $u=\sum_{j\geq0}u_jz^j$, then
 \[\wt f_i=\sum_{k=0}^i f_k\sum_{\substack{\alpha\in\N^k\\|\alpha|=i-k}}u_\alpha\]
 where $u_\alpha=\prod_{o=1}^k u_{\alpha_o}$.
 \end{enumerate}
\end{lemma}
\begin{proof}
 (1): The proof is a computation.
 \[f=\sum_{i\geq0}f_iz^i=\sum_{i\geq0}f_i(\wt z+g)^i\]
 \[=\sum_{i\geq0}f_i\sum_{k=0}^i\binom{i}{k}g^{i-k}\wt z^k=\sum_{k\geq0}\sum_{i\geq k}\binom{i}{k}f_ig^{i-k}\wt z^k.\]
 (2): Again, the proof is a computation.
 \[f=\sum_{i\geq0}f_i(\sum_{j\geq0}u_j\wt z^k)^i\wt z^i\]
 \[=\sum_{i\geq0}f_i\sum_{\alpha\in\N^i}u_\alpha \wt z^{i+|\alpha|}\]
 \[=\sum_{k\geq0}\sum_{i=0}^kf_i\sum_{\substack{\alpha\in\N^i\\|\alpha|=k-i}}u_\alpha \wt z^k.\]
\end{proof}

\begin{proposition} \label{coeff_ideal_well_def}
 Let $R=K[[\x,z]]$, $J\subseteq R$ an ideal and $c\in\N$ a non-negative integer. Let $u\in R^*$ be a unit and set $\wt z=uz$. Then $\coeff_{(\x,z)}^c(J)=\coeff^c_{(\x,\wt z)}(J)$.
\end{proposition}
\begin{proof}
 Let $f\in J$ have the expansion $f=\sum_{i\geq0}f_iz^i$ with $f_i\in K[[\x]]$. By Lemma \ref{coordinate_change_g_z} (2) we know that $f$ has the expansion $f=\sum_{i\geq0}\wt f_i\wt z^i$ where
 \[\wt f_i=\sum_{k=0}^i f_k\sum_{\substack{\alpha\in\N^k\\|\alpha|=i-k}}u_\alpha.\]
 Consequently,
 \[\coeff^c_{(\x,\wt z)}(J)=(\wt f_i^{\frac{c!}{c-i}}:f\in J,i<c).\]
 Fix an element $f\in J$ and an integer $i<c$. We want to show that $\wt f_i^{\frac{c!}{c-i}}\in\coeff_{(\x,z)}(J)$. To this end, define a power series $v=\sum_{k\geq0}v_k z^k$ where 
 \[v_{i-k}=\sum_{\substack{\alpha\in\N^k\\|\alpha|=i-k}}u_\alpha\]
 for $k=0,\ldots,i$ and $v_k=0$ for $k>i$. Then the power series $v\cdot f$ has the expansion 
 \[v\cdot f=\sum_{j\geq0}\sum_{k=0}^jf_kv_{j-k}z^j.\]
 In particular, the $i$-th term in the expansion of $v\cdot f$ is
 \[\sum_{k=0}^if_kv_{i-k}=\sum_{k=0}^i f_k\sum_{\substack{\alpha\in\N^k\\|\alpha|=i-k}}u_\alpha=\wt f_i.\]
 This proves that $\coeff^c_{(\x,\wt z)}(J)\subseteq\coeff^c_{(\x,z)}(J)$. By a symmetric argument, equality of the coefficient ideals holds.
\end{proof}

\begin{remark}
 The analogue of Proposition \ref{coeff_ideal_well_def} is not true when applied to the remaining parameters $\x=(x_1,\ldots,x_n)$. To see this, consider the following example:
 
 Let $R=K[[x,y,z]]$ and 
 \[J=(z^2+y(y+x)).\]
Then 
\[\coeff_{(x,y,z)}^2(J)=(y(y+x)).\] 
 
 Now consider a change of coordinates $x=x_1(1+z)$. Notice that $x$ and $x_1$ define the same residue modulo $(z)$. Then 
 \[J=(z^2+x_1yz+y(y+x_1)).\] 
 Consequently, 
 \[x_1^2y^2\in\coeff_{(x_1,y,z)}^2(J).\]
 This proves that $\coeff^2_{(x,y,z)}(J)\neq\coeff^2_{(x_1,y,z)}(J)$.
\end{remark}

\subsection{Relationship with the order function and the directrix}

 The remaining results in this section will establish the basic relationship between the coefficient ideal, the order function and the directrix.

\begin{lemma} \label{coeff_is_zero}
 Let $R=K[[\x,z]]$ and $J\subseteq R$ an ideal of order $\ord J=c$. Then $\coeff^c_{(\x,z)}(J)=0$ holds if and only if $J=(z^c)$.
\end{lemma}
\begin{proof}
 
 By the definition of the coefficient ideal it is clear that $J=(z^c)$ implies that $\coeff^c_{(\x,z)}(J)=0$.
 
 Now assume that $\coeff^c_{(\x,z)}(J)=0$. Let $f\in J$ have expansion $f=\sum_{i\geq0}f_iz^i$. Then by the definition of the coefficient ideal, $f_i=0$ for $i<c$. Hence, $J\subseteq (z^c)$. But since $\ord J=c$, it is clear that $J=(z^c)$. 
\end{proof}

\begin{lemma} \label{ord_coeff}
 Let $R=K[[\x,z]]$, $J\subseteq R$ an ideal and $c>0$ an integer. Set $J_{-1}=\coeff^c_{(\x,z)}(J)$. Let each element $f\in J$ have an expansion $f=\sum_{i\geq0}f_iz^i$ with $f_i\in K[[\x]]$. Then
 \[\ord J_{-1}=\min_{\substack{f\in J\\ i<c}}\frac{c!}{c-i}\ord f_i.\]
 Consequently, for all elements $f\in J$ and indices $i\geq0$ the inequality
 \[\ord f_i\geq \frac{c-i}{c!}\ord J_{-1}\]
 holds.
\end{lemma}
\begin{proof}
 This is immediate from the definition of the coefficient ideal.
\end{proof}

\begin{lemma} \label{ogeqc!}
 Let $R=K[[\x,z]]$, $J\subseteq R$ an ideal and $c\leq \ord J$ a positive integer. Then 
 \[\ord\coeff^c_{(\x,z)}(J)\geq c!.\]
\end{lemma}
\begin{proof}
 Let each element $f\in J$ have an expansion $f=\sum_{i\geq0}f_iz^i$ with $f_i\in K[[\x]]$. Since $\ord f\geq c$, we know that $\ord f_i\geq c-i$. Then by Lemma \ref{ord_coeff},
 \[\ord\coeff^c_{(\x,z)}(J)=\min_{\substack{f\in J\\ i<c}}\frac{c!}{c-i}\underbrace{\ord f_i}_{\geq c-i}\geq c!.\]
\end{proof}

\begin{lemma} \label{coeff_ideal_and_directrix}
 Let $R=K[[\x,z]]$ and $J\subseteq R$ an ideal of order $c=\ord J$. The following two statements are equivalent:
 \begin{enumerate}[(1)]
  \item $\ord\coeff^c_{(\x,z)}(J)>c!$.
  \item $\tau(J)=1$ and $\Dir(J)=(\ol z)$.
 \end{enumerate}
\end{lemma}
\begin{proof}
 $(2)\implies (1)$: Let each element $f\in J$ have the expansion $f=\sum_{i\geq0}f_iz^i$ with $f_i\in K[[\x]]$. It is clear that $\ord f_i\geq c-i$. Now assume that there is an index $i<c$ such that $\ord f_i=c-i$. Then $\ord f=c$ and $\init(f)\neq z^c$. But this is a contradiction to the fact that $\Dir(J)=(\ol z)$. Hence, $\ord f_i>c-i$ for all $f\in J$ and $i<c$. Thus, by Lemma \ref{ord_coeff} we conclude that $\ord\coeff^c_{(\x,z)}(J)>c!$.
 
 $(1)\implies (2)$: Assume that $\Dir(J)\neq(\ol z)$. Then there is an element $f\in J$ such that $\ord f=c$ and $\init(f)\neq z^c$. Let $f$ have the expansion $f=\sum_{i\geq0}f_iz^i$ with $f_i\in K[[\x]]$. Then there is an index $i<c$ such that $\ord f_i=c-i$. Hence, by Lemma \ref{ord_coeff} we know that $\ord\coeff_{(\x,z)}^c(J)=c!$. 
\end{proof}

\begin{lemma} \label{coeff_ideal_of_powers}
 Let $R=K[[\x,z]]$, $J\subseteq R$ an ideal and $c,n>0$ positive integers. Then
 \[\ord \coeff_{(\x,z)}^{nc}(J^n)=\frac{(nc)!}{c!}\ord \coeff_{(\x,z)}^c(J).\]
\end{lemma}
\begin{proof}
 Set $o_1=\ord\coeff_{(\x,z)}^c(J)$ and $o_n=\ord \coeff_{(\x,z)}^{nc}(J^n)$.
 
 Let each element $f_k\in J$ have an expansion $f_k=\sum_{i\geq0}f_{k,i}z^i$ with $f_{k,i}\in K[[\x]]$. Then $J^n$ is generated by elements of the form $\prod_{k=1}^n f_k$ which have the expansion
 \[\prod_{k=1}^n f_k=\sum_{i\geq0}\sum_{\substack{\alpha\in\N^n\\|\alpha|=i}}\prod_{k=1}^nf_{k,\alpha_k} z^i.\]
 
 Thus, we can compute with Lemma \ref{ord_coeff} that
 \[o_n=\min_{\substack{f_1,\ldots,f_k\in J\\i<nc}}\frac{(nc)!}{nc-i}\ord\Big(\sum_{\substack{\alpha\in\N^n\\|\alpha|=i}}\prod_{k=1}^nf_{k,\alpha_k}\Big)\]
 \[\geq \min_{\substack{f_1,\ldots,f_k\in J\\i<nc}}\min_{\substack{\alpha\in\N^n\\|\alpha|=i}}\frac{(nc)!}{nc-i}\sum_{k=1}^n\underbrace{\ord f_{k,\alpha_k}}_{\geq\frac{c-\alpha_k}{c!}o_1}=\frac{(nc)!}{c!}o_1.\]
 
 By Lemma \ref{ord_coeff} there is an element $f\in J$ with expansion $f=\sum_{i\geq0}f_iz^i$ and an index $j<c$ such that
 \[\ord f_j=\frac{c-j}{c!}o_1.\]
 Let $j$ be minimal with this property. Consequently, for all multi-indices $\alpha\in\N^n$ with $|\alpha|=nj$ and $\alpha\neq(j,\ldots,j)$ the strict inequality
 \[\ord \prod_{k=1}^n f_{\alpha_k}>n\frac{c-j}{c!}o_1\]
 holds. The element $f^n\in J^n$ has the expansion
 \[f^n=\sum_{i\geq0}\sum_{\substack{\alpha\in\N^n\\|\alpha|=i}}\prod_{k=1}^n f_{\alpha_k}z^i.\]
 Notice that
 \[\ord \sum_{\substack{\alpha\in\N^n\\|\alpha|=nj}}\prod_{k=1}^n f_{\alpha_k}=n\frac{c-j}{c!}o_1.\]
 Thus, we know by Lemma \ref{ord_coeff} that
 \[o_n\leq \frac{(nc!)}{nc-nj}\cdot n\frac{c-j}{c!}o_1=\frac{(nc)!}{c!}o_1.\]
 This proves the assertion.
\end{proof}

\section{$z$-regularity} \label{section_z_regularity}

 In this section, the simple yet powerful notion of $z$-regularity over a power series ring will be introduced and some basic properties will be proved.

\begin{definition}
 Let $R=K[[\x,z]]$. An element $f\in R$ is said to be \emph{$z$-regular of order $c$ with respect to the parameters $(\x,z)$} if
 \[f(0,\ldots,0,z)=z^c\cdot u(z)\]
 for a unit $u\in K[[z]]^*$. If it is clear from the context which parameters we consider, we will just say that $f$ is $z$-regular of order $c$. An element $f$ which is $z$-regular of order $1$ will just be called $z$-regular.
 
 Notice that an element $f\in R$ of order $\ord f=c$ with expansion $f=\sum_{i\geq0}f_iz^i$ is $z$-regular of order $c$ if and only if $\ord f_c=0$.
 
\end{definition}

 The notion of $z$-regularity is a prerequisite for the \emph{cleaning}-techniques that we will develop in Chapter \ref{chapter_cleaning}. Heuristically speaking, for an element $f\in K[[\x,z]]$ with expansion $f=\sum_{i\geq0}f_iz^i$ that is $z$-regular of order $c$, the coefficients $f_i$ for $i<c$ behave in a very controlled way under coordinate changes $z\mapsto z+g(\x)$. Recall that the coefficient ideal $\coeff_{(\x,z)}^c(J)$ of an ideal $J\subseteq K[[\x,z]]$ is generated by powers of such coefficients $f_i$ with $i<c$ for $f\in J$. The existence of an element $f\in J$ that is $z$-regular of order $c$ will give us great control over the behavior of invariants (like weighted orders) associated to the coefficient ideal of $J$ under coordinate changes $z\mapsto z+g(\x)$. In particular, we will be able to devise techniques to maximize these invariants over all such coordinate changes.

 Since having an element $f\in J$ that is $z$-regular of order $c$ is such a desirable property, we will prove in the following lemmas some basic results on the existence of such elements after generic coordinate changes in the parameters $x_i$ and the stability of $z$-regularity under coordinate changes $z\mapsto \wt z$ and blowup.

\begin{lemma} \label{z_regular_after_generic_coord_change}
 Let $R=K[[\x,z]]$ with $\x=(x_1,\ldots,x_n)$ and $f\in R$ an element of order $\ord f=c$. If $K$ is an infinite field, then there are constants $\lambda_1,\ldots,\lambda_n\in K$ such that $f$ is $z$-regular of order $c$ with respect to the parameters $(\wt\x,z)$ where $\wt \x=(\wt x_1,\ldots,\wt x_n)$ is defined by $\wt x_i=x_i-\lambda_iz$ for $i=1,\ldots,n$.
\end{lemma}
\begin{proof}
 We can write the initial form of $f$ as
 \[\init(f)=\sum_{\substack{\alpha\in\N^n\\|\alpha|\leq c}}c_\alpha\x^{\alpha}z^{c-|\alpha|}\]
 for certain constants $c_\alpha\in K$. Let $t$ denote the vector of variables $t=(t_1,\ldots,t_n)$ and set $\x_t=(x_1-t_1z,\ldots,x_n-t_nz)$. Then we can rewrite $\init(f)$ as
 \[\init(f)=\sum_{\substack{\alpha\in\N^n\\|\alpha|\leq c}}c_\alpha(\x_t+tz)^\alpha z^{c-|\alpha|}\]
 \[=\sum_{\substack{\alpha\in\N^n\\|\alpha|\leq c}}\sum_{\substack{\beta\in\N^n\\ |\beta|\leq c}}c_\alpha \binom{\alpha}{\beta}(tz)^{\alpha-\beta}\x_t^\beta  z^{c-|\alpha|}\]
 \[=\sum_{\substack{\beta\in\N^n\\|\beta|\leq c}}\underbrace{\sum_{\substack{\alpha\in\N^n\\|\alpha|\leq c}}c_\alpha \binom{\alpha}{\beta}t^{\alpha-\beta}}_{=:\wt c_{\beta}(t)}\x_t^\beta  z^{c-|\beta|}.\]
 In particular, 
 \[\wt c_{(0,\ldots,0)}(t)=\sum_{\substack{\alpha\in\N^n\\|\alpha|\leq c}}c_\alpha t^\alpha.\]
 Since $K$ is infinite and there is an index $\alpha\in\N^n$ with $|\alpha|\leq c$ such that $c_\alpha\neq0$, we can find values $\lambda=(\lambda_1,\ldots,\lambda_n)\in K^n$ such that $\wt c_{(0,\ldots,0)}(\lambda)\neq0$. This implies that $f$ is $z$-regular of order $c$ with respect to $\wt\x=(x_1-\lambda_1z,\ldots,x_n-\lambda_nz)$.
\end{proof}

\begin{lemma} \label{z_regular_under_coord_change}
 Let $R=K[[\x,z]]$ and $f\in R$ an element of order $\ord f=c$ that is $z$-regular of order $c$ with respect to the parameters $(\x,z)$. Consider one of the following changes of coordinates:
 \begin{enumerate}[(1)] 
  \item $z=\wt z+g$ with $g\in K[[\x]]$, $\ord g\geq1$.
  \item $z=u\wt z$ with $u\in R^*$ a unit.
 \end{enumerate}
 Then $f$ is also $\wt z$-regular of order $c$ with respect to the parameters $(\x,\wt z)$.
\end{lemma}
\begin{proof}
 Let $f$ have the expansions $f=\sum_{i\geq0}f_iz^i$ and $f=\sum_{i\geq0}\wt f_iz^i$ with $f_i,\wt f_i\in K[[\x]]$. We need to show that $\ord\wt f_c=0$.
 
 (1): By Lemma \ref{coordinate_change_g_z} (1) we know that
 \[\wt f_c=f_c+\sum_{k>c}\binom{k}{c}f_kg^{k-c}.\]
 Since $\ord f_c=0$ and $\ord g\geq1$, it follows immediately that $\ord \wt f_c=0$.
 
 (2): Let $u$ have the expansion $u=\sum_{j\geq0}u_jz^j$ with $u_j\in K[[\x]]$. By Lemma \ref{coordinate_change_g_z} (2) we know that
 \[\wt f_c=f_cu_0^c+\sum_{k=0}^{c-1}f_k\sum_{\substack{\alpha\in\N^k\\|\alpha|=c-k}}u_\alpha.\]
 Since $u\in R^*$, we know that $\ord u_0=0$. Further, $\ord f_k\geq c-k>0$ for $k<c$ since $\ord f=c$. Thus, $\ord\wt f_c=0$.
\end{proof}

\begin{lemma} \label{z_regular_stable_under_blowup}
 Let $R=K[[\x,z]]$ with $\x=(x_1,\ldots,x_n)$ and $f\in R$ an element of order $\ord f=c$ which is $z$-regular of order $c$.
 
 Let $k\leq n$ be an index such that $\ord_{(x_1,\ldots,x_k,z)}f=c$. Consider the blowup-map $\pi:R\to R$ with center $(x_1,\ldots,x_k,z)$ in the $x_1$-chart given by
 \[\begin{array}{ll}
 \pi(x_1)=x_1, &\\
 \pi(x_i)=x_1(x_i+t_i) & \text{for $1<i\leq k$,}\\ 
 \pi(x_j)=x_j & \text{for $j>k$,} \\
 \pi(z)=x_1z &
\end{array}\]
 for certain constants $t_2,\ldots,t_k\in K$.
 
 Set $f'=x_1^{-c}\pi(f)$ and assume that $\ord f'=c$. Then $f'$ is again $z$-regular of order $c$.
\end{lemma}
\begin{proof}
 It is clear that $f'$ has the expansion $f'=\sum_{i\geq0}f_i'z^i$ with $f_i\in K[[\x]]$ of the form $f_i'=x_1^{i-c}\pi(f_i)$. In particular, $f_c'=\pi(f_c)$. Since $\ord f_c=0$, we know that $\ord f_c'=\ord \pi(f_c)=0$. Hence, $f'$ is $z$-regular of order $c$.
\end{proof}

 The following two lemmas will establish the relationship between the notion of $z$-regularity and the directrix of an ideal.

\begin{lemma} \label{z_regularity_from_directrix}
 Let $R=K[[\x,z]]$ and $J\subseteq R$ an ideal of order $c=\ord J$. Assume that $\tau(J)=1$ and $\Dir(J)=(\ol z)$. Let $f\in J$ be an element of order $\ord f=c$. Then $f$ is $z$-regular of order $c$.
\end{lemma}
\begin{proof}
 Consider the expansion $f=\sum_{i\geq0}f_iz^i$ with $f_i\in K[[\x]]$. Since $\ord f=c$ and $\Dir(J)=(\ol z)$, we know that $\init(f)=\lambda z^c$ for some constant $\lambda\in K^*$. Consequently, $\ord f_c=0$. Thus, $f$ is $z$-regular of order $c$.
\end{proof}

\begin{lemma} \label{directrix_from_z_regularity}
 Let $R=K[[\x,z]]$ with $\x=(x_1,\ldots,x_n)$ and $J\subseteq R$ an ideal of order $c=\ord J$. Let $f\in J$ be an element that is $z$-regular of order $c$ with respect to the parameters $(\x,z)$. Then there is a linear combination $g=\sum_{i=1}^nc_ix_i$ with $c_i\in K$ such that $\ol{z+g}\in\Dir(J)$.
\end{lemma}
\begin{proof}
 Assume that $\ol{z+g}\notin\Dir(J)$ for all $g$ of the form $g=\sum_{i=1}^nc_ix_i$. This implies that $\Dir(J)\subseteq (\ol x_1,\ldots,\ol x_n)$. But this implies that $\init(f)\in (x_1,\ldots,x_n)^c$. This contradicts the fact that $f$ is $z$-regular of order $c$ with respect to $(\x,z)$.
\end{proof}

 As mentioned before, the existence of an element $f\in J$ that is $z$-regular of order $c$ gives us great control over the behavior of weighted orders of the coefficient ideal $\coeff^c_{(\x,z)}(J)$ under coordinate changes $z\mapsto z+g$. On the other hand, the coefficient ideal behaves largely chaotic under a coordinate change that permutes $z$ with one of the parameters $x_i$. The following Lemma ensures that we still have some control over weighted orders of the coefficient ideal under such coordinate changes under the condition that there is an element $f\in J$ which is $z$-regular of order $c$.

\begin{lemma} \label{z_regular_blocks_other_parameters}
 Let $R=K[[\x,y,z]]$ with $\x=(x_1,\ldots,x_n)$, $J\subseteq R$ an ideal of order $c=\ord J$ and $f\in J$ an element which is $z$-regular of order $c$. Consider the coefficient ideal $J_{-1}=\coeff^c_{(\x,z,y)}(J)\subseteq K[[\x,z]]$ with respect to the hypersurface $V(y)$. Let $\w:K[[\x,z]]\to\Ni^k$ be a weighted order function defined on $(\x,z)$.
 
 Then $\w(J_{-1})\leq c!\cdot\w(z)$.
\end{lemma}
\begin{proof}
 Let $f$ have the expansion $f=\sum_{i,j\geq0}f_{j,i}y^iz^j$ with $f_{j,i}\in K[[\x]]$. Then $\ord f_{c,0}=0$. Now consider the expansion $f=\sum_{i\geq0}f_iy^i$ with $f_i=\sum_{j\geq0}f_{j,i}z^j$.
 
 By definition of the coefficient ideal, $f_0^{\frac{c!}{c}}\in J_{-1}$. Thus, $\w(J_{-1})\leq \frac{c!}{c}\w(f_0)$. Furthermore, we know that 
 \[\w(f_0)\leq \w(f_{c,0}z^c)=c\cdot\w(z)\]
 since $f_{c,0}$ is a unit. Thus, $\w(J_{-1})\leq c!\cdot \w(z)$.
\end{proof}

\section{Binomial coefficients over fields of positive characteristic} \label{section_binom_in_pos_char}

In this section we will briefly recall some basic facts about binomial coefficients over fields $K$ of positive characteristic.

The biggest difference from the situation in characteristic zero is the fact that binomial coefficients $\binom{n}{k}$ can vanish even if $n\geq k$ holds. Surprisingly, this vanishing of binomial coefficients is the only reason why fields of positive characteristic will have to be treated differently from fields of characteristic zero in this thesis.

The central result in this section is the classical theorem of Lucas.

\begin{proposition} \label{lucas}
 Let $p$ be a prime number and $n,k\in\N$ non-negative integers. Let $n,k$ have the $p$-adic expansions $n=\sum_{i\geq0}n_ip^i$ and $k=\sum_{i\geq0}k_ip^i$ where $0\leq n_i,k_i<p$. Then
 \[\binom{n}{k}\equiv\prod_{i\geq0}\binom{n_i}{k_i}\pmod p.\]
\end{proposition}
\begin{proof}
 The identity can be verified by considering the binomial coefficients as coefficients in the expansion of $(1+x)^n\in\mathbb{Z}[x]$:
 \[\sum_{k=0}^n\binom{n}{k}x^k=(1+x)^n=(1+x)^{\sum_{i\geq0}n_ip^i}=\prod_{i\geq0}(1+x)^{n_ip^i}\]
 \[\equiv\prod_{i\geq0}\Big(1+x^{p^i}\Big)^{n_i}\pmod p\]
 \[=\prod_{i\geq0}\sum_{k_i=0}^{n_i}\binom{n_i}{k_i}x^{k_ip^i}=\sum_{k=0}^n\prod_{i\geq0}\binom{n_i}{k_i}x^k.\]
\end{proof}

Lucas' theorem has the following immediate application to the theory of differential operators in positive characteristic:

\begin{lemma} \label{diffops_on_p_th-powers}
 Let $R=K[[\x]]$ where $K$ is a field of characteristic $\chara(K)=p>0$. Set $q=p^e$ for a positive integer $e>0$. Consider a multi-index $\alpha\in\N^n$ with $\alpha\notin q\cdot \N^n$.
 
 Then $\partial_{\x^\alpha}(f^q)=0$ for all elements $f\in R$.
\end{lemma}
\begin{proof}
 It suffices to verify the statement for a monomial $f^q=\x^\gamma$ where $\gamma=(\gamma_1,\ldots,\gamma_n)\in q\cdot\N^n$. There exists an index $1\leq i\leq n$ such that $\alpha_i\notin q\cdot \N$. By Lemma \ref{lucas} it is clear that
 \[\partial_{\x^\alpha}(\x^\gamma)=\binom{\gamma_1}{\alpha_1}\cdots \underbrace{\binom{\gamma_i}{\alpha_i}}_{=0}\cdots\binom{\gamma_n}{\alpha_n}\x^{\gamma-\alpha}=0\]
\end{proof}

We will now define the number $q_K(c)$ which will allow us in many contexts to treat the cases of characteristic zero and positive characteristic uniformly.

\begin{definition}
 Let $K$ be a field and $c>0$ a positive integer. Set 
 \[q_K(c)=\begin{cases}
           1 & \text{if $\chara(K)=0$,} \\
           p^{\ord_pc} & \text{if $\chara(K)=p>0.$}
          \end{cases}\]
 If it is clear from the context which field $K$ and which integer $c$ is considered, we will always just write $q$ for $q_K(c)$.
\end{definition}

 The significance of the number $q_K(c)$ will be made clear in the following lemma.

\begin{lemma} \label{c-q}
 Let $p$ be a prime number and $c\in\N$ a non-negative integer. Set $q=p^{\ord_pc}$. Let $k$ be an integer in the range $c-q<k<c$. Then the following hold:
 \begin{enumerate}[(1)]
  \item $\binom{c}{q}=\binom{c}{c-q}\not\equiv0\pmod p$.
  \item $\binom{c}{k}\equiv0\pmod p$.
  \item $\binom{k}{c-q}\equiv1\pmod p$.
 \end{enumerate}
\end{lemma}
\begin{proof}
 Let $c$ have the $p$-adic expansion $c=\sum_{i\geq0}c_ip^i$ with $0\leq c_i<p$. Set $e=\ord_pc$. Thus, $c=\sum_{i\geq e}c_ip^i$ and $q=p^e$.
 
 (1): Since $e=\ord_pc$, we know that $c_e\neq0$. Consequently, 
 \[\binom{c}{q}\equiv\binom{c_e}{1}\cdot\prod_{i>e}\binom{c_i}{0}=c_e\not\equiv0\pmod p\]
 by Proposition \ref{lucas}.
 
 (2): Since $c-q<k<c$, we know that $e\geq1$. It is clear that $k$ has the $p$-adic expansion 
 \[k=\sum_{i<e}k_ip^i+(c_e-1)p^e+\sum_{i>e}c_ip^i\]
 for certain $k_i$ with $0\leq k_i<p$. Further, there is an index $0\leq i<e$ such that $k_i\neq0$. Consequently,
 \[\binom{c}{k}\equiv\underbrace{\prod_{i<e}\binom{0}{k_i}}_{=0}\cdot\binom{c_e}{c_e-1}\cdot\prod_{i>e}\binom{c_i}{c_i}=0\pmod p\]
 by Proposition \ref{lucas}.
 
 (3): Using the $p$-adic expansion of $k$ from above, we compute that
 \[\binom{k}{c-q}\equiv\prod_{i<e}\binom{k_i}{0}\cdot \binom{c_e-1}{c_e-1}\cdot \prod_{i>e}\binom{c_i}{c_i}=1\pmod p\]
 by Proposition \ref{lucas}.
\end{proof}

 Although it does not explicitly involve binomial coefficients, the following result belongs to the same class of arithmetic phenomena that are specific to characteristic $p>0$.

\begin{lemma} \label{char_p_sum_lemma}
 Let $R$ be an integral domain of characteristic $p>0$. Set $q=p^e$ for a non-negative integer $e\in\N$. Further, let $g_j\in R$ be elements for $j\in\N$. For a multi-index $\alpha\in\N^q$ set $g_\alpha=\prod_{o=1}^q g_{\alpha_o}$. Let $k\geq0$ be a non-negative integer. Then
 \[\sum_{\substack{\alpha\in\N^q\\|\alpha|=k}}g_\alpha=\begin{cases}
                                                         g_j^q & \text{if $k=jq$ for an integer $j\in\N$,}\\
                                                         0 & \text{if $k$ is not divisible by $q$.}
                                                        \end{cases}
\]
\end{lemma}
\begin{proof}
 Consider the set $X=\{\alpha\in\N^q:|\alpha|=k\}$. The map $\sigma:X\to X$, $\sigma(\alpha)=(\alpha_2,\ldots,\alpha_q,\alpha_1)$ is a bijection with the property that $\sigma^q=\textnormal{id}_X$. Further, $g_{\sigma(\alpha)}=g_\alpha$ for all $\alpha\in X$. If $k=jq$ for an integer $j\in\N$, then $(j,\ldots,j)$ is the only fixed point of $\sigma$. If $k$ is not divisible by $q$, the map $\sigma$ has no fixed point. Thus, there is a subset $Y\subseteq X$ such that
 \[\sum_{\substack{\alpha\in\N^q\\|\alpha|=k}}g_\alpha=\sum_{\alpha\in X}g_\alpha=\begin{cases}
                                                         g_j^q+q\cdot\sum_{\alpha\in Y}g_\alpha & \text{if $k=jq$ for an integer $j\in\N$,}\\
                                                         q\cdot\sum_{\alpha\in Y}g_\alpha & \text{if $k$ is not divisible by $q$.}
                                                        \end{cases}.\]
 The statement follows from the fact that $q\equiv0\pmod p$ if $e\geq1$ and $Y=\emptyset$ if $e=0$.
\end{proof}

\begin{remark}
There is some notational ambiguity since $\binom{n}{k}$ will be used both to denote the binomial coefficient of $n$ over $k$ as an integer and as an element of a field $K$. In particular, the statement $\binom{n}{k}=0$ is ambiguous. As a general rule, binomial coefficients that appear as coefficients of power series over a field $K$ will always be considered as elements of $K$.
\end{remark}

\chapter{Obstructions in positive characteristic and our strategy for the surface case} \label{chapter_pathologies}

The purpose of this chapter is twofold: First we are going to discuss the well-known pathologies that appear when trying to generalize the usual resolution invariant from characteristic zero to the setting of positive characteristic by replacing hypersurfaces of maximal contact with regular formal hypersurfaces which maximize certain invariants that are associated to the coefficient ideal. To this end, we will begin this chapter by defining the \emph{residual order}, a characteristic independent invariant that serves as a refinement of the order function. As we will see, the residual order does not constitute a feasible resolution invariant over fields of positive characteristic. It is not upper semicontinuous and it can increase under blowup, even if the center only consists of a closed point.

The second purpose of this chapter is to outline how these problems can be overcome in the surface case. We will describe a modification of the residual order which will be used in the later chapters as a resolution invariant for the embedded resolution of surface singularities. A number of examples will be provided which serve to motivate the specific definitions we make.

The arguments and definitions throughout this chapter are largely heuristic in nature. For results that are relevant to our proof for the embedded resolution of surface singularities, we will give references to the later chapters in which proofs are given. The resolution invariant that is described in this chapter will be defined rigorously in Chapter \ref{chapter_invariant}. 




\section{The residual order} \label{section_residual_order}

 As mentioned in Section \ref{section_coeff_ideals}, the order of the coefficient ideal is not a suitable resolution invariant since its order might increase under blowup. In characteristic zero, this problem is solved by factoring the equations of exceptional components stemming from previous blowups from the coefficient ideal. Hence, we consider a factorization of the coefficient ideal $J_{-1}$ of the form $J_{-1}=M\cdot I$ where $M$ is a principal monomial ideal of the form $M=(\prod x_i^{r_i})$ where $r_i=\ord_{(x_i)}J_{-1}$ and each $x_i$ locally defines an exceptional component. While there is no common name for this order of the ideal $I$, it will be called \emph{residual order} in this thesis. The name was introduced by Hironaka in \cite{Hironaka_CMI}. 
 
 Over fields of characteristic zero, the residual order with respect to a hypersurface of maximal contact is upper semicontinuous along the strata defined by the order function. It also behaves well under blowup in the sense that it does not increase if the center of blowup is chosen sufficiently small. In proofs of embedded resolution of singularities over fields of characteristic zero, the residual order is used as a refinement of the order function and appears as a component of the resolution invariant. In this section, we will give a definition of the residual order which is independent of the characteristic and does not use hypersurfaces of maximal contact.

\subsection{Consequences of the failure of maximal contact}

 The fact that hypersurfaces of maximal contact do not exist over fields of positive characteristic has serious consequences for proving resolution of singularities over these fields. In particular, measuring improvement at an equiconstant points via descent in dimension is a lot more difficult over fields of positive characteristic.

 While hypersurfaces of maximal contact provide a natural choice for the hypersurface used for the descent in dimension in characteristic zero, there is no obvious choice in the situation of positive characteristic. Our approach to remedy this drawback is as follows: Instead of only working with hypersurfaces of maximal contact (which might not even exist), we will consider the set of \emph{all} regular formal hypersurfaces. For each of these hypersurfaces, we consider the associated coefficient ideal and assign certain invariants to it. The hypersurface that we use for the descent in dimension will be the one that maximizes these invariants. This technique will be described in much more detail in the following sections. The idea of using hypersurfaces that maximize certain invariants for the descent in dimension was originally introduced by Abhyankar \cite{Abhyankar_Plane_Curves} and has been used under the name \emph{hypersurfaces of weak maximal contact} in \cite{EH}, \cite{Ha_BAMS_2}. 
 
 Another problem that the failure of maximal contact over fields of positive characteristic entails is that the hypersurface used for the descent in dimension will inevitably have to be changed during the resolution process. No matter how well we choose the hypersurface $H$ for the descent in dimension, we know that there might be a sequence of order-permissible blowups and a sequence $a^{(m)},\ldots,a^{(1)}$ of equiconstant points lying over $a$ such that either $a^{(m)}$ is not contained in the $m$-th strict transform of $H$ or the strict transform of $H$ becomes singular. Thus, a new hypersurface has to be used for the descent in dimension at the point $a^{(m)}$. Consequently, the invariants that we define need to exhibit a certain robustness that makes them work well even in the case when the hypersurface used for the descent in dimension has to be changed after blowup.

\subsection{The residual order with respect to a formal hypersurface}

 Since the residual order is defined with respect to the exceptional components produced by previous blowups, let us assume that we are already in the middle of the resolution process. Let $W$ be a regular variety, $X\subseteq W$ a closed subset and $E$ a simple normal crossings divisor on $W$ which is defined as the collection of exceptional components. For simplicity's sake, we will assume throughout this chapter that $X$ is a hypersurface. Let $a\in X$ be a closed point of order $c=\ord_aX$. Let $E_a$ be defined as the union of those components of $E$ which contain $a$ and which were created by the most recent blowups in the following sense: Let
 \[W\overset{\pi_1}{\longrightarrow} W^{(-1)}\overset{\pi_2}{\longrightarrow}\cdots\overset{\pi_m}{\longrightarrow} W^{(-m)}\]
 be the sequence of blowups which have been performed in the resolution process so far. Set $a^{(-i)}=(\pi_i\circ\cdots\circ\pi_1)(a)\in W^{(-i)}$. Assuming that all blowups $\pi_i$ are order-permissible, there is a maximal number $k$ such that
 \[\ord_aX=\ord_{a^{(-1)}}X^{(-1)}=\ldots=\ord_{a^{(-k)}}X^{(-k)}.\]
 We then define $\Ea$ as the union of those components of $E$ that were produced by the blowups $\pi_1,\ldots,\pi_k$ and contain the point $a$.
 
 A regular formal hypersurface $H\subseteq\Spec(\hOWa)$ at $a$ is said to be \emph{compatible} with $\Ea$ if $H\cup \Ea$ has simple normal crossings at $a$ and $H\not\subseteq\Ea$. (This guarantees that the intersection $H\cap \Ea$ again has simple normal crossings and has the same number of components as $\Ea$.) Denote by $\HH$ the set of all regular formal hypersurfaces $H\subseteq\Spec(\hOWa)$ that are compatible with $\Ea$. The dependence of $\HH$ on the point $a$ will be suppressed in this notation. It can be verified from the definition of $E_a$ that $\HH$ is non-empty.
 
 Let $H\in\HH$ be a regular formal hypersurface at $a$ that is compatible with $\Ea$. Then there exists a regular system of parameters $(\x,z)=(x_1,\ldots,x_n,z)$ for $\hOWa$ such that $H=V(z)$ and $\Ea=V(\prod_{i\in\Delta}x_i)$ for some set $\Delta\subseteq\{1,\ldots,n\}$. Such parameters $(\x,z)$ are said to be \emph{subordinate} to $H$ and $\Ea$. Let the ideal $J_{n,(\x,z)}$ be defined as the coefficient ideal 
 \[J_{n,(\x,z)}=\coeff^c_{(\x,z)}(\hIXa).\]
 Further, consider the factorization of ideals
 \[J_{n,(\x,z)}=M_{n,(\x,z)}\cdot I_{n,(\x,z)}\]
 where $M_{n,(\x,z)}=(\prod_{i\in\Delta}x_i^{r_i})$ with $r_i=\ord_{(x_i)}J_{n,(\x,z)}$. We will show in Proposition \ref{residual_order_coord_indep} that the numbers 
 \[m_H=\ord M_{n,(\x,z)}\text{ and }d_H=I_{n,(\x,z)}\]
 are independent of the choice of subordinate parameters $(\x,z)$. Thus, they are invariants of the formal hypersurface $H$. (They also depend on the divisor $\Ea$, but this will be suppressed in the notation.) The number $d_H$ will be called the \emph{residual order} of $X$ at $a$ with respect to the formal hypersurface $H$. 
 
 Different choices of regular hypersurfaces $H$ will produce different values for $m_H$ and $d_H$. This is illustrated by the following example:

 \begin{example}
  Consider the situation $\hOWa\cong K[[x,y,z]]$ where $K$ is a field of characteristic $2$. Let $\hIXa$ be generated by an element $f$ of the form
  \[f=z^2+x^2y^2(y^3+x^7),\]
  and let $\Ea=V(xy)$. We will consider the three regular hypersurfaces $H_1=V(z)$, $H_2=V(z+xy)$ and $H_3=V(z+x^4+y^4)$. Notice that all three are compatible with $\Ea$.
  
  For the hypersurface $H_1$, it is clear that $m_{H_1}=4$ and $d_{H_1}=3$.
  
  For the hypersurface $H_2$, consider the expansion
  \[f=(z+xy)^2+x^2y^2(1+y^3+x^7).\]
  Thus, we see that $m_{H_2}=4$ and $d_{H_2}=0$.
  
  For the hypersurface $H_3$, consider the expansion
  \[f=(z+x^4+y^4)^2+x^8+y^8+x^2y^5+x^9y^2.\]
  Thus, $m_{H_3}=0$ and $d_{H_3}=7$.
 \end{example}

 Although all three hypersurfaces are adjacent to $X$ at $a$ in the above example, they do not all determine a residual order that is useful for measuring an improvement of $X$ under blowup. The hypersurface $H_2$ determines the minimal value $d_{H_2}=0$. Clearly, this is not a suitable value for measuring improvement. On the other hand, the hypersurface $H_3$ determines the value $d_{H_3}=7$, which is the maximal value for $d_H$ over all formal hypersurfaces $H\in\HH$. But since $m_{H_3}=0$, this value coincides with the order of the coefficient ideal with respect to $H_3$. As we already discussed, the order of the coefficient ideal is not a suitable resolution invariant since it might increase under blowup.
 
 In particular, we see that neither the minimal nor the maximal value of $d_H$ are necessarily a significant invariant for measuring the complexity of $X$ at $a$.

 \subsection{Valid and maximizing hypersurfaces}

 To find a regular formal hypersurface $H$ that determines a significant value for $d_H$, we will use $m_H$ as an auxiliary invariant that has to be maximized first before maximizing $d_H$.
 
 We say that a hypersurface $H\in\HH$ is \emph{valid} if it fulfills $m_{\wt H}\leq m_H$ for all hypersurfaces $\wt H\in\HH$. Further, we say that a valid hypersurface $H\in\HH$ \emph{realizes the residual order} of $X$ at $a$ if it fulfills $d_{\wt H}\leq d_H$ for all valid hypersurfaces $\wt H\in\HH$. In other words, a hypersurface $H\in\HH$ realizes the residual order if it lexicographically maximizes the pair $(m_H,d_H)$ over all hypersurfaces in $\HH$. Using the cleaning techniques of Section \ref{section_w_cleaning}, it is easy to show that there always exists a hypersurface which realizes the residual order. The residual order (with respect to $\Ea$) of $X$ at $a$ is then defined as 
 \[\resord_aX=d_{H_{\max}}\]
 where $H_{\max}\in\HH$ realizes the residual order.
 
 Over fields of characteristic zero, there are regular hypersurfaces which have maximal contact with $X$ at $a$ and also realize the residual order. Moreover, their strict transforms at equiconstant points again realize the residual order of the strict transform $X'$ of $X$. As we will see later, this ensures that the residual order behaves well under blowups over fields of characteristic zero. A formal hypersurface with such properties can be constructed via Tschirnhausen transformation. 
 
 Notice that we defined the residual order of $X$ only at closed points $a$. A simple way to extend its definition to non-closed points $\xi\in X$ is to set
 \[\resord_\xi X=\min\{\resord_aX:a\in\ol{\{\xi\}}\text{ is a closed point}\}.\]
 
 In the following sections, we will investigate how well the residual order is suited as a refinement of the order function. Hence, we will study the behavior of the invariant $(\ord X,\resord X):X\to\Ni^2$ under blowup, where $\Ni^2$ is considered with the lexicographic order.

 \subsection{The monomial case}

 Naturally, the residual order is not able to measure improvement under blowup anymore once it has been lowered to zero. Still, $\resord_aX=0$ does not imply that the order function will necessarily decrease under the next blowup. Instead, specific centers of blowup have to be chosen and a different invariant has to be employed to measure further improvement until the order function decreases.

 Since $\resord_aX=0$ implies that the coefficient ideal $J_{n,(\x,z)}$ associated to any valid hypersurface $H\in\HH$ coincides with the principal monomial ideal $M_{n,(\x,z)}=(\prod_{i\in\Delta}x_i^{r_i})$, this situation is usually called the \emph{monomial case}. To resolve the monomial case over fields of characteristic zero, a process called \emph{combinatorial resolution} is used to successively lower the exponents $r_i$ in the coefficient ideal. It owes its name to the fact that the centers of blowup can be chosen in a purely combinatorial way from the exponents $r_i$. As soon as $\sum_{i\in\Delta}r_i<c!$ is achieved, we know that the order function has decreased by Lemma \ref{ogeqc!}. Combinatorial resolution over fields of characteristic zero is described in more detail in \cite{BM_Canonical_Desing} and \cite{EV_Good_Points}.
 
 Thus, our strategy for ultimately lowering the order function consists of two steps:
 \begin{enumerate}[(1)]
  \item Lower the residual order until it becomes zero. In other words, monomialize the coefficient ideal $J_{n,(\x,z)}$.
  \item Use combinatorial resolution to lower the exponents $r_i$ in $M_{n,(\x,z)}=(\prod_{i\in\Delta}x_i^{r_i})$ and eventually make the order function decrease.
 \end{enumerate}
 
 The monomial case is also called a \emph{terminal case} since it is the final stage in the resolution process before the order function decreases. If the centers are chosen correctly, the monomial case is stable under blowup. In this thesis we will also make use of a second terminal case which is called the \emph{small residual case}. It will be discussed in the later sections of this chapter.

 Again, the absence of hypersurfaces of maximal contact poses a potential threat for applying combinatorial resolution over fields of positive characteristic. As it will be shown later though, at least in the surface case, we will be able to use combinatorial resolution independently of the characteristic to lower the order function once we reach the monomial case. Thus, although we will see that the residual order has several deficiencies over fields of positive characteristic, $\resord_aX=0$ is still a sufficiently strong property to show that the order function can be lowered via combinatorial resolution.

\section{Behavior of the residual order under blowup} \label{section_residual_order_under_blowup}

In this section we will see that the behavior of the residual order under blowup greatly depends on the characteristic of the ground field. In particular, we will see an example over a field of positive characteristic in which the residual order increases under a point-blowup.

\subsection{The situation over fields of characteristic zero}

Let $W$, $X$, $a$, $\Ea$ and $\HH$ be defined as in the previous section. Further, consider a blowup $\pi:W'\to W$ of $W$ at a center $Z$ that is permissible with respect to the order function and has the property that $Z\cup E$ has simple normal crossings. Let $X'$ be the strict transform of $X$ under $\pi$. Let $a'\in\pi^{-1}(a)$ be an equiconstant point. Let $\Eap$ be defined as $\Eap=\Ea\st\cup\Dnew$ where $\Ea\st$ denotes the strict transform of $\Ea$ and $\Dnew=\pi^{-1}(Z)$ the exceptional divisor of $\pi$. We want to compare the residual order of $X$ with respect to $\Ea$ at $a$ with the residual order of $X'$ with respect to $\Eap$ at $a'$.

Let $H\in\HH$ be a formal hypersurface that is adjacent to $X$ at $a$ and locally contains $Z$. Thus, the point $a'$ lies on the strict transform $H'$ of $H$. It is easy to see that $H'$ is again compatible with $\Eap$. Let $(\x,z)=(x_1,\ldots,x_n,z)$ be a regular system of parameters for $\hOWa$ that is subordinate to $H$ and $\Ea$. Then $a'$ is not contained in the $z$-chart. Let $(\x',z')$ be the induced parameters for $\hOWap$ in some $x_i$-chart. Then the parameters $(\x',z')$ are again subordinate to $H'$ and $\Eap$. Further, the ideal $I_{n,(\x',z')}$ contains the the weak transform of the ideal $I_{n,(\x,z)}$ by Lemma \ref{coeff_ideal_under_blowup}. Thus, we know by Proposition \ref{ord_under_blowup} that $d_{H'}\leq d_H$ holds as long as the center $Z$ is small enough so that the order of the ideal $I_{n,(\x,z)}$ is locally constant along $Z$. (In the following, we will always tacitly that the center $Z$ fulfills this condition.)

Hence, we can make the following conclusion: If there exists a formal hypersurface $H\in\HH$ that is adjacent, realizes the residual order of $X$ at $a$ and has the property that its strict transform $H'$ at $a'$ realizes the residual order of $X'$ at $a'$, then $\resord_{a'}X'\leq\resord_aX$ holds. As we have already discussed, such a hypersurface $H$ exists over fields characteristic zero and $H$ can even be chosen as a hypersurface of maximal contact. Over fields of positive characteristic, the situation is much more delicate.

\subsection{Positive characteristic: The purely inseparable equation}

\
 To investigate how the residual order behaves under blowup over fields of positive characteristic, we first have to explain how to find a formal hypersurface $H\in\HH$ that realizes the residual order. In this section, we will only consider the case in which $X$ is a hypersurface that is locally at $a$ defined by an element $f$ of the form
 \[f=z^{p^e}+F(\x)\]
 where $p>0$ is the characteristic of the ground field, $e\geq1$ a positive integer and $F\in K[[\x]]$ a power series of order $\ord F>p^e$. Further, we assume that $\Ea\subseteq V(\prod_{i=1}^nx_i)$. A power series of this form will be called \emph{purely inseparable}. Similarly, the equation $z^{p^e}+F(\x)=0$ is called a purely inseparable equation. The reason why we restrict ourselves to this sort of equation is that it already exhibits all relevant pathologies.
 
 Notice that a change of coordinates $z=\wt z+g$ with $g\in K[[\x]]$ changes the expansion of $f$ into
 \[f=\wt z^{p^e}+g(\x)^{p^e}+F(\x).\]
 Consequently, it is possible to eliminate all $p^e$-th powers that appear in the expansion of $F$ via such a coordinate change. We say that $f$ is \emph{clean} with respect to the parameters $(\x,z)$ if no $p^e$-th powers appear in the expansion of $F(\x)$. Notice that the property of $f$ being clean is not only dependent on the parameter $z$, but also on $\x$. Thus, $f$ being clean is not an intrinsic property of the formal hypersurface $V(z)$. The change of coordinates $z\mapsto z+g(\x)$ which eliminates all $p^e$-th powers from the expansion of $F(\x)$ is referred to as \emph{cleaning} of $f$ with respect to the parameters $\x$.
 
 We will show in Section \ref{section_w_cleaning} that if the parameters $(\x,z)$ are subordinate to a formal hypersurface $H=V(z)\in\HH$ and $\Ea$ and the element $f$ is clean with respect to $(\x,z)$, then the hypersurface $H$ realizes the residual order and consequently, $\resord_aX=d_H$. In Chapter \ref{chapter_cleaning}, we will discuss how to find hypersurfaces which realize the residual order for more general equations that are not necessarily of purely inseparable form.
 
 We will now investigate how the property of being clean behaves under blowup. Thus, we consider again a blowup $\pi:W'\to W$ along a center $Z$ that is locally contained in $H$, is permissible with respect to the order function and has the property that $Z\cup E$ has simple normal crossings. Let $X'$, $a'$ and $\Eap$ be defined as above. Since we assumed that $\ord F>p^e$, we know that $H=V(z)$ is adjacent. Hence, $a'$ is not contained in the $z$-chart. Denote the induced parameters for $\hOWap$ in some $x_i$-chart by $(\x',z')$. Then  $X'$ is defined locally at $a'$ by an element $f'$ of the form
 \[f'=(z')^{p^e}+F'(\x').\]
 Further, if $f$ is clean with respect to $(\x,z)$ and the local blowup map $\hOWa\to\hOWap$ is monomial with respect to these parameters, then it is easy to see that $f'$ is again clean with respect to the induced parameters $(\x',z')$. Since this implies that both $H=V(z)$ and its strict transform $H'=V(z')$ realize the residual order, we know that $\resord_{a'}X'\leq\resord_aX$ holds in this case.
 
 Consequently, we know that $\resord_{a'}X'\leq\resord_aX$ holds if we can choose the parameters $\x$ in such a way that they are subordinate to $\Ea$ and that the local blowup map $\hOWa\to\hOWap$ is monomial in an $x_i$-chart. Such parameters $\x$ exist if and only if at most one component of $\Ea$ is \emph{lost} in the transition from $a$ to $a'$. Here we say that a component $D$ of $\Ea$ is lost if $a'$ is not contained in its strict transform $D'$.
 
 If two or more components of $E$ are lost in the transition from $a$ to $a'$, the residual order might increase. The following example for this phenomenon is due to Hauser \cite{Ha_BAMS_2}.
 
 \begin{example}\
  Let $X$ be defined at $a$ by the purely inseparable polynomial
  \[f=z^2+xy(x^2+y^2)\]
  over a field $K$ of characteristic $2$ and set $\Ea=V(xy)$. Since $f$ is clean with respect to the parameters $(x,y,z)$, we know that the hypersurface $H=V(z)$ realizes the residual order of $X$ at $a$. Consequently, $\resord_aX=d_H=2$.
  
  Now consider the blowup $\pi:W'\to W$ with center the point $a$. Let $a'$ have the affine coordinates $(1,0)$ in the $x$-chart. Thus, both components of $\Ea$ are lost in the transition from $a$ to $a'$. Denote the induced parameters for $\hOWap$ again by $(x,y,z)$. Then $X'$ is defined at $a'$ by the purely inseparable polynomial
  \[f'=z^2+x^2(y+1)y^2\]
  \[=z^2+x^2(y^2+y^3).\]
  Since $\Eap=V(x)$, we know that $d_{H'}=2$ holds for the strict transform $H'=V(z)$ of $H$. But $f'$ is not clean with respect to the induced parameters and in fact, $H'$ does not realize the residual order of $X'$ at $a'$. Applying a coordinate change $z_1=z+xy$ changes the expansion to
  \[f'=z_1^2+x^2\cdot y^3.\]
  Since $f'$ is clean with respect to the parameters $(x,y,\wt z)$, we know that the regular formal hypersurface $\wt H=V(z_1)\subseteq\Spec(\hOWap)$  realizes the residual order of $X'$ at $a'$ and consequently, $\resord_{a'}X'=d_{\wt H}=3$. Hence,
  \[\resord_{a'}X'>\resord_aX\]
  holds.
 \end{example}
 
 Since the residual order can increase under blowup while the order function remains constant, we know that the pair $(\ord X,\resord X)$ does not constitute a feasible resolution invariant over fields of positive characteristic.
 
 It is important to point out that, although the residual order might increase under blowup, it cannot increase arbitrarily. Moh proved in \cite{Moh} that for a purely inseparable equation, the increase of the residual order is bounded by
 \[\resord_{a'}X'\leq \resord_aX+(p^e-1)!\cdot p^{e-1}=\resord_aX+\frac{(p^e)!}{p}.\]
 In this formula, the factor $(p^e-1)!$ only appears due to our definition of the coefficient ideal since 
 \[\coeff^{p^e}_{(\x,z)}(z^{p^e}+F(\x))=(F(\x))^{(p^e-1)!}.\]
 In Moh's notation, the bound reads $d'\leq d+p^{e-1}$.
 
 Hauser introduced the name \emph{Kangaroo points} for equiconstant points $a'$ at which the residual order increases. Further information on this phenomenon and necessary conditions for the increase of the residual order to happen can be found in \cite{Ha_BAMS_2}.

\section{Modifying the residual order in the surface case} \label{section_modifying_the_residual_order}

 In this section we will describe a modification of the residual order that does not increase under point-blowup in the setting where the ambient space $W$ is $3$-dimensional. The concept of maximizing invariants over all regular formal hypersurfaces will be extended to all formal flags. The discussion in this section will provide the motivation for the definition of the resolution invariant in Chapter \ref{chapter_invariant}.
 
\subsection{A closer analysis of Kangaroo points in the surface case}
 
 Since the increase of the residual order only happens in positive characteristic, we will restrict the discussion in this section to the case that $W$ is a regular $3$-dimensional variety over a field $K$ of characteristic $\chara(K)=p>0$. Further, let $X\subseteq W$ be a hypersurface and $a\in X$ a closed point of order $c=\ord_aX$. We will assume that $\hIXa$ is generated by an element $f$ of purely inseparable form
 \[f=z^{p^e}+F(x,y)\]
 for a positive integer $e\geq1$ and $F\in K[[x,y]]$ a power series of order $\ord F>p^e$. Thus, $c=p^e$. Assume that $f$ is clean with respect to the parameters $(x,y,z)$.
 
 Further, consider the blowup $\pi:W'\to W$ at the point $a$ and let $a'\in\pi^{-1}(a)$ be an equiconstant point lying over $a$. Since $H=V(z)$ is an adjacent hypersurface, we know that $a'$ is not contained in the $z$-chart.
 
 Assume now further that $\Ea$ is given as $\Ea=V(xy)$. If $a'$ is either the origin of the $x$-chart or the $y$-chart, then only one component of $\Ea$ is lost in the transition from $a$ to $a'$. Hence, $\resord_{a'}X'\leq\resord_aX$ holds. So let $a'$ have the affine coordinates $(t,0)$ in the $x$-chart for some non-zero constant $t\in K^*$. As we have seen, $f$ being clean with respect to the parameters $(x,y,z)$ does not imply that its transform $f'$ is again clean with respect to the induced parameters. This happens because the local blowup map $\hOWa\to\hOWap$ is not monomial in $(x,y,z)$.
 
 Set 
 \[y_t=y-tx.\]
 This change of coordinates may introduce new $p^e$-th powers to the expansion of $F$ with respect to the parameters $(x,y_t)$. Eliminating all $p^e$-th powers in the expansion of $F(x,y_t)$ yields a new parameter 
 \[z_t=z+g_t(x,y_t).\]
 Let the formal hypersurface $H_t\in\HH$ be defined via $H_t=V(z_t)$. Then $H_t$ is again adjacent. Most importantly, the blowup map $\hOWa\to\hOWap$ is monomial in the $x$-chart with respect to the parameters $(x,y_t,z_t)$. Thus, $f'$ is clean with respect to the induced parameters $(x',y_t',z_t')$ for $\hOWap$. Since $\Eap=V(x')$ and $f'$ is clean, we know that the strict transform $H_t'$ of $H_t$ realizes the residual order of $X'$ at $a'$. Consequently, $\resord_{a'}X'=d_{H_t'}$. Further, we know that $d_{H_t'}\leq d_{H_t}$ holds. But since $\Ea$ is not defined by a monomial in the parameters $(x,y_t)$, the hypersurface $H_t$ is not necessarily valid, as the following example illustrates:

\begin{example}
 Consider again the equation
 \[f=z^2+xy(x^2+y^2)\]
 over a field $K$ of characteristic $2$ with $\Ea=V(xy)$. Set $y_1=y+x$. Then
 \[f=z^2+x(y_1+x)y_1^2\]
 \[=z^2+xy_1^3+x^2y_1^2.\]
 Set $z_1=z+xy_1$. Then
 \[f=z_1^2+xy_1^3.\]
 
 Let $H_1\subseteq\Spec(\hOWa)$ be the formal hypersurface $H_1=V(z_1)$. Then the hypersurface $\wt H\subseteq\Spec(\hOWap)$ in the previous example is the strict transform of $H_1$. Consequently, $\resord_{a'}X'=d_{\wt H}\leq d_{H_1}$.
 
 But $H_1$ is not a valid hypersurface since $m_{H_1}=1$, but $m_H=2$ for $H=V(z)$.
\end{example}

\subsection{An estimate for the residual order at equiconstant points}

 We now want to define for each non-zero constant $t\in K^*$ a number $d_t$ which has the following qualities:
 \begin{itemize}
  \item The number $d_t$ can be computed from via  the clean expansion of $f$ with respect to the parameters $(x,y_t,z_t)$.
  \item It fulfills $\resord_{a'}X'=d_{H_t'}\leq d_t$.
  \item The number $d_t$ should still be sufficiently close to $\resord_aX$ in the sense that there exists a bound $\varepsilon\in\N$ which is independent of $X$ such that 
  \[d_t\leq\resord_{a}X+\eps\]
  holds for all $t\in K^*$. 
  By Moh's bound we know that $\resord_{a'}X'\leq\resord_{a}X+\frac{c!}{p}$ holds. This suggests to require that the same bound 
  \[d_t\leq\resord_aX+\frac{c!}{p}\]
  holds all $t\in K^*$.
 \end{itemize}
 To this end, consider the expansion
 \[f=z_t^{p^e}+F_t(x,y_t).\]
 Set $\x_t=(x,y_t,z_t)$. Further, set
 \[J_{2,\x_t}=\coeff_{\x_t}^c(f)=(F_t)^{(p^e-1)}.\]
 By construction, $f$ is clean with respect to the parameters $\x_t$. A number $d_t$ that has all of the requested qualities is given by the order of the weak initial ideal of $J_{2,\x_t}$ along the regular formal curve 
 \[C_t=V(y_t,z_t)=V(y+tx,z_t)\subseteq\Spec(\hOWa).\]
 Hence, we set
 \[d_t=\ord_{(y_t)}\minit(J_{2,\x_t})=(p^e-1)!\cdot \ord_{(y_t)}\init(F_t).\]
 We will show in Chapter \ref{chapter_invariance} that $\ord F>p^e$ implies that the number $d_t$ is independent of the chosen parameters $\x_t$ as long as they fulfill $H_t=V(z_t)$ and $C_t=V(z_t,y_t)$. We say that the parameters $\x_t$ are \emph{subordinate} to the flag given by the curve $C_t$ and the hypersurface $H_t$.
 
 Generally, a \emph{formal flag} $\F$ at $a$ consists of a regular curve $\F_1$ and a regular hypersurface $\F_2$ in $\Spec(\hOWa)$ such that $\F_1\subseteq\F_2$ holds.
 
 Consider now the flag $\F$ given by $\F_2=H_t$ and $\F_1=C_t$ as above. Notice that the flag $\F$ has a particular geometric configuration with respect to $\Ea$. The hypersurface $\F_2$ is compatible with $\Ea$ and thus, the intersection 
 \[\F_2\cap \Ea=V(z_t,xy)=V(z_t,x)\cup V(z_t,y)\]
 is the union of two regular formal curves with simple normal crossings. The curve $\F_1$ meets each component of $\F_2\cap\Ea$ transversally, but the union $\F_1\cup(\F_2\cap\Ea)$ does not have simple normal crossings since
 \[\F_1\cup(\F_2\cap\Ea)=V(z_t,x)\cup V(z_t,y)\cup V(z_t,y+tx).\]
 
 
\subsection{Modifying the residual order via orders along formal curves}
 
 Let us now return to the more general situation that $W$ is a $3$-dimensional variety over a field $K$ of characteristic $p>0$, $X\subseteq W$ a hypersurface and $E\subseteq W$ a simple normal crossings divisor. Let $a\in X$ be a closed point of order $\ord_aX=c$ and denote by $\Ea$ the union of those components of $E$ that contain $a$.
 
 We want to define a modification $d_X(a)$ of the residual order which fulfills
 \[\resord_aX\leq d_X(a)\leq \resord_aX+\frac{c!}{p}\]
 and 
 \[d_{X'}(a')\leq d_X(a)\]
 for all equiconstant points $a'$ lying over $a$ under the blowup of the point $a$. To this end, we will now define a particular class of formal flags with similar properties to the flag we used to define $d_t$.
 
 Let $\FF$ denote the set of all formal flags $\F$ at $a$ with the property that the hypersurface $\F_2$ is compatible with $\Ea$. Clearly, $\FF\neq\emptyset$ is equivalent to the fact that $\Ea$ has less than three components. By induction on the previous blowups which have been performed in the resolution process, we can assume that this is the case. If we want to emphasize the dependence of $\FF$ on the point $a$, we will denote it by $\FF(a)$
 
 Further, let $\FF_1$ denote the set of flags $\F\in\FF$ with the following properties:
 \begin{itemize}
  \item The union $\F_1\cup(\F_2\cap\Ea)$ is not simple normal crossings.
  \item For all components $D$ of $\Ea$, the curves $\F_1$ and $\F_1\cap D$ meet transversally at $a$.
 \end{itemize}
 
 As we will show in Lemma \ref{form_of_associated_components} (5), $\FF_1\neq\emptyset$ is equivalent to $\Ea$ having exactly two components.
 
 Let $\F\in\FF_1$ be a  flag and $\x=(x,y,z)$ a regular system of parameters for $\hOWa$ that fulfills $\F_2=V(z)$ and $\F_1=V(y,z)$. We say that such parameters are subordinate to $\F$. To ease notation, set $\x=(x,y,z)$. Notice that we do not require (and it is not possible) that $\Ea$ is defined by a monomial in $x$ and $y$. Set
 \[J_{2,\x}=\coeff^c_{(x,y,z)}(\hIXa)\]
 and define the numbers
 \[m_{\F,\x}=\ord J_{2,\x},\]
 \[d_{\F,\x}=\ord_{(y)}\minit(J_{2,\x}).\]
 
 While the numbers $m_{\F,\x}$ and $d_{\F,\x}$ might depend on a chosen subordinate parameter system $\x=(x,y,z)$, we ignore this dependence for the time being and treat them as if they are invariants of the flag $\F$.
 
 Now consider the blowup $\pi:W'\to W$ at the point $a$ and let $a'$ be an equiconstant point over $a$. Recall that $\resord_{a'}X'\leq\resord_aX$ holds if one or less components of $\Ea$ are lost at $a'$. On the other hand, if two components of $\Ea$ are lost at $a'$, then there is a flag $\F\in\FF_1(a)$ such that $a'$ is contained in the strict transforms of $\F_2$ and $\F_1$. Further, $\F$ can even be chosen in such a way that it maximizes $m_{\F,\x}$ over all flags in $\FF_1(a)$ and the inequality
 \[\resord_{a'}X'\leq d_{\F,\x}\]
 is fulfilled. 
 
 A flag $\F\in\FF_1$ which maximizes $m_{\F,\x}$ over all flags in $\FF_1$ will be called \emph{valid}. If $\hIXa$ is generated by a purely inseparable power series $z^{p^e}+F(x,y)$ and the parameters $(x,y,z)$ are subordinate to a flag $\F\in\FF_1$, then the flag is valid if $f$ is clean. As the following example shows, the value $d_{\F,\x}$ is not significant if a flag $\F\in\FF_1$ is not valid:
 
 \begin{example}
  Let $\hIXa$ be generated by the purely inseparable polynomial
  \[f=z^2+x^7y(x+y)^2\]
  over a field of characteristic $2$ and let $\Ea=V(xy)$. Then $\resord_aX=2$.
  
  Set $y_1=y+x$. Then
  \[f=z^2+x^7(y_1+x)y_1^2\]
  \[=(z+x^4y_1)^2+x^7y_1^3.\]
  Set $z_1=z+x^4y_1$ and let $\F\in\FF_1$ be the flag $\F_2=V(z_1)$, $\F_1=V(z_1,y_1)$. Set $\x=(x,y_1,z_1)$. Then $m_{\F,\x}=10$ and $d_{\F,\x}=3$. The flag $\F$ is valid.
  
  Now consider the flag $\G\in\FF_1$ defined by $\G_2=V(z_1+y_1^4)$, $\G_1=V(z_1+y_1^4,y_1)$. The expansion of $f$ with respect to the subordinate parameters is
  \[f=(z_1+y_1^4)^2+y_1^8+x^7y_1^3.\]
  Set $\wt\x=(x,y_1,z_1+y_1^4)$. Then $m_{\G,\wt\x}=8$ and $d_{\G,\wt\x}=8$. In particular, the estimate $d_{\G,\wt\x}\leq\resord_aX+\frac{c!}{p}$ is not fulfilled. The flag $\G$ is invalid.
 \end{example}

 \begin{remark}
  We could define a preliminary modification of the residual order as
  
  \[\wt d_X(a)=\sup(\{\resord_aX\}\cup\{d_{\F,\x}:\text{$\F\in\FF_1(a)$ is valid}\}).\]
  This invariant fulfills
  \[\resord_{a'}X'\leq \wt d_X(a)\]
  at all equiconstant points $a'$ lying over $a$.
  
  Further, the modification $\wt d_X(a)$ is within the range
  \[\resord_aX\leq \wt d_X(a)\leq \resord_a(X)+\frac{c!}{p}.\]
  
  Still, the invariant $\wt d_X(a)$ might increase under blowup, as the following example shows:
  
  Let $\hIXa$ be generated by the purely inseparable polynomial
  \[f=z^2+x(x+y^2)^2,\]
  and set $\Ea=V(x)$. Since $\Ea$ has only one component, we know that $\FF_1(a)=\emptyset$ and thus, $\wt d_X(a)=\resord_aX=2$.
 
   Now consider the blowup at the point $a$ and let $a'$ be the origin of the $y$-chart over $a$. Then $\hIXap$ is generated by the element
  \[f'=z^2+xy(x+y)^2,\]
  where $(x,y,z)$ denote the induced parameters for $\hOWap$. Further, $\Eap=V(xy)$. Consider the flag $\F\in\FF_1(a')$ that is given by $\F_2=V(z+xy_1)$ and $\F_1=V(z+xy_1,y_1)$ where $y_1=y+x$. Since $f'$ has the expansion
  \[f'=(z+xy_1)^2+xy_1^3,\]
  the flag $\F$ is valid and $d_\F=3$. Hence, $\wt d_{X'}(a')>\wt d_X(a)$.
 \end{remark}

 %

\subsection{Associated multiplicity $n_\F$ and the definition of $d_\F$ for all flags}

To correctly define the modification $d_X(a)$ of the residual order, we have to define the invariants $m_\F$ and $d_\F$ for all flags $\F\in\FF$.

 To classify the flags in $\FF$, we will define for each flag $\F\in\FF$ the \emph{associated multiplicity} $n_\F$. If the union $\F_1\cup(\F_2\cap\Ea)$ has simple normal crossings, we set $n_\F=0$. Otherwise, we set
 \[n_\F=\max\{\mult_a(\F_1,\F_2\cap D):\text{$D$ is a component of $\Ea$}\}\]
 where $\mult_a(.,.)$ denotes the intersection multiplicity of two curves at $a$.
 
 Clearly, $\FF_1$ consists of the flags $\F\in \FF$ with $n_\F=1$. Now let $\F\in\FF$ be a flag with $n_\F\geq2$. We will show in Lemma \ref{form_of_associated_components} that this implies that there exists exactly one component $D$ of $\Ea$ with the property that $n_\F=\mult_a(\F_1,\F_2\cap D)$. (The intersection multiplicity for any other component is always $1$.) This component of $\Ea$ is called the \emph{associated component} of $\F$ and will be denoted by $D_\F$.
 
 Now let $\F\in\FF$ be any flag with $n_\F>0$. Set $n=n_\F$ and let $\x=(x,y,z)$ be parameters that are subordinate to $\F$. (Again, we do not require that the parameters $\x$ are subordinate to $\Ea$.) Let $J_{2,\x}$ be defined as 
 \[J_{2,\x}=\coeff^c_{(x,y,z)}(\hIXa).\]
 Define the weighted order function $\w_n:K[[x,y]]\to\Ni$ by $\w_n(x)=1$ and $\w_n(y)=n$. Then we set
 \[m_{\F,\x}=\w_n(J_{2,\x}),\]
 \[d_{\F,\x}=\ord_{(y)}\minit_{\w_n}(J_{2,\x}).\]
 Notice that $d_{\F,\x}$ equals the order of the weak initial ideal of $J_{2,\x}$ with respect to $\w_n$ along the formal curve $\F_1$. We set
 \[m_\F=\begin{cases}
         m_{\F,\x} & \text{if $d_{\F,\x}\geq nc!$,}\\
         nc! & \text{if $d_{\F,\x}<nc!$.}
        \end{cases}
       \]
 \[d_\F=\begin{cases}
         d_{\F,\x} & \text{if $d_{\F,\x}\geq c!$,}\\
        d_{\F,\x} & \text{if $0<d_{\F,\x}<c!$ and $c!\nmid m_\F$,}\\
         -1 & \text{if $0<d_{\F,\x}<c!$ and $c!\mid m_\F$,}\\
         -1 & \text{if $d_{\F,\x}=0$.}
        \end{cases}
 \]
 As we will show in Proposition \ref{m_d_coord_indep}, the numbers $m_\F$ and $d_\F$ are independent of the chosen subordinate parameter system $\x$. Hence, they are invariants of the flag $\F$.
 
 If $n_\F=0$ holds, the numbers $m_\F$ and $d_\F$ are defined differently. Let $\F\in\FF$ be a flag with $n_\F=0$ and $\x=(x,y,z)$ a regular system of parameters that is subordinate to $\F$ and $\Ea$. Then we set 
\[m_\F=m_{\F_2}=\ord M_{2,\x},\]
\[d_\F=d_{\F_2}=\ord I_{2,\x}\]
 where the ideals $M_{2,\x}$ and $I_{2,\x}$ are defined as in Section \ref{section_residual_order} via the factorization
 \[J_{2,\x}=M_{2,\x}\cdot I_{2,\x}.\]
 Notice that these numbers are invariants of the hypersurface $\F_2$ and do not depend on the curve $\F_1$ as long as $\F_1$ is chosen in such a way that $n_\F=0$. They are also independent of the chosen subordinate parameters $\x$ by Proposition \ref{residual_order_coord_indep}.
 







 
\subsection{Definition of $d_X(a)$}
 
 Two flags $\F,\G\in\FF$ are said to be \emph{comparable} if their associated multiplicity and their associated component (if applicable) coincide.
 
 A flag $\F\in\FF$ is said to be \emph{valid} if all comparable flags $\G\in\FF$ fulfill the inequality $m_\G\leq m_\F$.
 
 We define our modification $d_X(a)$ of the residual order as
 \[d_X(a)=\sup\{d_\F:\text{$\F\in\FF$ is valid}\}.\]
 
 We will show in Proposition \ref{maximizing_flags_exist} that the supremum in this definition is actually a maximum.
 
 Notice that $d_X(a)\geq \resord_aX$ holds naturally since
 \[\resord_aX=\max\{d_\F:\text{$\F\in\FF$ with $n_\F=0$ is valid}\}.\]

\subsection{Behavior of $d_X(a)$ under blowup}
 
 The details for all of the statements in this section will be provided in Section \ref{section_inv_drops_for_d>0}.
 
 Consider the blowup $\pi:W'\to W$ at the point $a$ and let $a'$ be an equiconstant point lying over $a$. By definition of the invariant, the inequality
 \[d_{X'}(a')\leq d_X(a)\]
 is equivalent to the fact that for each valid flag $\G\in \FF(a')$ there exists a valid flag $\F\in\FF(a)$ which fulfills
 \[d_\G\leq d_\F.\]
 
 To prove this, we first have to consider how the different types of flags transform under blowup. To this end, let $\F\in\FF(a)$ be a formal flag with the property that the point $a'$ is contained in the strict transform $\F_1'$ of the curve $\F_1$. (Thus, $a'$ is also contained in the strict transform of $\F_2$). Denote by $\F'$ the formal flag at $a'$ that is defined by the strict transform $\F_2'$ of $\F_2$ and the strict transform $\F_1'$ of $\F_1$. It is called the \emph{induced flag} at $a'$. As we will show in Proposition \ref{flag_invs_under_blowup}, $\F'\in\FF(a')$ holds.
 
 According to their type, the flags $\F\in\FF(a)$ behave in the following way under blowup:
 \begin{itemize}
  \item For a flag $\F$ with $n_\F>1$, the property $a'\in\F_1'$ implies that the point $a'$ also lies on the strict transform $(D_\F)'$ of the associated component $D_\F$. Any other component of $\Ea$ is lost in the transition from $a$ to $a'$. The induced flag $\F'$ then fulfills $n_{\F'}=n_\F-1$. If $n_{\F'}>1$, then $D_{\F'}=(D_\F)'$. Further, $d_{\F'}=d_\F$ holds.
  \item For a flag $\F$ with $n_\F=1$, the property $a'\in\F_1'$ implies that both components of $\Ea$ are lost in the transition from $a$ to $a'$. Further, the induced flag $\F'$ fulfills $n_{\F'}=0$ and, if $d_\F\neq-1$ holds, then $d_{\F'}\leq d_\F$.
  \item For a flag $\F$ with $n_\F=0$, the property $a'\in\F_1'$ implies that at most one component of $\Ea$ is lost in the transition from $a$ to $a'$. Also, the induced flag $\F'$ fulfills $n_{\F'}=0$ and $d_{\F'}\leq d_\F$.
 \end{itemize}
 
 It is easy to see that not all valid flags $\G\in\FF(a')$ are induced by flags $\F\in\FF(a)$. This poses only a minor problem if there is a flag $\F\in\FF(a)$ such that its induced flag $\F'$ is comparable to $\G$. In this case, we will be able be able to construct a valid flag $\F\in\FF(a)$ such that the chain of inequalities
 \[d_\G\leq d_{\F'}\leq d_\F\]
 holds.
 
 The problem is more serious if there is no flag $\F\in\FF(a)$ such that the induced flag $\F'$ is comparable to $\G$. This holds if and only if $n_\G>1$ and $D_\G=\Dnew$ where $\Dnew=\pi^{-1}(a)$ denotes the exceptional divisor of the blowup $\pi$. In fact, for all such flags $\G$ the image $\pi(\G_1)$ of the curve $\G_1$ is a singular curve in $\Spec(\hOWa)$.
 
 To estimate $d_\G$ for these kinds of flags, a modified version of Moh's bound will be used. We will show that for all valid flags $\G\in\FF(a')$ with $n_\G>1$ and $D_\G=\Dnew$, there exists a valid flag $\F\in\FF(a)$ such that
 \[d_\G\leq \frac{1}{n_\G}d_\F+\eps\]
 holds, where
 \[\eps=\begin{cases}
         0 & \text{if $\chara(K)=0$ or $c!\nmid m_\G$,}\\
         \frac{c!}{p} & \text{if $\chara(K)=p>0$ and $c!\mid m_\G$.}
        \end{cases}
\]
 If $\eps=0$, this already implies that $d_\G<d_\F$. So consider the case that $\eps=\frac{c!}{p}$. Since we know that $n_\G\geq2$ and $p\geq2$, we can compute that
 \[d_\G\leq \frac{1}{2}d_\F+\frac{c!}{2}.\]
 Since $c!\mid m_\G$ holds by assumption and we may assume that $d_\G\neq -1$, the estimate $d_\G\geq c!$ holds by definition. Hence, it follows from above inequality that $d_\G\leq d_\F$. (Using refined techniques, it is actually possible to show that $d_\G<d_\F$ holds in this case.) 
 
 
 
 \subsection{The small residual case}
 
 By what we have seen so far, the inequality $d_\F\leq d_{\F'}$ for the induced flag $\F'$ holds in all but one case. Namely, if $\F\in\FF(a)$ is a flag with $n_\F=1$ and $d_\F=-1$. Then the induced flag $\F'\in\FF(a')$ fulfills $d_{\F'}>d_\F$.
 
 A closer analysis of this case reveals that coefficient ideal with respect to the induced flag $\F'\in\FF(a')$ has a very good property in this case.
 
 If $d_\F=-1$ because $d_{\F,\x}=0$, then $d_{\F'}=0$ holds for the induced flag. On the other hand, assume that $d_\F=-1$ because $0<d_{\F,\x}<c!$ and $c!\mid m_\F$. Let $\x'$ be subordinate parameters for the induced flag $\F'$ and denote them by $\x'=(x,y,z)$. Then it can be shown that the coefficient ideal $J_{2,\x'}$ has the form
 \[J_{2,\x'}=(y^{mc!})\cdot I\]
 for a positive integer $m>0$ and an ideal $I\subseteq K[[x,y]]$ with $\ord_{(y)}I=0$ and $0<\ord I<c!$.
 
 In this case, we say that $X'$ is in the \emph{small residual case} at $a'$. The small residual case is a terminal case similar to the monomial case that can be resolved by combinatorial resolution. This follows from the simple fact that blowing up the regular curve $V(y,z)$ changes the form of the coefficient ideal to
 \[(y^{(m-1)c!})\cdot I\]
 where $I$ is the same ideal as before. Repeating this $m$ times, we end up with the situation that the coefficient ideal equals $I$. By Lemma \ref{ogeqc!}, this implies that the order function has decreased. 
 
 An exact definition of the small residual case is given in Section \ref{section_monom_s_r_case}. Terminal cases which are similar, but more general than the small residual cases have been used by other authors. (Cf. the notion of \emph{exceptional and good} in \cite{EV_Good_Points} p. 116 and Abhyankar's notion of \emph{good points} in \cite{Abhyankar_Good_Points}.)

 
 


\subsection{The flag invariant $\inv(\F)=(d_\F,n_\F,s_\F)$}

 Combining the results we have stated so far, we conclude that
 \[d_{X'}(a')\leq d_X(a)\]
 holds at all equiconstant points $a'$ over $a$ at which $X'$ is not in a terminal case.

 Although the invariant $d_X(a)$ does not increase under point-blowups at equiconstant points, it will still remain constant in many cases. To measure improvement under point-blowups until we reach a terminal case that can be resolved via combinatorial resolution, we have to further refine our invariant.
 
 As a first step, we consider the pair $(d_\F,n_\F)$ for each valid flag $\F\in\FF$. By the results we established so far, the strict inequality
 \[(d_{\F'},n_{\F'})<(d_\F,n_\F)\]
 holds for the induced flag $\F'\in\FF(a')$ in all cases except if either $n_\F=0$ and $d_{\F'}=d_\F$ or $n_\F=1$ and $d_\F=-1$. Following the discussion in the previous section, only the first of these cases is relevant.
 
 Thus, it remains to measure improvement in the case $n_{\F'}=n_\F=0$ and $d_{\F'}=d_\F$.
 
 To this end, we define $s_\F$ as the order of the \emph{second coefficient ideal} $J_{1,\x}$ as follows: Let $\F\in\FF$ be a flag with $n_\F=0$ and subordinate parameters $\x=(x,y,z)$. Recall the definition of $M_{2,\x}$, $I_{2,\x}$ and $d_\F$. Then we set  $s_\F=\ord J_{1,\x}$ where
 \[J_{1,\x}=\begin{cases}
             \coeff_{(x,y)}^{d_\F}(I_{2,\x}) & \text{if $d_\F\geq c!$,}\\
             \coeff_{(x,y)}^{d_\F(c!-d_\F)}(I_{2,\x}^{c!-d_\F}+M_{2,\x}^{d_\F}) & \text{if $0<d_\F<c!$,}\\
             0 & \text{if $d_\F=0$.}
            \end{cases}
\]
It will will be shown in Proposition \ref{slope_coord_indep} that $s_\F$ only depends on the flag $\F$ and not on the choice of subordinate parameters $\x$. The definition of $J_{1,\x}$ for $0<d_\F<c!$ makes use of the \emph{companion ideal}, a well-known technique used in the proof of resolution of singularities in characteristic zero (Cf. \cite{EH}). It ensures that the value of $s_\F$ is always independent of the choice of $\x$.

For flags $\F\in\FF$ with $n_\F>0$, we set $s_\F=0$. Hence, no second coefficient ideal is considered for these types of flags.

The invariant $s_\F$ behaves the following way under blowup: Let $\F\in\FF(a)$ be a flag with $n_\F=0$. Then the induced flag $\F'\in\FF(a')$ fulfills $n_{\F'}=0$ and $d_{\F'}\leq d_\F$. If the equality $d_{\F'}=d_\F$ holds, then $s_{\F'}<s_\F$ holds under the condition that $s_\F<\infty$. 

Consequently, we define for each flag $\F\in\FF$ the \emph{flag invariant}
\[\inv(\F)=(d_\F,n_\F,s_\F).\]
Then we define
\[\inv_X(a)=\sup\{\inv(\F):\text{$\F\in\FF$ is valid}\}.\]

It will be shown in Proposition \ref{maximizing_flags_exist} that the supremum in this definition is actually a maximum. Thus, there exists a \emph{maximizing} flag $\F\in\FF(a)$ such that $\F$ is valid and $\inv(\F)=\inv_X(a)$.

There are two terminal cases in which we will set $\inv_X(a)$ to be the trivial value $(0,0,0)$ instead of the flag invariant $\inv(\F)$ of a maximizing flag $\F$. One of them is the small residual case which was already discussed. The other terminal case is the \emph{monomial case}. The exact definitions of both terminal cases will be given in Section \ref{section_monom_s_r_case}. 

We say that $X$ is in the monomial case at $a$ if there exists a regular system of parameters $\x=(x,y,z)$ for $\hOWa$ with $\Ea\subseteq V(xy)$ such that the coefficient ideal $J_{2,\x}$ is a principal monomial ideal of the form
\[J_{2,\x}=(x^{r_x}y^{r_y})\]
and an additional maximality condition is fulfilled. This maximality condition is the existence of an element $f\in\hIXa$ which is $\ord$-clean with respect to the coefficient ideal $J_{2,\x}$, a notion that will be introduced in Chapter \ref{chapter_cleaning}. It ensures that the order of $J_{2,\x}$ is maximal over all coordinate changes $z\mapsto z+g(x,y)$.

We will show in Lemma \ref{inv_finite} that if $X$ is \emph{not} in a terminal case at $a$, then $\inv_X(a)=(d_\F,n_\F,s_\F)\in\N^3$ (hence, no component is infinite) and $\inv_X(a)>(0,0,0)$. This guarantees that successively lowering this invariant under point-blowups eventually leads to one of the terminal cases.

By combining all statements about the behavior of $\inv(\F)$ under blowups that we made so far, we can make the following conclusion: Let $W'\to W$ be the blowup at the point $a$ and let $a'\in\pi^{-1}(a)$ be an equiconstant point. If $X$ is not in a terminal case at $a$ and $X'$ is not in a terminal case at $a'$, then the strict inequality
\[\inv_{X'}(a')<\inv_X(a)\]
holds.

Since we set $\inv_X(a)=(0,0,0)$ in the terminal cases, we consider these cases to be \emph{better} than the case in which $\inv_X(a)=\inv(\F)>(0,0,0)$. Indeed, both terminal cases can be resolved via combinatorial resolution. To measure the improvement during combinatorial resolution, we will use additional invariants that will not be defined here.


Thus, we can reformulate our strategy to eventually lower the order function for the surface case in the following way:

 \begin{enumerate}[(1)]
  \item Lower the invariant $\inv_X(a)$ until it becomes $(0,0,0)$.
  \item As soon as $\inv_X(a)=(0,0,0)$ is achieved, we have reached a terminal case. Then combinatorial resolution can be applied to lower the order function.
 \end{enumerate}

\section{Failure of upper semicontinuity of the residual order} \label{section_generic_down}

 Over fields of characteristic zero, hypersurfaces of maximal contact can be used to show that the map $(\ord X,\resord X):X\to\Ni^2$ is upper semicontinuous.
 
 On the other hand, the function is generally not upper semicontinuous over fields of positive characteristic. This is illustrated by the following example:

 \begin{example}
  Consider the ambient space $W=\Spec(K[x,y,z])$ over an algebraically closed field $K$ of characteristic $\chara(K)=p>0$ and the hypersurface $X=V(f)\subseteq W$ defined by the polynomial
  \[f=z^p+xy^{np}.\]
  for some positive integer $n>0$. Further, let $E$ be given as $E=V(x)$. The top locus of $X$ consists of the line $V(z,y)$.
  
  Since $f$ is clean with respect to the parameters $(x,y,z)$, the hypersurface $H_0=V(z)$ realizes the residual order of $X$ at the origin $a_0=(0,0,0)$ and 
  \[\resord_{a_0}X=d_H=(p-1)!\cdot np\]
  holds.
  
  Now consider a closed point on $\topp(X)$ with affine coordinates $a_t=(t,0,0)$ for some $t\in K^*$. Set $x_t=x-t$. The expansion of $f$ at $a_t$ with respect to the local parameters $(x_t,y,z)$ is
  \[f=z^p+x_ty^{np}+ty^{np}.\]
  Set $z_t=z+t^\frac{1}{p}y^n$ and let $H_t$ be the regular hypersurface $H_t=V(z_t)$. Notice that the $p$-th root of $t$ exists since $K$ is algebraically closed. With respect to the parameters $(x_t,y,z_t)$, the expansion of $f$ is
  \[f=z_t^p+x_ty^{np}.\]
  Hence, $X$ is pointwise trivial along the line $V(y,z)$ in the strong sense that all local rings $\OO_{X,a}$ for $a\in V(y,z)$ are isomorphic.
  
  Since the expansion of $f$ is clean with respect to the parameters $(x_t,y,z_t)$, the hypersurface $H_t$ realizes the residual order of $X$ at $a_t$. Consequently,
  \[\resord_{a_t}X=d_{H_t}=(p-1)!\cdot (np+1)\]
  for all non-zero constants $t\in K^*$, a bigger value than at the special point $a_0$.
  
  Consequently, the function $(\ord X,\resord X):X\to\N^2$ is not upper semicontinuous.
 \end{example}

 Another important thing to observe in above example is the fact that at each closed point of the top locus, a different regular hypersurface has to be used to maximize the residual order. This is in stark contrast to the situation in characteristic zero where hypersurfaces of maximal contact which realize the residual order of $X$ are known to exist on open neighborhoods. 
 
 We will now explain how the problem of upper semicontinuity will be handled in the surface case. To this end, we will first analyze the situation over fields of characteristic zero to get a picture of how our invariant \emph{should} behave.
 
 Under the condition that $W$ is a regular $3$-dimensional ambient variety over a field $K$ and $X\subseteq W$ is a hypersurface, we know that $\topp(X)$ is at most $1$-dimensional. Since upper semicontinuity is clear at isolated points of $\topp(X)$, let us consider a curve $C$ lying inside $\topp(X)$.
 
 If the field $K$ is of characteristic zero, one can use hypersurfaces of maximal contact to show that there exists a non-negative integer $m\in\N$ such that the following hold:
 \begin{enumerate}[(1)]
  \item $\resord_aX\geq m$ for all points $a\in C$.
  \item For all but finitely many points $a\in C$ the equality $\resord_aX=m$ holds.
  \item If the equality $\resord_aX=m$ holds for a closed point $a\in C$, then $J_{2,\x}(a)=(y^m)$ where $\x=(x,y,z)$ are regular parameters for $\hOWa$ which are subordinate to a hypersurface $H$ which realizes the residual order of $X$ at $a$. In particular, $J_{2,\x}(a)$ is a principal monomial ideal.
  \item If $C$ is singular at $a$, then $\resord_aX>m$.
 \end{enumerate}
 An important consequence of these properties, other than the residual order being upper semicontinuous on $C$, is that the coefficient ideal $J_{2,\x}(a)$ is a principal monomial ideal at almost all points $a\in C$. This is important since we only want to blow up curves once we have reached a terminal case. The above tells us that $X$ is already in the monomial case at all but finitely many points $a\in C$ and once we have lowered the residual order of $X$ at all points of $C$ to $m$, then the curve $C$ is regular and $X$ is in the monomial case at all of its points. Hence, we can proceed to apply combinatorial resolution.
 
 As we can see from the previous example, these properties do generally not hold if the characteristic of $K$ is positive. Yet, there is another thing to observe from the example. Namely, at each closed point $a_t\in\topp(X)$, $X$ is in the small residual case. Since we set $\inv_a(X)=(0,0,0)$ whenever $X$ is in a terminal case at a point $a$, this shows that the previously defined invariant $\inv(X)$ is (trivially) upper semicontinuous on $\topp(X)$.
 
 Surprisingly, the addition of the small residual case to the terminal cases already suffices to make the map $(\ord X,\inv(X)):X\to\Ni^4$ upper semicontinuous. As we will show in Proposition \ref{key_proposition_for_usc}, for a curve $C$ as before, the hypersurface $X$ is in a terminal case at all but finitely many points $a\in C$. Naturally, this suffices to show that $\inv(X)$ is upper semicontinuous on $C$. Further, we will show that if $C$ is singular at a point $a$, then $\inv_a(X)>(0,0,0)$ holds.
 
 The actual resolution invariant $\ivX$ that we will define in Chapter \ref{chapter_invariant} will not only consist of the order function and $\inv(X)$, but also of the \emph{combinatorial pair}, an invariant used to measure further improvement under blowup once we have reached a terminal case. It will require some additional effort to show that this combinatorial pair is also upper semicontinuous. The details for all of these statements and their proofs can be found in Chapter \ref{chapter_usc}.

\chapter{Coordinate-independence of invariants associated to coefficient ideals} \label{chapter_invariance}

Our definition of the coefficient ideal comes with the drawback that it is highly coordinate dependent. This poses a problem since we want to use certain invariants associated to coefficient ideals as components of our resolution invariant. Different choices of parameters might yield different invariants. One way to solve this problem would be to prescribe in each situation a particular regular system of parameters with respect to which the coefficient ideal is considered. We will not follow this approach since prescribing fixed parameters seems little geometric.

Instead, our approach is to always choose parameters which are \emph{subordinate} to certain geometric objects, like regular hypersurfaces or flags. Different choices of subordinate parameters will generally yield different coefficient ideals. But if the invariants associated to those coefficient ideals are the same for all subordinate parameters, we can consider them as invariants of the underlying geometric objects.

In this chapter we will investigate under which conditions certain invariants associated to coefficient ideals are invariants of geometric objects. The results will then be applied to show that the invariants $m_\F$, $d_\F$ and $s_\F$ which were defined in Chapter \ref{chapter_pathologies} only depend on the flag $\F$ and not on a particular choice of subordinate parameters.

Most of the argumentation in the chapter is rather technical. The main results that will be used in Chapter \ref{chapter_invariant} to ensure that the resolution invariant is well defined are Proposition \ref{residual_order_coord_indep}, Proposition \ref{m_d_coord_indep} and Proposition \ref{slope_coord_indep}.

\section{Coordinate-independence of weighted orders of coefficient ideals} \label{section_invariance_of_weighted_orders}

 In this section we will investigate how weighted orders of the coefficient ideal behave under coordinate changes that stabilize the underlying hypersurface. 
 
 To be more precise, consider the power series ring $R=K[[\x,z]]$ with $\x=(x_1,\ldots,x_n)$ and let $J\subseteq R$ be an ideal. Let $H\subseteq\Spec(R)$ be the regular hypersurface $H=V(z)$. Further, let $\w:K[[\x]]\to\Ni^k$ be a weighted order function defined on the parameters $\x$. Let $J_{-1}$ be the coefficient ideal $J_{-1}=\coeff_{(\x,z)}^c(J)$ with respect to the hypersurface $H$ and a positive integer $c>0$.
 
 Using the inverse function theorem and the Weierstrass preparation theorem, it is easy to see that coordinate changes $(\x,z)\mapsto (\wt \x,\wt z)$ which stabilize the hypersurface $H$ in the sense that $H=V(\wt z)$ are composed of the following types of simple coordinate changes:
 \begin{itemize}
  \item Triangular coordinate changes $x_i\mapsto x_i+g(\x_-,z)$ where $\x_-=(x_1,\ldots,x_{i-1},x_{i+1},\ldots,x_n)$.
  \item Multiplications with units $x_i\mapsto u x_i$ where $u\in R^*$.
  \item Transpositions of variables $x_i\mapsto x_j$, $x_j\mapsto x_i$.
  \item Multiplications with units $z\mapsto u z$ where $u\in R^*$.
 \end{itemize}
 A transposition of $x_i$ and $x_j$ clearly leaves the weighted order $\w$ of $J_{-1}$ unchanged if $\w(x_i)=\w(x_j)$ holds. For a multiplication of the parameter $z$ with a unit, we know by Proposition \ref{coeff_ideal_well_def} that this coordinate change stabilizes the coefficient ideal $J_{-1}$.
 
 In the main result of this section, Proposition \ref{invariance_of_w}, we will prove that multiplication of a parameter $x_i$ with a unit also leaves the weighted order of $J_{-1}$ unchanged, but a triangular coordinate change $x_i\mapsto x_i+g(\x_-,z)$ only does so under certain conditions.
 
 The result will then be applied in Proposition \ref{residual_order_coord_indep} and Proposition \ref{m_d_coord_indep} to show that the residual order and the invariants $m_\F$ and $d_\F$ as they were defined in Chapter \ref{chapter_pathologies} are independent of the choice of subordinate parameters.
 
 As a preparation for proving Proposition \ref{invariance_of_w}, we will state two lemmas. Lemma \ref{w_coeff} provides us with a technique to compute the weighted order of the coefficient ideal $J_{-1}$ from the weighted orders of the coefficients of elements $f$ of $J$. This technique will be used many times throughout the entire thesis. In Lemma \ref{coordinate_change_g} we will give formulae how the power series expansion of an element $f\in R$ changes under coordinate changes in $x_i$. To make notation more transparent, the parameter $x_i$ will be denoted by $y$ in the following.

\begin{lemma} \label{w_coeff}
 Let $R=K[[\x,y,z]]$, $J\subseteq R$ an ideal and $c>0$ a positive integer. Let $\w$ be a weighted order function on $K[[\x,y]]$ that is defined on the parameters $(\x,y)$.
 
 Further, let each element $f\in J$ have the expansions $f=\sum_{i\geq0}f_iz^i$ and $f=\sum_{i,j\geq0}f_{i,j}y^jz^i$ with $f_i\in K[[\x,y]]$ and $f_{i,j}\in K[[\x]]$.
 
 Set $J_{-1}=\coeff_{(\x,y,z)}^c(J)$.
 
 Then
 \[\w(J_{-1})=\min_{\substack{f\in J\\i<c}}\frac{c!}{c-i}\w(f_i)\]
 \[=\min_{f\in J}\min_{\substack{i<c\\j\geq0}}\frac{c!}{c-i}(\w(f_{i,j})+j\w(y)).\]
 Consequently, for each element $f\in J$ and all indices $i,j\geq0$ the inequalities
 \[\w(f_i)\geq \frac{c-i}{c!}\w(J_{-1}),\]
 \[\w(f_{i,j})\geq \frac{c-i}{c!}\w(J_{-1})-j\w(y)\]
 hold.
 
 Further, if $G\subseteq J$ is a generating set for $J$, then
 \[\w(J_{-1})=\min_{\substack{f\in G\\i<c}}\frac{c!}{c-i}\w(f_i).\]
\end{lemma}
\begin{proof}
 This is obvious by the definitions of the coefficient ideal and weighted order functions.
\end{proof}

%
%
%
%

\begin{lemma} \label{coordinate_change_g}
 Let $R=K[[\x,y,z]]$ and $f\in R$ an element. Consider a change of coordinates $y\mapsto\wt y$ that will be specified in the following.
 
 Let $f$ have the expansions $f=\sum_{i,j\geq0}f_{i,j}y^jz^i$ and $f=\sum_{i,j\geq0}\wt f_{i,j}\wt y^jz^i$ with $f_{i,j},\wt f_{i,j}\in K[[\x]]$. Then the following hold:
 \begin{enumerate}[(1)]
  \item If $y=\wt y+g$ where $g\in K[[\x,z]]$ has the expansion $g=\sum_{i\geq0}g_iz^i$ with $g_i\in K[[\x]]$, then
  \[\wt f_{i,j}=\sum_{\substack{0\leq k\leq i\\l\geq j}}\binom{l}{j}f_{k,l}\sum_{\substack{\alpha\in\N^{l-j}\\|\alpha|=i-k}}g_\alpha\]
  where $g_\alpha=\prod_{o=1}^{l-j}g_{\alpha_o}$.
  \item If $y=u\wt y$ where $u\in R^*$ is a unit with expansion $u=\sum_{i,j\geq0}u_{i,j}\wt y^j z^i$ where $u_{i,j}\in K[[\x]]$, then
  \[\wt f_{i,j}=\sum_{\substack{0\leq k\leq i\\0\leq l\leq j}}f_{k,l}\sum_{\substack{\alpha\in\N^l\\|\alpha|=i-k}}\sum_{\substack{\beta\in\N^l\\|\beta|=j-l}}u_{\alpha,\beta}\]
  where $u_{\alpha,\beta}=\prod_{o=1}^l u_{\alpha_o,\beta_o}$.
 \end{enumerate}
\end{lemma}
\begin{proof}
 (1): The proof is a computation.
 \[f=\sum_{i,j\geq0}f_{i,j}y^jz^i=\sum_{i,j\geq0}f_{i,j}(\wt y+\sum_{k\geq0}g_kz^k)^jz^i\]
 \[=\sum_{i,j\geq0}f_{i,j}\sum_{l=0}^j\binom{j}{l}\wt y^l(\sum_{k\geq0}g_kz^k)^{j-l}z^i\]
 \[=\sum_{i,j\geq0}f_{i,j}\sum_{l=0}^j\binom{j}{l}\wt y^l\sum_{\alpha\in\N^{j-l}}g_\alpha z^{|\alpha|+i}\]
 \[=\sum_{k,l\geq0}\sum_{\substack{0\leq i\leq k\\ j\geq l}}\binom{j}{l}f_{i,j}\sum_{\substack{\alpha\in\N^{j-l}\\|\alpha|=k-i}}g_\alpha \wt y^l z^k.\]
 
 (2):  Again, the proof is a computation.
 \[f=\sum_{i,j\geq0}f_{i,j}y^jz^i=\sum_{i,j\geq0}f_{i,j}(\sum_{k,l\geq0}u_{k,l}\wt y^l z^k)^j\wt y^j z^i\]
 \[=\sum_{i,j\geq0}f_{i,j}\sum_{\substack{\alpha\in\N^j\\ \beta\in\N^j}}u_{\alpha,\beta}\wt y^{j+|\beta|}z^{i+|\alpha|}\]
 \[=\sum_{k,l\geq0}\sum_{\substack{0\leq i\leq k\\0\leq j\leq l}}f_{i,j}\sum_{\substack{\alpha\in\N^j\\|\alpha|=k-i}}\sum_{\substack{\beta\in\N^j\\|\beta|=l-j}}u_{\alpha,\beta}\wt y^l z^k.\]
\end{proof}


\begin{proposition} \label{invariance_of_w}
 Let $R=K[[\x,y,z]]$ with $\x=(x_1,\ldots,x_n)$, $J\subseteq R$ an ideal and $c>0$ a positive integer.
 
 Let $\w$ be a weighted order function on $K[[\x,y]]$ that is defined on the parameters $(\x,y)$.
 
 Set $J_{-1}=\coeff^c_{(\x,y,z)}(J)$.
 
 Consider a coordinate change $y\mapsto \wt y$ of one of the following types:
 \begin{enumerate}[(1)]
  \item $y=\wt y+g$ for an element $g\in K[[\x,z]]$ with $\ord g\geq1$. Let $g$ have the expansion $g=\sum_{i\geq0}g_iz^i$ with $g_i\in K[[\x]]$.
  
  We require in this case that the following two properties hold:
  \begin{itemize}
   \item $\w(J_{-1})\geq c!\cdot\w(y)$.
   \item $\w(g_0)\geq \w(y)$.
  \end{itemize}
  \item $y=u\wt y$ for a unit $u\in R^*$.
 \end{enumerate}
 Set $\wt J_{-1}=\coeff_{(\x,\wt y,z)}^c(J)$ and let the weighted order function $\wt\w$ on $K[[\x,\wt y]]$ be defined by $\wt\w(x_i)=\w(x_i)$ for all $i=1,\ldots,n$ and $\wt\w(\wt y)=\w(y)$.
 
 Then $\w(J_{-1})=\wt \w(\wt J_{-1})$.
\end{proposition}
\begin{proof}
 By abuse of notation, we will denote the weighted order functions $\wt\w$ and $\w$ by the same letter. By Lemma \ref{w_coeff} there exists an element $f\in J$ with expansion $f=\sum_{i,j\geq0}\wt f_{i,j}\wt y^jz^i$ where $\wt f_{i,j}\in K[[\x]]$ and indices $i<c$, $j\geq0$ such that
 \[\w(\wt J_{-1})=\frac{c!}{c-i}(\w(\wt f_{i,j})+j\w(y)).\]
 Let the element $f$ have the expansion $f=\sum_{i,j\geq0}f_{i,j}y^jz^i$ with $f_{i,j}\in K[[\x]]$. We will now make use of the formulas for $\wt f_{i,j}$ we derived in Lemma \ref{coordinate_change_g}.
 
 (1): In this case we know that
 \[\wt f_{i,j}=\sum_{\substack{0\leq k\leq i\\l\geq j}}\binom{l}{j}f_{k,l}\sum_{\substack{\alpha\in\N^{l-j}\\|\alpha|=i-k}}g_\alpha.\]
 Thus, we can use Lemma \ref{w_coeff} to compute that
 \[\w(\wt J_{-1})=\frac{c!}{c-i}(\w(\wt f_{i,j})+j\w(y))\]
 \[\geq \frac{c!}{c-i}\min_{\substack{0\leq k\leq i\\l\geq j}}\min_{\substack{\alpha\in\N^{l-j}\\|\alpha|=i-k}}(\w(f_{k,l})+j\w(y)+\w(g_\alpha))\]
 \[\geq \frac{c!}{c-i}\min_{\substack{0\leq k\leq i\\l\geq j}}\min_{\substack{\alpha\in\N^{l-j}\\|\alpha|=i-k}}\Big(\frac{c-k}{c!}\w(J_{-1})-(l-j)\w(y)+\w(g_\alpha)\Big)\]
 \[=\w(J_{-1})+\frac{c!}{c-i}\min_{\substack{0\leq k\leq i\\l\geq j}}\min_{\substack{\alpha\in\N^{l-j}\\|\alpha|=i-k}}\Big(\underbrace{\frac{i-k}{c!}\w(J_{-1})-(l-j)\w(y)+\w(g_\alpha)}_{=:G(k,l,\alpha)}\Big).\]
 If we can show for all indices $k\leq i$, $l\geq j$ and $\alpha\in\N^{l-j}$ with $|\alpha|=i-k$ that $G(k,l,\alpha)\geq0$ holds, then we have proved that $\w(\wt J_{-1})\geq\w(J_{-1})$. It then follows by a symmetric argument that $\w(\wt J_{-1})=\w(J_{-1})$. 
 
 So let $k\leq i$, $l\geq j$ and $\alpha\in\N^{l-j}$ be indices with $|\alpha|=i-k$. Thus, we know that
 \[\w(g_\alpha)\geq ((l-j)-(i-k))\w(g_0)\geq ((l-j)-(i-k))\w(y).\]
 Further, we know by assumption that
 \[\frac{i-k}{c!}\w(J_{-1})\geq (i-k)\w(y).\]
 Together, this proves that $G(k,l,\alpha)\geq0$. 
 
 (2): In this case we know that
 \[\wt f_{i,j}=\sum_{\substack{0\leq k\leq i\\0\leq l\leq j}}f_{k,l}\sum_{\substack{\alpha\in\N^l\\|\alpha|=i-k}}\sum_{\substack{\beta\in\N^l\\|\beta|=j-l}}u_{\alpha,\beta}.\]
 Thus, we can compute by Lemma \ref{w_coeff} that
 \[\w(\wt J_{-1})=\frac{c!}{c-i}(\w(\wt f_{i,j})+j\w(y))\]
 \[\geq \frac{c!}{c-i}\min_{\substack{0\leq k\leq i\\0\leq l\leq j}}\min_{\substack{\alpha\in\N^l\\|\alpha|=i-k}}\min_{\substack{\beta\in\N^l\\|\beta|=j-l}}\Big(\w(f_{k,l})+j\w(y))+\underbrace{\w(u_{\alpha,\beta})}_{\geq0}\Big)\]
 \[\geq \frac{c!}{c-i}\min_{\substack{0\leq k\leq i\\0\leq l\leq j}}(\frac{c-k}{c!}\w(J_{-1})+(j-l)\w(y))\]
 \[=\w(J_{-1})+\frac{c!}{c-i}\min_{\substack{0\leq k\leq i\\0\leq l\leq j}}\Big(\underbrace{\frac{i-k}{c!}\w(J_{-1})}_{\geq0}+\underbrace{(j-l)\w(y)}_{\geq0}\Big)\geq \w(J_{-1}).\]
 By a symmetric argument we conclude that $\w(\wt J_{-1})=\w(J_{-1})$ holds.
\end{proof}

\begin{example}
 The following example shows that the previous proposition does not hold if the conditions for the coordinate change of type (1) are not fulfilled.
 
 Let $R=K[[x,y,z]]$ and consider the weighted order function $\w:K[[x,y]]\to\Ni$ defined by $\w(x)=1$ and $\w(y)=2$. Let $J$ be the ideal generated by the element
 \[f=z^2+xy.\]
 Then $J_{-1}=\coeff_{(x,y,z)}^2(J)=(xy)$ and $\w(J_{-1})=3<2\cdot\w(y)$. Consider the change of coordinates $y=\wt y+z$. Then
 \[f=z^2+xz+x\wt y\]
 and $\wt J_{-1}=\coeff_{(x,\wt y,z)}^2(J)=(x^2,x\wt y)$. Define the weighted order function $\wt\w:K[[x,y_1]]\to\Ni$ by $\wt\w(x)=1$ and $\wt\w(\wt y)=2$. Then $\wt\w(\wt J_{-1})=2\neq\w(J_{-1})$.
 
 Notice that $V(z,y)$ and $V(z,\wt y)$ define the same curve in $\Spec(R)$.
 \end{example}

 Using this last proposition, we can now prove that the residual order with respect to a regular hypersurface and some related numerals are independent of the choice of subordinate parameters.
 

\begin{proposition} \label{residual_order_coord_indep}
 Let $R$ be the power series ring in $(n+1)$ variables over a field $K$, $J\subseteq R$ an ideal and $c>0$ a positive integer with $\ord J\geq c$.
 
 Define for each regular system of parameters $(\x,z)$ for $R$ the coefficient ideal
 \[J_{n,(\x,z)}=\coeff_{(\x,z)}^c(J).\]
 Let $H\subseteq\Spec(R)$ be a regular hypersurface. Then the following hold:
 \begin{enumerate}[(1)]
  \item There exists a number $o_H\in\Ni$ such that for all regular systems of parameters $(\x,z)$ for $R$ with the property that $H=V(z)$, the equality
  \[\ord J_{n,(\x,z)}=o_H\]
  holds.
  \item Let $D\subseteq\Spec(R)$ be a regular hypersurface distinct from $H$ such that $H\cup D$ has simple normal crossings.
  
  Then there exists a number $r_{H,D}\in\Ni$ such that for all regular systems of parameters $(\x,z)$ for $R$ with the property that $H=V(z)$ and $D=V(x_j)$ for some index $j$, the equality
  \[\ord_{(x_j)} J_{n,(\x,z)}=r_{H,D}\]
  holds.
  
  Further, if $r_{H,D}\geq c!$ and $D_1\subseteq\Spec(R)$ is another regular hypersurface such that $H\cup D_1$ has simple normal crossings and $H\cap D=H\cap D_1$, then $r_{H,D}=r_{H,D_1}$.
  \item Let $E\subseteq\Spec(R)$ be a simple normal crossings divisor such that $H\not\subseteq E$ and $H\cup E$ has simple normal crossings.
  
  Then there exists a number $d_{H,E}$ such that for all regular system of parameters $(\x,z)$ for $R$ with the property that $H=V(z)$ and $E=V(\prod_{i\in\Delta}x_i)$ for some subset $\Delta\subseteq\{1,\ldots,n\}$ the following holds:
  
  Let $J_{n,(\x,z)}$ have the factorization
  \[J_{n,(\x,z)}=M_{n,(\x,z)}\cdot I_{n,(\x,z)}\]
  with $M_{n,(\x,z)}=(\prod_{i\in\Delta}x_i^{r_i})$ and $r_i=\ord_{(x_i)}J_{n,(\x,z)}$. Then
  \[\ord I_{n,(\x,z)}=d_{H,E}.\]
 \end{enumerate}

\end{proposition}
\begin{proof}
 (1): To prove the result, it suffices to show that the following types of coordinate changes stabilize the order of the coefficient ideal:
 \begin{enumerate}[(i)]
  \item Triangular coordinate changes $x_i\mapsto x_i+g(\x_-,z)$ where $\x_-=(x_1,\ldots,x_{i-1},x_{i+1},\ldots,x_n)$ and $\ord g\geq1$.
  \item Multiplications with units $x_i\mapsto u x_i$ where $u\in R^*$.
  \item Transpositions of variables $x_i\mapsto x_j$, $x_j\mapsto x_i$.
  \item Multiplications with units $z\mapsto u z$ where $u\in R^*$.
 \end{enumerate}
 Consider a coordinate change of type (i). By Lemma \ref{ogeqc!} we know that
 \[\ord J_{n,(\x,z)}\geq c!=c!\cdot\ord(x_i).\]
 Let $g$ have the expansion $g=\sum_{i\geq0}g_iz^i$ with $g_i\in K[[\x_-]]$. Then it is clear that 
 \[\ord g_0\geq \ord g\geq 1=\ord(x_i).\]
 Hence, the coordinate change stabilizes the order of the coefficient ideal by Proposition \ref{invariance_of_w} (1). 
 
 By Proposition \ref{invariance_of_w} (2) also coordinate changes of type (ii) leave the order of the coefficient ideal invariant.
 
 Coordinate changes of type (iii) obviously leave the coefficient ideal itself invariant.
 
 By Proposition \ref{coeff_ideal_well_def} coordinate changes of type (iv) leave the coefficient ideal itself invariant as well.
 
 (2): Following the same reasoning as above, we only have to show that a triangular coordinate change of the form $x_i\mapsto x_i+g(\x_-,z)$ for an index $i\neq j$ fulfills the conditions of Proposition \ref{invariance_of_w} (1). But this is clear since
 \[\ord_{(x_j)}J_{n,(\x,z)}\geq 0=c!\cdot \ord_{(x_j)}(x_i)\]
 and 
 \[\ord_{(x_j)}g_0\geq 0=\ord_{(x_j)}(x_i).\]
 
 Now assume that $r_{H,D}\geq c!$. We then have to show that a triangular coordinate change of the form $x_j\mapsto x_j+g$ with $\ord_{(z)} g\geq1$ fulfills the conditions of Proposition \ref{invariance_of_w} (1). Since $g_0=0$, this is immediate from
 \[\ord_{(x_j)}J_{n,(\x,z)}=r_{H,D}\geq c!=c!\cdot\ord_{(x_j)}(x_j).\]

 (3): This follows from (1) and (2) since $d_{H,E}$ can be expressed as
 \[d_{H,E}=o_H-\sum_{i\in\Delta} r_{H,V(x_i)}.\]
\end{proof}

 As a second application, we will prove that the invariants $m_\F$ and $d_\F$ that were introduced in Section \ref{section_modifying_the_residual_order} for flags $\F$ with $n_\F>0$ are well-defined.

\begin{proposition} \label{m_d_coord_indep}
 Let $R$ be the power series ring in $3$ variables over a field $K$, $J\subseteq R$ an ideal and $c,n>0$ positive integers.
 
 Define for each regular system of parameters $\x=(x,y,z)$ the coefficient ideal
 \[J_{2,\x}=\coeff_{(x,y,z)}^c(J).\]
 
 Let $\F$ be a flag consisting of a regular hypersurface $\F_2$ and a regular curve $\F_1$ in $\Spec(R)$ which fulfill $\F_1\subseteq\F_2$. A regular system of parameters $\x=(x,y,z)$ is said to be subordinate to $\F$ if $\F_2=V(z)$ and $\F_1=V(z,y)$.
 
 Let $\w_n:K[[x,y]]\to\N$ be the weighted order function defined on $(x,y)$ by $\w_n(x)=1$ and $\w_n(y)=n$. Define for each subordinate system of parameters $\x$ the numbers
 \[m_{\F,\x}=\w_n(J_{2,\x}),\]
 \[d_{\F,\x}=\ord_{(y)}\minit_{\w_n}(J_{2,\x}).\]
 Then the following hold:
 \begin{enumerate}[(1)]
  \item The number $m_\F$ which is defined via 
  \[m_\F=\begin{cases}
         m_{\F,\x} & \text{if $m_{\F,\x}\geq nc!$,} \\
         nc! & \text{if $m_{\F,\x}< nc!$}
        \end{cases}
\]
 is independent of the choice of subordinate parameters $\x$.

  \item The number $d_\F$ which is defined via 
 \[d_\F=\begin{cases}
         d_{\F,\x} & \text{if $d_{\F,\x}\geq c!$,}\\
         d_{\F,\x} & \text{if $0<d_{\F,\x}<c!$ and $c!\nmid m_\F$,}\\
         -1 & \text{if $0<d_{\F,\x}<c!$ and $c!\mid m_\F$,}\\
         -1 & \text{if $d_{\F,\x}=0$.}
        \end{cases}
\]  
  is independent of the choice of subordinate parameters $\x$.
 \end{enumerate}
\end{proposition}
\begin{proof}
 Let $\x=(x,y,z)$ be a regular system of parameters that is subordinate to $\F$. Coordinate changes that stabilize the flag $\F$ are composed of the following types of coordinate changes:
 \begin{enumerate}[(i)]
  \item Triangular coordinate changes $x\mapsto x+g(y,z)$ with $g\in K[[y,z]]$, $\ord g\geq1$.
  \item Triangular coordinate changes $y\mapsto y+g(x,z)$ with $g\in K[[x,z]]$, $\ord_{(z)} g\geq1$.
  \item Multiplications of $x,y$ and $z$ with units.
 \end{enumerate}
 By Proposition \ref{coeff_ideal_well_def} and Proposition \ref{invariance_of_w} (2), coordinate changes of type (iii) always leave the numbers $m_{\F,\x}$ and $d_{\F,\x}$ invariant.
 
 For the coordinate changes of type (i) and (ii) we have to verify that the conditions of Proposition \ref{invariance_of_w} (1) are fulfilled.
 
 (1): Assume first that $m_{\F,\x}\geq nc!$. For a coordinate change of type (i), the conditions of Proposition \ref{invariance_of_w} (1) are fulfilled since
 \[\w_n(J_{2,\x})\geq nc!\geq c!=c!\cdot \w_n(x)\]
 and
 \[\w_n(g_0)=\ord g_0\geq 1=\w_n(x).\]
 For a coordinate change of type (ii), notice that $g_0=0$ since $\ord_{(z)}g\geq1$. Thus, we can compute that
 \[\w_n(J_{2,\x})\geq nc!=c!\cdot \w_n(y)\]
 and 
 \[\w_n(g_0)=\infty\geq\w_n(y).\]
 Hence, $m_\F$ is independent of $\x$ whenever $m_{\F,\x}\geq nc!$. But from this, it also follows that $m_\F$ is independent of $\x$ if $m_{\F,\x}<nc!$.
 
 (2): Define the weighted order function $\y_n:K[[x,y]]\to\Ni^2$ via $\y_n(x)=(1,0)$ and $\y_n(y)=(n,1)$. Then, by Lemma \ref{double_weighted order function}, we know that
 \[\y_n(J_{2,\x})=(m_{\F,\x},d_{\F,\x}).\]
 As a preparation for showing that $d_\F$ is independent of $\x$, we will first show that, $\y_n(J_{2,\x})$ (and hence also $d_{\F,\x}$) is independent of the choice of subordinate parameters $\x$ under the condition that $\y_n(J_{2,\x}\geq (nc!,c!)$ holds. 
 
 So assume that $\y_n(J_{2,\x})\geq (nc!,c!)$. For a coordinate change of type (i), we can compute that
 \[\y_n(J_{2,\x})\geq (nc!,c!)\geq (c!,0)=c!\cdot \y_n(x)\]
 and 
 \[\y_n(g_0)\geq (\w_n(g_0),0)=(\ord g_0,0)\geq (1,0)=\y_n(x).\]
 For a coordinate change of type (ii), we compute that
 \[\y_n(J_{2,\x})\geq (nc!,c!)=c!\cdot\y_n(y)\]
 and 
 \[\y_n(g_0)=(\infty,\infty)\geq\y_n(y).\]
 
 Thus, the conditions of Proposition \ref{invariance_of_w} (1) are fulfilled. Further, notice that the inequality 
 \[d_{\F,\x}\leq\frac{1}{n}m_{\F,\x}\]
 holds. Using these two facts, can now prove that $d_\F$ is independent of $\x$.
 
 If $\y_n(J_{2,\x})\geq (nc!,c!)$ holds, then $d_{\F,\x}$ is independent of $\x$ and hence, it is clear by the definition of $d_\F$ that it is also independent.
 
 So we may assume that $\y_n(J_{2,\x})<(nc!,c!)$. Hence, either $m_{\F,\x}=nc!$ and $d_{\F,\x}<c!$ or $m_{\F,\x}<nc!$. If the latter holds, then also $d_{\F,\x}\leq\frac{1}{n}m_{\F,\x}<c!$. In either case, $m_\F=nc!$ holds by definition. In particular, $c!\mid m_\F$. Since $d_{\F,\x}<c!$ holds for all subordinate parameters $\x$, it is clear that $d_\F$ is also independent of $\x$ in this case.
\end{proof}

\section{Coordinate-independence of the order of the second coefficient ideal} \label{section_invariance_of_slope}

 In this section we will investigate how the order of the second coefficient ideal behaves under coordinate changes. The goal is to prove that the invariant $s_\F$ as it was defined in Chapter \ref{chapter_pathologies} is independent of the choice of subordinate parameters. 
 
 Throughout this section, we will always consider the following setting:
 
 Let $R=K[[\x,y,z]]$ with $\x=(x_1,\ldots,x_n)$ 
 and $J\subseteq R$ an ideal of order $c=\ord J>0$. Let $J_{-1}$ be the coefficient ideal
 \[J_{-1}=\coeff^c_{(\x,y,z)}(J).\]
 Let $J_{-1}$ have a factorization of the form
 \[J_{-1}=(\x^r y^{r_y})\cdot I_{-1}\]
 where $r=(r_1,\ldots,r_n)\in\N^n$ and $I_{-1}$ is an ideal in $K[[\x,y]]$. Set $d=\ord I_{-1}$. Then we define
 \[J_{-2}=\coeff^d_{(\x,y)}(I_{-1})\]
 and set $s=\ord J_{-2}$. Further, we define the auxiliary number
 \[\D=\frac{1}{c!}\Big((d+r_y)\sd+|r|\Big).\]

 We are interested in the behavior of the order of the second coefficient ideal $J_{-2}$ under coordinate changes which stabilize all geometric objects in $\Spec(R)$ that are involved in its definition. These objects are:
 \begin{itemize}
  \item The regular hypersurface $V(z)$.
  \item The regular subscheme $V(y,z)$. (With respect to which $J_{-2}$ is defined.)
  \item The regular hypersurfaces $V(x_i)$ for which $r_i>0$ holds.
  \item The regular hypersurface $V(y)$ if $r_y>0$ holds.
 \end{itemize}

 Using the same argumentation as in Section \ref{section_invariance_of_weighted_orders}, it is clear that we have to consider the following four types of coordinate changes:
 \begin{enumerate}[(i)]
  \item Triangular coordinate changes $x_i\mapsto x_i+g(\x_-,y,z)$. 
  (Only if $r_i=0$.)
  \item Multiplications with units $x_i\mapsto u x_i$ where $u\in R^*$.
  \item Triangular coordinate changes $y\mapsto y+g(\x,z)$ with $\ord_{(z)}g\geq1$. (Only if $r_y=0$.)
  \item Multiplications with units $y\mapsto uy$ where $u\in R^*$.
 \end{enumerate}

 Naturally, it only makes sense to compare the order of the second coefficient ideal if the same monomial can be factored from the first coefficient ideal after the coordinate change. Hence, we make the following definitions:
 
 Let $(\wt\x,\wt y,\wt z)$ be another regular system of parameters for $R$. Set 
 \[\wt J_{-1}=\coeff^c_{(\wt\x,\wt y,\wt z)}(J).\]
 We say that the change of coordinates $(\x,y,z)\mapsto (\wt\x,\wt y,\wt z)$ \emph{preserves the setting} if $\wt J_{-1}$ has a factorization
 \[\wt J_{-1}=(\wt \x^r \wt y^{r_y})\cdot \wt I_{-1}\]
 for an ideal $\wt I_{-1}$ that fulfills $\ord \wt I_{-1}=d$.
 
 We then define
 \[\wt J_{-2}=\coeff^d_{(\wt \x,\wt y)}(\wt I_{-1}),\]
 and set $\wt s=\ord \wt J_{-2}.$ We also define the auxiliary number
 \[\wt \D=\frac{1}{c!}\Big((d+r_y)\wsd+|r|\Big).\]
 
 In the main technical results of this section, Proposition \ref{s_well_def_w} and Proposition \ref{s_well_def_y}, we will show that coordinate changes of type (i), (ii) and (iv) always preserve the setting and leave the order of the second coefficient ideal unchanged, but a coordinate change of type (iii) will only do so if the condition $\D\geq\sd$ holds. By Lemma \ref{Dgeq1}, this inequality automatically holds if $d\geq c!$. If $d<c!$, the definition of the second coefficient ideal has the be modified for its order to still be invariant under the coordinate changes (i)-(iv). In this case, we consider instead of $J_{-2}$ the coefficient ideal
 \[\coeff_{(\x,y)}^{d(c!-d)}(I_{-1}^{c!-d}+M_{-1}^d)\]
 where $M_{-1}=(\x^ry^{r_y})$. The invariance will be proven in Proposition \ref{slope_coord_indep}.
 
 As a preparation for proving the propositions \ref{s_well_def_w} and \ref{s_well_def_y}, we will first prove three rather technical lemmas. 
 
 
 In the following, consider for each element $f\in J$ the expansions $f=\sum_{i\geq0}f_iz^i$, $f=\sum_{i,j\geq0}f_{i,j}y^iz^i$ and $f=\sum_{i,j\geq0}\wt f_{i,j}\wt y^j\wt z^i$ where $f_i\in K[[\x,y]]$, $f_{i,j}\in K[[\x]]$ and $\wt f_{i,j}\in K[[\wt\x]]$.
 
 

The following lemma gives a lower bound for the order of the coefficients $f_{i,j}$.

\begin{lemma} \label{factorization_lemma}
 Let $f\in J$ be an element and $i,j\geq0$ indices. Then the following inequality holds:
  \[\ord f_{i,j}\geq \frac{c-i}{c!}|r|.\]
\end{lemma}
\begin{proof}
 Let $x_k$ be one of the parameters in $\x=(x_1,\ldots,x_n)$. Then by Lemma \ref{w_coeff} we know that 
 \[\ord_{(x_k)} f_{i,j}\geq \ord_{(x_k)}f_i\geq \frac{c-i}{c!}r_k.\]
 Consequently,
 \[\ord f_{i,j}\geq \sum_{k=1}^n\ord_{(x_k)}f_{i,j}\geq \frac{c-i}{c!}|r|.\]
\end{proof}

The following lemma gives a formula for how to compute the order of $J_{-2}$ from the orders of the coefficients $f_{i,j}$ of elements $f=\sum f_{i,j}y^jz^i$ of $J$. Similarly to Lemma \ref{w_coeff}, this will be used many times throughout the thesis.

\begin{lemma} \label{ord_J_2}
 The equality 
 \[\ord J_{-2}=\min_{f\in J}\min_{\substack{i<c\\j<\frac{c-i}{c!}(d+r_y)}}\frac{d!}{d+r_y-\frac{c!}{c-i}j}\Big(\frac{c!}{c-i}\ord f_{i,j}-|r|\Big).\]
 holds. Furthermore, for all elements $f\in J$ and all indices $i,j\geq0$, the inequality
 \[\ord f_{i,j}\geq(c-i)\D-j\sd\]
 holds. 
 
 If $G\subseteq J$ is a generating set for $J$, then the equality 
 \[\ord J_{-2}=\min_{f\in G}\min_{\substack{i<c\\j<\frac{c-i}{c!}(d+r_y)}}\frac{d!}{d+r_y-\frac{c!}{c-i}j}\Big(\frac{c!}{c-i}\ord f_{i,j}-|r|\Big)\]
 holds.
\end{lemma}
\begin{proof}
 Define the number $s_0$ as 
 \[s_0=\min_{f\in J}\min_{\substack{i<c\\j<\frac{c-i}{c!}(d+r_y)}}\frac{d!}{d+r_y-\frac{c!}{c-i}j}(\frac{c!}{c-i}\ord f_{i,j}-|r|)\]
 and 
 \[\D_0=\frac{1}{c!}\Big((d+r_y)\frac{s_0}{d!}+|r|\Big).\]
 We will prove that $s=s_0$.
 
 First we are going to prove that the inequality 
 \[\ord f_{i,j}\geq(c-i)\D_0-j\frac{s_0}{d!}\]
 holds for all elements $f\in J$ and indices $i,j\geq0$. It is clear by the definition of $s_0$ and $\D_0$ that the inequality holds for indices $i<c$ and $j<\frac{c-i}{c!}(d+r_y)$. The inequality is trivial for indices $i\geq c$ and $j\geq0$. So let $i<c$ and $j\geq\frac{c-i}{c!}(d+r_y)$. In this case, we can compute by Lemma \ref{factorization_lemma} that   
 \[\ord f_{i,j}\geq \frac{c-i}{c!}|r|\geq\frac{c-i}{c!}|r|+(\frac{c-i}{c!}(d+r_y)-j)\frac{s_0}{d!} \]
 \[=(c-i)\D_0-j\frac{s_0}{d!}.\] 
 
 By the definition of the coefficient ideal, $J_{-1}$ is generated by elements $g_i=f_i^{\frac{c!}{c-i}}$ for $i<c$. It is easy to see that the elements $g_i$ have the expansion $g_i=\sum_{k\geq0}g_{i,k}y^k$ with
 \[g_{i,k}=\sum_{\substack{\alpha\in\N^\frac{c!}{c-i}\\|\alpha|=k}}f_{i,\alpha}.\]
 We know that there are elements $h_{i,k}\in K[[\x]]$ such that $g_{i,k}=\x^r\cdot h_{i,k}$ for all indices $i<c$ and $k\geq0$. Consequently,
 \[J_{-2}=(h_{i,k}^\frac{d!}{d+r_y-k}:f\in J,i<c,k<d+r_y).\]
 Thus, we can compute that
 \[s=\min_{f\in J}\min_{\substack{i<c\\k<d+r_y}}\frac{d!}{d+r_y-k}\ord h_{i,k}\]
 \[=\min_{f\in J}\min_{\substack{i<c\\k<d+r_y}}\frac{d!}{d+r_y-k}(\ord g_{i,k}-|r|)\]
 \[=\min_{f\in J}\min_{\substack{i<c\\k<d+r_y}}\min_{\substack{\alpha\in\N^{\frac{c!}{c-i}}\\|\alpha|=k}}\frac{d!}{d+r_y-k}(\ord f_{i,\alpha}-|r|)\]
 \[\geq \min_{f\in J}\min_{\substack{i<c\\k<d+r_y}}\min_{\substack{\alpha\in\N^{\frac{c!}{c-i}}\\|\alpha|=k}}\frac{d!}{d+r_y-k}\frac{c!}{c-i}(c-i)\D_0-k\frac{s_0}{d!}-|r|=s_0.\]
 
 To show that $s=s_0$ holds, we will first show that there is an element $f\in J$ and indices $i<c$, $k<d+r_y$ such that 
 \[\ord g_{i,k}=c!\D_0-k\frac{s_0}{d!}.\]
 To this end, let $i<c$ and $j<\frac{c!}{c-i}(d+r_y)$ be indices such that
 \[\ord f_{i,j}=(c-i)\D_0-j\frac{s_0}{d!}.\]
 Further, assume that $j$ is minimal with this property. This implies that $\ord f_{i,\alpha}>\frac{c!}{c-i}\ord f_{i,j}$ for all $\alpha\in\N^{\frac{c!}{c-i}}$ that fulfill $|\alpha|=\frac{c!}{c-i}j$ and $\alpha\neq(j,\ldots,j)$. Set $k=\frac{c!}{c-i}j$. Then
 \[g_{i,k}=\frac{c!}{c-i}\ord f_{i,j}\]
 \[=\frac{c!}{c-i}((c-i)\D_0-j\frac{s_0}{d!})\]
 \[=c!\D_0-k\frac{s_0}{d!}.\]
 Using the above calculations, we conclude that
 \[s\leq\frac{d!}{d+r_y-k}(\ord g_{i,k}-|r|)\]
 \[=\frac{d!}{d+r_y-k}(c!D_0-k\frac{s_0}{d!}-|r|)=s_0.\]
 Hence, we have proved the equality $s=s_0$.
 
 The final claim follows from a straightforward calculation.
\end{proof}

 The following Lemma will give us sufficient conditions to verify that either $\ord \wt J_{-2}\geq \ord J_{-2}$ or $\ord\wt J_{-2}\leq \ord J_{-2}$ holds:

\begin{lemma} \label{s_changes}
 Consider a coordinate change $(\x,y,z)\mapsto(\wt\x,\wt y,\wt z)$ that preserves the setting. The following hold:
 \begin{enumerate}[(1)]
  \item If for all elements $f\in J$ and indices $i<c$, $j<\frac{c-i}{c!}(d+r_y)$ the inequality
  \[\ord \wt f_{i,j}\geq (c-i)\D-j\sd\]
  holds, then $\ord \wt J_{-2}\geq \ord J_{-2}$.
  \item If there exists an element $f\in J$ and indices $i<c$, $j<\frac{c-i}{c!}(d+r_y)$ such that
  \[\ord \wt f_{i,j}=(c-i)\D-j\sd\]
  holds, then $\ord \wt J_{-2}\leq \ord J_{-2}$.
  \item If there exists an element $f\in J$ and indices $i<c$, $j<\frac{c-i}{c!}(d+r_y)$ such that
  \[\ord \wt f_{i,j}<(c-i)\D-j\sd\]
  holds, then $\ord \wt J_{-2}<\ord J_{-2}$.
 \end{enumerate}
\end{lemma}
\begin{proof}
 This follows from Lemma \ref{ord_J_2} by a straightforward calculation.
\end{proof}

\begin{lemma} \label{Dgeq1}
 The following hold:
 \begin{enumerate}[(1)]
  \item $\D\geq1$.
  \item If $d\geq c!$, then $D\geq\sd$.
 \end{enumerate}
\end{lemma}
\begin{proof}
 (1): Notice that $s\geq d!$ and $\ord J_{-1}\geq c!$ hold by Lemma \ref{ogeqc!}. Hence,
 \[\D=\frac{1}{c!}((d+r_y)\sd+|r|)\geq \frac{1}{c!}(d+r_y+|r|)=\frac{\ord J_{-1}}{c!}\geq 1.\]
 
 (2): This is obvious.
\end{proof}


\begin{proposition} \label{s_well_def_w}
 Let $x_i$ be one of the parameters $\x=(x_1,\ldots,x_n)$. Consider one of the following types of coordinate changes:
  \begin{enumerate}[(1)]
  \item $x_i=\wt x_i+g$ with $g\in K[[\x_-,y,z]]$, where $\x_-=(x_1,\ldots,x_{i-1},x_{i+1},\ldots,x_n)$ and $\ord g\geq1$.
  
  We require that $r_i=\ord_{(x_i)}J_{-1}=0$ holds for this type of coordinate change.
  \item $x_i=u\wt x_i$ for a unit $u\in R^*$.
 \end{enumerate}
 Set $\wt\x=(x_1,\ldots,x_{i-1},\wt x_i,x_{i+1},\ldots,x_n)$. Then the change of coordinates $(\x,y,z)\mapsto (\wt\x,y,z)$ preserves the setting and $\ord\wt J_{-2}=\ord J_{-2}$ holds.
\end{proposition}
\begin{proof}
 The coordinate change preserves the setting by Proposition \ref{residual_order_coord_indep}.
 
 It remains to show that $\wt s=s$. To this end, let each element $f\in J$ have expansions $f=\sum_{i,j,k\geq0}f_{i,j,k}x_i^ky^jz^i$ and $f=\sum_{i,j,k\geq0}\wt f_{i,j,k}\wt x_i^ky^jz^i$ with $f_{i,j,k},\wt f_{i,j,k}\in K[[\x_-]]$. By Lemma \ref{ord_J_2}, for all indices $i,j,k\geq0$ the inequality 
 \[\ord f_{i,j,k}\geq (c-i)\D-j\frac{s}{d!}-k\]
 holds.
 
 (1): Let $g$ have the expansion $g=\sum_{i,j\geq0}g_{i,j}y^jz^i$ with $g_{i,j}\in K[[\x_-]]$. Using the formula in Lemma \ref{coordinate_change_g} (1), we can compute that
 \[\wt f_{i,j,k}=\sum_{\substack{0\leq a\leq i\\0\leq b\leq j\\l\geq k}}f_{a,b,l}\binom{l}{k}\sum_{\substack{\alpha,\beta\in\N^{l-k}\\|\alpha|=i-a\\|\beta|=j-b}}g_{\alpha,\beta}\]
 where $g_{\alpha,\beta}=\prod_{o=1}^{l-k}g_{\alpha_o,\beta_o}$. 
 
 This implies that
 \[\ord \wt f_{i,j,k}\geq\min_{\substack{0\leq a\leq i\\0\leq b\leq j\\l\geq k}}\min_{\substack{\alpha,\beta\in\N^{l-k}\\|\alpha|=i-a\\|\beta|=j-b}}\ord f_{a,b,l} g_{\alpha,\beta}.\]
 Further, since $\ord g\geq1$, we know that $\ord g_{0,0}\geq1$. Thus, for multi-indices $\alpha,\beta\in\N^{l-k}$ with $|\alpha|=i-a$ and $|\beta|=j-b$, we can conclude that
 \[\ord g_{\alpha,\beta}\geq (l-k)-(i-a)-(j-b).\]
 So let $i<c$, $j<\frac{c-i}{c!}(d+r_y)$ and $k\geq0$ be indices. We know that there exist indices $a\leq i$, $b\leq j$ and $l\geq k$ such that 
 \[\ord\wt f_{i,j,k}\geq \ord f_{a,b,l}+(l-k)-(i-a)-(j-b).\]
 Since we know by Lemma \ref{ogeqc!} that $\sd\geq 1$ and by Lemma \ref{Dgeq1} (1) that $\D\geq 1$, we can compute with Lemma \ref{ord_J_2} that
 \[\ord\wt f_{i,j,k}\geq \ord f_{a,b,l}+(l-k)-(i-a)-(j-b)\]
 \[\geq ((c-a)\D-b\sd-l)+(l-k)-(i-a)\D-(j-b)\sd\]
 \[=(c-i)\D-j\sd-k.\]
 This implies by Lemma \ref{s_changes} (1) that $\wt s\geq s$. By a symmetric argument, $\wt s=s$ holds.
 
 (2): Using the formula in Lemma \ref{coordinate_change_g} (2), we can compute that
 \[\ord \wt f_{i,j,k}\geq\min_{\substack{0\leq a\leq i\\0\leq b\leq j\\0\leq l\leq k}}\ord f_{a,b,l}.\]
 
 So let $i<c$, $j<\frac{c-i}{c!}(d+r_y)$ and $k\geq0$ be indices. We know that there are indices $a\leq i$, $b\leq j$ and $l\leq k$ such that $\ord \wt f_{i,j,k}\geq\ord f_{a,b,l}$. Thus, we can compute that
 \[\ord  \wt f_{i,j,k}\geq\ord f_{a,b,l}\]
 \[\geq (c-a)\D-b\sd-l\]
 \[\geq (c-i)\D-j\sd-k.\]
 This implies by Lemma \ref{s_changes} (1) that $\wt s\geq s$. By a symmetric argument, $\wt s=s$ holds.
\end{proof}

\begin{proposition} \label{s_well_def_y}
 Consider one of the following types of coordinate changes:
  \begin{enumerate}[(1)]
  \item $y=\wt y+g$ with $g\in K[[\x,z]]$ and $\ord_{(z)} g\geq1$.
  
  We require for this type of coordinate change that the following properties hold:
  \begin{itemize}
   \item $r_y=0$.
   \item $\D\geq\sd$.
  \end{itemize}

  \item $y=u\wt y$ for a unit $u\in R^*$.
 \end{enumerate}
 The change of coordinates $(\x,y,z)\mapsto (\x,\wt y,z)$ preserves the setting and $\ord\wt J_{-2}=\ord J_{-2}$ holds.
\end{proposition}
\begin{proof}
 The coordinate change preserves the setting by Proposition \ref{residual_order_coord_indep}.
 
 It remains to show that $\wt s=s$. By Lemma \ref{ord_J_2}, for all indices $i,j\geq0$ the inequality 
 \[\ord f_{i,j}\geq (c-i)\D-j\frac{s}{d!}\]
 holds.
 
 (1): Let $g$ have the expansion $g=\sum_{j\geq0}g_jz^j$ with $g_j\in K[[\x]]$. Notice that $g_0=0$. By Lemma \ref{coordinate_change_g} we know that
 \[\wt f_{i,j}=\sum_{\substack{0\leq k\leq i\\j\leq l}}\sum_{\substack{\alpha\in\N^{l-j}\\|\alpha|=i-k}}f_{k,l}\binom{l}{j}g_\alpha.\]
 Notice that $g_\alpha\neq0$ implies that $i-k\geq l-j$ since $g_0=0$. Hence, we know that
 \[\ord \wt f_{i,j}\geq\min_{\substack{0\leq k\leq i\\j\leq l\leq j+i-k}}\ord f_{k,l}.\]
 Let $i<c$ and $j<\frac{c-i}{c!}d$ be indices. Then there are indices $k\leq i$ and $j\leq l$ with $i-k\geq l-j$ such that $\ord \wt f_{i,j}\geq\ord f_{k,l}$. This allows us to compute that
 \[\ord \wt f_{i,j}\geq\ord f_{k,l}\]
 \[\geq (c-k)\D-l\sd\]
 \[=(c-i)\D-j\sd+\underbrace{(i-k)}_{\geq(l-j)}\D-(l-j)\sd\]
 \[=(c-i)\D-j\sd+(l-j)\underbrace{(\D-\sd)}_{\geq0}.\]
 This implies by Lemma \ref{s_changes} (1) that $\wt s\geq s$ holds. By a symmetric argument, $\wt s=s$ holds.
 
 (2): By Lemma \ref{coordinate_change_g} we know that
 \[\ord f_{i,j}\geq \min_{\substack{0\leq k\leq i\\0\leq l\leq j}}\sum_{\substack{\alpha\in\N^l\\|\alpha|=i-k}}\ord f_{k,l}.\]
 So let $i<c$ and $j<\frac{c-i}{c!}(d+r_y)$ be indices. We know that there are indices $k\leq i$ and $l\leq k$ such that $\ord \wt f_{i,j}\geq\ord f_{k,l}$. Thus, we can compute that by Lemma \ref{ord_J_2} that
 \[\ord  \wt f_{i,j}\geq\ord f_{k,l}\]
 \[\geq (c-k)\D-l\sd\]
 \[\geq (c-i)\D-j\sd.\]
 This implies by Lemma \ref{s_changes} (1) that $\wt s\geq s$. By a symmetric argument, $\wt s=s$ holds.
\end{proof}


\begin{proposition} \label{slope_coord_indep}
 Let $R$ be the power series ring in $(n+2)$ variables over a field $K$ and $J\subseteq R$ an ideal of order $\ord J=c$.
 
 Let $H_{-1}\subseteq\Spec(R)$ be a regular hypersurface and $H_{-2}\subseteq H_{-1}$ a regular hypersurface in $H_{-1}$. Let $E\subseteq\Spec(R)$ be a simple normal crossings divisor with the following properties:
 \begin{itemize}
  \item $H_{-1}\not\subseteq E$.
  \item $H_{-1}\cup E$ has simple normal crossings.
  \item $(H_{-1}\cap E)\cup H_{-2}$ has simple normal crossings.
 \end{itemize}
 Then there exist numbers $d,s\in\Ni$ with the following properties:
 
 Let $(\x,y,z)$ be a regular system of parameters for $R$ subject to the conditions:
 \begin{itemize}
  \item $H_{-1}=V(z)$.
  \item $H_{-2}=V(z,y)$.
  \item $E=V(\prod_{i\in\Delta}x_i)$ or $E=V(\prod_{i\in\Delta}x_i\cdot y)$ for some subset $\Delta\subseteq\{1,\ldots,n\}$.
 \end{itemize}
 Then the coefficient ideal $J_{n+1,(\x,y,z)}=\coeff_{(\x,y,z)}^c(J)$ has a factorization
  \[J_{n+1,(\x,y,z)}=M_{n+1,(\x,y,z)}\cdot I_{n+1,(\x,y,z)}\]
  with $M_{n+1,(\x,y,z)}=(\prod_{i\in\Delta}x_i^{r_i}y^{r_y})$ where $r_i=\ord_{(x_i)}J_{n+1,(\x,y,z)}$,
  \[r_y=\begin{cases}
         \ord_{(y)}J_{n+1,(\x,y,z)} & \text{if $V(y)\subseteq E$,}\\
         0 & \text{if $V(y)\not\subseteq E$,}
        \end{cases}
\]
and
  \[\ord I_{n+1,(\x,y,z)}=d.\]
  Further, the following hold:
  \begin{enumerate}[(1)]
   \item If $d\geq c!$, then 
   \[\ord \coeff_{\x}^d(I_{n+1,(\x,y,z)})=s.\]
   \item If $d<c!$, then 
   \[\ord \coeff_{\x}^{d(c!-d)}(I_{n+1,(\x,y,z)}^{c!-d}+M_{n+1,(\x,y,z)}^d)=s.\]
  \end{enumerate}
\end{proposition}
\begin{proof}
 The existence of the number $d$ has already been proved in Proposition \ref{residual_order_coord_indep}.
 
 (1): Let $(\x,z)$ be a regular system of parameters for $R$ that fulfills the stated conditions. Set
 \[s_0=\coeff_{\x}^d(I_{n+1,(\x,y,z)})\]
 and
 \[\D_0=\frac{1}{c!}\Big((d+r_n)\frac{s_0}{d!}+|r_-|\Big)\]
 where $r_-=(r_1,\ldots,r_{n-1})$.
 
 We have to show the invariance of $s_0$ under the following types of coordinates changes:
 \begin{enumerate}[(i)]
  \item Triangular coordinate changes $x_i\mapsto x_i+g$ for $i\notin\Delta$ where $g\in K[[\x_-,y,z]]$ for $\x_-=(x_1,\ldots,x_{i-1},x_{i+1},\ldots,x_n)$ .
  \item Multiplications with units $x_i\mapsto u x_i$ where $u\in R^*$.
  \item Triangular coordinate changes $y\mapsto y+g$ if $V(y)\not\subseteq E$ where $g\in K[[\x,z]]$ and $\ord_{(z)}g\geq1$.
  \item Multiplications with units $y\mapsto uy$ where $u\in R^*$.
  \item Multiplications with units $z\mapsto uz$ where $u\in R^*$.
 \end{enumerate}

 Coordinate changes of type (i) leave $s_0$ invariant by Proposition \ref{s_well_def_w} (1).
 
 Coordinate changes of type (ii) leave $s_0$ invariant by Proposition \ref{s_well_def_w} (2).
 
 Coordinate changes of type (iii) leave $s_0$ invariant by Proposition \ref{s_well_def_y} (1) since $\frac{s_0}{d!}\leq \D_0$ is fulfilled by Lemma \ref{Dgeq1} (2).
 
 Coordinate changes of type (iv) leave $s_0$ invariant by Proposition \ref{s_well_def_y} (2).
 
 Coordinate changes of type (v) leave $s_0$ invariant by Proposition \ref{coeff_ideal_well_def}.
 
 (2): Set 
 \[m_0=\ord\coeff^{c!-d}_{\x}(M_{n+1,(\x,y,z)}).\]
 By Lemma \ref{coeff_ideal_of_powers} we know that
 \[\ord \coeff_{\x}^{d(c!-d)}(I_{n+1,(\x,y,z)}^{c!-d}+M_{n+1,(\x,y,z)}^d)=(d(c!-d))!\cdot\min\Big\{\frac{s_0}{d!},\frac{m_0}{(c!-d)!}\Big\}.\]
 Notice that $m_0$ is invariant under coordinate changes of types (i)-(iv) by Proposition \ref{residual_order_coord_indep} since $\ord M_{n+1,(\x,y,z)}\geq c!-d$ by Lemma \ref{ogeqc!}. Hence, we can assume without loss of generality that
 \[s_0\leq\frac{d!}{(c!-d)!}m_0\]
 and it remains to show that $s_0$ is invariant under the coordinate changes of type (i)-(iv) under this condition.
 
 For coordinate changes of the types (i),(iii) and (iv) the proof for the invariance of $s_0$ is the same as before.
 
 Now consider a coordinate change of type (ii). Since $r_y=0$ in this case, we can compute that
 \[m_0=\frac{(c!-d)!}{c!-d}|r|\]
 and consequently,
 \[\D_0=\frac{1}{c!}\B(d\frac{s_0}{d!}+|r|\B)\]
 \[\geq \frac{1}{c!}(d\frac{s_0}{d!}+(c!-d)\frac{s_0}{d!})=\frac{s_0}{d!}.\]
 Hence, the coordinate change leaves $s_0$ invariant by Proposition \ref{s_well_def_y} (1).
\end{proof}

\chapter{Maximizing invariants associated to coefficient ideals} \label{chapter_cleaning}

In the Chapter \ref{chapter_invariance} we showed that certain invariants which are associated to the coefficient ideal can be regarded as invariants of the geometric objects which are involved in their definition. Hence, their value does not depend on the choice of a subordinate system of parameters. In this chapter we will analyze how these invariants behave when changing the geometric objects with respect to which the coefficient ideal is defined. The most important example for this is to consider the coefficient ideal $J_{-1}=\coeff_{(\x,z)}^c(J)$ of an ideal $J$ with respect to the formal hypersurface $V(z)$ and apply a coordinate change $z\mapsto z+g(\x)$ with $g\in K[[\x]]$ which moves the underlying hypersurface. In particular, we are interested in finding coordinate changes which \emph{maximize} an associated invariant. The maximal value does then not depend on the choice of a particular hypersurface anymore and can thus be seen as an intrinsic invariant of the ideal $J$.

In Section \ref{section_w_cleaning} we will consider weighted orders $\w$ of the coefficient ideal $J_{-1}$ and their behavior under coordinate changes $z\mapsto z+g(\x)$ which move the underlying hypersurface. We will develop the notion of \emph{$\w$-cleanness} which will be shown to be a sufficient condition for the weighted order $\w(J_{-1})$ to be maximal over all coordinate changes $z\mapsto z+g(\x)$. Further, we will develop the \emph{$\w$-cleaning process} which will allow us to find a specific coordinate change which maximizes the weighted order $\w$ of the coefficient ideal.

In Section \ref{section_s_cleaning} and \ref{section_maximizing_over_y_and_z}, similar considerations will be made for the behavior of the order of the second coefficient ideal $J_{-2}$ which is defined with respect to a formal hypersurface $V(y,z)$ inside $V(z)$. In Section \ref{section_s_cleaning} we will investigate the behavior of the order of $J_{-2}$ under coordinate changes $z\mapsto z+g(\x,y)$, leaving the parameter $y$ fixed. We will again devise a notion of cleanness which guarantees maximality $\ord J_{-2}$ and a cleaning process which allows us to construct a coordinate change that maximizes this order.

In Section \ref{section_maximizing_over_y_and_z} we will then consider the behavior of $\ord J_{-2}$ under simultaneous coordinate changes $z\mapsto z+g(\x,y)$ and $y\mapsto y+h(\x)$. This significantly increases the complexity. For this situation, we will not devise a cleaning process, but only prove that there exists a coordinate change $z\mapsto z+g(\x,y)$, $y\mapsto y+h(\x)$ which maximizes the order of $J_{-2}$ over all such coordinate changes.

The cleaning techniques in this chapter were developed both as a generalization of the Tschirnhausen transformation used in characteristic zero to construct hypersurfaces of maximal contact and the cleaning of a purely inseparable equation $z^{p^e}+F(\x)=0$ where all $p^e$-th powers in the expansion of $F$ can be eliminated via a coordinate change $z\mapsto z+g(\x)$. Unlike these, our cleaning techniques work in any dimension and in arbitrary characteristic. They can be applied to any ideal $J$ as long as it contains an element $f$ which is $z$-regular of order $c=\ord J$.

The techniques of this chapter will be used in Section \ref{section_maximizing_flags} to construct flags $\F$ which maximize the flag invariant $\inv(\F)$ that was introduced in Section \ref{section_modifying_the_residual_order}. The results of Section \ref{section_maximizing_flags} are fundamental to Chapter \ref{chapter_usc} and Chapter \ref{chapter_decrease} where it will be shown that the resolution invariant $\ivX$ is upper semicontinuous and decreases under blowup.

Cleaning techniques similar to the ones introduced in this chapter have been used many times in the literature on resolution of singularities in positive characteristic: \cite {Abhyankar_67}, \cite{Hironaka_Bowdoin}, \cite{Moh}, \cite{Cossart_Piltant_1}, \cite{Cossart_Piltant_2}, \cite{Ha_BAMS_2}, \cite{Cutkosky_Skeleton},  \cite{Hironaka_CMI}, \cite{BV_Monoidal}, \cite{HW}, \cite{Kawanoue_Matsuki_Surfaces}.



\section{Maximizing a weighted order of the coefficient ideal} \label{section_w_cleaning}
 
In this section we will investigate the effect of coordinate changes $z\mapsto z+g(\x)$ on weighted orders of the coefficient ideal $\coeff_{(\x,z)}^c(J)$ and develop a technique to maximize weighted orders over all such coordinate changes. 

Throughout this section, we will always use the following setting:

Let $R=K[[\x,z]]$ with $\x=(x_1,\ldots,x_n)$ and $J\subseteq R$ an ideal of order $c=\ord J$. Let $J_{-1}$ denote the coefficient ideal 
\[J_{-1}=\coeff^c_{(\x,z)}(J)\]
with respect to the regular hypersurface $H=V(z)\subseteq\Spec(R)$.

Let $\w:K[[\x]]\to\Ni^l$ be a weighted-order function that is defined on the parameters $\x$. Set 
\[m=\w(J_{-1}).\]
We will always assume that $m>0$ holds.

Let now $\wt H=V(\wt z)\subseteq\Spec(R)$ be another regular hypersurface. If the element $\wt z$ is not $z$-regular with respect to the parameters $(\x,z)$, then the effect of the coordinate change $z\mapsto \wt z$ on the coefficient ideal will generally be chaotic. Hence, we will exclude such coordinate changes in this chapter. Notice though, that if an element $f\in J$ exists which is $z$-regular of order $c$, then Lemma \ref{z_regular_blocks_other_parameters} gives a strong bound on weighted orders of the coefficient ideal with respect to $V(\wt z)$.

On the other hand, if the element $\wt z$ is $z$-regular, we may assume by the Weierstrass preparation theorem that $z=\wt z+g$ for an element $g\in K[[\x]]$ with $\ord g\geq1$. This is the situation that we will investigate in this section. We set 
\[\wt J_{-1}=\coeff_{(\x,\wt z)}^c(J)\]
to be the coefficient ideal with respect to the hypersurface $\wt H=V(\wt z)$ and 
\[\wt m=\w(\wt J_{-1}).\]

In Lemma \ref{m_under_coord_changes} we will give a basic estimate for the value of $\w(\wt J_{-1})$ dependent on the weighted order $\w(g)$ of $g$. In particular, we will show that $\w(\wt J_{-1})>\w(J_{-1})$ can only hold if $\w(g)=\frac{m}{c!}$ under the condition that there exists an element $f\in J$ which is $z$-regular of order $c$. We will then introduce the definition of \emph{$\w$-cleanness} of such an element $f$. In Proposition \ref{w_cleaning_maximizes_w} we will show that if there exists an element $f\in J$ which is $\w$-clean with respect to $J_{-1}$, then the weighted order $\w(J_{-1})$ is maximal over all coordinate changes $z\mapsto z+g(\x)$. After this, we will describe the \emph{$w$-cleaning process} which is an algorithm to construct from a given parameter $z$ and an element $f\in J$ which is $z$-regular of order $c$ a new parameter $\wt z=z-g(\x)$ so that $f$ is $\w$-clean with respect to $\wt J_{-1}$.

For elements $f\in J$ we will denote their power series expansions with respect to the parameter systems $(\x,z)$ and $(\x,\wt z)$ by $f=\sum_{i\geq0}f_iz^i$ and $f=\sum_{i\geq0}\wt f_i\wt z^i$ with $f_i,\wt f_i\in K[[\x]]$. By Lemma \ref{coordinate_change_g_z} we know that
\[\wt f_i=\sum_{k\geq i}\binom{k}{i}f_kg^{k-i}.\]

Further, we know by Lemma \ref{w_coeff} that for all elements $f\in J$ and indices $i\geq0$ the inequality
\[\w(f_i)\geq\frac{c-i}{c!}m\]
holds. Also, there exists an element $f\in J$ and an index $i<c$ such that equality holds.

We will use the notation $q=q_K(c)$ from Section \ref{section_binom_in_pos_char}, where
\[q_K(c)=\begin{cases}
           1 & \text{if $\chara(K)=0$,} \\
           p^{\ord_pc} & \text{if $\chara(K)=p>0.$}
          \end{cases}\]

\begin{lemma} \label{m_under_coord_changes}
 The following hold:
 \begin{enumerate}[(1)]
  \item If $\w(g)\geq\frac{m}{c!}$, then $\w(\wt J_{-1})\geq \w(J_{-1})$.
  \item If $\w(g)>\frac{m}{c!}$, then $\w(\wt J_{-1})=\w(J_{-1})$.
  \item If $\w(g)<\frac{m}{c!}$ and there exists an element $f\in J$ which is $z$-regular of order $c$, then $\w(\wt J_{-1})=c!\cdot\w(g)<\w(J_{-1})$.
 \end{enumerate}
\end{lemma}
\begin{proof}
 (1): By the formula for $\wt f_i$ we can compute that
 \[\w(\wt f_i)\geq \min_{k\geq i}(\underbrace{\w(f_k)}_{\geq\frac{c-k}{c!}m}+(k-i)\underbrace{\w(g)}_{\geq\frac{m}{c!}})\geq \frac{c-i}{c!}m.\]
 Since this holds for all elements $f\in J$ and all indices $i<c$, we know by Lemma \ref{w_coeff} that $\wt m\geq m$.
 
 (2): By (1) we know that $\wt m\geq m$ holds.
 
 Now let $f\in J$ and $i<c$ be such that $\w(f_i)=\frac{c-i}{c!}m$ holds. We can compute that
 \[\w\Big(\sum_{k>i}\binom{k}{i}f_kg^{k-i}\Big)\geq\min_{k>i}(\underbrace{\w(f_k)}_{\geq\frac{c-k}{c!}m}+(k-i)\underbrace{\w(g)}_{>\frac{m}{c!}})>\frac{c-i}{c!}m=\w(f_i).\]
 Hence, $\w(\wt f_i)=\w(f_i)=\frac{c-i}{c!}m$. By Lemma \ref{w_coeff} this implies that $\wt m\leq m$.
 
 In total, we conclude that $\wt m=m$.

 (3): We will first consider the case that $\w(g)>0$. Let $f\in J$ be $z$-regular of order $c$. Consider the term
 \[\wt f_0=\sum_{k\geq0}f_kg^k.\]
 Observe that $\w(f_cg^c)=c\cdot\w(g)$ since $f_c$ is a unit. We claim that for all indices $k\geq0$ with $k\neq c$, the strict inequality $\w(f_k g^k)>c\cdot\w(g)$ holds. First, consider the case $k>c$. Clearly, 
 \[\w(f_kg^k)\geq\w(g^k)=k\cdot\w(g)>c\cdot\w(g).\]
 On the other hand, consider the case $k<c$. By Lemma \ref{w_coeff} we know that $\w(f_k)\geq\frac{c-k}{c!}m$. Thus,
 \[\w(f_kg^k)\geq\underbrace{\frac{c-k}{c!}m}_{>(c-k)\w(g)}+k\cdot\w(g)>c\cdot\w(g).\]
 We conclude that $\w(\wt f_0)=c\cdot\w(g)$. By Lemma \ref{w_coeff} this implies that 
 \[\wt m\leq c!\cdot \w(g)<m.\]
 
 Assume that $\wt m<c!\cdot \w(g)$ holds. Then statement (2) would imply that $m=\wt m$, which contradicts $\wt m<m$.
 
 Now consider the case that $\w(g)=0$. Define the weighted order function $\wt\w:K[[\x]]\to\Ni^{l+1}$ as $\wt\w(x_i)=(\w(x_i),1)$ for $i=1,\ldots,n$. Then by Lemma \ref{double_weighted order function} we know that
 \[\wt\w(g)=(\w(g),\ord\init_\w(g))\geq(0,1)\]
 since $\ord g\geq1$. Further, $\wt\w(J_{-1})=(m,\ord \minit_\w(J_{-1}))$. Since $\wt\w(g)>0$, the assertion follows from what we have already shown.
\end{proof}

\begin{example}
 The third statement in Lemma \ref{m_under_coord_changes} is wrong when there is no element $f\in J$ which is $z$-regular of order $c$.
 
 For example, consider the ideal $J=(yz-xy)$ with $c=2$ in the ring $R=K[[x,y,z]]$ with the weighted order function $\w=\ord_{(y)}$. Then
 \[J_{-1}=(y^2,xy)\]
 and $m=\ord_{(y)}J_{-1}=1$. Now consider the change of coordinates $z=\wt z+x$. It fulfills $\ord_{(y)}(x)=0<\frac{m}{c!}$. But since $J=(y\wt z)$, we know that
 \[\wt J_{-1}=(y^2)\]
 and consequently, $\ord_{(y)}\wt J_{-1}=2>\ord_{(y)}J_{-1}$.
\end{example}

\subsection{The $\w$-cleanness property}

 

\begin{definition}
 Let $f\in J$ be an element that is $z$-regular of order $c$. Then $f$ is said to be \emph{$\w$-clean} with respect to the coefficient ideal $J_{-1}$ 
 if one of the following properties holds:
 \begin{itemize}
  \item[$(1)_\w$] There is an index $i$ such that $c-q<i<c$ and $\w(f_i)=\frac{c-i}{c!}m$.
  \item[$(2)_\w$] $\w(f_{c-q})>\frac{q}{c!}m$.
  \item[$(3)_\w$] There is no element $G\in K[[\x]]$ such that $\init_\w(f_{c-q})=\init_\w(f_c)\cdot G^q$.
 \end{itemize}
 The definition of $\w$-cleanness depends on the entire regular system of parameters $(\x,z)$. This is suppressed in the notation since the parameters $(\x,z)$ are considered to be part of the information of the coefficient ideal $J_{-1}$. 
\end{definition}

\begin{remark}
 \begin{enumerate}[(1)]
  \item Consider the special case $q=1$. This holds if either $K$ has characteristic zero or $c$ is not divisible by the characteristic of $K$. In this case, neither of the properties $(1)_\w$ or $(3)_\w$ can be fulfilled. Thus, $f$ is $\w$-clean if and only if $\w(f_{c-1})>\frac{m}{c!}$ holds.
  
  This holds in particular if $f_{c-1}=0$. As we discussed in Section \ref{section_woah_maximal_contact}, this can always be achieved by Weierstrass preparation and Tschirnhausen transformation. Thus, the Tschirnhausen transformation guarantees that $f$ is $\w$-clean for all weighted order functions $\w$ which are defined on the parameters $\x$.
  
  \item Consider a purely inseparable power series
  \[f=z^{p^e}+F(\x)\]
  over a field of characteristic $p>0$. In this case, $c=q=p^e$. Since $F=f_0$ is the only coefficient for $i<c$, neither of the properties $(1)_\w$ or $(2)_\w$ can be fulfilled. Thus, $f$ is $\w$-clean if and only if $\init_\w(F)$ is not a $p^e$-th power.
  
  This holds in particular if no $p^e$-th powers appear in the expansion of $F(\x)$. In this case, $f$ is $\w$-clean for all weighted order functions $\w$ which are defined on the parameters $\x$.
 \end{enumerate}
\end{remark}

In Proposition \ref{w_cleaning_maximizes_w} we will show that the existence of an element $f\in J$ which is $\w$-clean with respect to $J_{-1}$ is a sufficient condition for $\w(J_{-1})$ to be maximal over all coordinate changes $z\mapsto z+g(\x)$. As a preparation, we will show in Lemma \ref{clean_lemma} which consequences the conditions $(1)_\w-(3)_\w$ have on the power series expansion of $f$ with respect to the parameters $(\x,\wt z)$.

\begin{remark}
 In the following, we will also often make use of the negations of the properties $(1)_\w-(3)_\w$. They can be formulated in the following way: 
 \begin{itemize}
  \item [$\neg(1)_\w$] For all indices $i$ with $c-q<i<c$ the inequality $\w(f_i)>\frac{c-i}{c!}m$ holds.
  \item [$\neg(2)_\w$] $\w(f_{c-q})=\frac{q}{c!}m$.
  \item [$\neg(3)_\w$] There exists an element $G\in K[[\x]]$ such that $\init_\w(f_{c-q})=\init_\w(f_c)\cdot G^q$.
 \end{itemize}
\end{remark}


\begin{lemma} \label{clean_lemma}
 Let $f\in J$ be an element that is $z$-regular of order $c$. Assume that $\w(g)\geq\frac{m}{c!}$. Then the following hold:
 \begin{enumerate}[$(i)$]
  \item If the property $(1)_\w$ holds and the index $c-q<i<c$ is maximal with the property that $\w(f_i)=\frac{c-i}{c!}m$, then $\w(\wt f_i)=\frac{c-i}{c!}m$ and
  \[\init_\w(\wt f_i)=\init_\w(f_i).\]
  \item If the properties $\neg(1)_\w$ and $(2)_\w$ hold and $\w(g)=\frac{m}{c!}$, then $\w(\wt f_{c-q})=\frac{q}{c!}m$ and
  \[\init_\w(\wt f_{c-q})=\binom{c}{q}\init_\w(f_c)\init_\w(g)^q.\]
  \item If the properties $\neg(1)_\w$ and $\neg(2)_\w$ hold and $\w(g)=\frac{m}{c!}$, then either $\w(\wt f_{c-q})=\frac{q}{c!}m$ and
  \[\init_\w(\wt f_{c-q})=\init_\w(f_{c-q})+\binom{c}{q}\init_\w(f_c)\init_\w(g)^q\]
  or the right-hand term vanishes and $\w(\wt f_{c-q})>\frac{q}{c!}m$.
 \end{enumerate}
\end{lemma}
\begin{proof}
 (i): Notice that $\binom{c}{i}=0$ by Lemma \ref{c-q} and thus the term $f_cg^{c-i}$ does not appear in the expansion of $\wt f_i$. We claim that $\w(f_kg^{k-i})>\frac{c-i}{c!}m$ holds for all indices $k>i$ with $k\neq c$. Clearly, if $k>c$, then
 \[\w(f_kg^{k-i})\geq (k-i)\w(g)>(c-i)\w(g)\geq\frac{c-i}{c!}m.\]
 On the other hand, consider the case $i<k<c$. Since $i$ was chosen maximally, we know that
 \[\w(f_kg^{k-i})\geq \underbrace{\w(f_k)}_{>\frac{c-k}{c!}m}+(k-i)\underbrace{\w(g)}_{\geq\frac{m}{c!}}>\frac{c-i}{c!}m.\]
 This proves that $\w(\wt f_i)=\w(f_i)$ and $\init_\w(\wt f_i)=\init_\w(f_i)$.
 
 (ii): Consider the expansion
 \[\wt f_{c-q}=\sum_{k\geq c-q}\binom{k}{c-q}f_kg^{k-(c-q)}.\]
 Using the same arguments as before, we conclude that $\w(f_kg^{k-(c-q)})>\frac{q}{c!}m$ holds for all indices $k>c-q$ with $k\neq c$. Also, we know that $\w(f_{c-q})>\frac{q}{c!}m$ by assumption. But notice that $\binom{c}{c-q}\neq0$ by Lemma \ref{c-q} and $\w(f_cg^q)=\frac{q}{c!}m$ since $f_c$ is a unit. This proves that $\w(\wt f_{c-q})=\frac{q}{c!}m$ and 
 \[\init_\w(\wt f_{c-q})=\binom{c}{q}\init_\w(f_c)\init_\w(g)^q.\]
 
 Assertion (iii) can be proved in the same way as (ii).
\end{proof}

\begin{proposition} \label{w_cleaning_maximizes_w}
 Let $f\in J$ be an element that is $\w$-clean with respect to $J_{-1}$.
 
 Then $\w(\wt J_{-1})\leq \w(J_{-1})$.
\end{proposition}
\begin{proof}
 By Lemma \ref{m_under_coord_changes} we may assume without loss of generality that $\ord g=\frac{m}{c!}$. By Lemma \ref{w_coeff} it suffices to find an index $i<c$ such that $\w(\wt f_i)=\frac{c-i}{c!}m$ to conclude $\wt m\leq m$.
 
 First assume that the property $(1)_\w$ holds. Let $i<c$ be maximal with the property that $\w(f_i)=\frac{c-i}{c!}m$. By Lemma \ref{clean_lemma} (i) this implies that $\w(\wt f_i)=\frac{c-i}{c!}m$. 
 
 If the properties $\neg(1)_\w$ and $(2)_\w$ hold, we know by Lemma \ref{clean_lemma} (ii) that $\w(\wt f_{c-q})=\frac{q}{c!}m$. 
 
 Finally, assume that the properties $\neg(1)_\w$, $\neg(2)_\w$ and $(3)_\w$ hold. By Lemma \ref{clean_lemma} (iii) we know that either $\w(\wt f_{c-q})=\frac{q}{c!}m$ and
 \[\init_\w(\wt f_{c-q})=\init_\w(f_{c-q})+\binom{c}{q}\init_\w(f_c)\init_\w(g)^q\]
 or the right-hand term vanishes. But this term cannot vanish by property $(3)_\w$. Thus, we know that $\w(\wt f_{c-q})=\frac{q}{c!}m$.
\end{proof}

\subsection{The $\w$-cleaning process}

We will now devise a process to successively construct coordinate changes $z\mapsto z+g(\x)$ that increase the weighted order of the coefficient ideal until it reaches its maximal value. A single such coordinate change will be referred to as an $\w$-cleaning step, while the successive application of $\w$-cleaning steps will be called the $\w$-cleaning process. 

\begin{definition}
 Let $f\in J$ be an element that is $z$-regular of order $c$ and not $\w$-clean with respect to $J_{-1}$. By definition, there exists an element $G\in K[[\x]]$ such that 
 \[\init_\w(f_{c-q})=\init_\w(f_c)\cdot G^q.\]
 
 An \emph{$\w$-cleaning step} with respect to $f$ and $J_{-1}$ is defined as the coordinate change $z=\wt z+g$ where
 \[g=-\binom{c}{q}^{-1}G.\]
 
 Notice that $g$ is weighted homogeneous with respect to $\w$ and $\w(g)=\frac{m}{c!}$.
\end{definition}

In Proposition \ref{w_cleaning_improves_w} we will show that by each application of an $\w$-cleaning step, either the weighted order of the coefficient ideal is increased or $\w$-cleanness (and hence, by Proposition \ref{w_cleaning_maximizes_w}, maximality) is achieved. This will enable us to prove in Proposition \ref{w_cleaning_terminates} that the $\w$-cleaning process either terminates after finitely many iterations or there is a coordinate change $z=\wt z+g(\x)$ with $\w(\wt J_{-1})=\infty$ and hence, $\wt J_{-1}=0$.

\begin{remark}
 Recall that the property of $f$ being $z$-regular of order $c$ is stable under coordinate changes $z\mapsto z+g(\x)$ by Lemma \ref{z_regular_under_coord_change} (1). We will implicitly make use of this in all of the following statements.
\end{remark}

\begin{proposition} \label{w_cleaning_improves_w}
 Let $f\in J$ be an element that is $z$-regular of order $c$ and not $\w$-clean with respect to $J_{-1}$. Let $z=\wt z+g$ be an $\w$-cleaning step with respect to $f$ and $J_{-1}$. Then one of the following holds:
 \begin{itemize}
  \item $\w(\wt J_{-1})=\w(J_{-1})$ and $f$ is $\w$-clean with respect to $\wt J_{-1}$.
  \item $\w(\wt J_{-1})>\w(J_{-1})$.
 \end{itemize}
\end{proposition}
\begin{proof}
 Since $\w(g)=\frac{m}{c!}$, we know that by Lemma \ref{m_under_coord_changes} that $\wt m\geq m$. Further, since 
 \[\init_\w(f_{c-q})+\binom{c}{q}\init_\w(f_c)\init_\w(g)^q=0,\]
 by construction, we know by Lemma \ref{clean_lemma} (iii) that $\w(\wt f_{c-q})>\frac{q}{c!}m$. If $\wt m=m$, this means that $\w(\wt f_{c-q})>\frac{q}{c!}\wt m$. Thus, the property $(2)_\w$ holds and $f$ is $\w$-clean with respect to $\wt J_{-1}$.
\end{proof}

 We will now describe the \emph{$\w$-cleaning process}:
 
 Let $R=K[[\x,z]]$ and $J\subseteq R$ an ideal of order $c=\ord J$. Let $\w:K[[\x,z]]\to\Ni^l$ be a weighted order function defined on the parameters $\x$ such that $\w(\coeff_{(\x,z)}^c(J))>0$. Let $f\in J$ be an element that is $z$-regular of order $c$.
 
 Set $z_0=z$. We will now describe a process to successively construct certain parameters $z_i$ for $i\geq1$.
 
 In each iteration of the process, set 
 \[J_{-1}^{(i)}=\coeff_{(\x,z_i)}^c(J).\]
 If $f$ is $\w$-clean with respect to $J_{-1}^{(i)}$, the process terminates.

 Otherwise, let $z_i=z_{i+1}+g_i$ be an $\w$-cleaning step with respect to $f$ and $J_{-1}^{(i)}$.

\begin{proposition} \label{w_cleaning_terminates}
 Consider the $\w$-cleaning process as described above. One of the following occurs:
 \begin{itemize}
  \item The process terminates in finitely many steps.
  \item The process does not terminate. In this case, $z_\infty=z-\sum_{i\geq0} g_i$ is a well-defined power series and 
  \[\coeff^c_{(\x,z_\infty)}(J)=0.\]
 \end{itemize}
\end{proposition}
\begin{proof}
 Assume that the process does not terminate. Set $m_i=\w(J_{-1}^{(i)})$. By Proposition \ref{w_cleaning_improves_w} we know that $m_{i+1}>m_i$ holds for all $i\geq0$. Thus, also $\w(g_{i+1})>\w(g_i)$ holds for all $i\geq0$. Consequently, $\lim_{i\to\infty}\ord g_i=\infty$. This guarantees that $z_\infty=z-\sum_{i\geq0}g_i$ is a well-defined power series. Set $m_\infty=\w(\coeff^c_{(\x,z_\infty)}(J))$. By Lemma \ref{m_under_coord_changes} we know that $m_\infty>m_i$ holds for all $i\geq0$. Thus, $m_\infty=(\infty,\ldots,\infty)$ and $\coeff_{(\x,z_\infty)}^c(J)=0$.
\end{proof}

\begin{remark}
 Recall that $\coeff^c_{(\x,z_\infty)}(J)=0$ implies that $J=(z_\infty^c)$ by Lemma \ref{coeff_is_zero}.
\end{remark}

\subsection{Additional results on $\w$-cleanness}

Since $\w$-cleanness is a fundamental notion that will be used extensively throughout the remainder of this thesis, we will prove several additional results in the remainder of this section.

In Lemma \ref{c!_does_not_divide_m} we will show that $\w$-cleanness automatically holds if one of the components of $m=\w(J_{-1})\in\Ni^l$ is not divisible by $c!$. In Lemma \ref{coord_changes_which_preserve_w_clean} we will state two sufficient properties for a coordinate change $z\mapsto z+g(\x)$ to preserve the $\w$-cleanness of an element $f\in J$. This will be used in Lemma \ref{w_cleaning_preserves_v_cleaning} to show that $\nu$-cleaning with respect to another weighted order function $\nu$ that is defined on the parameters $\x$ preserves $\w$-cleanness. This is an important result since it guarantees that cleaning processes with respect to finitely many given weighted order functions $\w_1,\ldots,\w_k$ can be applied consecutively to achieve simultaneous $\w_i$-cleanness for $i=1,\ldots,k$ as long as all $\w_i$ are defined on the same parameters $\x$.

Other important questions about $\w$-cleanness will be addressed in Chapter \ref{chapter_6}. In Section \ref{section_weighted_orders_under_blowup} we will show in which sense the cleanness property is preserved under monomial blowup maps. In Section \ref{section_global_local_expansion} we will address the intricate question how the cleanness property and procedure, which are defined in the power series ring, can be extended to a Zariski neighborhood of a closed point on a regular variety. 

\begin{lemma} \label{c!_does_not_divide_m}
 Let $f\in J$ be an element that is $z$-regular of order $c$. If one of the components of $m$ is not divisible by $c!$, then $f$ is $\w$-clean with respect to $J_{-1}$.
\end{lemma}
\begin{proof}
 Assume that $f$ is not $\w$-clean with respect to $J_{-1}$. By property $\neg(2)_\w$ we know that $\w(f_{c-q})=\frac{q}{c!}m$. By $\neg(3)_\w$ we further know that $\w(f_{c-q})=q\cdot\w(G)$ for some element $G\in K[[\x]]$. Thus, $m=c!\cdot\w(G)$.
\end{proof}

 

\begin{lemma} \label{coord_changes_which_preserve_w_clean}
 Let $f\in J$ be an element which is $z$-regular of order $c$. Consider a coordinate change $z=\wt z+g$ where $g\in K[[\x]]$ fulfills one of the following two properties:
 \begin{itemize}
  \item $\w(g)\geq\frac{1}{q}\w(f_{c-q})$.
  \item $\w(g)>\frac{m}{c!}$.
 \end{itemize}
 Then the following hold:
 \begin{enumerate}[(1)]
  \item $\w(\wt J_{-1})\geq \w(J_{-1})$.
  \item If $f$ is $\w$-clean with respect to $J_{-1}$, then $\w(\wt J_{-1})=\w(J_{-1})$ and $f$ is also $\w$-clean with respect to $\wt J_{-1}$.
 \end{enumerate}
 Consequently, we say that such a coordinate change $z=\wt z+g$ \emph{preserves $\w$-cleanness of $f$}.
\end{lemma}
\begin{proof}
 Notice that in both cases $\w(g)\geq\frac{m}{c!}$ holds.

 (1) is immediate by Lemma \ref{m_under_coord_changes} (1).
 
 (2): By Proposition \ref{w_cleaning_maximizes_w} we know that $\wt m=m$ holds. It remains to show that $f$ is $\w$-clean with respect to $\wt J_{-1}$.
 
 First assume that $(1)_\w$ holds for $f$ with respect to $J_{-1}$ and $i<c$ is maximal with $\w(f_i)=\frac{c-i}{c!}m$. By Lemma \ref{clean_lemma} (i) this implies that $\w(\wt f_i)=\frac{c-i}{c!}m$. Hence, the property $(1)_\w$ still holds for $f$ with respect to $\wt J_{-1}$.
 
 Now assume that $\neg(1)_\w$ and $(2)_\w$ hold for $f$ with respect to $J_{-1}$. Thus, $\w(g)>\frac{m}{c!}$ holds by assumption. Consider the expansion 
 \[\wt f_{c-q}=f_{c-q}+\sum_{k>c-q}\binom{k}{c-q}f_kg^{k-(c-q)}.\]
 Let $k$ be an index such that $k>c-q$. Then
 \[\w(f_kg^{k-(c-q)})=\underbrace{\w(f_k)}_{\geq\frac{c-k}{c!}m}+(k-(c-q))\underbrace{\w(g)}_{>\frac{m}{c!}}>\frac{q}{c!}m.\]
 Hence, $\w(\wt f_{c-q})>\frac{q}{c!}m$. Consequently, property $(2)_\w$ holds for $f$ with respect to $\wt J_{-1}$.
 
 Finally, assume that $\neg(1)_\w$, $\neg(2)_\w$ and $(3)_\w$ hold for $f$ with respect to $J_{-1}$. We can write 
 \[\wt f_c=f_c+\sum_{k>c}\binom{k}{c}f_kg^{k-c}.\]
 Since $\w(f_c)=0$ and $\w(g)\geq\frac{m}{c!}>0$, it is clear that $\init_{\w}(\wt f_c)=\init_\w(f_c)$.
 
 If $\w(g)>\frac{m}{c!}$, we can use the same arguments as before to show that $\init_\w(\wt f_{c-q})=\init_\w(f_{c-q})$. Hence, property $(3)_\w$ holds for $f$ with respect to $\wt J_{-1}$ in this case.
 
 So assume now that $\w(g)=\frac{m}{c!}$. By Lemma \ref{clean_lemma} (iii) and property $(3)_\w$ we know that
 \[\init_\w(\wt f_{c-q})=\init_\w(f_{c-q})+\binom{c}{q}\init_\w(f_c)g^q.\]
 Now assume that there is an element $G\in K[[\x]]$ such that $\init_\w(\wt f_c)G^q=\init_\w(\wt f_{c-q})$. Then this implies that
 \[\init_\w(f_{c-q})=\init_\w(f_c)\Big(G-\binom{c}{q}g\Big)^q.\]
 But this contradicts the assumption that property $(3)_\w$ holds for $f$ with respect to $J_{-1}$. Hence, $(3)_\w$ holds for $f$ with respect to $\wt J_{-1}$.
\end{proof}

\begin{lemma} \label{w_cleaning_preserves_v_cleaning}
 Let $f\in J$ be an element that is $z$-regular of order $c$. Consider a weighted order function $\y:K[[\x]]\to\Ni^k$ that is defined on the parameters $\x$. 
 
 A $\y$-cleaning step $z=\wt z+g$ preserves $\w$-cleanness of $f$.
\end{lemma}
\begin{proof}
 By definition $g=-\binom{c}{q}^{-1}G$ for an element $G\in K[[\x]]$ that fulfills 
 \[\init_{\y}(f_c)\cdot G^q=\init_{\y}(f_{c-q}).\]
 Since $f_c$ is a unit, we know that 
 \[\w(g)=\frac{1}{q}\w(\init_{\y}(f_{c-q}))\geq\frac{1}{q}\w(f_{c-q}).\]
 Thus, the coordinate change preserves $\w$-cleanness of $f$ by Lemma \ref{coord_changes_which_preserve_w_clean}.
\end{proof}

\begin{remark}
 One drawback of our definition of $\w$-cleanness is that the definition requires $\w(J_{-1})>0$. There is a simple trick though to extend the definition to all weighted order functions in the following way:
 
 Let $\w:K[[\x]]\to\Ni^l$ be a weighted order function on $\x$ such that $\w(J_{-1})=0$ holds. Define the weighted order function $\w_+:K[[\x]]\to\Ni^{l+1}$ on $\x$ via $\w_+(x_i)=(\w(x_i),1)$ for $i=1,\ldots,n$. Then we know by Lemma \ref{double_weighted order function} and Lemma \ref{ogeqc!} that
 \[\w_+(J_{-1})=(\w(J_{-1}),\ord \minit_{\w}(J_{-1}))\geq (0,\ord J_{-1})\geq (0,c!).\]
 Hence, $\w_+(J_{-1})>0$ holds.
 
 We will say by abuse of notation that an element $f\in J$ is $\w$-clean with respect to $J_{-1}$ if $f$ is in fact $\w_+$-clean with respect to $J_{-1}$. Then $\w_+$-cleanness implies maximality of $\w(J_{-1})$ by Proposition \ref{w_cleaning_maximizes_w}.
\end{remark}

\section{Maximizing the order of the second coefficient ideal} \label{section_s_cleaning}

 In this section we will investigate the effect of coordinate changes $z\mapsto z+g(\x,y)$ on the order of the second coefficient ideal $J_{-2}$ as it was defined in Section \ref{section_invariance_of_slope}. Similarly to the previous section, we will define a cleanness property and devise a cleaning process to achieve maximality.
 
 Throughout this section we will use the following setting:
 
 Let $R=K[[\x,y,z]]$ with $\x=(x_1,\ldots,x_n)$ and $J\subseteq R$ an ideal of order $c=\ord J$. Let $J_{-1}$ be the coefficient ideal
 \[J_{-1}=\coeff^c_{(\x,y,z)}(J)\]
 with respect to the hypersurface $V(z)$. Let $J_{-1}$ have a factorization of the form
 \[J_{-1}=(\x^r y^{r_y})\cdot I_{-1}\]
 where $r=(r_1,\ldots,r_n)\in\N^n$ and $I_{-1}$ is in an ideal in $K[[\x,y]]$. Set $d=\ord I_{-1}$. Then we define
 \[J_{-2}=\coeff^d_{(\x,y)}(I_{-1})\]
 as the coefficient ideal with respect to $V(y,z)$ and set
 \[s=\ord J_{-2}.\]
 
 Further, we will consider coordinate changes $z=\wt z+g$ with $g\in K[[\x,y]]$, $\ord g\geq1$ which \emph{preserve the setting} in the sense that the coefficient ideal
 \[\wt J_{-1}=\coeff_{(\x,y,\wt z)}^c(J)\]
 with respect to the hypersurface $V(\wt z)$ also has a factorization
 \[\wt J_{-1}=(\x^ry^{r_y})\cdot \wt I_{-1}\]
 for an ideal $\wt I_{-1}$ that fulfills $\ord \wt I_{-1}=d$. Let
 \[\wt J_{-2}=\coeff^d_{(\x,y)}(\wt I_{-1})\]
 be the coefficient ideal with respect to $V(y,\wt z)$ and set
 \[\wt s=\ord \wt J_{-2}.\]
 
 Notice that we do not apply a coordinate change to the parameter $y$ and regard it as fixed instead. The situation in which coordinate changes are applied to both $z$ and $y$ will be investigated in Section \ref{section_maximizing_over_y_and_z}.
 
 We define the auxiliary numbers
 \[\D=\frac{1}{c!}\Big((d+r_y)\sd+|r|\Big),\]
 \[\wt \D=\frac{1}{c!}\Big((d+r_y)\wsd+|r|\Big).\]
 
 In Lemma \ref{s_under_coord_changes} we will give a basic estimate for the order of $J_{-2}$ dependent on the orders of the coefficients $g_j$ in the expansion $g=\sum_{j\geq0}g_j(\x)y^j$. If there is an element $f\in J$ which is $z$-regular of order $c$, we will show that $\ord\wt J_{-2}>\ord J_{-2}$ can only hold if for all indices $j\geq0$ the estimate $\ord g_j\geq\D-j\sd$ holds and there is an index $j$ such that equality holds. We will then introduce the notion of secondary $\ord$-cleanness of an element $f$. The definition is similar to $\w$-cleanness, but more technically involved. In Proposition \ref{s_cleaning_maximizes_slope} we will show that the existence of an element $f\in J$ which is secondary $\ord$-clean is sufficient for the order of the second coefficient ideal to be maximal over all coordinate changes $z\mapsto z+g(\x,y)$ which preserve the setting. We will then define the secondary $\ord$-cleaning process and show in Proposition \ref{s_cleaning_terminates} that this process either terminates in finitely many steps or there is a change of coordinates $z=\wt z+g(\x,y)$ such that $\ord\wt J_{-2}=\infty$.
 
 Denote for each element $f\in J$ its power series expansions with respect to the parameter systems $(\x,y,z)$ and $(\x,y,\wt z)$ by $f=\sum_{i,j\geq0}f_{i,j}y^jz^i$ and $f=\sum_{i,j\geq0}\wt f_{i,j}y^j\wt z^i$ with $f_{i,j},\wt f_{i,j}\in K[[\x]]$. Let $g$ have the expansion $g=\sum_{j\geq0}g_jy^j$ with $g_j\in K[[\x]]$. We know by Lemma \ref{coordinate_change_g} that
 \[\wt f_{i,j}=f_{i,j}+\sum_{\substack{k>i\\0\leq l\leq j}}\binom{k}{i}f_{k,l}\sum_{\substack{\alpha\in\N^{k-i}\\|\alpha|=j-l}}g_\alpha.\]
 Further, we know by Lemma \ref{ord_J_2} (1) that for each element $f\in J$ and all indices $i,j\geq0$ the inequality
 \[\ord f_{i,j}\geq (c-i)\D-j\sd\]
 holds. There exists an element $f\in J$ and indices $i<c$, $j<\frac{c-i}{c!}(d+r_y)$ such that equality holds.
 
 The following lemma gives a basic estimate for the orders of the coefficients $g_j$.

\begin{lemma} \label{stabilizing_coordinate_changes}
 If there is an element $f\in J$ that is $z$-regular of order $c$, the following estimates hold:
 \begin{enumerate}[(1)]
  \item $\ord g_j\geq\frac{|r|}{c!}$ for all indices $j\geq0$.
  \item $\ord g_j\geq \D-j\sd$ for all indices $j\geq \frac{d+r_y}{c!}$. For indices $j>\frac{d+r_y}{c!}$ then the strict inequality holds.
 \end{enumerate}
\end{lemma}
\begin{proof}
 (1): Let $x_k$ be one of the parameters in $\x=(x_1,\ldots,x_n)$. Since the coordinate change $z=\wt z+g$ preserves the setting, we know that $\ord_{(x_k)}\wt J_{-1}\geq r_k$. By Lemma \ref{m_under_coord_changes} this implies that 
 \[\ord_{(x_k)}g\geq \frac{r_k}{c!}\]
 holds for all indices $k=1,\ldots,n$. Consequently, for all indices $j\geq0$ the following estimate holds:
 \[\ord g_j\geq \sum_{k=1}^n\ord_{(x_k)}g_j\geq \sum_{k=1}^n\ord_{(x_k)}g\geq \frac{|r|}{c!}.\]
 
 (2): By (1) and the definition of $\D$, we can compute that
 \[\ord g_j\geq \frac{|r|}{c!}=\D-j\sd+\underbrace{(\frac{d+r_y}{c!}-j)}_{\geq 0}\sd.\]
 If $j>\frac{d+r_y}{c!}$, it is clear that strict inequality holds.
\end{proof}

\begin{lemma} \label{s_under_coord_changes}
 \begin{enumerate}[(1)]
  \item If for all indices $j\geq0$ the inequality
  \[\ord g_j\geq \D-j\sd\]
  holds, then $\ord \wt J_{-2}\geq \ord J_{-2}$.
  \item If for all indices $j\geq0$ the strict inequality 
  \[\ord g_j>\D-j\sd\]
  holds, then $\ord \wt J_{-2}=\ord J_{-2}$.
  \item If there is an element $f\in J$ that is $z$-regular of order $c$ and there is an index $j\geq0$ such that the strict inequality
  \[\ord g_{j}<\D-j\sd\]
  holds, then
  \[\ord \wt J_{-2}=d!\cdot\min_{j<\frac{d+r_y}{c!}}\frac{c!\cdot \ord g_j-|r|}{d+r_y-jc!}<\ord J_{-2}.\]
 \end{enumerate}
\end{lemma}
\begin{proof}
 (1): Let $f\in J$ be arbitrary and $i<c$, $j<\frac{d+r_y}{c!}$ indices. Let $k>i$, $l\leq j$ and $\alpha\in\N^{k-i}$ with $|\alpha|=j-l$ be further indices. Then
 \[\ord f_{k,l}g_\alpha\geq (c-k)\D-l\sd+\sum_{o=1}^{k-i}\underbrace{\ord g_{\alpha_o}}_{\geq \D-\alpha_o\sd}\]
 \[\geq (c-k)\D-l\sd+(k-i)\D-(j-l)\sd\]
 \[=(c-i)\D-j\sd.\]
 Hence, $\ord \wt f_{i,j}\geq (c-i)\D-j\sd$. By Lemma \ref{s_changes} (1) this proves that $\wt s\geq s$.
 
 (2): Let $f\in J$ and indices $i<c$, $j<\frac{d+r_y}{c!}$ be such that $\ord f_{i,j}=(c-i)\D-j\sd$. Now let $k>i$, $l\leq j$ and $\alpha\in\N^{k-i}$ with $|\alpha|=j-l$ be further indices. Then
  \[\ord f_{k,l}g_\alpha\geq (c-k)\D-l\sd+\sum_{o=1}^{k-i}\underbrace{\ord g_{\alpha_o}}_{>\D-\alpha_o\sd}\]
 \[>(c-k)\D-l\sd+(k-i)\D-(j-l)\sd\]
 \[=(c-i)\D-j\sd.\]
 Hence, $\ord \wt f_{i,j}=\ord f_{i,j}=(c-i)\D-j\sd$. By Lemma \ref{s_changes} (2) this implies that $\wt s\leq s$. In combination with (1) this proves that $\wt s=s$.
 
 (3): Let $f\in J$ be $z$-regular of order $c$. Define for all integers $j\geq0$ the number 
 \[L(j)=\ord g_j+j\sd.\]
 Since $\ord g_0>0$ by assumption and $s\geq d!$ by Lemma \ref{ogeqc!}, it is clear that $L(j)>0$ for all $j\geq0$. Let now $\jz\geq0$ be a fixed index such that the number $L(\jz)$ is minimal and that $\jz$ is minimal with this property. Notice that the assumption of the lemma asserts that $L(\jz)<\D$. Further, we know that $\jz<\frac{d+r_y}{c!}$ by Lemma \ref{stabilizing_coordinate_changes}.

 We are now going to prove that
 \[\ord \wt f_{0,c\jz}=c\B(L(\jz)-\jz\sd\B).\]
 To this end, let $k\geq0$, $l\leq c\cdot\jz$ and $\alpha\in\N^k$ with $|\alpha|=c\cdot\jz-l$ be indices. Notice that 
 \[\ord g_\alpha=\sum_{o=1}^{k}\underbrace{L(\alpha_o)}_{\geq L(\jz)}-\alpha_o\sd\]
 \[\geq k L(\jz)-(c\cdot\jz-l)\sd.\]
 If $k<c$, then we can compute that
 \[\ord f_{k,l}g_\alpha\geq (c-k)\underbrace{\D}_{>L(b)}-l\sd+kL(\jz)-(c\cdot\jz-l)\sd\]
 \[>c\B(L(\jz)-\jz\sd\B).\]
 On the other hand, if $k\geq c$ and $(k,l)\neq(c,0)$, then it is clear that
 \[\ord f_{k,l}g_\alpha\geq\ord g_\alpha\geq kL(\jz)-(c\cdot\jz-l)\sd\]
 \[=c\B(L(\jz)-\jz\sd\B)+\underbrace{(c-k)L(\jz)+l\sd}_{>0}.\]
 Now consider the indices $k=c$ and $l=0$. By assumption, $\ord f_{c,0}=0$. Let $\alpha\in\N^c$ be a multi-index such that $|\alpha|=c\cdot\jz$. If $\alpha\neq(\jz,\ldots,\jz)$, we know by the minimality of $\jz$ that
 \[\ord f_{c,0}g_\alpha=\ord g_\alpha>c\B(L(\jz)-\jz\sd\B).\]
 On the other hand, it is clear that
 \[\ord f_{c,0}g_\jz^c=c\B(L(\jz)-\jz\sd\B).\]
 This concludes the proof that $\ord \wt f_{0,c\jz}=c(L(\jz)-\jz\sd)$. Since $L(\jz)<\D$ and $c\cdot\jz<\frac{c}{c!}(d+r_y)$, this proves $\wt s<s$ by Lemma \ref{s_changes} (3).
 
 It follows from statement (2) that $\ord g_\jz=\wt \D-j\wsd$. It is then a straightforward computation to determine the exact value of $\wt s$ from Lemma \ref{ord_J_2}.
\end{proof}


\subsection{The secondary $\ord$-cleanness property}


\begin{definition} \label{def_s-clean}
 Let $f\in J$ be an element that is $z$-regular of order $c$.
 
 Then $f$ is said to be \emph{secondary $\ord$-clean with respect to $J_{-2}$} if for each index $\jz<\frac{d+r_y}{c!}$ one of the following properties holds:
 \begin{enumerate}[$(i)_\jz$]
  \item There exist indices $i,j$ with $c-q<i<c$ and $j\leq (c-i)\jz$ such that
  \[\ord f_{i,j}=(c-i)\D-j\sd.\]
  \item $\ord f_{c-q,\jz q}>q\D-\jz q\sd.$
  \item There is no element $G\in K[[\x]]$ such that
  \[\init(f_{c-q,\jz q})=\init(f_{c,0})\cdot G^q.\]
 \end{enumerate}
 The definition depends on the entire regular system of parameters $(\x,y,z)$. This will be suppressed in the notation since the parameters are considered to be part of the setting.
\end{definition}

\begin{remark}
 \begin{enumerate}[(1)]
  \item If $q=1$, then $f$ being secondary $\ord$-clean is equivalent to
  \[\ord f_{c-1,b}>\D-\jz\sd\]
  for all indices $\jz<\frac{d+r_y}{c!}$. In particular, this is fulfilled if $f_{c-1}=0$.
  \item Consider a purely inseparable power series 
  \[f=z^{p^e}+F(\x,y)\]
  where $p$ is the characteristic of $K$ and $F$ has the expansion $F=\sum_{j\geq0}F_jy^j$ with $F_j\in K[[\x]]$. Then $f$ being secondary $\ord$-clean is equivalent to the fact that for all indices $\jz<\frac{d+r_y}{c!}$ either
  \[\ord F_{\jz p^e}>p^e\D-\jz p^e\sd\]
  holds or the initial form $\init(F_{\jz p^e})$ is not a $p^e$-th power. In particular, this is fulfilled if no $p^e$-th powers appear in the expansion of $F(\x,y)$.
 \end{enumerate}
\end{remark}

We will show in Proposition \ref{s_cleaning_maximizes_slope} that the existence of an element $f\in J$ which is secondary $\ord$-clean with respect to $J_{-2}$ ensures that $\ord J_{-2}$ is maximal over all coordinate changes $z\mapsto z+g(\x,y)$ which preserve the setting. As a preparation, we first prove two technical lemmas. Lemma \ref{too_big_j} will enable us to assume without loss of generality that the coefficients $g_j$ are zero for $j\geq\frac{d+r_y}{c!}$. In Lemma \ref{s_clean_lemma} we will show which consequences the properties $(i)_\jz-(iii)_\jz$ have on coefficients of the power series expansion of $f$ with respect to the parameters $(\x,y,\wt z)$.

\begin{lemma} \label{too_big_j}
 Assume that $\ord \wt J_{-2}\geq\ord J_{-2}$ holds. Let $f\in J$ be an element which is $z$-regular of order $c$. Consider the coordinate change $z=\wh z+\wh g$ where $\wh g=\sum_{j<\frac{d+r_y}{c!}}g_jy^j$. This coordinate change preserves the setting. Set $\wh J_{-1}=\coeff_{(\x,y,\wh z)}^c(J)$ and consider the factorization $\wh J_{-1}=(\x^ry^{r_y})\cdot \wh I_{-1}$ where $\ord\wh I_{-1}=d$. Set $\wh J_{-2}=\coeff_{(\x,y)}^d(\wh I_{-1})$ and $\wh s=\ord\wh J_{-2}$.
 
 Then the identity $\ord\wh J_{-2}=\ord \wt J_{-2}$ holds.
\end{lemma}
\begin{proof}
 We first prove that the coordinate change preserves the setting. Since $\ord \wh g\geq\ord g$, $\ord_{(y)} \wh g\geq\ord_{(y)} g$ and $\ord_{(x_i)}\wh g\geq\ord_{(x_i)}g$ for all indices $i=1,\ldots,n$, it follows from Lemma \ref{m_under_coord_changes} that $\ord \wh J_{-1}\geq\ord J_{-1}$ and $\wh J_{-1}$ has the factorization $\wh J_{-1}=(\x^ry^{r_y})\cdot \wh I_{-1}$. It remains to prove that $\ord \wh J_{-1}=\ord J_{-1}$. Assume that $\ord \wh J_{-1}>\ord J_{-1}$ holds. By Lemma \ref{m_under_coord_changes} this implies that there exists an index $j<\frac{d+r_y}{c!}$ such that the equality
 \[\ord g_j+j=\frac{\ord J_{-1}}{c!}=\frac{d+r_y+|r|}{c!}\]
 holds. This implies that
 \[\ord g_j=\frac{d+r_y+|r|}{c!}-j\]
 \[=\delta-j\sd-\B(\frac{d+r_y}{c!}-j\B)\B(\sd-1\B)<\delta-j\sd\]
 by Lemma \ref{ogeqc!}. This contradicts the assumption $\ord \wt J_{-2}\geq\ord J_{-2}$ by Lemma \ref{s_under_coord_changes} (3).

  By Proposition \ref{w_cleaning_maximizes_w} we know that $\ord \wh J_{-1}=\ord J_{-1}$, hence $\ord \wh I_{-1}=d$ and the coordinate change preserves the setting.
 
 It remains to show that $\wh s=\wt s$. To this end, consider the change of coordinates $\wh z=\wt z+\sum_{j\geq\frac{d+r_y}{c!}}g_jy^j$. It follows from Lemma \ref{stabilizing_coordinate_changes} (2) in combination with Lemma \ref{s_under_coord_changes} (1) that $\wt s\geq\wh s$. Since the argument is independent of the value of $\wt s$, it follows by a symmetric argument that $\wh s\geq \wt s$. Hence, $\wh s=\wt s$.
\end{proof}

\begin{remark}
 In the following, we will also often make use of the negations of the properties $(i)_\jz-(iii)_\jz$. They can be formulated in the following way: 
 \begin{itemize}
  \item [$\neg(i)_\jz$] For all indices $i,j$ with $c-q<i<c$ and $j\leq (c-i)\jz$ the following strict inequality holds:
  \[\ord f_{i,j}>(c-i)\D-j\sd.\]
  \item [$\neg(ii)_\jz$] $\ord f_{c-q,\jz q}=q\D-\jz q\sd.$
  \item [$\neg(iii)_\jz$] There is an element $G\in K[[\x]]$ such that $\init(f_{c-q,\jz q})=\init(f_{c,0})\cdot G^q$.
 \end{itemize}
\end{remark}


\begin{lemma} \label{s_clean_lemma}
 Let $f\in J$ be an element that is $z$-regular of order $c$. Let $g$ have the expansion $g=\sum_{j=0}^{\jz}g_jy^j$ where $g_\jz\neq0$. Assume that for all indices $j\leq\jz$ either $g_j=0$ or $\ord g_j=\D-j\sd$ holds. Let $\jzz$ be an index with $\jzz\leq\jz$. Then the following hold:
 \begin{enumerate}[(1)]
  \item If $(i)_\jzz$ holds and the index $c-q<i<c$ is maximal with the property that there exists an index $j\leq(c-i)\jzz$ such that
  \[\ord f_{i,j}=(c-i)\D-j\sd\]
  holds, then $\ord\wt f_{i,j}=\ord f_{i,j}$.
  \item If $\neg(i)_\jz$ and $(ii)_{\jz}$ hold, then 
  \[\ord\wt f_{c-q,\jz q}=q\D-\jz q\sd.\]
  \item If $\neg(i)_\jz$ and $\neg(ii)_\jz$ hold, then either
  \[\init(\wt f_{c-q,\jz q})=\init(f_{c-q,\jz q})+\binom{c}{q}\init(f_{c,0})\cdot \init(g_\jz)^{q}\]
  or the right-hand term vanishes and 
  \[\ord \wt f_{c-q,\jz q}>q\D-\jz q\sd.\]
  \item If $\jzz<\jz$ and $\neg(i)_\jzz$ holds, then 
  \[\wt f_{c-q,\jzz q}=f_{c-q,\jzz q}+G\]
  for an element $G\in K[[\x]]$ with
  \[\ord G>q\D-\jzz q\sd.\]
 \end{enumerate}
\end{lemma}
\begin{proof}
 (1): We want to determine $\ord \wt f_{i,j}$. Thus, let $k,l,\alpha$ be indices such that $k>i$, $l\leq j$ and $\alpha\in\N^{k-i}$ with $|\alpha|=j-l$. It is clear that either $g_\alpha=0$ or
 \[\ord g_\alpha=(k-i)\D-(j-l)\sd.\]
 Assume that $g_\alpha\neq0$. Since $\jz$ is minimal with $g_\jz\neq0$ and $\jzz\leq\jz$, we know that
 \[j-l=|\alpha|\geq (k-i)\jz\geq (k-i)\jzz.\]
 Since $j\leq (c-i)\jzz$, this implies that $l\leq (c-k)\jzz$. Thus, $k\leq c$. By maximality of $i$, we know that either $k<c$ and
 \[\ord f_{k,l}g_\alpha>(c-k)\D-l\sd+(k-i)\D-(j-l)\sd\]
 \[=(c-i)\D-j\sd=\ord f_{i,j}\]
 or $k=c$. But since $\binom{c}{i}=0$ by Lemma \ref{c-q} (2), this proves that $\ord \wt f_{i,j}=\ord f_{i,j}$.
 
 (2): By property $(ii)_\jz$ we know that
 \[\ord f_{c-q,\jz q}>q\D-\jz q\sd.\]
 Let $k,l,\alpha$ be indices such that $k>c-q$, $l\leq \jz q$ and $\alpha\in\N^{k-(c-q)}$ with $|\alpha|=\jz q-l$. Again, it is clear that either $g_\alpha=0$ or
 \[\ord g_\alpha=(k-(c-q))\D-(\jz q-l)\sd.\]
 Assume that $g_\alpha\neq0$. Since $\jz$ is minimal with $g_\jz\neq0$, this implies that $l\leq (c-k)\jz$ in the same way we showed before. Thus, $k\leq c$. Since $\neg(i)_\jz$ holds, we know that either
 \[\ord f_{k,l}g_\alpha>(c-k)\D-l\sd+(k-(c-q))\D-(\jz q-l)\sd\]
 \[=q\D-\jz q\sd\]
 holds or $k=c$. Notice that $\binom{c}{c-q}\neq0$ by Lemma \ref{c-q} (1). By Lemma \ref{char_p_sum_lemma} we know that
 \[f_{c,0}\sum_{\substack{\alpha\in\N^q\\|\alpha|=\jz q}}g_\alpha=f_{c,0}g_\jz^q\]
 and by assumption
 \[\ord f_{c,0}g_\jz^q=q\B(\D-\jz\sd\B)=q\D-\jz q\sd.\]
 This proves that
 \[\init(\wt f_{c-q,\jz q})=\binom{c}{q}\init(f_{c,0})\cdot \init(g_\jz)^q\]
 and consequently,
 \[\ord \wt f_{c-q,\jz q}=q\D-\jz q\sd.\]
 
 (3): This can be proved in the same way as statement (2).
 
 (4): Let $k,l,\alpha$ be indices such that $k>c-q$, $l\leq \jzz q$ and $\alpha\in\N^{k-(c-q)}$ with $|\alpha|=\jzz q-l$. Again, it is clear that either $g_\alpha=0$ or
 \[\ord g_\alpha=(k-(c-q))\D-(\jzz q-l)\sd.\]
 Assume that $g_\alpha\neq0$. Since $\jz$ is minimal with $g_\jz\neq0$ and $\jzz<\jz$, this implies that
 \[\jzz q-l=|\alpha|\geq (k-(c-q))\jz> (k-(c-q))\jzz.\]
 Consequently, $l<(c-k)\jzz$. Thus, $k<c$. But by property $\neg(i)_\jzz$ we know that 
 \[\ord f_{k,l}g_\alpha>(c-k)\D-l\sd+(k-(c-q))\D-(\jzz q-l)\sd\]
 \[=q\D-\jzz q\sd.\]
 This proves the assertion.
\end{proof}


\begin{proposition} \label{s_cleaning_maximizes_slope}
 Let $f\in J$ be an element that is secondary $\ord$-clean with respect to $J_{-2}$. Then $\ord \wt J_{-2}\leq \ord J_{-2}$.
\end{proposition}
\begin{proof}
 By Lemma \ref{s_under_coord_changes} and Lemma \ref{too_big_j} we can assume without loss of generality that $g=\sum_{j<\frac{d+r_y}{c!}}g_jy^j$ and for all indices $j<\frac{d+r_y}{c!}$ either $g_j=0$ or $\ord g_j=D-j\sd$ holds. Let $\jz$ be minimal with the property that $g_\jz\neq0$. Since $\jz<\frac{d+r_y}{c!}$, we know by assumption that one of the properties $(i)_\jz-(iii)_\jz$ holds. 
 
 Assume first that the property $(i)_\jz$ holds. Let $i<c$ be maximal with the property that there exists an index $j\leq (c-i)\jz$ such that
 \[\ord f_{i,j}=(c-i)\D-j\sd.\]
 By Lemma \ref{s_clean_lemma} (1) we know that $\ord \wt f_{i,j}=\ord f_{i,j}$. Since $j\leq (c-i)\jz<\frac{c-i}{c!}(d+r_y)$, this proves that $\wt s\leq s$ by Lemma \ref{s_changes} (2).
 
 Now assume that the properties $\neg(i)_\jz$ and $(ii)_\jz$ hold. By Lemma \ref{s_clean_lemma} (2) this implies that
 \[\ord\wt f_{c-q,\jz q}=q\D-\jz q\sd.\]
 Again, since $\jz q<\frac{q}{c!}(d+r_y)$, this proves $\wt s\leq s$ by Lemma \ref{s_changes} (2).
 
 Finally, assume that the properties $\neg(i)_\jz$, $\neg(ii)_\jz$ and $(iii)_\jz$ hold. Since the term
 \[\init(f_{c-q,\jz q})+\binom{c}{q}\init(f_{c,0})\cdot \init(g_\jz)^{q}\]
 cannot vanish by assumption, we know by Lemma \ref{s_clean_lemma} (3) that
 \[\ord \wt f_{c-q,\jz q}=\ord f_{c-q,\jz q}=qD-\jz q \sd.\]
 Again, this proves $\wt s\leq s$ by Lemma \ref{s_changes} (2).
\end{proof}

\subsection{The secondary $\ord$-cleaning process}

 Similarly to the last section, we will now develop a process to successively construct coordinate changes $z\mapsto z+g_\jz(\x)y^\jz$ that preserve the setting and increase the order of the second coefficient ideal until it reaches its maximal value and cleanness is achieved. Since secondary $\ord$-cleanness means that one of the properties $(i)_\jz-(iii)_\jz$ has to hold for all indices $\jz<\frac{d+r_y}{c!}$, this process is inherently more complicated than the $\w$-cleaning process. Not only do we have to find coordinate changes that ensure that one of the properties $(i)_\jz-(iii)_\jz$ holds for a fixed index $\jz$, we also have to make sure that we do not invalidate the properties we already achieved for other indices $\jzz<\frac{d+r_y}{c!}$. Our strategy for doing so will be to start our cleaning process with the minimal index $\jz$ for which none of the properties $(i)_\jz-(iii)_\jz$ hold and then successively raise $\jz$ until either cleanness is achieved or the order of the second coefficient ideal increases.

\begin{definition} \label{s_cleaning_step}
 Let $f\in J$ be an element that is $z$-regular of order $c$ and not secondary $\ord$-clean with respect to $J_{-2}$.
 
 Let $\jz<\frac{d+r_y}{c!}$ be the index that is minimal with the property that $\neg(i)_\jz$, $\neg(ii)_\jz$ and $\neg(iii)_\jz$ hold. Thus, there is an element $G\in K[[\x]]$ such that
 \[\init(f_{c-q,\jz q})=\init(f_{c,0})\cdot G^q.\]
 
 A \emph{secondary $\ord$-cleaning step} with respect to $f$ and $J_{-2}$ is defined as the coordinate change $z=\wt z+g_\jz y^\jz$ where
 \[g_\jz=-\binom{c}{q}^{-1}G\in K[[\x]].\]
 Notice that $g_\jz$ is by construction homogeneous and $\ord g_\jz=\D-\jz\sd$.
\end{definition}

 Naturally, a cleaning process for the second coefficient ideal only makes sense if each iteration of the process preserves the setting. In Lemma \ref{s_cleaning_step_preserves_setting} we will prove that a secondary $\ord$-cleaning step preserves the setting under the condition that there exists an element $f\in J$ which is $\ord$-clean with respect to the coefficient ideal $J_{-1}$. Hence, the $\ord$-cleanness of $f$ with respect to $J_{-1}$ will be a prerequisite for the secondary $\ord$-cleaning process. By Lemma \ref{s_cleaning_preserves_v_clean}, the $\ord$-cleanness of $f$ with respect to $J_{-1}$ is also preserved by a secondary $\ord$-cleaning step.
 
 In Proposition \ref{s_cleaning_improves_smthg} we will then show that secondary $\ord$-cleaning steps either increase the order of the second coefficient ideal or successively establish the properties $(i)_\jz-(iii)_\jz$ for the indices $\jz<\frac{d+r_y}{c!}$. This will enable us to prove in Proposition \ref{s_cleaning_terminates} that the cleaning process either terminates after finitely many steps or there is a coordinate change $z\mapsto \wt z+g(\x,y)$ which preserves the setting and for which $\ord \wt J_{-2}=\infty$ holds.

\begin{lemma} \label{s_cleaning_preserves_v_clean}
 Let $\w$ be a weighted order function on $K[[\x,y]]$ that is defined on the parameters $(\x,y)$. Then a secondary $\ord$-cleaning step $z=\wt z+g_\jz y^\jz$ preserves $\w$-cleanness. In particular, $\w(\wt J_{-1})\geq\w(J_{-1})$.
\end{lemma}
\begin{proof}
 By definition, there is an element $G\in K[[\x]]$ such that 
 \[\init(f_{c-q,\jz q})=\init(f_{c,0})\cdot G^q\]
 and $\w(g_\jz)=\w(G)$. Since $\w(f_{c,0})=0$, we know that 
 \[\w(g_{\jz})=\frac{1}{q}\w(\init(f_{c-q,q\jz}))\geq\frac{1}{q}\w(f_{c-q,q\jz}).\]
 Thus, we can compute that
 \[\w(g_\jz y^\jz)=\w(g_\jz)+\jz\cdot\w(y)\geq\frac{1}{q}\w(f_{c-q,q\jz}y^{q\jz})\geq \frac{1}{q}\w(f_{c-q}).\]
 This proves the assertion by Lemma \ref{coord_changes_which_preserve_w_clean}.
\end{proof}

\begin{lemma} \label{s_cleaning_step_preserves_setting}
 If there exists an element $f\in J$ that is $\ord$-clean with respect to $J_{-1}$, then a secondary $\ord$-cleaning step $z=\wt z+g_\jz y^\jz$ preserves the setting.
\end{lemma}
\begin{proof}
 By Lemma \ref{s_cleaning_preserves_v_clean} we know that $\ord \wt J_{-1}\geq\ord J_{-1}$, $\ord_{(y)} \wt J_{-1}\geq \ord_{(y)}J_{-1}$ and $\ord_{(x_k)} \wt J_{-1}\geq \ord_{(x_k)}J_{-1}$ for all indices $k=1,\ldots,n$. By Proposition \ref{w_cleaning_maximizes_w} we know that $\ord \wt J_{-1}=\ord J_{-1}$. This implies that $\wt J_{-1}$ has a factorization
 \[\wt J_{-1}=(\x^ry^{r_y})\cdot\wt I_{-1}\]
 where the ideal $\wt I_{-1}$ fulfills $\ord\wt I_{-1}=d$.
\end{proof}


\begin{lemma} \label{s_cleaning_improves_smthg}
 Let $f\in J$ be an element that is $\ord$-clean with respect to $J_{-1}$. Consider a secondary $\ord$-cleaning step $z=\wt z+g_\jz y^\jz$. Then one of the following holds:
 \begin{itemize}
  \item $\ord \wt J_{-2}=\ord J_{-2}$ and for all indices $\jzz\leq\jz$ one of the properties $(i)_\jzz-(iii)_\jzz$ holds for $f$ with respect to $\wt J_{-2}$.
  \item $\ord \wt J_{-2}>\ord J_{-2}$.
 \end{itemize}
\end{lemma}
\begin{proof}
 Since $\ord g_\jz=\D-\jz\sd$, we know by Lemma \ref{s_under_coord_changes} that $\wt s\geq s$. So assume that $\wt s=s$ holds.
 
 Since the term
 \[\init(f_{c-q,\jz q})+\binom{c}{q}\init(f_{c,0})\cdot \init(g_\jz)^{q}\]
 vanishes by construction of $g_\jz$, we know by Lemma \ref{s_clean_lemma} (3) that 
 \[\ord \wt f_{c-q,\jz q}>q\D-\jz q\sd.\]
 Thus, the property $(ii)_\jz$ holds for $f$ with respect to $\wt J_{-2}$.
 
 Now consider an index $\jzz<\jz$. By the choice of the index $\jz$, we know that one of the properties $(i)_\jzz-(iii)_\jzz$ holds for $f$ with respect to $J_{-2}$.
 
 First assume that $(i)_\jzz$ holds. Then by Lemma \ref{s_clean_lemma} (1) we know that the property $(i)_\jzz$ also holds for $f$ with respect to $\wt J_{-2}$.
 
 Now assume that $\neg(i)_\jzz$ and one of the properties $(ii)_\jzz$ or $(iii)_\jzz$ holds. Then by Lemma \ref{s_clean_lemma} (4) we know that
 \[\wt f_{c-q,\jzz q}=f_{c-q,\jzz q}+G\]
 for an element $G\in K[[\x]]$ with
 \[\ord G>qD-\jzz q\sd.\]
 If $(ii)_\jzz$ holds, this proves that
 \[\ord \wt f_{c-q,\jzz q}>qD-\jzz q\sd\]
 and hence, the property $(ii)_\jzz$ also holds for $f$ with respect to $\wt J_{-2}$. On the other hand, if $\neg(ii)_\jzz$ and $(iii)_\jzz$ hold, then we know that
 \[\init(\wt f_{c-q,\jzz q})=\init(f_{c-q,\jzz q}).\]
 Further, we know by Lemma \ref{coordinate_change_g} that
 \[\wt f_{c,0}=f_{c,0}+\sum_{k>c}\binom{k}{c}f_{k,0}g_0^{k-c}.\]
 Hence, $\init(\wt f_{c,0})=\init(f_{c,0})$. Thus, the property $(iii)_\jzz$ holds for $f$ with respect to $\wt J_{-2}$.
 
 In total, we have shown that one of the properties $(i)_\jzz-(iii)_\jzz$ holds for $f$ with respect to $\wt J_{-2}$ for all indices $\jzz\leq\jz$.
\end{proof}

 We will now describe the \emph{secondary $\ord$-cleaning process}:
 
 Let $R=K[[\x,y,z]]$ and $J\subseteq R$ an ideal of order $c=\ord J$. Let the coefficient ideal $J_{-1}=\coeff_{(\x,y,z)}^c(J)$ have a factorization $J_{-1}=(\x^ry^{r_y})\cdot I_{-1}$ for an ideal $I_{-1}$ of order $d=\ord I_{-1}$. Let $f\in J$ be an element that is $\ord$-clean with respect to $J_{-1}$.
 
 Set $z_0=z$. We will now describe a process to successively construct certain parameters $z_i$ for $i\geq1$.
 
 In each iteration of the process set 
 \[J_{-1}^{(i)}=\coeff_{(\x,y,z_i)}^c(J).\]
 By Lemma \ref{s_cleaning_preserves_v_clean}, Lemma \ref{s_cleaning_step_preserves_setting} and induction on $i$, $f$ is $\ord$-clean with respect to $J_{-1}^{(i)}$ and the ideal $J_{-1}^{(i)}$ has a factorization $J_{-1}^{(i)}=(\x^ry^{r_y})\cdot I_{-1}^{(i)}$ with $\ord I_{-1}^{(i)}=d$. Set
 \[J_{-2}^{(i)}=\coeff_{(\x,y)}^c(I_{-1}^{(i)}).\]
 If $f$ is secondary $\ord$-clean with respect to $J_{-2}^{(i)}$, the process terminates.

 Otherwise, let $z_i=z_{i+1}+g_i$ be a secondary $\ord$-cleaning step with respect to $f$ and $J_{-2}^{(i)}$.

\begin{proposition} \label{s_cleaning_terminates}
 Consider the secondary $\ord$-cleaning process as described above. One of the following holds:
 \begin{itemize}
  \item The process terminates in finitely many steps.
  \item The process does not terminate. In this case, $z_\infty=z-\sum_{i\geq0} g_i$ is a well-defined power series and 
  \[J_{-2}^{(\infty)}=0\]
  where $J_{-2}^{(\infty)}$ is defined analogously to $J_{-2}^{(i)}$.
 \end{itemize}
\end{proposition}
\begin{proof}
 Assume that the process does not terminate. Then for all $i\geq0$ the element $g_i$ has the form $g_i=g_{i,j}y^j$ where $g_{i,j}\in K[[\x]]$ and $j<\frac{d+r_y}{c!}$. We know that
 \[\ord g_{i,j}=\frac{1}{c!}\B(\frac{d+r_y}{d!}s_i+|r|\B)-j\frac{s_i}{d!}\geq\B(\frac{d+r_y}{c!}-j\B)\frac{s_i}{d!}\]
 where $s_i=\ord J_{-2}^{(i)}$. By Lemma \ref{s_cleaning_improves_smthg} we further know that $\lim_{i\to\infty} s_i=\infty$. Hence, $\lim_{i\to\infty}\ord g_i=\infty$ and $z_\infty$ is a well-defined power-series.
 
 It follows from Lemma \ref{s_under_coord_changes} that
 \[\ord J_{-2}^{(\infty)}\geq s_i\]
 for all $i\geq0$. Hence, $\ord J_{-2}^{(\infty)}=\infty$ and $J_{-2}^{(\infty)}=0$.
\end{proof}

\section[Maximizing under simultaneous coordinate changes in $z$ and $y$]{Maximizing the order of the second coefficient ideal under simultaneous coordinate changes in $z$ and $y$} \label{section_maximizing_over_y_and_z}

In this section we will investigate the effect of simultaneous coordinate changes $z\mapsto z+g(\x,y)$, $y\mapsto y+h(\x)$ on the order of the second coefficient ideal $J_{-2}$. As it turns out, the situation is drastically more complicated than when only considering coordinate changes in $z$ as we did in the last section. Due to the high complexity, we will not devise a cleaning process, but we will show that there exists a coordinate change which maximizes the order of the second coefficient ideal over all such coordinate changes.

We consider the following setting:

Let $R=K[[\x,y,z]]$ with $\x=(x_1,\ldots,x_n)$ for a field $K$ and let $J\subseteq R$ be an ideal of order $c=\ord J$. Let $J_{-1}$ be the coefficient ideal
 \[J_{-1}=\coeff^c_{(\x,y,z)}(J)\]
 with respect to the hypersurface $V(z)$. Let $J_{-1}$ have a factorization of the form
 \[J_{-1}=(\x^r)\cdot I_{-1}\]
 where $r=(r_1,\ldots,r_n)\in\N^n$ and $I_{-1}$ is in an ideal in $K[[\x,y]]$. Notice that we do not factor a power of $y$ from the coefficient ideal since we do not fix the hypersurface $V(y)$. Set $d=\ord I_{-1}$. We define
 \[J_{-2}=\coeff^d_{(\x,y)}(I_{-1})\]
 as the coefficient ideal of $I_{-1}$ with respect to $V(y,z)$ and set
 \[s=\ord J_{-2}.\]
 
 We consider coordinate changes $z=\wt z+g$, $y=\wt y+h$ where $g\in K[[\x,y]]$ and $h\in K[[\x]]$ are elements with $\ord g,\ord h\geq1$. The coordinate changes we consider are all assumed to preserve the setting. Hence, the coefficient ideal
 \[\wt J_{-1}=\coeff_{(\x,y,\wt z)}^c(J)=\coeff_{(\x,\wt y,\wt z)}^c(J)\]
 with respect to the hypersurface $V(\wt z)$ has a factorization
 \[\wt J_{-1}=(\x^r)\cdot \wt I_{-1}\]
 for an ideal $\wt I_{-1}$ of order $\ord \wt I_{-1}=d$. Let
 \[\wt J_{-2}=\coeff^d_{(\x,\wt y)}(\wt I_{-1})\]
 be the coefficient ideal with respect to $V(\wt y,\wt z)$ and set
 \[\wt s=\ord \wt J_{-2}.\]
 
 We define the auxiliary numbers
\[\D=\frac{1}{c!}\Big(d\sd+|r|\Big),\]
\[\wt\D=\frac{1}{c!}\Big(d\wsd+|r|\Big).\]

Let each element $f\in J$ have the power series expansions $f=\sum_{i,j\geq0}f_{i,j}y^jz^i$ and $f=\sum_{i,j\geq0}\wt f_{i,j}\wt y^j\wt z^i$ with $f_{i,j},\wt f_{i,j}\in K[[\x]]$. Further, let $g$ have the expansion $g=\sum_{j\geq0}g_jy^j$ with $g_j\in K[[\x]]$. The form of the coefficients $\wt f_{i,j}$ will be established in Lemma \ref{double_coord_change}.

The goal of this section is to prove in Proposition \ref{maximum_over_y_and_z_exists} that there exists a specific coordinate change which realizes the maximal value of the order of the second coefficient ideal over all coordinate changes $z\mapsto z+g(\x,y)$, $y\mapsto y+h(\x)$ that preserve the setting. As a preparation for this result, we will prove in Lemma \ref{s_under_double_coord_changes} that the order of the second coefficient ideal can only increase under such a coordinate change if the estimates $\ord h\geq\sd$ and $\ord g_j\geq\D-j\sd$ hold for all indices $j\geq0$. While the corresponding results in the previous sections, Lemma \ref{m_under_coord_changes} and Lemma \ref{s_under_coord_changes}, only required the existence of an element $f\in J$ which is $z$-regular of order $c$, Lemma \ref{s_under_double_coord_changes} only holds if there exists an element $f\in J$ which is $\ord$-clean with respect to $J_{-1}$. The proof of Lemma \ref{s_under_double_coord_changes} is lengthy and very technical. It needs as a preparation the Lemma \ref{new_edge_lemma} which will enable us to interpret the $\ord$-cleanness of an element $f\in J$ in terms of the coefficients $f_{i,j}$. 




From the results of the two previous sections, one could make the optimistic guess that the order of $J_{-2}$ is maximal if there exist an element $f\in J$ that is secondary $\ord$-clean with respect to $J_{-2}$ and an element $f_{-1}\in I_{-1}$ that is $\ord$-clean with respect to $J_{-2}$. But as the following example shows, this is not true:

\begin{example}
 Consider the ideal $J$ generated by the element 
 \[f=z^2+y^3+yx^4\]
 in the ring $R=K[[x,y,z]]$ over a field $K$ of characteristic $2$. The coefficient ideals $J_{-1}$ and $J_{-2}$ have the form
 \[J_{-1}=(y^3+yx^4),\]
 \[J_{-2}=(x^{12}).\]
 Hence, $\ord J_{-2}=12$.
 
 Notice that the element $f$ is a clean purely inseparable power series. Hence, we know that $f$ is $\ord$-clean with respect to $J_{-1}$ and secondary $\ord$-clean with respect to $J_{-2}$. Further, the element
 \[f_{-1}=y^3+yx^4\]
 that generates the ideal $J_{-1}$ is $\ord$-clean with respect to $J_{-2}$ since $q_K(3)=1$ and the coefficient of $y^2$ is zero.
 
 By the results that we proved in the previous two sections, we know that neither a coordinate change $z\mapsto z+g(x,y)$ nor a coordinate change $y\mapsto y+h(x)$ that preserves the setting can increase the order of the second coefficient ideal by itself. The situation is different though when we consider simultaneous coordinate changes.
 
 Set $z=\wt z+yx+x^3$ and $y=\wt y+x^2$.
 Then
 \[f=\wt z^2+\wt y^3.\]
 In particular,
 \[\wt J_{-1}=(\wt y^3),\]
 \[\wt J_{-2}=0.\]
 Thus, $\ord \wt J_{-2}=\infty$.
\end{example}



\begin{lemma} \label{double_coord_change}
 For all indices $i,j\geq0$ the equality
 \[\wt f_{i,j}=\sum_{\substack{k\geq i\\l\geq0}}\binom{k}{i}f_{k,l}\sum_{\substack{0\leq m\leq j\\ \alpha\in\N^{k-i}\\|\alpha|=j-m}}\binom{l}{m}g_\alpha h^{l-m}\]
 holds, where $g_\alpha=\prod_{o=1}^{k-i}g_{\alpha_o}$.
\end{lemma}
\begin{proof}
 Consider the auxiliary expansion $f=\sum_{i,j\geq0}\wc f_{i,j}\wt y^jz^i$ with $\wc f_{i,j}\in K[[\x]]$. Then it follows from Lemma \ref{coordinate_change_g_z} and Lemma \ref{coordinate_change_g} that
 \[\wt f_{i,j}=\sum_{\substack{k\geq i\\m\leq j}}\binom{k}{i}\wc f_{k,m}\sum_{\substack{\alpha\in\N^{k-i}\\|\alpha|=j-m}}g_\alpha\]
 \[=\sum_{\substack{k\geq i\\m\leq j}}\binom{k}{i}\sum_{l\geq m}\binom{l}{m} f_{k,l} h^{l-m}\sum_{\substack{\alpha\in\N^{k-i}\\|\alpha|=j-m}}g_\alpha\]
 \[=\sum_{\substack{k\geq i\\l\geq0}}\binom{k}{i}f_{k,l}\sum_{\substack{0\leq m\leq j\\ \alpha\in\N^{k-i}\\|\alpha|=j-m}}\binom{l}{m}g_\alpha h^{l-m}.\]
\end{proof}

\begin{lemma} \label{new_edge_lemma}
 Assume that $s>d!$ holds. Let $f\in J$ be an element. Then for all indices $i<c$ the following two statements are equivalent:
 \begin{enumerate}[(1)]
  \item $\ord f_i=\frac{c-i}{c!}\ord J_{-1}$.
  \item $\frac{c-i}{c!}d\in\N$, $\ord f_{i,\frac{c-i}{c!}d}=\frac{c-i}{c!}|r|$ and $\init(f_i)=\init(f_{i,\cid})\cdot y^{\cid}$.
 \end{enumerate}
\end{lemma}
\begin{proof}
 $(2)\implies(1)$: We can compute that
 \[\ord f_i\leq \ord f_{i,\cid}y^{\cid}=\frac{c-i}{c!}|r|+\cid=\frac{c-i}{c!}\ord J_{-1}.\]
 By Lemma \ref{ord_coeff}, equality has to hold.
 
 $(1)\implies(2)$: There is an index $j\geq0$ such that
 \[\ord f_{i,j}=\frac{c-i}{c!}\ord J_{-1}-j=\frac{c-i}{c!}|r|+\frac{c-i}{c!}d-j\]
 holds. Notice that by Lemma \ref{ord_J_2} the inequality
 \[\ord f_{i,j}\geq(c-i)\D-j\sd=\frac{c-i}{c!}|r|+\Big(\frac{c-i}{c!}d-j\Big)\sd\]
 holds. Since $s>d!$, this implies that $j\geq\cid$. But we also know by Lemma \ref{factorization_lemma} that
 \[\ord f_{i,j}\geq\frac{c-i}{c!}|r|.\]
 Obviously, this implies that $j\leq\cid$. Hence, $j=\cid$ and 
 \[\ord f_{i,\cid}=\frac{c-i}{c!}|r|.\]
 Further, it is clear that $\init(f_i)=\init(f_{i,\cid})\cdot y^{\cid}$ holds.
\end{proof}

\begin{lemma} \label{s_under_double_coord_changes}
 The following hold:
 \begin{enumerate}[(1)]
  \item If $\ord h\geq\frac{s}{d!}$ and for all indices $j\geq0$ the inequality
  \[\ord g_j\geq \D-j\sd\]
  holds, then $\ord \wt J_{-2}\geq \ord J_{-2}$.
  \item If $\ord h>\frac{s}{d!}$ and for all indices $j\geq0$ the strict inequality
  \[\ord g_j>\D-j\sd\]
  holds, then $\ord \wt J_{-2}=\ord J_{-2}$.
  \item If there is an element $f\in J$ which is $\ord$-clean with respect to $J_{-1}$ and either $\ord h<\frac{s}{d!}$ holds or there is an index $j\geq0$ such that
  \[\ord g_j<\D-j\sd\]
  holds, then $\ord \wt J_{-2}<\ord J_{-2}$.
 \end{enumerate}
\end{lemma}
\begin{proof}
 The assertions (1) and (2) follow immediately from Lemma \ref{m_under_coord_changes} and Lemma \ref{s_under_coord_changes}.

 (3): We begin the proof by noticing that $J_{-1}=\coeff^c_{(\x,y,z)}(J)=\coeff^c_{(\x,\wt y,z)}(J)$ since $y=\wt y+h$ with $h\in K[[\x]]$. Define the ideal $\wc J_{-2}\subseteq K[[\x]]$ as 
 \[\wc J_{-2}=\coeff_{(\x,\wt y)}^d(I_{-1}).\]
 Set $\wc s=\ord\wc J_{-2}$.
 
 If $\ord h\geq\frac{s}{d!}$, we know that $\wc s\geq s$ by Lemma \ref{m_under_coord_changes} (1). It then follows that $\wt s<s$ holds by the formula for $\wt s$ in Lemma \ref{s_under_coord_changes} (3). So we can assume that $\ord h<\frac{s}{d!}$ and 
 \[\wc s=d!\ord h<s\]
 by Lemma \ref{m_under_coord_changes}. This implies that $s>d!$. Assume that $\wt s\geq s$. This implies that also $\wt s>\wc s$ holds. Set 
 \[\wc \D=\frac{1}{c!}\B(d\td+|r|\B).\]
 Then by Lemma \ref{s_under_coord_changes} the inequality 
 \[\ord g_j\geq \wc \D-j\td\]
 holds for all indices $j\geq0$ and there is an index $\jz$ such that
 \[\ord g_\jz=\wc \D-\jz\td.\]
 
 As a further preparation for the proof, we are going to show that for all indices $k,l\geq0$ with $l\neq \frac{c-k}{c!}d$ the inequality
 \[\ord f_{k,l}>(c-k)\wc \D-l\td\]
 holds. If $l<\frac{c-k}{c!}d$, it is clear by Lemma \ref{ord_J_2} that
 \[\ord f_{k,l}\geq (c-k)\D-l\sd=\underbrace{\B(\frac{c-k}{c!}d-l\B)}_{>0}\underbrace{\sd}_{>\td}+\frac{c-k}{c!}|r|\]
 \[>\B(\frac{c-k}{c!}d-l\B)\td+\frac{c-k}{c!}|r|=(c-k)\wc \D-l\td.\]
 Now consider the case $l>\frac{c-k}{c!}d$. By Lemma \ref{factorization_lemma} we can compute that
 \[\ord f_{k,l}\geq \frac{c-k}{c!}|r|\]
 \[>\B(\frac{c-k}{c!}-l\B)\td+\frac{c-k}{c!}|r|=(c-k)\wc \D-l\td.\]
 On the other hand, consider indices $k,l$ with $l=\frac{c-k}{c!}d$. We know by Lemma \ref{factorization_lemma} that
 \[\ord f_{k,\frac{c-k}{c!}d}\geq \frac{c-k}{c!}|r|=(c-k)\wc \D-l\td.\]
 
 We will now consider different cases according to which conditions of $(1)_{\ord}-(3)_{\ord}$ hold for $f$ with respect to $J_{-1}$.
 
 Assume that the property $(1)\corner$ holds for $f$. Let $i$ with $c-q<i<c$ be maximal such that $\ord f_i=\frac{c-i}{c!}\ord J_{-1}$. By Lemma \ref{new_edge_lemma} this implies that $\cid\in\N$ and
 \[\ord f_{i,\cid}=\frac{c-i}{c!}|r|.\]
 We now want to show that 
 \[\ord\wt f_{i,0}=(c-i)\wc \D.\]
 By Lemma \ref{double_coord_change} we know that
 \[\wt f_{i,0}=\sum_{\substack{k\geq i\\l\geq0}}\binom{k}{i}f_{k,l}g_0^{k-i}h^l.\]
 So let $k$ and $l$ be indices with $k\geq i$ and $l\geq0$. We know that
 \[\ord g_0^{k-i}h^l\geq (k-i)\wc \D+l\td.\]
 By the maximality of $i$ we know by Lemma \ref{new_edge_lemma} that for all indices $i<k<c$ either $\frac{c-k}{c!}d\notin\N$ holds or
 \[\ord f_{k,\frac{c-k}{c!}d}>\frac{c-k}{c!}|r|.\] 
 Thus, the strict inequality
 \[\ord f_{k,l}>(c-k)\wc \D-l\td\]
 holds if $(k,l)\notin\{(i,\frac{c-i}{c!}d),(c,0)\}$. Consequently,
 \[\ord f_{k,l}g_0^{k-i}h^l>(c-i)\wt \D\]
 holds for all such indices $k,l$. Further, we know by Lemma \ref{c-q} (1) that $\binom{c}{i}=0$. On the other hand, we know that
 \[\ord f_{i,\frac{c-i}{c!}d}h^{\frac{c-i}{c!}d}=\frac{c-i}{c!}|r|+\frac{c-i}{c!}d\td=(c-i)\wc \D.\]
 This proves that
 \[\ord \wt f_{i,0}=(c-i)\wc \D.\]
 By Lemma \ref{s_changes} (2) this implies that $\wt s\leq \wc s<s$.
 
 From now on we will assume that the property $\neg(1)\corner$ holds for $f$. By Lemma \ref{new_edge_lemma} this implies that for all indices $c-q<i<c$ either $\cid\notin\N$ holds or
 \[\ord f_{i,\cid}>\frac{c-i}{c!}|r|.\]
 
 Now assume that the property $(2)\corner$ holds for $f$. Thus, we know by Lemma \ref{new_edge_lemma} that either $\frac{q}{c!}d\notin\N$ holds or
 \[\ord f_{c-q,\qd}>\frac{q}{c!}|r|.\]
 We now want to show that
 \[\ord \wt f_{c-q,\jz q}=q\wc \D-q\jz \td.\]
 By Lemma \ref{double_coord_change} we know that 
 \[\wt f_{c-q,\jz q}=\sum_{\substack{k\geq c-q\\l\geq0}}\binom{k}{c-q}f_{k,l}\sum_{\substack{0\leq m\leq \jz q\\ \alpha\in\N^{k-(c-q)}\\|\alpha|=\jz q-m}}\binom{l}{m}g_\alpha h^{l-m}.\]
 So let $k\geq c-q$, $l\geq0$, $m\leq\jz q$ and $\alpha\in\N^{k-(c-q)}$ be indices fulfilling $|\alpha|=\jz q-m$. We know that
 \[\ord g_\alpha h^{l-m}\geq (k-(c-q))\wc \D-(\jz q-m)\td+(l-m)\td\]
 \[=(k-(c-q))\wc \D-(\jz q-l)\td.\]
 Using the properties $\neg(1)\corner$ and $(2)\corner$, we know that
 \[\ord f_{k,l}>(c-k)\wc \D-l\td\]
 holds if $(k,l)\neq(c,0)$. Consequently,
 \[\ord f_{k,l}g_\alpha h^{l-m}>q\wc \D-q\jz\td\]
 holds whenever $(k,l)\neq(c,0)$. By Lemma \ref{c-q} we know that $\binom{c}{c-q}\neq0$. Also, we know by Lemma \ref{char_p_sum_lemma} that
 \[\sum_{\substack{\alpha\in\N^q\\|\alpha|=\jz q}}g_\alpha=g_\jz^q.\]
 Consequently,
 \[\ord \binom{c}{c-q}f_{c,0}\sum_{\substack{\alpha\in\N^{k-(c-q)}\\|\alpha|=\jz q}}g_\alpha=q\ord g_\jz\]
 \[=q\wc \D-\jz q\td.\]
 Hence, we have shown that
 \[\ord \wt f_{c-q,\jz q}=q\wc \D-\jz q\td.\]
 By Lemma \ref{s_changes} (2) this proves that $\wt s\leq \wh s<s$.
 
 Finally, assume that the properties $\neg(1)\corner$, $\neg(2)\corner$ and $(3)\corner$ hold. Thus, we know by Lemma \ref{new_edge_lemma} that $\qd\in\N$, 
 \[\ord f_{c-q,\qd}=\frac{q}{c!}|r|\]
 and 
 \[\init(f_{c-q})=\init(f_{c-q,\qd})\cdot y^{\qd}.\]
 Property $(3)\corner$ thus implies that either $\frac{d}{c!}\notin\N$ or there is no element $G\in K[[\x]]$ such that
 \[\init(f_{c-q,\qd})=\init(f_{c,0})\cdot G^q\]
 holds.
 
 Assume first that $\frac{d}{c!}\in\N$. We will consider the coefficient $\wt f_{c-q,0}$ and want to show that
 \[\ord \wt f_{c-q,0}=q\wc \D.\]
 By Lemma \ref{double_coord_change} we know that
 \[\wt f_{c-q,0}=\sum_{\substack{k\geq c-q\\ l\geq0}}\binom{k}{c-q}f_{k,l}g_0^{k-(c-q)}h^l.\]
 Let $k\geq c-q$ and $l\geq0$ be indices. It is clear that
 \[\ord g_0^{k-(c-q)}h^l\geq (k-(c-q))\wc \D+l\sd.\]
 By property $\neg(1)_{\ord}$ we know that the strict inequality
 \[\ord f_{k,l}>(c-k)\wc \D-l\sd\]
 holds if $(k,l)\notin \{(c,0),(c-q,\frac{q}{c!}d)\}$. Consequently,
 \[\ord f_{k,l}g_0^{k-(c-q)}h^l>q\wc \D\]
 holds for all such indices $k,l$. On the other hand,
 \[\ord f_{c-q,\frac{q}{c!}d}h^{\frac{q}{c!}d}=\frac{q}{c!}|r|+\frac{q}{c!}d\td=q\wc \D.\]
 If $\ord g_0>\wc \D$, then
 \[\ord f_{c,0}g_0^q>q\wc \D\]
 and consequently,
 \[\ord \wt f_{c-q,0}=q\wc \D.\]
 On the other hand, assume that $\ord g_0=\wc \D$. Then either $\ord \wt f_{c-q,0}=q\wc \D$ or the term
 \[\binom{c}{c-q}\init(f_{c,0})\cdot \init(g_0)^q+\init(f_{c-q,\qd})\cdot \init(h)^{q\cdot\frac{d}{c!}}\]
 vanishes. But this violates condition $(3)\corner$. Hence, $\ord \wt f_{c-q,0}=q\wc \D$. By Lemma \ref{s_changes} (2) this implies that $\wt s\leq \wc s<s$.
 
 Finally, consider the case $\frac{d}{c!}\notin\N$. Consequently, $q\nmid \frac{q}{c!}d$. Set $q_1=q_K(\frac{q}{c!}d)$. Thus, $q_1<q$ and $\binom{\frac{q}{c!}d}{q_1}\neq0$ by Lemma \ref{c-q} (1). We will now consider the coefficient $\wt f_{c-q,q_1}$ and want to show that 
 \[\ord \wt f_{c-q,q_1}=q\wc \D-q_1\td.\]
 By Lemma \ref{double_coord_change} we know that 
 \[\wt f_{c-q,q_1}=\sum_{\substack{k\geq c-q\\l\geq0}}\binom{k}{c-q}f_{k,l}\sum_{\substack{0\leq m\leq q_1\\ \alpha\in\N^{k-(c-q)}\\|\alpha|=q_1-m}}\binom{l}{m}g_\alpha h^{l-m}.\]
 So let $k\geq c-q$, $l\geq0$, $m\leq q_1$ and $\alpha\in\N^{k-(c-q)}$ be indices fulfilling $|\alpha|=q_1-m$. As before, we can conclude from property $\neg(1)\corner$ that
 \[\ord f_{k,l}g_\alpha h^{l-m}>q\wt \D-q_1\td\]
 holds if $(k,l)\notin\{(c,0),(c-q,\frac{q}{c!} d)\}$. Further, we know that
 \[f_{c,0}\sum_{\substack{\alpha\in\N^q\\|\alpha|=q_1}}g_\alpha=0\]
 by Lemma \ref{char_p_sum_lemma}. On the other hand, if $(k,l)=(c-q,\frac{q}{c!}d)$, then necessarily $m=q_1$. Further, we know by property $\neg(2)\corner$ that
 \[\ord\B(\binom{\frac{q}{c!}d}{q_1}f_{c-q,\qd}\cdot h^{\frac{q}{c!}d-q_1}\B)=q\wh \D-\qd\td+\B(\qd-q_1\B)\td\]
 \[=q\wh \D-q_1\td.\]
 This proves that
 \[\ord \wt f_{c-q,q_1}=q\wc \D-q_1\td.\]
 By Lemma \ref{s_changes} (2) this implies that $\wt s\leq \wc s<s$.
\end{proof}

\begin{example}
 As the following example shows, assertion (3) of Lemma \ref{double_coord_change} does not hold anymore if there exists an element $f\in J$ which is $z$-regular of order $c$, but not $\ord$-clean.
 
 Let the ideal $J$ be generated by the element
 \[f=z^2+y^2+x^4\]
 in the ring $R=K[[x,y,z]]$ over a field $K$ of characteristic $2$. The coefficient ideals $J_{-1}$ and $J_{-2}$ have the form
 \[J_{-1}=(y^2+x^4),\]
 \[J_{-2}=(x^{4}).\]
 Hence, $d=2$, $s=4$ and $\D=2$.
 
 Now consider the change of coordinates $z=\wt z+g_0$ and $y=\wt y+h$ with $g_0=x+x^2$ and $h=x$. Notice that they fulfill
 \[\ord g_0=1<\D-0\sd,\]
 \[\ord h=1<\sd.\]
 But the expansion of $f$ with respect to the new coordinates is
 \[f=\wt z^2+\wt y^2.\]
 Consequently, $\ord\wt J_{-2}=\infty$.
\end{example}



\begin{lemma} \label{ord_cleaning_preserves_s}
 Consider an $\ord$-cleaning step $z=\wt z+g$ with respect to $J_{-1}$ and an element $f\in J$ that is $z$-regular of order $c$.
 
 If $\ord\wt J_{-1}=\ord J_{-1}$ and $s>d!$ hold, then the coordinate change $z\mapsto \wt z$ preserves the setting and $\ord \wt J_{-2}=\ord J_{-2}$ holds.
\end{lemma}
\begin{proof}
 It follows from Lemma \ref{w_cleaning_preserves_v_cleaning} that the coordinate change $z\mapsto \wt z$ preserves the setting.
 
 Since $f$ is not $\ord$-clean with respect to $J_{-1}$, the property $\neg(2)_{\ord}$ holds. Thus, we know by Lemma \ref{new_edge_lemma} that $\qd\in\N$ and
 \[\init(f_{c-q})=\init(f_{c-q,\qd})\cdot y^{\qd}.\]
 Thus, by definition of an $\ord$-cleaning step, we know that
 \[\ord_{(y)}g=\frac{1}{q}\ord_{(y)}\init(f_{c-q})=\frac{d}{c!}.\]
 By Lemma \ref{too_big_j} this proves that $\wt s=s$.
\end{proof}

 We are now ready to prove that the main result of this section.
 
 Consider pairs $(g,h)$ with $g\in K[[\x,y]]$ and $h\in K[[\x]]$ such that $\ord g,\ord h\geq1$. Define the set
 \[\G=\{(g,h):\text{The coordinate change $z=\wt z+g$, $y=\wt y+h$ preserves the setting.}\}\]
 For each pair $(g,h)\in\G$ define the parameters $y_h=y-h$ and $z_g=z-g$. Further, we define the coefficient ideal
 \[J_{-1}^{(g,h)}=\coeff_{(\x,y_h,z_h)}^c(J).\]
 By assumption, this ideal has a factorization
 \[J_{-1}^{(g,h)}=(\x^r)\cdot I_{-1}^{(g,h)}\]
 where $I_{-1}^{(g,h)}$ is an ideal of order $\ord I_{-1}^{(g,h)}=d$. Set 
 \[J_{-2}^{(g,h)}=\coeff_{(\x,y)}^d(I_{-1}^{(g,h)}),\]
 \[s_{(g,h)}=\ord J_{-2}^{(g,h)}.\]

\begin{proposition} \label{maximum_over_y_and_z_exists}
 If there exists an element $f\in J$ which is $\ord$-clean with respect to $J_{-1}$, then there is a maximizing pair $(g_{\max},h_{\max})\in\G$ that fulfills
 \[s_{(g_{\max},h_{\max})}\geq s_{(g,h)}\]
 for all pairs $(g,h)\in\G$ .
\end{proposition}
\begin{proof}
 Assume that such a maximizing pair $(g_{\max},h_{\max})$ does not exist. This implies that there exists a sequence of pairs $(g_i,h_i)\in\G$ for $i\in\N$ such that 
 \[s_{(g_{i+1},h_{i+1})}>s_{(g_i,h_i)}\]
 holds for all $i\geq0$. Without loss of generality we can set $g_0=0$ and $h_0=0$. 
 
 By Lemma \ref{ord_cleaning_preserves_s} we may assume that $f$ is $\ord$-clean with respect to $J_{-1}^{(g_i,h_i)}$ for all $i\geq0$.
 
 Let each element $g_i$ have the expansion $g_i=\sum_{j\geq0}g_{i,j}y_i^j$ with $g_{i,j}\in K[[\x]]$. 
 
 Set $z_i=z_{g_i}$, $y_i=y_{g_i}$ and $s_i=s_{(g_i,h_i)}$. 
 
 Define for all indices $i\in\N$ the differences $G_i=g_{i+1}-g_i$ and $H_i=h_{i+1}-h_i$. Thus, $z_i=z_{i+1}+G_i$ and $y_i=y_{i+1}+H_i$. It is clear that each element $G_i$ has an expansion 
 \[G_i=\sum_{j<\frac{d}{c!}}G_{i,j}y_i^j\]
 with $G_{i,j}\in K[[\x]]$.
 
 We know by Lemma \ref{s_under_coord_changes} that for all indices $i,j\geq0$ the inequalities 
 \[\ord G_{i,j}\geq\B(\frac{d}{c!}-j\B)\frac{s_i}{d!}+\frac{|r|}{c!}.\]
 and
 \[\ord H_i\geq\frac{s_i}{d!}\]
 hold. Thus, $\lim_{i\to\infty}\ord G_i=\infty$ and $\lim_{i\to\infty}\ord H_i=\infty$. Consequently, the power series $g_\infty=\sum_{i\geq0}G_i$ and $h_\infty=\sum_{i\geq0}H_i$ are well-defined. By Lemma \ref{m_under_coord_changes} and since $f$ is $\ord$-clean, it follows that $(g_\infty,h_\infty)\in\G$. Define the parameters $z_\infty=z-g_\infty$ and $y_\infty=y-h_\infty$. Set $s_\infty=s_{(g_\infty,h_\infty)}$.
 
 We will now show that $s_\infty=\infty$. Let $k\in\N$ be an index. Then $z_k=z_\infty+\wt G_k$ and $y_k=y_\infty+\wt H_k$, where $\wt G_k=\sum_{i\geq k}G_i$ and $\wt H_k=\sum_{i\geq k}H_i$. Further, define $H_{k,i}=-\sum_{l=k}^{i-1}H_l$. It fulfills $y_i=y_k+H_{k,i}$ for $i>k$. Notice that
 \[\ord H_{k,i}\geq \min_{k\leq l<i}\ord H_l\geq \frac{s_k}{d!}.\]
 Further, we can compute that
 \[\wt G_k=\sum_{i\geq k}G_i=\sum_{\substack{i\geq k\\j\geq0}}G_{i,j}y_{i}^j\]
 \[=\sum_{\substack{i\geq k\\j\geq0}}G_{i,j}(y_k+H_{i,k})^j\]
 \[=\sum_{l\geq0}\sum_{\substack{i\geq k\\j\geq l}}\binom{j}{l}G_{i,j}H_{k,i}^{j-l}y_k^l.\]
 Thus, $\wt G_k$ has the expansion $\wt G_k=\sum_{l\geq0}\wt G_{k,l}y_k^l$ with
 \[\wt G_{k,l}=\sum_{\substack{i\geq k\\j\geq l}}\binom{j}{l}G_{i,j}H_{k,i}^{j-l}.\]
 In particular, for indices $l<\frac{d}{c!}$ the inequality
 \[\ord \wt G_{k,l}\geq \min_{\substack{i\geq k\\j\geq l}}\underbrace{\ord G_{i,j}}_{\geq(\frac{d}{c!}-j)\frac{s_k}{d!}+\frac{|r|}{c!}}+(j-l)\cdot \underbrace{\ord H_{k,i}}_{\geq \frac{s_k}{d!}}\]
 \[\geq \frac{1}{c!}\B(d\frac{s_k}{d!}+|r|\B)-\frac{s_k}{d!}\]
 holds. Also, the inequality
 \[\ord \wt H_k\geq \min_{i\geq k}\ord H_i\geq \frac{s_k}{d!}\]
 holds. Consequently, $s_\infty\geq s_k$ holds by Lemma \ref{s_under_double_coord_changes} (1) and Lemma \ref{too_big_j}. Since this holds for arbitrary indices $k\in\N$, we know that $s_\infty=\infty$.
\end{proof}

\chapter{Technical results for the resolution of surfaces} \label{chapter_6}

 In this chapter we will establish various technical results that will be used to prove the resolution of surface singularities in Chapters 7-9. Most results involve coefficient ideals and many involve the cleaning techniques that were developed in Chapter 5. While some of the results that we are going to prove hold in arbitrary dimension, others are specific to the surface case. We will state in the opening of each section in which generality the discussed results hold and how they will be used for proving the resolution of surface singularities.

\section{Stability of cleanness under blowup} \label{section_weighted_orders_under_blowup}

In the first section of this chapter we will investigate how the invariants associated to the coefficient ideal we have introduced so far behave under monomial blowup maps. We will also show in which sense the cleanness properties which were introduced in Chapter \ref{chapter_cleaning} are stable under such maps. The premise that the blowup maps we consider are monomial is fundamental to the results in this section. If translations appear, none of the claimed results holds anymore.


The main results of this section are Proposition \ref{w_clean_stable_under_blowup} and Proposition \ref{s_clean_stable_under_blowup}. In Proposition \ref{w_clean_stable_under_blowup} we will show how weighted orders of the coefficient ideal behave under monomial blowup maps and in which sense $\w$-cleanness is preserved. The result holds in arbitrary dimension.  Proposition \ref{s_clean_stable_under_blowup} states how the order of the second coefficient ideal behaves under a monomial point-blowup and that secondary $\ord$-cleanness is preserved in this case. This result only holds for point-blowups in a $3$-dimensional ambient space. Apart from these results, we will show in Lemma \ref{cleaning_preserves_directrix} that the cleaning procedures which were introduced in Chapter \ref{chapter_cleaning} preserve the directrix under certain assumptions.


The results of this section will be applied in Chapter 9 when proving that the resolution invariant $\ivX$ decreases under blowup. 

\subsection[Behavior of weighted orders of the coefficient ideal under blowup]{Behavior of weighted orders of the coefficient ideal under monomial blowup maps}

For the first result, we consider the following setting:

 Let $R=K[[\x,z]]$ with $\x=(x_1,\ldots,x_n)$ and $J\subseteq R$ an ideal of order $c=\ord J$. Set $J_{-1}=\coeff^c_{(\x,z)}(J)$.
 
 Let $\w:K[[\x]]\to\Ni^l$ be a weighted order function defined on the parameters $\x$. Denote 
 \[\w(x_i)=(\w_1(x_i),\ldots,\w_l(x_i)).\]

 Let $k\leq n$ be an index such that the center $(x_1,\ldots,x_k,z)$ is permissible with respect to the order function in the sense that $\ord_{(x_1,\ldots,x_k,z)}(J)=c$. Further, assume that for all indices $i\leq k$ the inequality $\w_j(x_i)\geq \w_j(x_1)$ holds for $j=1,\ldots,l$.
  
 Consider the monomial blowup map $\pi$ along the center $(x_1,\ldots,x_k,z)$ in the $x_1$-chart. Thus, the map $\pi:R\to R$ is given by
 \[\begin{array}{ll}
 \pi(x_1)=x_1, &\\
 \pi(x_i)=x_1x_i & \text{for $1<i\leq k$,}\\ 
 \pi(x_j)=x_j & \text{for $j>k$,} \\
 \pi(z)=x_1z. &
\end{array}\]
 
 Let $J^*=R\pi(J)$ be the total transform and $J'=x_1^{-c}J^*$ the weak transform of $J$ under $\pi$. Assume that $\ord J'=c$ and set $J_{-1}'=\coeff^c_{(\x,z)}(J')$.
 
 Denote the induced map $K[[\x]]\to K[[\x]]$ again by $\pi$. Let $J_{-1}^*=K[[\x]]\pi(J_{-1})$ be the total transform of $J_{-1}$. Recall from Lemma \ref{coeff_ideal_under_blowup} that the inclusion $x_1^{-c!}J_{-1}^*\subseteq J_{-1}'$ holds, but in general, equality does not hold.
 
 Consider the \emph{induced weighted order function} $\w':K[[\x]]\to\N_\infty^l$ which is defined on the parameters $\x$ via 
  \[\begin{array}{ll}
 \w'(x_1)=\w(x_1), &\\ 
 \w'(x_i)=\w(x_i)-\w(x_1) & \text{for $1<i\leq k$,} \\
 \w'(x_j)=\w(x_j) & \text{for $k<j\leq n$.}
\end{array}\]

\begin{proposition} \label{w_clean_stable_under_blowup}

 The following hold:
 
 \begin{enumerate}[(1)]
  \item $\w'(J_{-1}')=\w'(x_1^{-c!}J_{-1}^*)=\w(J_{-1})-c!\cdot\w(x_1)$.
  \item If an element $f\in J$ is $\w$-clean with respect to $J_{-1}$, then its transform $f'=x_1^{-c}\pi(f)\in J'$ is $\w'$-clean with respect to $J_{-1}'$.
 \end{enumerate}
 
\end{proposition}
 
\begin{proof}  Let each element $f\in J$ have an expansion $f=\sum_{i\geq0}f_iz^i$ with $f_i\in K[[\x]]$. Then the elements $f'=x_1^{-c}\pi(f)\in J'$ have the expansion $f'=\sum_{i\geq0}f_i'z^i$ with $f_i'=x_1^{i-c}\pi(f_i)\in K[[\x]]$.

 (1): Let $G$ be a generating set for $J$. Then $G'=\{f'=x_1^{-c}\pi(f):f\in G\}$ is a generating set for the weak transform $J'$.
 
 Let $\x^\alpha$ be a monomial with $\alpha\in\N^n$. Then 
 \[\pi(\x^{\alpha})=x_1^{\alpha_1+\ldots+\alpha_k}\prod_{i=2}^n x_i^{\alpha_i}.\]
 Hence, the equality $\w'(\pi(\x^{\alpha}))=\w(\x^\alpha)$ holds by definition. Consequently, for all elements $f\in J$ and indices $i\geq0$ the equality $\w'(\pi(f_i))=\w(f_i)$ holds.
 
 This allows us to compute with Lemma \ref{w_coeff} that
 \[\w'(J_1')=\min_{\substack{f'\in G'\\i<c}}\frac{c!}{c-i}\w'(f_i')=\min_{\substack{f\in G\\i<c}}\frac{c!}{c-i}\w'(x_1^{i-c}\pi(f_i))\]
 \[=\min_{\substack{f\in G\\i<c}}\frac{c!}{c-i}\w'(\pi(f_i))-c!\cdot\w'(x_1)=\w'(x_1^{-c!}J_{-1}^*)\]
 \[=\min_{\substack{f\in G\\i<c}}\frac{c!}{c-i}\w(f_i)-c!\cdot\w(x_1)=\w(J_{-1})-c!\cdot\w(x_1).\]
 
 (2): 
 Set $m=\w(J_{-1})$ and $m'=\w'(J_{-1}')$. 
 
 By Lemma \ref{z_regular_stable_under_blowup} we know that $f'$ is again $z$-regular of order $c$.
 
 Assume first that $(1)_\w$ holds for $f$ with respect to $J_{-1}$. Let $i$ be an index with $c-q<i<c$ such that $\w(f_i)=\frac{c-i}{c!}m$. Then by what we have already shown,
 \[\w'(f_i')=(i-c)\w(x_1)+\w'(\pi(f_i))=(i-c)\w(x_1)+\w(f_i)=\frac{c-i}{c!}m'.\]
 Hence, the property $(1)_{\w'}$ holds for $f'$ with respect to $J_{-1}'$.
 
 It can be shown in exactly the same way that if $(2)_\w$ holds for $f$ with respect to $J_{-1}$, also $(2)_{\w'}$ holds for $f'$ with respect to $J_{-1}'$.
 
 Finally, assume that $(3)_\w$ holds for $f$ with respect to $J_{-1}$. Assume that the property $(3)_{\w'}$ does not hold for $f'$ with respect to $J_{-1}'$. Thus, there is an element $G\in K[[\x]]$ such that
 \[\init_{\w'}(f_{c-q}')=\init_{\w'}(f_c')\cdot G^q.\]
 From what we have shown, it is easy to see that $\init_{\w'}(f_{c-q}')=x_1^{-q}\pi(\init_\w(f_{c-q}))$ and $\init_{\w'}(f_c')=\pi(\init_\w(f_c))$. Thus, we see that the equality
 \[\pi(\init_\w(f_{c-q})\cdot(\init_\w(f_c))^{-1})=(x_1G)^q\]
 holds. Further, it is easy to check that a monomial $\x^\alpha$ is a $q$-th power if and only if its image $\pi(\x^\alpha)$ is a $q$-th power. Hence, this implies that there exists an element $H\in K[[\x]]$ such that
 \[\init_\w(f_{c-q})\cdot(\init_\w(f_c))^{-1}=H^q.\]
 But this obviously contradicts the property $(3)_\w$ for $f$. Thus, we have shown that $(3)_{\w'}$ holds for $f'$ with respect to $J_{-1}'$.
\end{proof}

\begin{remark}
 Notice that \emph{any} weighted order function $\y:K[[\x]]\to\Ni^l$ can be expressed as an induced weighted order function $\nu=\w'$ by setting $\w(x_1)=\y(x_1)$, $\w(x_i)=\y(x_i)+\y(x_1)$ for $1<i\leq k$ and $\w(x_j)=\y(x_j)$ for $k<j\leq n$. This is particularly useful when combining it with the observation in Lemma \ref{w_cleaning_preserves_v_cleaning} that cleaning processes with respect to different weighted order functions preserve one another.
 
 Thus, if there is an element $f\in J$ that is $z$-regular of order $c$, we can successively apply cleaning with respect to weighted order functions to ensure that both $f$ is $\w_i$-clean with respect to $J_{-1}$ and $f'$ is $\y_j$-clean with respect to $J_{-1}'$ for finitely many given weighted order functions $\w_1,\ldots,\w_b$ and $\y_1,\ldots,\y_d$ that are all defined on the parameters $\x$.
\end{remark}

 Recall from Section \ref{section_directrix} that the directrix was our main tool for determining the position of equiconstant points $a'$ lying over a point $a$. The following result will ensure that cleaning processes stabilize the directrix as long as the involved weighted order functions are defined on parameters that are also generators of the directrix.

\begin{lemma} \label{cleaning_preserves_directrix}
 Let $R=K[[\x,z]]$ with $\x=(x_1,\ldots,x_n)$ and $J\subseteq R$ an ideal of order $c=\ord J$. Assume that $\Dir(J)=(x_1,\ldots,x_k,z)$ for an index $k\leq n$. Set $J_{-1}=\coeff_{(\x,z)}^c(J)$. Let $J_{-1}$ have a factorization $J_{-1}=(\x^r)\cdot I_{-1}$ and set $J_{-2}=\coeff_{\x}^d(I_{-1})$ where $d=\ord I_{-1}$. Further, let $f\in J$ be an element that is $z$-regular of order $c$.
 
 Let $z=z_1+g$ be either an $\w$-cleaning step with respect to $J_{-1}$ and $f$ for a weighted order function $\w$ that is defined on the parameters $\x$ or a secondary $\ord$-cleaning step with respect to the second coefficient ideal $J_{-2}$ and $f$.
 
 Then $\ol z_1\in\Dir(J)$.
\end{lemma}
\begin{proof}
 Let $f$ have the expansion $f=\sum_{i\geq0}f_iz^i$ with $f_i\in K[[\x]]$. Since $\init(f)\in (x_1,\ldots,x_k,z)^c$, we know that $f_{c-q}$ has the form
 \[f_{c-q}=\sum_{\substack{\alpha\in\N^k\\|\alpha|=q}}c_\alpha x_1^{\alpha_1}\cdots x_k^{\alpha_k}+Q(\x)\]
 where $c_\alpha\in K$ and $Q\in K[[\x]]$ with $\ord Q>q$. By definition of the cleaning step, $g$ then has the form
 \[g=\sum_{i=1}^k\lambda_ix_i+\wt Q\]
 where $\lambda_i\in K$ and $\wt Q\in K[[\x]]$ with $\ord\wt Q\geq2$. Thus, $\ol z_1\in \Dir(J)$.
\end{proof}

\subsection[Behavior of the order of the second coefficient ideal under blowup]{Behavior of the order of the second coefficient ideal under monomial blowup maps}

For the second result of this section, we consider the following setting:

Let $R=K[[x,y,z]]$ and $J\subseteq R$ an ideal of order $c=\ord J$. Set $J_2=\coeff^c_{(x,y,z)}(J)$.

Let the ideal $J_2$ have a factorization $J_2=(x^{r_x}y^{r_y})\cdot I_2$ for an ideal $I_2\subseteq K[[x,y]]$. Set $d=\ord I_2$.

Further, set $J_1=\coeff_{(x,y)}^d(I_1)$ and $s=\ord J_1$.
 
Consider the point-blowup map $\pi$ in the center of the $x$-chart. Thus, the map $\pi:R\to R$ is given by
\[\begin{array}{ll}
 \pi(x)=x, &\\
 \pi(y)=xy &\\ 
 \pi(z)=xz. &
\end{array}\]
 
Let $J^*=R\pi(J)$ be the total transform and $J'=x^{-c}J^*$ the weak transform of $J$ under $\pi$. Assume that $\ord J'=c$ and set $J_2'=\coeff_{(x,y,z)}^c(J')$.

By Proposition \ref{w_clean_stable_under_blowup} (1) we know that the ideal $J_2'$ has a factorization $J_2'=(x^{r_x'}y^{r_y'})\cdot I_2'$ for an ideal $I_2'\subseteq K[[x,y]]$ where
\[r_x'=\ord J_2-c!=r_x+r_y+d-c!,\]
\[r_y'=r_y.\]

Set $d'=\ord I_2'$. Since $I_2'$ has the same order as the weak transform of $I_2$ by Proposition \ref{w_clean_stable_under_blowup} (1), we know by Proposition \ref{ord_under_blowup} that $d'\leq d$ holds.

Further, define $J_1'=\coeff_{(x,y)}^{d'}(I_2')$ and $s'=\ord J_1'$.

\begin{proposition} \label{s_clean_stable_under_blowup}
 If $d'=d$, then following hold:
 \begin{enumerate}[(1)]
  \item $\ord J_1'=\ord J_1-d!$.
  \item If an element $f\in J$ is secondary $\ord$-clean with respect to $J_1$, then its transform $f'=x^{-c}\pi(f)$ is secondary $\ord$-clean with respect to $J_1'$.
 \end{enumerate}
\end{proposition}
\begin{proof}
Let each element $f\in J$ have an expansion $f=\sum_{i,j\geq0}f_{i,j}y^jz^i$ with $f_{i,j}\in K[[x]]$. Then the $f'=x^{-c}\pi(f)\in J'$ have the expansion $f'=\sum_{i,j\geq0}f_{i,j}'y^jz^i$ with $f_{i,j}'=x^{i+j-c}f_{i,j}$. Hence, 
 \[\ord f_{i,j}'=\ord f_{i,j}-(c-i)+j.\]

(1): Let $G$ be a generating set for $J$. Then $G'=\{f'=x^{-c}\pi(f):f\in G\}$ is a generating set for the weak transform $J'$. Hence, we can compute with Lemma \ref{ord_J_2} that
 \[s'=\min_{f'\in G'}\min_{\substack{i<c\\j<\frac{c-i}{c!}(d+r_y')}}\frac{d!}{d+r_y'-\frac{c!}{c-i}j}\Big(\frac{c!}{c-i}\ord f_{i,j}'-r_x'\Big)\]
 \[=\min_{f\in G}\min_{\substack{i<c\\j<\frac{c-i}{c!}(d+r_y)}}\frac{d!}{d+r_y-\frac{c!}{c-i}j}\Big(\frac{c!}{c-i}\ord f_{i,j}-r_x-\B(d+r_y-\frac{c!}{c-i}j\B)\Big)=s-d!.\]
 
 (2): Recall the notation
 \[\D=\frac{1}{c!}\B((d+r_y)\sd+r_x\B).\]
 Notice that
 \[\D'=\frac{1}{c!}\B((d'+r_y')\frac{s'}{d'!}+r_x'\B)\]
 \[=\frac{1}{c!}\Big((d+r_y)\sd\underbrace{-(d+r_y)+\ord J_2}_{=r_x}-c!\Big)=\D-1\]
 since $s'=s-d!$ by (1).
 
 Thus, the property
 \[\ord f_{i,j}=(c-i)\D-j\sd\]
 holds for indices $i,j\geq0$ if and only if the property
 \[\ord f_{i,j}'=(c-i)\D'-j\frac{s'}{d'!}\]
 holds. 
 
 By Lemma \ref{z_regular_stable_under_blowup} we know that $f'$ is again $z$-regular of order $c$. Let $\jz<\frac{d+r_y}{c!}$ be an index. By what we have shown, it is clear that the properties $(i)_\jz$ and $(ii)_\jz$ carry over from $f$ to $f'$.
 
 So assume that $f$ has the property $(iii)_\jz$ with respect to $J_{1}$, but $f'$ has the property $\neg(iii)_\jz$ with respect to $J_{1}'$. Thus, there is an element $G\in K[[x]]$ such that
 \[\init(f'_{c-q,\jz q})=\init(f'_{c,0})\cdot G^q.\]
 Since $f'_{c-q,\jz q}=x^{(\jz-1)q}f_{c-q,\jz q}$ and $f'_{c,0}=f_{c,0}$, this implies that
 \[\init(f_{c-q,\jz q})=\init(f_{c,0})\cdot (x^{-(\jz-1)}G)^q\]
 holds. But this contradicts the assumed property $(iii)_\jz$.
 
 Hence, we have shown that $f'$ is secondary $\ord$-clean with respect to $J_{1}'$.
\end{proof}

\section{Bounds for the invariant $d_\F$ for flags $\F$ with $n_\F>0$} \label{section_kangaroo_calc}

 We will briefly recall from Section \ref{section_modifying_the_residual_order} the definition of $d_\F$ in the case $n_\F>0$. Let $W$ be a $3$-dimensional regular variety, $X\subseteq W$ a closed subset and $a\in X$ a closed point of order $c=\ord_aX$. Further, let there be given a simple normal crossings divisor $\Eg\subseteq\Spec(\hOWa)$. Let $\F\in\FF$ be a formal flag. If $\F_1\cup(\F_2\cap\Eg)$ is not simple normal crossings, then the associated multiplicity $n_\F$ of $\F$ is defined as the maximal intersection multiplicity of $\F_1$ with a curve $\F_2\cap D$ where $D$ is a component of $\Eg$.
 
 Let $\x=(x,y,z)$ be a regular system of parameters for $\hOWa$ such that either $\Eg=V(y)$ or $\Eg=V(xy)$. Set $y_1=y+tx^n$ for a non-zero constant $t\in K^*$ and a positive integer $n$. Then a flag $\F$ of the form $\F_2=V(z_1)$, $\F_1=V(z_1,y_1)$ with $z_1=z+g$ for some $g\in K[[x,y]]$ has associated multiplicity $n_\F=n$. Set $\x_1=(x,y_1,z_1)$. The invariant $d_\F$ is defined by 
  
 \[d_{\F,\x_1}=\ord_{(y_1)}\minit_{\w}(J_{2,\x_1}),\]
 \[d_\F=\begin{cases}
         d_{\F,\x_1} & \text{if $d_{\F,\x_1}\geq c!$,}\\
         d_{\F,\x_1} & \text{if $0<d_{\F,\x_1}<c!$ and $c!\nmid m_\F$,}\\
         -1 & \text{if $0<d_{\F,\x_1}<c!$ and $c!\mid m_\F$,}\\
         -1 & \text{if $d_{\F,\x_1}=0$,}
        \end{cases}
\]
 where $J_{2,\x_1}=\coeff_{(x,y_1,z_1)}^c(\hIXa)$ and $\w:K[[x,y_1]]\to\Ni$ is the weighted order function defined on $(x,y_1)$ via $\w(x)=1$ and $\w(y_1)=n$. Notice that $\w$ is also defined on the parameters $(x,y)$ via $\w(x)=1$ and $\w(y)=n$.

 In this section, we will establish several technical results that serve as bounds for the invariant $d_\F$ which are independent of the constant $t$ and the multiplicity $n$. As we will show in Proposition \ref{cleaning_tangential_flags}, we may assume for this purpose that the coordinate change $z=z_1+g$ is $\y$-cleaning of an element $f\in\hIXa$ with respect to $J_{2,\x}$ where the weighted order function $\y:K[[x,y_1]]\to\Ni^2$ is defined via $\y(x)=(0,1)$, $\y(y_1)=(1,n)$. Unlike $\w$, the weighted order function $\y$ is not defined on the parameters $(x,y)$.

 The results of this section will be applied in Section \ref{section_maximizing_flags} to show that maximizing flags exist and in Section \ref{section_inv_drops_for_d>0} to show that the resolution invariant $\ivX$ decreases under point-blowup before reaching a terminal case.
 
 



\subsection{A modified version of Moh's bound}

The central result of this section, Proposition \ref{kangaroo_prop}, will provide a bound for $d_\F$ that can be calculated from the coefficient ideal $J_{2,\x}$ without needing to make the coordinate changes $y\mapsto y_1$ and $z\mapsto z_1$. It only requires the existence of an element $f\in \hIXa$ which is $\w$-clean with respect to $J_{2,\x}$. If the characteristic of $K$ is positive, the bound for $d_\F$ involves the characteristic $p$. In fact, this bound is very similar to the one that Moh devised for the increase of the residual order under permissible blowups in \cite{Moh}.

As a preparation for Proposition \ref{kangaroo_prop}, we will prove two technical lemmas on differential operators.


\begin{lemma} \label{diffop_lemma_for_kangaroo}
 Let $R=K[[x,y]]$ and $f\in R$ an element. Let $p$ be a prime number and $q=p^e$ with $e\in\N$ a $p$-th power. Let $k\in\N$ be a non-negative integer.
 
 Define the parameter $y_1=y+tx^n$ for a non-zero constant $t\in K^*$ and a positive integer $n>0$. Let $\partial_{y_1^k}:R\to R$ denote the respective differential operator with respect to the parameter system $(x,y_1)$.
 
 The following hold:
 \begin{enumerate}[(1)]
  \item If $\chara(K)=p$, then $\partial_{y^q}(f^q)=(\partial_y(f))^q$.
  \item If $\chara(K)=p$ and $f\in K[[x^q,y^q]]$, then $\partial_{y^k}(f)=0$ for $0<k<q$.
  \item $\ord_{(y)}f\leq\ord_{(y)}\partial_{y^k}(f)+k$.
  \item $\partial_{y_1^k}(f)=\partial_{y^k}(f)$. 
 \end{enumerate}
 For the remaining results, we will consider the weighted order function $\w:K[[x,y]]\to\Ni$ that is defined by $\w(x)=1$, $\w(y)=n$ and assume that $f$ is weighted homogeneous with respect to $\w$. Further, set $d(f)=\ord f-\ord_{(x)}f-\ord_{(y)}f$.
 \begin{enumerate}[(1)]\setcounter{enumi}{4}
  \item If $\partial_{y^k}(f)\neq0$, then 
  \[\ord_{(y_1)}\partial_{y^k}(f)\leq d(f).\]
  If the equality $\ord_{(y_1)}\partial_{y^k}(f)=d(f)$ holds, then 
  \[\partial_{y^k}(f)=x^ry^s(y+tx^n)^{d(f)}\]
  for certain $r,s\in\N$. 
  \item Assume that $\chara(K)=p$, $q\mid \w(f)$ and $f\notin K[[x^q,y^q]]$. Let $e_1\in\N$ be maximal with the property that $f\in K[[x^r,y^r]]$ for $r=p^{e_1}$. Then $\partial_{y^r}(f)\neq0$.
 \end{enumerate}
\end{lemma}
\begin{proof}
 (1): Since both the differential operator $\partial_{y^r}$ and the map $f\mapsto f^q$ are additive, it suffices to verify the statement for monomials $x^iy^j$. It follows from Proposition \ref{lucas} that
 \[\partial_{y^q}(x^{qi}y^{qj})=\binom{qj}{q}x^{qi}y^{q(j-1)}=(j\cdot x^iy^{j-1})^q=(\partial_y(x^iy^j))^q.\]
 
 (2): This is proved in Lemma \ref{diffops_on_p_th-powers}.
 
 (3): Let $f=y^mg$ with $m=\ord_{(y)}f$ and $g\in R$. The statement is trivial if $m<k$. So assume that $m\geq k$. Then by Proposition \ref{basis_for_diff_ops} (3), we can write
 \[\partial_{y^k}(f)=\sum_{j=0}^k\binom{m}{k-j}y^{(m-k)+j}\partial_{y^j}(g)=y^{m-k}\sum_{j=0}^k\binom{m}{k-j}y^{j}\partial_{y^j}(g).\]
 Hence, $\ord_{(y)}\partial_{y^k}(f)\geq m-k=\ord_{(y)}f-k$.
 
 (4): It suffices to verify the statement for monomials $x^iy_1^j$. We compute that
 \[\partial_{y^k}(x^iy_1^j)=\partial_{y^k}\Big(\sum_{l=0}^j\binom{j}{l}t^{j-l}x^{i+n(j-l)}y^l\Big)\]
 \[=\sum_{l=0}^j\underbrace{\binom{j}{l}\binom{l}{k}}_{=\binom{j}{k}\binom{j-k}{l-k}}t^{j-l}x^{i+n(j-l)}y^{l-k}=\binom{j}{k}x^iy_1^{j-k}=\partial_{y_1^k}(x^iy_1^j).\]
 
 (5): Let $f$ have the factorization
 \[f=x^{r_x}y^{r_y}\cdot g\]
 with $\ord g=d(f)$. Since $g$ is also weighted homogeneous with respect to $\w$, it is clear that $\w(g)=n\cdot d(f)$ holds. Then by (3) we know that 
 \[\partial_{y^k}(f)=x^{r_x}y^{s_y}\cdot G\]
 where $s_y=\max\{0,r_y-k\}$ for some element $G\in R$. Further, it is clear that $\partial_{y^k}(f)$ is again weighted homogeneous with respect to $\w$ and $\w(\partial_{y^k}(f))=\w(f)-kn$ holds. Thus,
 \[\w(G)=\w(\partial_{y^k}(f))-r_x-ns_y\leq \w(f)-r_x-nr_y=\w(g).\]
 Consequently, 
 \[\ord_{(y_1)}\partial_{y^k}(f)\leq \ord_{(y_1)}G\leq\frac{1}{n}\w(G)\]
 \[\leq\frac{1}{n}\w(g)=\ord g=d(f).\]
 
 If equality holds, then necessarily $\ord_{(y_1)}G=\frac{1}{n}\w(G)$ holds. Consequently, $G=(y+tx^n)^{d(f)}$ and $\partial_{y^k}(f)=x^{r_x}y^{s_y}(y+tx^n)^{d(f)}$.
 
 (6): Since $e_1$ is maximal with $f\in K[[x^r,y^r]]$ for $r=p^{e_1}$, a monomial $x^iy^j$ with $(i,j)\in r\cdot\N^2\setminus (pr\cdot\N^2)$ appears with non-zero coefficient in the expansion of $f$. Since $q\mid \w(f)=i+nj$, we know that $pr\nmid j$. Thus, $\partial_{y^r}(x^iy^j)=\binom{j}{r}x^iy^{j-r}\neq0$ by Proposition \ref{lucas}.
\end{proof}

\begin{lemma} \label{arithmetic_kang_lemma}
 Let $R=K[[x,y]]$ with $\chara(K)=p>0$. Let $r_x,r_y,m,n\in\N$ be non-negative integers and $t\in K^*$. Then there is no element $F\in R$ such that
 \[\partial_y(F)=x^{r_x}y^{r_y}(y+tx^n)^{mp-1}.\]
\end{lemma}
\begin{proof}
 Consider the expansion
 \[x^{r_x} y^{r_y}(y+tx^n)^{mp-1}=\sum_{i=0}^{mp-1}\binom{mp-1}{i}t^{mp-1-i}x^{r_x+n(mp-1-i)}y^{\ry+i}.\]
 There is an index $j$ with $0\leq j<p$ such that $r_y+j\equiv-1\pmod p$ holds. Notice that since $mp-1=(m-1)p+(p-1)$, the equality 
 \[\binom{mp-1}{j}=\binom{m-1}{0}\binom{p-1}{j}\neq0\]
 holds in $K$ by Proposition \ref{lucas}.
 
 Hence, a term of the form $x^{i}y^{kp-1}$ for some $k\in\N$ appears in the expansion of $x^{r_x} y^{r_y}(y+tx^n)^{mp-1}$ with non-zero coefficient. Since $\partial_y(x^iy^{kp})=0$, there can be no element $F\in R$ such that $\partial_y(F)=x^{r_x} y^{r_y}(y+tx^n)^{mp-1}$ holds.
\end{proof}


\begin{proposition} \label{kangaroo_prop}
 Let $R=K[[x,y,z]]$ and $J\subseteq R$ an ideal of order $\ord J=c$. Set $J_2=\coeff^c_{(x,y,z)}(J)$. Define the weighted order function $\w:K[[x,y]]\to\N_\infty$ via $\w(x)=1$ and $\w(y)=n$ for a positive integer $n>0$. Set $m=\w(J_2)$. Let $f\in J$ be an element that is $\w$-clean with respect to $J_2$.
 
 Let $\minit_\w(J_2)$ have a factorization 
 \[\minit_\w(J_2)=(x^{r_x} y^{r_y})\cdot I\]
 for an ideal $I$ with $\ord_{(x)}I=\ord_{(y)}I=0$. Set $d=\ord I$. Further, set $y_1=y+tx^n$ for a constant $t\in K^*$. Notice that $J_2=\coeff^c_{(x,y_1,z)}(J)$.
 
 Define the weighted order function $\y:K[[x,y_1]]\to\Ni^2$ via $\y(x)=(1,0)$ and $\y(y_1)=(n,1)$. Let $z=\wt z+g$ be $\y$-cleaning with respect to $J_2$ and $f$. Set $\wt J_2=\coeff^c_{(x,y_1,\wt z)}(J)$ and 
 \[d_*=\ord_{(y_1)}\minit_\w(\wt J_2).\]
 Then the following hold:
 \begin{enumerate}[(1)]
  \item If $\chara(K)=0$, then $d_*\leq d$.
  \item If $\chara(K)=p>0$, then $d_*\leq d+\frac{c!}{p}$.
  \item If $d_*>d$, then $m\in c!\cdot\N$.
  \item If $d\notin c!\cdot\N$, then $d_*<\ul\frac{d}{c!}\ur c!$.
 \end{enumerate}
\end{proposition}
\begin{proof}
 For an element $F\in K[[x,y]]$ denote by $d(F)$ the difference 
 \[d(F)=\ord F-\ord_{(x)}F-\ord_{(y)}F.\]
 By Lemma \ref{double_weighted order function} we know that $\y(F)=(\w(F),\ord_{(y_1)}\init_\w(F))$.
 
 It is easy to see that $\w(I)=nd$. Consequently, 
 \[m=r_x+n(r_y+d).\] 
 
 Let $f$ have the expansion $f=\sum_{i\geq0}f_iz^i$ with $f_i\in K[[x,y]]$. 
 Notice that the weighted order function $\w:K[[x,y]]=K[[x,y_1]]\to\N_\infty$ is also defined on the parameters $(x,y_1)$ via $\w(x)=1$ and $\w(y_1)=n$.
 
 Set $d_1=\ord_{(y_1)}\minit_\w(J_2)$. Then $\y(J_2)=(m,d_1)$ by Lemma \ref{double_weighted order function}. Since 
 \[\ord_{(y_1)}\minit_\w(J_2)\leq \ord_{(y_1)}I\leq \ord I=d,\]
 we know that $d_1\leq d$.
 
 Further, notice that $\w(f_i)=\frac{c-i}{c!}m$ implies that 
 \begin{equation}
  d(\init_\w(f_i))\leq\frac{c-i}{c!}d \tag{$\ast$} \label{deq}
 \end{equation}
 for the following reason: Let $\init_\w(f_i)$ have a factorization $\init_\w(f_i)=x^{a_i}y^{b_i}g_i$ with $d(\init_\w(f_i))=\ord g_i$. Then $a_i\geq\frac{c-i}{c!}r_x$ and $b_i\geq\frac{c-i}{c!}r_y$ by Lemma \ref{w_coeff}. Hence, we can compute that
 \[d(\init_\w(f_i))=\ord g_i=\frac{1}{n}\w(g_i)=\frac{1}{n}(\w(f_i)-\w(x^{a_i})-\w(y^{b_i}))\]
 \[\leq\frac{1}{n}\B(\frac{c-i}{c!}m-\frac{c-i}{c!}nr_y-\frac{c-i}{c!}r_x\B)=\frac{c-i}{c!}d.\]
 
 If $f$ is $\y$-clean with respect to $J_2$, then it is clear that $d_*=d_1\leq d$. If $c!\nmid m$, then $f$ is $\y$-clean by Lemma \ref{c!_does_not_divide_m}. This proves assertion (3). Assume from now on that $c!\mid m$.
 
 First consider the case that $\chara(K)=0$. Since $f$ is $\w$-clean with respect to $J_2$, we know that $\w(f_{c-1})>\frac{m}{c!}$. Hence, also $\y(f_{c-1})>(\frac{m}{c!},\frac{d_1}{c!})$. Thus, $f$ is also $\y$-clean with respect to $J_2$. This proves assertion (1).
 
 Assume from now on that $\chara(K)=p>0$ and $f$ is not $\y$-clean with respect to $J_2$. In particular, $\neg(2)_{\y}$ holds. Thus, $\w(f_{c-q})=\frac{q}{c!}m$ and $\ord_{(y_1)}\init_\w(f_{c-q})=\frac{q}{c!}d_1$ by Lemma \ref{double_weighted order function}.  Consequently, also $\neg(2)_{\w}$ holds for $f$ with respect to $J_2$. By $(\ast)$, this implies that $d(\init_\w(f_{c-q}))\leq \frac{q}{c!}d$.
 
 By definition of $\y$-cleaning, we know that $\w(g)=\frac{m}{c!}$ and $\ord_{(y_1)}\init_\w(g)=\frac{d_1}{c!}$ hold. Consider the expansion $f=\sum_{i\geq0}\wt f_i\wt z^i$ with $\wt f_i\in K[[x,y]]$. 
 
 First assume that the property $(1)_\w$ holds for $f$ with respect to $J_2$. Let $i$ be maximal with $c-q<i<c$ such that $\w(f_i)=\frac{c-i}{c!}m$ holds. Then we know by Lemma \ref{clean_lemma} (1) that $\init_\w(\wt f_i)=\init_\w(f_i)$. Consequently,
 \[\ord_{(y_1)}\init_\w(\wt f_i)=\ord_{(y_1)}\init_\w(f_i)\leq d(\init_\w(f_i))\leq\frac{c-i}{c!}d\]
 holds by (\ref{deq}). By Lemma \ref{w_coeff} this proves that $d_*\leq d$.
 
 Assume from now on that $\neg(1)_\w$ holds for $f$ with respect to $J_2$. Since we know that also $\neg(2)_{\w}$ holds, property $(3)_\w$ necessarily holds for $f$ with respect to $J_2$. Since $K$ is algebraically closed, this implies that $\init_\w(f_{c-q})\notin K[[x^q,y^q]]$. Let $r=p^e$ with $e\in\N$ be maximal such that $\init_\w(f_{c-q})\in K[[x^r,y^r]]$. By Lemma \ref{diffop_lemma_for_kangaroo} (6) this implies that $\partial_{y^r}(\init_\w(f_{c-q}))\neq0$. By Lemma \ref{clean_lemma} (3) we know that 
 \[\init_\w(\wt f_{c-q})=\init_\w(f_{c-q})+G^q\]
 for an element $G\in K[[x,y]]$. By using statements (2)-(5) of Lemma \ref{diffop_lemma_for_kangaroo} and the inequality (\ref{deq}) we can calculate that
 \[\ord_{(y_1)}\init_\w(\wt f_{c-q})=\ord_{(y_1)}(\init_\w(f_{c-q})+G^q)\overset{(3)}{\leq} \ord_{(y_1)}\partial_{y_1^r}(\init_\w(f_{c-q})+G^q)+r\]
 \[\overset{(2)}{=}\ord_{(y_1)}\partial_{y_1^r}(\init_\w(f_{c-q}))+r\overset{(4)}{=}\ord_{(y_1)}\partial_{y^r}(\init_\w(f_{c-q}))+r\overset{(5)}{\leq} d(\init_\w(f_{c-q}))+r\leq\frac{q}{c!}d+r.\]
 This proves that
 \[d_*\leq\frac{c!}{q}\ord_{(y_1)}(\init_\w(\wt f_{c-q}))\leq \frac{c!}{q}\B(\frac{q}{c!}d+r\B)=d+c!\frac{r}{q}\leq d+\frac{c!}{p}.\]
 Hence, we proved assertion (2).
 
 Now assume that $d$ is not divisible by $c!$. Thus, we know that 
 \[d(\init_\w(f_{c-q}))\leq\frac{q}{c!}d<q\Bigl\lceil\frac{d}{c!}\Bigr\rceil.\]
 Since $\init_\w(f_{c-q})\in K[[x^r,y^r]]$, we know that $d(\init_\w(f_{c-q}))\leq q\ul\frac{d}{c!}\ur-r$. Combining this with our earlier results shows that $d_*\leq \ul\frac{d}{c!}\ur c!$. Assume that $d_*=\ul\frac{d}{c!}\ur c!$ holds. Then the following equalities necessarily have to hold:
 \begin{itemize}
  \item $d(\init_\w(f_{c-q}))=\frac{q}{c!}d$.
  \item $d=(\ul\frac{d}{c!}\ur-\frac{r}{q})\cdot c!$. 
  \item $\ord_{(y_1)}\partial_{y^r}(\init_\w(f_{c-q}))=d(\init_\w(f_{c-q}))$.
 \end{itemize}
 By Lemma \ref{diffop_lemma_for_kangaroo} (5) this implies that 
 \[\partial_{y^r}(\init_\w(f_{c-q}))=x^{m_x}y^{m_y}(y+tx^n)^{\frac{q}{c!}d}\]
 for certain $m_x,m_y\in\N$. Since $\init_\w(f_{c-q})\in K[[x^r,y^r]]$, there is an element $F\in K[[x,y]]$ which is weighted homogeneous with respect to $\w$ and fulfills $F^r=\init_\w(f_{c-q})$. Hence, by Lemma \ref{diffop_lemma_for_kangaroo} (1) we know that $\partial_{y^r}(\init_\w(f_{c-q}))=(\partial_y(F))^r$. Thus, 
 \[\partial_y(F)=x^{\frac{m_x}{r}}y^{\frac{m_y}{r}}(y+tx^n)^{\frac{qd}{rc!}}.\]
 Notice that $\frac{qd}{rc!}=\frac{q}{r}\ul\frac{d}{c!}\ur-1$. Thus,
 \[\partial_y(F)=x^\frac{m_x}{r}y^\frac{m_y}{r}(y+tx^n)^{kp-1}\]
 for a certain integer $k\in\N$. But this is a contradiction to Lemma \ref{arithmetic_kang_lemma}. Hence, $d_*<\ul\frac{d}{c!}\ur c!$. This proves assertion (4).
\end{proof}


Recall from Section \ref{section_modifying_the_residual_order} that in the situation of a point-blowup $\pi:W'\to W$ and an equiconstant point $a'$ lying over $a$, there is no flag $\F\in\FF(a)$ such that its induced flag $\F'\in\FF(a')$ fulfills $n_{\F'}>1$ and $D_{\F'}=\Dnew$ where $\Dnew=\pi^{-1}(a)$ denotes the exceptional divisor of $\pi$. Thus, a special argument is required to show that the flag invariant $\inv(\G)$ of valid flags $\G\in\FF(a')$ with $n_\G>1$ and $D_\G=\Dnew$ is bounded by the flag invariant $\inv(\F)$ of some valid flag $\F\in\FF(a)$.

As already mentioned, we will use a modified version of Moh's bound for these flags which says that there is a valid flag $\F\in\FF(a)$ which fulfills
\[d_\G\leq \frac{1}{n_\F}d_\F+\eps,\]
where
 \[\eps=\begin{cases}
         0 & \text{if $\chara(K)=0$ or $c!\nmid m_\G$,}\\
         \frac{c!}{p} & \text{if $\chara(K)=p>0$ and $c!\mid m_\G$.}
        \end{cases}
 \]
 
 To derive this result from Proposition \ref{kangaroo_prop}, it is necessary to give a bound in terms of $d_\F$ for $d(\minit_\w(J_{2,\x'}(a'))$ where $J_{2,\x'}(a')$ is the coefficient ideal of $\hIXap$ with respect to the induced parameters $\x'$ for $\hOWap$. This is the purpose of the following result:

\begin{proposition} \label{kangaroo_blowup_prop}
 Let $R=K[[x,y,z]]$ and $J\subseteq R$ an ideal of order $\ord J=c$. Set $J_2=\coeff_{(x,y,z)}^c(J)$.
 
 Consider the point-blowup map $\pi$ in the origin of the $x$-chart. Thus, the map $\pi:R\to R$ is given by $\pi(x)=x$, $\pi(y)=xy$, $\pi(z)=xz$.
 
 Let $J'=x^{-c} R\pi(J)$ be the weak transform of $J$ and assume that $\ord J'=c$. Set $J_2'=\coeff^c_{(x,y,z)}(J')$.
 
 Consider the weighted order function $\w:K[[x,y]]\to\N_\infty$ that is defined on $(x,y)$ by $\w(x)=n$ and $\w(y)=1$ for a positive integer $n>0$.
 
 For an ideal $I\subseteq K[[x,y]]$ define $ d(I)=\ord I-\ord_{(x)}I-\ord_{(y)}I$.
 
 Then the following hold:
 \begin{enumerate}[(1)]
  \item $d(\minit_\w(J_2'))\leq \frac{1}{n} d(J_2)$.
  \item $d(\minit_\w(J_2'))\leq \frac{1}{n}\ord_{(y)}\minit(J_2)$.
 \end{enumerate}
\end{proposition}
\begin{proof}
 Define the weighted order functions $\wt\w,\y:K[[x,y]]\to\N_\infty$ via $\wt\w(x)=n$, $\wt\w(y)=n+1$ and $\y(x)=1$, $\y(y)=2$. Let $\minit_{\wt\w}(J_2)$ have the factorization 
 \[\minit_{\wt\w}(J_2)=(x^{r_x}y^{r_y})\cdot \wt I\]
 for an ideal $\wt I$ with $\ord \wt I=d(\minit_{\wt\w}(J_2))$. By Lemma \ref{double_weighted order function} and Proposition \ref{w_clean_stable_under_blowup} we can calculate that
 \[d(\minit_\w(J_2'))=\ord\minit_\w(J_2')-\ord_{(x)}\minit_\w(J_2')-\ord_{(y)}\minit_\w(J_2')\]
 \[=\underbrace{\y(\minit_{\wt\w}(J_2))}_{=r_x+2r_y+\y(\wt I)}-c!\underbrace{\y(x)}_{=1}-\underbrace{\ord\minit_{\wt\w}(J_2)}_{=r_x+r_y+\ord \wt I}+c!\underbrace{\ord x}_{=1}-\underbrace{\ord_{(y)}\minit_{\wt\w}(J_2)}_{=r_y}+c!\underbrace{\ord_{(y)}x}_{=0}\]
 \[=\y(\wt I)-\ord \wt I.\]
 
 Since $\wt I$ is generated by elements that are weighted homogeneous with respect to $\wt\w$, it is straightforward to verify that $\y(\wt I)=\frac{1}{n}\wt\w(\wt I)$ and $\ord \wt I=\frac{1}{n+1}\wt\w(\wt I)$. Consequently,
 \[d(\minit_\w(J_2'))=\B(\frac{1}{n}-\frac{1}{n+1}\B)\wt\w(\wt I)=\frac{1}{(n+1)n}\wt\w(\wt I)=\frac{1}{n}\ord \wt I.\]
 
 (1): It remains to show that $\ord \wt I\leq d(J_2)$ holds.
 
 To this end, consider the factorization $J_2=(x^{m_x}y^{m_y})\cdot I_2$ with $\ord I_2=d(J_2)$. Clearly, $r_x\geq m_x$ and $r_y\geq m_y$ hold.

 Further, we know that there is a term $x^iy^j$ which appears with non-zero coefficient in the expansion of an element of $J_2$ such that
 \[i+j=\ord J_2,\ i\geq m_x\]
 hold. Also, since $\wt\w(x)<\wt\w(y)$ holds, it is easy to see that there is a term $x^ky^l$ which fulfills
 \[nk+(n+1)l=\wt\w(J_2),\ l=r_y+\ord \wt I.\]
 Naturally, also the inequalities
 \[\ord J_2=i+j\leq k+l\]
 and 
 \[\wt\w(J_2)=nk+(n+1)l\leq ni+(n+1)j\]
 hold. This allows us to compute
 \[\ord \wt I=l-r_y=\underbrace{nk+(n+1)l}_{\leq ni+(n+1)j}\underbrace{-n(k+l)}_{\leq -n(i+j)}-r_y\]
 \[\leq j-r_y\leq j-m_y=i-m_x+d(J_2)\leq d(J_2).\]
 
 (2): We now want to show that $\ord \wt I\leq \ord_{(y)}\minit(J_2)$ holds.
 
 It is clear that there appears a term $x^iy^j$ in the expansion of an element of $J_2$ such that 
 \[i+j=\ord J_2,\ j=\ord_{(y)}\minit(J_2)\]
 hold. Further, we will again make use of a term $x^ky^l$ as above, fulfilling 
 \[nk+(n+1)l=\wt\w(J_2),\ l=r_y+\ord \wt I.\]
 hold. Thus, we can compute in the same way as before that
 \[\ord \wt I\leq j-r_y\leq j=\ord_{(y)}\minit(J_2).\]
\end{proof}

\subsection{Triviality of $d_\F$ for large $n_\F$}

In Section \ref{section_maximizing_flags} we will show that there exists a flag $\F$ which maximizes the flag invariant $\inv(\F)=(d_\F,n_\F,s_\F)$ over all valid flags. In particular, we will need to make sure that $d_\F$ is bounded and that there do not exist flags $\F$ with arbitrarily high associated multiplicity $n_\F$ that realize the maximal value for $d_\F$. In fact, we will show in Proposition \ref{maximal_n_and_d} that $d_\F$ is bounded and that there exists an $N\in\N$ such that $d_\F=-1$ holds for all valid flags $\F$ with $n_\F\geq N$. To get this result from Proposition \ref{kangaroo_prop}, it is necessary to have an element $f$ that is simultaneously clean with respect to all weighted order functions $\w_n:K[[x,y]]\to\Ni$ that are defined via $\w_n(x)=1$ and $\w_n(y)=n$. The following lemma will ensure that this can be achieved by only applying cleaning with respect to finitely many weighted order functions.

\begin{lemma} \label{automatic_w_n_cleaning_for_large_n}
 Let $R=K[[x,y,z]]$ and $J\subseteq R$ an ideal of order $c=\ord J$. Set $J_2=\coeff^c_{(x,y,z)}(J)$.
 
 For all positive integers $n>0$ consider the weighted order functions $\w_n:K[[x,y]]\to\N_\infty$ defined by $\w_n(x)=1$ and $\w_n(y)=n$. Let $\sigma:K[[x,y]]\to\Ni^2$ be the weighted order function defined by $\sigma(x)=(0,1)$ and $\sigma(y)=(1,0)$.
 
 Let $f\in J$ be an element which is $\sigma$-clean with respect to $J_2$. Then there is an integer $N\in\N$ such that $f$ is also $\w_n$-clean with respect to $J_2$ for all integers $n\geq N$.
 
 Further, there are integers $r,s\in\N$ such that $\minit_{\sigma}(J_2)=(x^{r}y^{s})$ and for all $n\geq N$ the identity $\minit_{\w_n}(J_2)=(x^{r}y^{s})$ holds.
\end{lemma}
\begin{proof}
 By Lemma \ref{double_weighted order function} we know that $\sigma(J_2)=(\ord_{(y)}J_2,\ord_{(x)}\minit_{(y)}(J_2))$. Let $r,s\in\N$ be such that $\sigma(J_2)=(s,r)$. Then it is clear that 
 \[\minit_\sigma(J_2)=\init_{(x)}(\minit_{(y)}(J_2))=(x^{r}y^{s})\]
 and $s=\ord_{(y)}J_2$.
 
 It is clear that for all $n>0$ the inequality
 \[\w_n(J_2)\leq \w_n(\minit_{\sigma}(J_2))=r+ns\]
 holds. Set $N=r+1$. We first want to prove that $\w_n(J_2)=r+ns$ and $\minit_{\w_n}(J_2)=(x^ry^s)$ for all $n\geq N$. Let $x^iy^{s+j}$ be a term appearing with non-zero coefficient in the power series expansion of an element of $J_2$. If $j>0$, then 
 \[\w_n(x^iy^{s+j})\geq (s+1)n>r+ns.\]
 If $j=0$, then we know that $i\geq r$. Thus, 
 \[\w_n(x^iy^{s})\geq r+ns\]
 is also clear in this case and equality holds only for $(i,j)=(r,s)$. Hence, we have shown that $\w_n(J_2)=r+ns$ and $\minit_{\w_n}(J_2)=(x^{r}y^{s})$. 
 
 Now we will prove that $f$ is $\w_n$-clean for $n\geq N$. Let $f$ have the expansion $f=\sum_{i\geq0}f_iz^i$ with $f_i\in K[[x,y]]$.
 
 First assume that the property $(1)_{\sigma}$ holds. Thus, there exists an index $c-q<i<c$ such that 
 \[(\ord_{(y)}f_i,\ord_{(x)}\init_{(y)}(f_i))=\sigma(f_i)=\B(\frac{c-i}{c!}s,\frac{c-i}{c!}r\B).\]
 This implies that $\w_n(f_i)\leq\frac{c-i}{c!}(r+ns)=\frac{c-i}{c!}\w_n(J_2)$. Hence, the property $(1)_{\w_n}$ holds for $f$.
 
 Next assume that the property $(2)_{\sigma}$ holds for $f$. Thus, either $\ord_{(y)} f_{c-q}>\frac{q}{c!}s$ holds or $\ord_{(y)}f_{c-q}=\frac{q}{c!}s$ and $\ord_{(x)}\init_{(y)}(f_{c-q})>\frac{q}{c!}r$. If $\ord_{(y)} f_{c-q}>\frac{q}{c!}s$, then 
 \[\w_n(f_{c-q})\geq n\cdot\ord_{(y)}f_{c-q}\geq n\B(\frac{q}{c!}s+1\B)=\frac{q}{c!}\B(ns+\underbrace{\frac{c!}{q}n}_{>r}\B)>\frac{q}{c!}(ns+r)=\frac{q}{c!}\w_n(J_2).\]
 On the other hand, if $\ord_{(y)} f_{c-q}=\frac{q}{c!}s$ and $\ord_{(x)}\init_{(y)}(f_{c-q})>\frac{q}{c!}r$, then it follows in a similar fashion that 
 \[\w_n(f_{c-q})\geq \frac{q}{c!}ns+\underbrace{\min\B\{\w_n(y),\frac{q}{c!}r+1\B\}}_{>\frac{q}{c!}r}>\frac{q}{c!}\w_n(J_2).\]
 Thus, in either case the property $(2)_{\w_n}$ holds for $f$.
 
 Finally, assume that the properties $\neg(2)_{\sigma}$ and $(3)_{\sigma}$ hold for $f$. Using the same techniques as before, it is straightforward to verify that $\init_\sigma(f_{c-q})=\init_{\w_n}(f_{c-q})$. Further, it is clear that $\init_\sigma(f_c)=\init_{\w_n}(f_c)=f_c(0)$. Thus, the property $(3)_{\w_n}$ is fulfilled for $f$.
\end{proof}

\section{Local description of the top locus in terminal cases} \label{section_form_of_top_locus}

 Recall from Section \ref{section_modifying_the_residual_order} that we employ two different terminal cases in the resolution of surfaces: The monomial case and the small residual case. In both terminal cases, the coefficient ideal $J_{2,\x}$ with respect to some regular formal hypersurface $H=V(z)$ has a very simple form. This simple form is then exploited by using blowups along particular centers to successively lower the order of the coefficient ideal. The exact definition of both terminal cases and the invariants which are used to measure improvement during combinatorial resolution will be given in Section \ref{section_monom_s_r_case}.
 
 While combinatorial resolution is by its nature generally a simple process, there is one technical complication that arises in positive characteristic due to the failure of maximal contact: Let in the following $W$ be a regular variety, $X\subseteq W$ a closed subset and $a\in X$ a closed point with $c=\ord_aX$. Let $H=V(z)\subseteq\Spec(\hOWa)$ be a regular formal hypersurface and suppose that the coefficient ideal with respect to $H$ is either of monomial or small residual form. If $\topp(X)=\{b\in X:\ord_bX=c\}$ is locally (on the level of the completion $\hOWa$) contained in the hypersurface $H$, then the centers which we need to blow up to lower the order of the coefficient ideal are automatically contained in $\topp(X)$, hence they are permissible with respect to the order-function. On the other hand, if $\topp(X)$ is not locally contained in $H$, then the simple form of the associated coefficient ideal is virtually useless since we might not be allowed to blow up the centers needed to measure further improvement.
 
 In characteristic zero, this problem is naturally resolved by allowing for $H$ only hypersurfaces of maximal contact, which locally contain $\topp(X)$ by definition (cf. Section \ref{section_woah_maximal_contact}). In positive characteristic, the problem is much more intricate. We know by Narasimhan's example \cite{Narasimhan} that for a general $X$, the top locus $\topp(X)$ might not be locally contained in any regular formal hypersurface. The goal of this section is to prove that in the monomial case and the small residual case (as they are defined in this thesis), the regular formal hypersurface $H$ whose associated coefficient ideal is of simple form \emph{automatically} contains the top locus. Hence, combinatorial resolution can be applied. Although the results of this section will be applied in the setting of a $3$-dimensional ambient space, they will be proved for arbitrary dimension.
 
 Recall from Proposition \ref{ord_via_diff_ops} that a formal hypersurface $H=V(z)\subseteq\Spec(\hOWa)$ contains $\topp(X)$ locally at $a$ if and only if the condition $z\in\rad(\Diff_{\hOWa/K}^{c-1}(\hIXa))$ is fulfilled. The proofs in this section will heavily involve differential operators which were introduced in Section \ref{section_diff_ops}.
 
 The results of this section will be used in Chapter \ref{chapter_usc} to prove the upper semicontinuity of our resolution invariant for surfaces $\ivX$ and in Section \ref{section_inv_drops_for_d=0} to show that the invariant decreases during combinatorial resolution. Since all results hold also in higher dimension, the contents of this section might also prove to be useful for future proofs of embedded resolution of singularities in higher dimension.
 
 

\subsection{The top locus in the monomial case}

Consider the power series ring $R=K[[\x,z]]$ with $\x=(x_1,\ldots,x_n)$ and an ideal $J\subseteq R$ of order $\ord J=c$. Assume that the coefficient ideal $J_{-1}=\coeff_{(\x,z)}^c(J)$ is a principal monomial ideal, say $J_{-1}=(\x^r)$ for some vector $r=(r_1,\ldots,r_n)\in\N^n$. In Proposition \ref{z_in_top_locus_in_monomial_case} we will show that if there is an element $f\in J$ which is $\ord$-regular with respect to $J_{-1}$, then the desired property $z\in\rad(\Diff_{R/K}^{c-1}(J))$ holds. 
As an application for $n=2$, we will give in Corollary \ref{form_of_top_locus_in_monomial_case} an explicit description of the ideal $\rad(\Diff_{R/K}^{c-1}(J))$ subject to the powers $r_1,r_2\in\N$ in the coefficient ideal $J_{-1}=(x_1^{r_1}x_2^{r_2})$.


 The following example shows that the result does generally not hold if there is no element $f\in J$ which is $\ord$-clean:

\begin{example}
 Consider the ring $R=K[[x,y,z]]$ over a field $K$ of characteristic $2$. Let $J\subseteq R$ be the ideal generated by the purely inseparable polynomial
 \[f=z^2+y^4(1+x^3).\]
 The coefficient ideal $J_{-1}=\coeff_{(x,y,z)}^2(J)$ is a principal monomial ideal of the form
 \[J_{-1}=(y^4).\]
 Notice though that the element $f$ is not $\ord$-clean with respect to $J_{-1}$ since the coordinate change $z=z_1+y^2$ eliminates the term $y^4$ from its expansion and increases the order of the coefficient ideal.
 
 Also, it can be computed that the ideal $\rad(\Diff_{R/K}^1(J))$ is of the form
 \[\rad(\Diff_{R/K}^1(J))=(z+y^2,xy).\]
 Hence, it does not contain the element $z$.
\end{example}

As a preparation for proving Proposition \ref{z_in_top_locus_in_monomial_case} we will first prove two technical lemmas about differential operators in positive characteristic.

 

\begin{lemma} \label{f_i_lemma}
 Let $R=K[[\x,z]]$ with $\chara(K)=p>0$. Let $f\in R$ be an element with expansion $f=\sum_{i\geq0}f_i z^i$ where $f_i\in K[[\x]]$. Set $q=p^e$ for a positive integer $e>0$. Then 
 \[f_iz^i\in\Diff_{R/K}^{q-1}(f)+(z^{q+1})\]
 for $i=1,\ldots,q-1$ and
 \[f_0+f_qz^q\in\Diff_{R/K}^{q-1}(f)+(z^{q+1}).\]
\end{lemma}
\begin{proof}
 By Lemma \ref{c-q} we know that $\binom{q}{k}=0$ for $0<k<q$. Thus,
 \[z^k\partial_{z^k}(f)=\sum_{\substack{i\geq k\\i\neq c}}\binom{i}{k}f_iz^i=\sum_{k\leq i<q}\binom{i}{k}f_iz^i+z^{q+1}F_k\]
 holds for $0<k<q$ and certain elements $F_k\in K[[\x]]$. This implies that the matrix-equation
 \[(z^k\partial_{z^k}(f)-z^{q+1}F_k)_{0<k<q}=\Big(\binom{i}{k}\Big)_{0<i,k<q}\cdot (f_iz^i)_{0<i<q}\]
 holds. Since the matrix $(\binom{i}{k})_{0<i,k<q}$ is invertible, this implies for all indices $i$ with $0<i<q$ that there is an element $G_i\in K[[\x]]$ such that
 \[f_iz^i+z^{q+1}G_i\in(z^k\partial_{z^k}(f),0<k<q)\subseteq\Diff_{R/K}^{q-1}(f).\]
 The last assertion is now obvious.
\end{proof}

\begin{lemma} \label{stefans_lemma}
 Let $R=K[[\x,z]]$ with $\chara(K)=p>0$. Let $f\in R$ have the expansion $f=\sum_{i\geq0}f_iz^i$ where $f_i\in K[[\x]]$. Let $e,e_1\in\N$ be non-negative integers such that $e>e_1$. Set $q=p^e$ and $w=p^{e_1}$. Then there is an element $G\in R$ such that
 \[\sum_{k=0}^{q-w-1}(-1)^k z^k\partial_{z^k}(f)=Gz^{q+1}+f_qz^q+(-1)^{q-w+1}f_{q-w}z^{q-w}+f_0.\]
\end{lemma}
\begin{proof}
 We compute that
 \[\sum_{k=0}^{q-w-1}(-1)^k z^k\partial_{z^k}(f)=\sum_{k=0}^{q-w-1}(-1)^kz^k\sum_{i\geq0}f_i\binom{i}{k}z^{i-k}\]
 \[=\sum_{i\geq0}\sum_{k=0}^{q-w-1}\binom{i}{k}(-1)^kf_iz^i=z^{q+1}G+\sum_{i=0}^q\underbrace{\sum_{k=0}^{q-w-1}\binom{i}{k}(-1)^k}_{=:c_i}f_iz^i\]
 for some element $G\in R$. Let $i$ be an index with $0\leq i\leq q$. Consider first the case $i<q-w$. Then clearly
 \[c_i=\sum_{k=0}^{q-w-1}\binom{i}{k}(-1)^k=\sum_{k=0}^i\binom{i}{k}(-1)^k=(1-1)^i=\delta_{0,i}\]
 where $\delta_{0,i}$ denotes the Kronecker delta.
 
 Now consider the case $q-w\leq i<q$. We can write $i$ as $i=q-w+j$ with $0\leq j<w$. Notice that the $p$-adic expansions $i=\sum_{0\leq n<e}i_n p^n$ of $i$ and $j=\sum_{0\leq n<e_1}j_n p^n$ are related by $i_n=j_n$ for $n<e_1$ and $i_n=p-1$ for $e_1\leq n<e$. Using Proposition \ref{lucas}, we compute
 \[c_i=\sum_{k=0}^{q-w-1}\binom{i}{k}(-1)^k=-\sum_{k=q-w}^i\binom{i}{k}(-1)^k=-\sum_{k=0}^j\binom{i}{q-w+k}(-1)^{(q-w)+k}\]
 \[=-\sum_{k=0}^j\underbrace{\prod_{0\leq n<e_1}\binom{j_n}{k_n}}_{=\binom{j}{k}}\underbrace{\prod_{e_1\leq n<e}\binom{p-1}{p-1}}_{=1}(-1)^{(q-w)+k}\]
 \[=(-1)^{q-w+1}\sum_{k=0}^j(-1)^k\binom{j}{k}=(-1)^{q-w+1}\delta_{0,j}=(-1)^{q-w+1}\delta_{q-w,i}.\]
 Further, $c_q=\sum_{k=0}^{q-w-1}\binom{q}{k}(-1)^k=\binom{q}{0}=1$ by Lemma \ref{c-q} (2). This finishes the proof of the Lemma.
\end{proof}

\begin{proposition} \label{z_in_top_locus_in_monomial_case}
 Let $R=K[[\x,z]]$ with $\x=(x_1,\ldots,x_n)$ and $J\subseteq R$ an ideal of order $c=\ord J$. Assume that the coefficient ideal $J_{-1}=\coeff^c_{(\x,z)}(J)$ is a principal monomial ideal. Thus, it has the form $J_{-1}=(\x^r)$ for a vector $r=(r_1,\ldots,r_n)\in\N^n$. Further, assume that there is an element $f\in J$ which is $\ord$-clean with respect to $J_{-1}$.
 
 Then $z\in\rad(\Diff_{R/K}^{c-1}(J))$.
\end{proposition}
\begin{proof}
 Let $P\in\Spec(R)$ be a prime ideal that contains $\Diff_{R/K}^{c-1}(J)$. Assume that $z\notin P$.
 
 Let $f$ have the expansion $f=\sum_{i\geq0}f_iz^i$ with $f_i\in K[[\x]]$. Set $F=\partial_{z^{c-q}}(f)$. It has the expansion $F=\sum_{i\geq0}F_iz^i$ with $F_i=\binom{i+(c-q)}{c-q}f_i$. By Lemma \ref{c-q} we know that $F_i=f_{i+(c-q)}$ for $0\leq i<q$ and $F_q=\binom{c}{q}f_{c-q}$ where $\binom{c}{q}\neq0$. Notice that $\Diff_{R/K}^{q-1}(F)\subseteq\Diff_{R/K}^{c-1}(J)\subseteq P$ holds by Lemma \ref{composition_of_diff_ops}. Further, we know by Lemma \ref{w_coeff} that
 \[\ord_{(x_j)}F_i=\ord_{(x_j)}f_{i+(c-q)}\geq\frac{c-(i+(c-q))}{c!}r_j=\frac{q-i}{c!}r_j\]
 for indices $0\leq i<q$ and $1\leq j\leq n$. Thus, $\ord F_i\geq\frac{q-i}{c!}|r|$ and $\ord F_i=\frac{q-i}{c!}|r|$ implies that $F_i=\x^{\frac{q-i}{c!}r}\cdot u$ for a unit $u\in K[[\x]]^*$. Since $f$ is $z$-regular of order $c$, we know that $F_q$ is a unit. Consequently, by Lemma \ref{f_i_lemma} there is a unit $u_1\in R^*$ such that 
 \[u_1z^q-F_0\in\Diff_{R/K}^{q-1}(F).\]
 
 First assume that the property $(1)_{\ord}$ holds for $f$ with respect to $J_{-1}$. Notice that this implies that $\chara(K)=p>0$ and $q>1$. There is an index $i$ with $0<i<q$ such that $F_i=\x^{\frac{q-i}{c!}r}u$ for a unit $u\in K[[\x]]^*$. Thus, $F_i^q\mid F_0^{q-i}$. By Lemma \ref{f_i_lemma} we know that there is an element $G_i\in R$ such that $z^{q+1}G_i+F_iz^i\in\Diff_{R/K}^{q-1}(F)$. This implies also that 
 \[z^{q(q+1)}G_i^q+F_i^qz^{qi}\in\Diff_{R/K}^{q-1}(F).\]
 Since $z\notin P$, this implies that 
 \[z^{q(q+1-i)}G_i^q+F_i^q\in P.\]
 Since $F_i^q\mid F_0^{q-i}$ and $u_1z^q+F_0\in P$, there is also an element $h\in R$ such that
 \[z^{q(q-i)}+F_i^qh\in P.\]
 Thus,
 \[h(z^{q(q+1-i)}G_i^q+F_i^q)-(z^{q(q-i)}+F_i^qh)=z^{q(q-i)}\underbrace{(z^qG_i^q-1)}_{\in R^*}\in P.\]
 This is a contradiction to the assumption that $z\notin P$.
 
 Now assume that the property $(2)_{\ord}$ holds for $f$ with respect to $J_{-1}$. We know by Lemma \ref{w_coeff} that there exists an element $g\in J$ with expansion $g=\sum_{i\geq0}g_iz^i$ where $g_i\in K[[\x]]$ and an index $i<c$ such that $g_i=\x^{\frac{c-i}{c!}r}\cdot u$ for a unit $u\in K[[\x]]^*$. Notice further that
 \[\partial_{z^i}(g)=\sum_{k\geq i}\binom{k}{i}g_k z^{k-i}\in P.\]
 Hence, also $\sum_{k\geq i}\binom{k}{i}g_k^q z^{q(k-i)}\in P$. Since $u_1z^q-F_0\in P$, this implies that 
 \[\sum_{k\geq i}\binom{k}{i}g_k^qu_1^{-(k-i)}F_0^{k-i}\in P.\]
 By the property $(2)_{\ord}$ we know that $\ord F_0>\frac{q}{c!}|r|$. Further, we know by Lemma \ref{w_coeff} that $\ord_{(x_j)}g_k\geq\frac{c-k}{c!}r_j$ for $k>i$ and $1\leq j\leq n$. Thus, it follows that for each index $k>i$ there exists an element $G_k\in K[[\x]]$ with $\ord G_k>0$ such that
 \[g_k^qu_1^{-(k-i)}F_0^{k-i}=\x^{q\frac{c-i}{c!}r}G_k.\]
 Hence,
 \[\sum_{k\geq i}\binom{k}{i}g_k^q u_1^{-(k-i)} F_0^{k-i}=\x^{q\frac{c-i}{c!}r}\underbrace{(u^q+\sum_{k>i}\binom{k}{i}G_k)}_{\in R^*}\in P.\]
 Thus, there exists an index $j$ such that $r_j>0$ and $x_j\in P$. Since $\ord_{(x_j)}f_i\geq\frac{c-i}{c!}r_j>0$ for $i<c$, it follows that $\sum_{i\geq c}f_iz^i\in P$. But since $f$ is $z$-regular of order $c$, we know that $\sum_{i\geq c}f_iz^i=z^c \wt u$ for a unit $\wt u\in R^*$. Hence, $z\in P$.
 
 Finally, assume that the properties $\neg(1)_{\ord}$, $\neg(2)_{\ord}$ and $(3)_{\ord}$ hold for $f$ with respect to $J_{-1}$. Again, this implies that $\chara(K)=p>0$ and $q>1$. We know that $F_0=\x^{\frac{q}{c!}r}u$ for a unit $u\in R^*$ and $\frac{q}{c!}r=(\frac{q}{c!}r_1,\ldots,\frac{q}{c!}r_n)\notin q\cdot\N^n$. Assume without loss of generality that $u=1$ and $q\nmid \frac{q}{c!}r_1$. Set $w=q_K(\frac{q}{c!}r_1)$. Then $1\leq w<q$. Further, $x_1^w\partial_{x_1^w}(F_0)=\lambda\cdot F_0$ where $\lambda=\binom{\frac{q}{c!}r_1}{w}\neq0$ by Proposition \ref{lucas}. Set $F_i'=x_1^w\partial_{x_1^w}(F_i)$ for $i>0$. Notice that $\ord F_q'>0$. Set 
 \[F'=x_1^w\partial_{x_1^w}(F)=\sum_{i>0}F_i'z^i+\lambda\cdot F_0\in\Diff_{R/K}^w(F).\]
 Notice that $\Diff_{R/K}^{q-w-1}(F')\subseteq\Diff_{R/K}^{q-1}(F)\subseteq P$ by Lemma \ref{composition_of_diff_ops}. By Lemma \ref{stefans_lemma} we know that
 \[\sum_{k=0}^{q-w-1}(-1)^k z^k\partial_{z^k}(F')=Gz^{q+1}+F_q'z^q+(-1)^{q-w+1}F_{q-w}'z^{q-w}+\lambda\cdot F_0\in P\]
 for an element $G\in R$. Set $G_1=Gz+F_q'$ and notice that $\ord G_1>0$. Then we have that 
 \[G_1z^q+(-1)^{q-w+1}F_{q-w}'z^{q-w}+\lambda\cdot F_0\in P.\]
 Further, we know that 
 \[\ord_{(x_j)}F_{q-w}'\geq \ord_{(x_j)}F_{q-w}\geq \frac{w}{c!}r_j\]
 for $j=1,\ldots,n$ and since $\neg(1)_{\ord}$ holds, also
 \[\ord F_{q-w}'\geq \ord F_{q-w}>\frac{w}{c!}|r|.\]
 Consequently, there is an element $H\in K[[\x]]$ with $\ord H>0$ such that $(F_{q-w}')^q=F_0^wH$. Thus,
 \[G_1^qz^{q^2}+(-1)^{q-w+1}F_0^w H z^{q(q-w)}+\lambda^q\cdot F_0^q\in P.\]
 Since $u_1 z^q-F_0\in P$, this implies that 
 \[z^{q^2}\underbrace{(G_1^q+(-1)^{q-w+1}u_1^{w}H+u_1^{q})}_{\in R^*}\in P.\]
 Consequently, $z\in P$.
 
 Since $P$ was an arbitrary prime ideal of $R$ that contains $\Diff_{R/K}^{c-1}(J)$, this proves that $z\in\rad(\Diff_{R/K}^{c-1}(J))$.
\end{proof}


\begin{corollary} \label{form_of_top_locus_in_monomial_case}
 Let $R=K[[x,y,z]]$ and $J\subseteq R$ an ideal of order $c$. Set $J_{-1}=\coeff_{(x,y,z)}^c(J)$ and assume that $J_{-1}$ has the form
 \[J_{-1}=(x^{r_x}y^{r_y})\]
 for certain integers $r_x,r_y\in\N$. Further, let there be an element $f\in J$ which is $\ord$-clean with respect to $J_{-1}$.
 
 Then \[\rad(\Diff^{c-1}_{R/K}(J))=\begin{cases}
                   (xy,z) & \text{if $r_x,r_y\geq c!$,}\\
                   (x,z) & \text{if $r_x\geq c!$ and $r_y<c!$,}\\
                   (y,z) & \text{if $r_y\geq c!$ and $r_x<c!$,}\\
                   (x,y,z) & \text{if $r_x,r_y<c!$.}
                  \end{cases}.\]
\end{corollary}
\begin{proof}
 By Proposition \ref{z_in_top_locus_in_monomial_case} and Proposition \ref{top_locus_defined_via_coeff_ideal} we know that 
 \[\rad(\Diff^{c-1}_{R/K}(J))=(z)+\rad(\Diff^{c!-1}_{K[[x,y]/K}(x^{r_x}y^{r_y})).\]
 This implies the assertion.
\end{proof}

\subsection{The top locus in the small residual case}

 Consider now the power series ring $R=K[[\x,y,z]]$ and an ideal $J\subseteq R$ of order $\ord J=c$. Assume that the coefficient ideal $J_{-1}=\coeff_{(\x,y,z)}^c(J)$ is of the form
 \[J_{-1}=(y^{mc!})\cdot I\]
 where $m>0$ is a positive integer and $I\subseteq K[\x,y]]$ is an ideal of order $0<\ord I<c!$. We will show in Proposition \ref{form_of_top_locus_in_g_p_case} that the ideal $\rad(\Diff_{R/K}^{c-1}(J))$ is of the form $(y,z)$. In particular, it contains the parameter $z$.
 
 Notice that we do not require the existence of an element $f\in J$ which is $\ord$-clean with respect to $J_{-1}$. This is a moot point though, since the order of the coefficient ideal $J_{-1}$ is not divisible by $c!$. Hence, by Lemma \ref{c!_does_not_divide_m}, every element $f\in J$ that is $z$-regular of order $c$ is already $\ord$-clean with respect to $J_{-1}$.
 
 
 Again, we need to two technical lemmas as a preparation for proving Proposition \ref{form_of_top_locus_in_g_p_case}.

\begin{lemma} \label{coolness_lemma}
 Let $f\in K[x]$ be a polynomial and $n\in\N$ an integer such that $\deg(f)<n$. Then
 \[\sum_{k=0}^n\binom{n}{k}(-1)^kf(k)=0.\]
\end{lemma}
\begin{proof}
 We can write $f$ as
 \[f=\sum_{i<n}a_ix(x-1)\cdots(x-i+1)\]
 with $a_i\in K$. Thus, we can assume without loss of generality that $f$ has the form 
 \[f=x(x-1)\cdots(x-i+1)\]
 for an index $i<n$. Hence, we compute
 \[\sum_{k=0}^n\binom{n}{k}(-1)^kf(k)=\sum_{k=0}^n\binom{n}{k}(-1)^k k(k-1)\cdots(k-i+1)\]
 \[=n(n-1)\cdots(n-i+1)(-1)^i\underbrace{\sum_{k=0}^n\binom{n-i}{k-i}(-1)^{k-i}}_{=(1-1)^{n-i}}=0.\]
\end{proof}

\begin{lemma} \label{s_r_lemma}
 Let $R=K[[\x,y,z]]$ and $f\in R$ an element of order $\ord f=c$. Let $f$ have the expansion $f=\sum_{i=0}^cf_iz^i$ with $f\in K[[\x,y]]$ and assume that there is a positive integer $m>0$ such that all coefficients $f_i$ have a factorization
 \[f_i=y^{(c-i)m}g_i\]
 for certain elements $g_i\in K[[\x,y]]$. Define for all indices $k,j\geq0$ the element
 \[F_{k,j}=\sum_{i=k}^c\binom{i}{k}\partial_{y^j}(g_i)y^{(c-i)m+j}z^i.\]
 Then $F_{k,j}\in \Diff^{c-1}_{R/K}(f)$ for $k+j<c$.
\end{lemma}
\begin{proof}
 By Lemma \ref{basis_for_diff_ops} we can compute for indices $k,j\geq0$ the following:
 \[y^jz^k\partial_{y^j}\partial_{z^k}(f)=\sum_{i=k}^c\binom{i}{k}y^j\partial_{y^j}(f_i)z^i\]
 \[=\sum_{i=k}^c\sum_{l=0}^j\binom{i}{k}\binom{(c-i)m}{j-l}\partial_{y^l}(g_i)y^{(c-i)m+l}z^i=:\wt F_{k,j}.\]
 By construction, $\wt F_{k,j}\in \Diff^{c-1}_{R/K}(f)$ if $k+j<c$.
 
 We prove the Lemma by induction on $j$. For $j=0$ we just note that $F_{k,0}=\wt F_{k,0}$ for all indices $k<c-1$.
 
 So let $k,j$ be indices with $k+j<c$ and assume that we have already proved that $F_{n,l}\in \Diff^{c-1}_{R/K}(f)$ for all indices $n,l$ with $n+l<c$ and $l<j$. Define the coefficients $\mu_{n,k,j-l}$ as
 \[\mu_{n,k,j-l}=\binom{n}{k}\sum_{b=k}^{n}\binom{n-k}{b-k}(-1)^{n-b}\binom{(c-b)m}{j-l}.\]
 Notice that $\mu_{n,k,j-l}=0$ for $j-l<n-k$ by Lemma \ref{coolness_lemma}. Consequently, we know that
 \[\sum_{n=k}^c\mu_{n,k,j-l}F_{n,l}=\sum_{n=k}^{k+j-l}\mu_{n,k,j-l}F_{n,l}\in  \Diff^{c-1}_{R/K}(f)\]
 for all indices $l<j$. But we can compute that
 \[\sum_{n=k}^c\mu_{n,k,j-l}F_{n,l}\]                                     
 \[=\sum_{n=k}^c\binom{n}{k}\sum_{b=k}^{n}\binom{n-k}{b-k}(-1)^{n-b}\binom{(c-b)m}{j-l}\sum_{i=n}^c\binom{i}{n}\partial_{y^l}(g_i)y^{(c-i)m+l}z^i.\]
 \[=\sum_{i=k}^c  \lambda_{i,k,j-l}   \partial_{y^l}(g_i)y^{(c-i)m+l}z^i\]
 where
 \[\lambda_{i,k,j-l}=\sum_{b=k}^i\sum_{n=b}^i\underbrace{\binom{i}{n}\binom{n}{k}\binom{n-k}{b-k}}_{=\binom{i}{b}\binom{b}{k}\binom{i-b}{n-b}}(-1)^{n-b}\binom{(c-b)m}{j-l}\]
 \[=\sum_{b=k}^i\binom{i}{b}\binom{b}{k}\binom{(c-b)m}{j-l}\underbrace{\sum_{n=b}^i\binom{i-b}{n-b}(-1)^{n-b}}_{=\delta_{b,i}}\]
 \[=\binom{i}{k}\binom{(c-i)m}{j-l}.\]
 Hence, we have shown that
 \[F_{k,j}=\wt F_{k,j}-\sum_{l<j}\sum_{n=k}^c\mu_{n,k,j-l}F_{n,l}\in \Diff^{c-1}_{R/K}(f).\]
\end{proof}

\begin{proposition} \label{form_of_top_locus_in_g_p_case}
 Let $R=K[[\x,y,z]]$ with $\x=(x_1,\ldots,x_n)$ and $J\subseteq R$ an ideal of order $\ord J=c$. Set $J_{-1}=\coeff_{(\x,y,z)}^c(J)$.
 
 Assume that $J_{-1}$ has the form 
 \[J_{-1}=(y^{mc!})\cdot I\]
 for a positive integer $m>0$ and an ideal $I$ of order $0<\ord I<c!$. Then
 \[\rad(\Diff_{R/K}^{c-1}(J))=(y,z).\]
\end{proposition}
\begin{proof}
 Assume that $z\in\rad(\Diff_{R/K}^{c-1}(J))$ holds. Then we know by Proposition \ref{top_locus_defined_via_coeff_ideal} that
 \[\rad(\Diff_{R/K}^{c-1}(J))=(z)+\rad(R\cdot\Diff_{K[[\x,y]]}^{c!-1}(J_{-1}).\]
 But it is easy to see that $\rad(\Diff_{K[[\x,y]]}^{c!-1}(J_{-1})=(y)$.
 
 So assume that $z\notin\rad(\Diff_{R/K}^{c-1}(J))$. Let $P\subseteq R$ be a prime ideal that contains $\Diff_{R/K}^{c-1}(J)$ and fulfills $z\notin P$.
 
 Let each element $f\in J$ have the expansion $f=\sum_{i\geq0}f_iz^i$ with $f_i\in K[[\x,y]]$. By Lemma \ref{w_coeff} we know that for all indices $i\geq0$ we can write $f_i=y^{(c-i)m}g_i$ for elements $g_i\in K[[\x,y]]$. Further, we know that $\ord g_i>0$ for all indices $i<c$ since $\ord I>0$.

 Let $f\in J$ have order $\ord f=c$. By Lemma \ref{coeff_ideal_and_directrix} we know that $\Dir(J)=(\ol z)$. Thus, we know by Lemma \ref{z_regularity_from_directrix} that $f$ is $z$-regular of order $c$. By the Weierstrass preparation theorem, we may even assume that $f=z^c+\sum_{i<c}f_iz^i$ holds.
 
 If $y\in P$, then also $z\in P$ since we can write $f=z^c+y^mG$ for some element $G\in R$. Hence, we assume from now on that $y\notin P$ holds.
 
 Set $q=q_K(c)$. Set $F=\partial_{z^{c-q}}(f)$. It has the expansion $F=\sum_{i=0}^qF_iz^i$ with $F_i=\binom{i+(c-q)}{c-q}f_i$. By Lemma \ref{c-q} we know that $F_i=f_{i+(c-q)}$ for $0\leq i<q$ and $F_q=\binom{c}{q}$ where $\binom{c}{q}\neq0$. Notice that $\Diff_{R/K}^{q-1}(F)\subseteq\Diff_{R/K}^{c-1}(J)\subseteq P$ holds by Lemma \ref{composition_of_diff_ops}.
 
 It follows from Lemma \ref{f_i_lemma} that $F_0+\binom{c}{q}z^q+G_1z^{q+1}\in P$ for some element $G_1\in R$. Thus, there exists an element $G\in R$ with $\ord G>0$ such that
 \[z^q-y^{qm}G\in P.\]
 
 Since $\ord I<c!$, we may assume by Lemma \ref{w_coeff} that there is an index $k<c$ such that $\ord g_k<c-k$. Set $j=\ord g_k$. Thus, $k+j<c$. After a change of coordinates $x_i\mapsto x_i+t_iy$ with $t_i\in K$, we may assume that $\partial_{y^j}(g_k)\in R^*$. By Lemma \ref{s_r_lemma} we know that
 \[\sum_{i=k}^c\binom{i}{k}\partial_{y^j}(g_i)y^{(c-i)m+j}z^i\in P.\]
 Since $z^q-y^{qm}G\in P$, we also know that
 \[\sum_{i=k}^c\binom{i}{k}\partial_{y^j}(g_i)^qy^{q(cm+j)}G^i\in P.\]
 Since $y\notin P$, this implies that
 \[\sum_{i=k}^c\binom{i}{k}\partial_{y^j}(g_i)^qG^i\in P.\]
 But since $\partial_{y^j}(g_k)^q$ is a unit, we know that
 \[=G^k\cdot \underbrace{\sum_{i=k}^c\binom{i}{k}\partial_{y^j}(g_i)^qG^{i-k}}_{\in R^*}\in P.\]
 This implies that $G\in P$ and subsequently, $z\in P$. This contradicts our assumption.
 \end{proof}

\section{Cleaning on an open neighborhood} \label{section_global_local_expansion}

 This section provides the technical background for Chapter 8, in which we will prove that our resolution invariant is upper semicontinuous and its top locus constitutes a permissible center of blowup.
 
 In the proofs of embedded resolution over fields of characteristic zero, hypersurfaces of maximal contact are used to show that the resolution invariant is upper semicontinuous. As mentioned in Section \ref{section_woah_maximal_contact}, hypersurfaces of maximal contact are known to exist on open neighborhoods along the strata $X_{\geq c}=\{a\in X:\ord_aX\geq c\}$ defined by the order function.
 
 On the other hand, the cleaning techniques we devised as a replacement for maximal contact in Chapter 5 are constructions in the completed local ring $\hOWa$. They generally do not globalize. As shown in Section \ref{section_generic_down}, it may happen that a different hypersurface is needed at each closed point of the top locus to maximize the order of the coefficient ideal.
 
 Since the top locus $\topp(X)$ of a surface $X$ is at most $1$-dimensional, it suffices to show that our resolution invariant $\ivX$ is upper semicontinuous along curves $C$ in $\topp(X)$. In fact, we will show in Proposition \ref{key_proposition_for_usc} that $X$ is in a terminal case at all but finitely points of such a curve $C$. The argument for this result will be sketched in the following example:
 
 \begin{example}
  Let $W=\Spec(K[x,y,z])$ be the $3$-dimensional affine space over a field $K$. Let $X=V(f)$ be defined by a polynomial $f$ of the form
  \[f=z^c+\sum_{i<c}f_i(x,y)z^i\]
  and assume that each coefficient $f_i$ has a factorization $f_i=y^{(c-i)m}\cdot g_i(x,y)$ for some positive integer $m>0$. Let $m$ be maximal in the sense that there exists an index $i<c$ such that $\ord_{(y)}g_i=0$. Let $i$ be maximal with this property.
  
  It is easy to see that the curve $C=V(y,z)$ is contained in the top locus of $X$. Consider for each closed point $a_t=(t,0,0)$ on the curve $C$ the coefficient ideal
  \[J_{-1}(a_t)=\coeff_{(x-t,y,z)}^c(f)\subseteq \wh\OO_{W,a_t}/(z).\]
  Then each ideal $J_{-1}(a_t)$ has a factorization $J_{-1}(a_t)=(y^{mc!})\cdot I_t$ for some ideal $I_t$. Sine $\ord_{(y)}g_i=0$, we know that $\ord_{(x-t,y)}g_i=0$ holds for all but finitely many $t\in K$. Let $C_1$ be a non-empty open subset of $C$ such that $\ord_{(x-t,y)}g_i=0$ holds for all $a_t\in C_1$. Then for all points $a_t\in C_1$, the equalities $I_t=(1)$ and $J_{-1}(a_t)=(y^{mc!})$ hold. Hence, the coefficient ideal $J_{-1}(a_t)$ is monomial at all points of $C_1$.
  
  Recall from Section \ref{section_modifying_the_residual_order} (or Section \ref{section_monom_s_r_case}) that the existence of an element in $J$ which is $\ord$-clean is a required condition for our monomial case.
  
  Set $q=q_K(c)$. If $i>c-q$, then $f$ fulfills as an element of the ring $\wh\OO_{W,a_t}$ for any point $a_t\in C_1$ the property $(1)_{\ord}$ with respect to $J_{-1}(a_t)$. Hence, $f$ is $\ord$-clean at all points of $C_1$ and we know that $X$ is in the monomial case at all points of $C_1$. Similarly, if $i<c-q$, then we know by maximality of $i$ that $f$ fulfills the property $(2)_{\ord}$. Hence, $X$ is again in the monomial case at all points of $C_1$.
  
  Finally, consider the case $i=c-q$. Set $g=\init_{(y)}(g_{c-q})\in K[x]$. If the element $g$ is a $q$-th power, then the element $f$ is not $\ord_{(y)}$-clean with respect to $J_{-1}(a_t)$ at any point $a_t\in C_1$. Hence, we may apply an $\ord_{(y)}$-cleaning step $z=z_1-\binom{c}{q}^{-1}y^mg^{\frac{1}{q}}$. Notice that $z_1\in K[x,y,z]$ and $C=V(y,z_1)$. Hence, the $\ord_{(y)}$-cleaning process can be applied globally along the curve $C_1$.
  
  So let us assume that (possibly after cleaning) the element $f$ is $\ord_{(y)}$-clean. Hence, $g$ has an expansion 
  \[g=\sum_{j\geq0}c_jx^j\]
  and there is an index $b$ with $b\not\equiv0\pmod q$ such that $c_b\neq0$. The expansion of $g$ at a point $a_t$ has the form
  \[g=\sum_{j\geq0}\wt c_j(t)(x-t)^j\]
  with $\wt c_j(t)=\sum_{k\geq0}\binom{k}{j}t^{k-j}c_k$. Let $0\leq \ol b<q$ be such that $b\equiv\ol b\pmod q$. Then for all but finitely many $t$, the coefficient $c_{\ol b}(t)$ does not vanish. Let $C_2\subseteq C_1$ be the non-empty open subset of points $a_t$ such that $c_{\ol b}(t)\neq0$. Let $a_t\in C_2$ be such a point. Apply $\ord$-cleaning $z=z_t+g_t(x,y)$ with respect to $f$ and $J_{-1}(a_t)$. Set 
  \[\wt J_{-1}(a_t)=\coeff_{(x-t,y,z_t)}^c(f).\]
  By Proposition \ref{w_cleaning_improves_w}, two things can happen: Either $\ord \wt J_{-1}(a_t)=\ord J_{-1}(a_t)$ and $f$ is $\ord$-clean with respect to $\wt J_{-1}(a_t)$. This implies that $\wt J_{-1}=(y^{mc!})$ and $X$ is in the monomial case at $a$.
  
  In the other case, $\ord \wt J_{-1}(a_t)>\ord J_{-1}(a_t)$. Then one can deduce from the fact that $c_{\ol b}(t)\neq0$ and $\ol b<q$ that $\ord\wt J_{-1}(a_t)<\ord J_{-1}(a_t)+c!$. Further, since $f$ is $\ord_{(y)}$-clean, we know by Lemma \ref{w_cleaning_preserves_v_cleaning} that $\ord_{(y)}\wt J_{-1}(a_t)=\ord_{(y)}J_{-1}(a_t)=mc!$. Consequently, $\wt J_{-1}(a_t)$ has the form
  \[\wt J_{-1}(a_t)=(y^{mc!})\cdot \wt I_t\]
  for some ideal $\wt I_t$ with $0<\ord \wt I_t<c!$ and $\ord_{(y)}\wt I_t=0$. Hence, $X$ is in the small residual case at $a_t$.
  
  Altogether, we have shown that $X$ is in one of the terminal cases at all but finitely many points of $C=V(y,z)$.
 \end{example}

 
 
 The following Proposition \ref{neighborhood_cleaning} will serve to generalize the ideas of above example to an arbitrary $3$-dimensional ambient variety $W$ and any hypersurface $X\subseteq W$. The idea of this proposition is to use specific differential operators to characterize certain concepts which were only defined in the completed local ring $\hOWa$, such as the order of the coefficient ideal, $z$-regularity of an element $f$ and $\ord_{(y)}$-cleanness. 
 
 We will consider the following setting:
 
 Let $R$ be the coordinate ring of a regular affine variety $W$ over a field $K$. Further, let $x_1,\ldots,x_n,y,z\in R$ be elements with the property that
 \[\Omega_{R/K}=(d{x_1},\ldots,d{x_n},dy,dz).\]
 Let $a\in W$ be the closed point corresponding to the maximal ideal $\mWa$ of $R$. Assume that the elements $x_1,\ldots,x_n,y,z$ are contained in the maximal ideal $\mWa$. Notice that $\hOWa=K[[\x,y,z]]$ where $\x=(x_1,\ldots,x_n)$.
 
 Let $J\subseteq R$ be an ideal. Set $J_a=\hOWa\cdot J$ and $J_{-1}=\coeff_{(\x,y,z)}^c(J_a)$ where $c=\ord J_a$.
 
 Let each element $f\in J$ have the power series expansion $f=\sum_{i\geq0}f_iz^i$ with $f_i\in K[[\x,y]]$.

\begin{proposition} \label{neighborhood_cleaning}
 Let $r\in\N$ be a non-negative integer. The following hold:
 \begin{enumerate}[(1)]
  \item $\ord_{(y)}J_{-1}\geq r$ holds if and only if for all indices $i<c$, $j<\frac{c-i}{c!}r$ the following inclusion holds:
  \[\partial_{y^j}\partial_{z^i}(J)\subseteq (y,z).\]
 \end{enumerate}
 Let for the remaining results $r$ be such that the inequalities $r\geq c!$ and $\ord_{(y)}J_{-1}\geq r$ hold. Let $f\in J$ be an element and $i<c$ an index such that $\frac{c-i}{c!}r\in\N$.
 \begin{enumerate}[(1)] \setcounter{enumi}{1}
  \item $\ord_{(y)}f_i=\frac{c-i}{c!}r$ holds if and only if $\partial_{y^j}\partial_{z^i}(f)\notin(y,z)$ for $j=\frac{c-i}{c!}r$. If these properties hold, then $\ord_{(y)}J_{-1}=r$ by Lemma \ref{w_coeff}.
  \item $\ord f_i=\frac{c-i}{c!}r$ holds if and only if $\partial_{y^j}\partial_{z^i}(f)\notin\mWa$ for $j=\frac{c-i}{c!}r$. If these properties hold, then $J_{-1}=(y^r)$ by Lemma \ref{w_coeff}.
  \item $f$ is $z$-regular of order $c$ with respect to the parameters $(\x,y,z)$ as an element of the ring $\hOWa$ if and only if $\partial_{z^c}(f)\notin\mWa$.
  \item If $\partial_{z^c}(f)=u\in R^*$ is a unit, then $\partial_{z^c}(u^{-1}\cdot f)=1+F$ for some element $F\in (y,z)$.
  \end{enumerate}
  Assume for the remaining statements that $r=\ord_{(y)}J_{-1}\in c!\cdot\N$ and set $m=\frac{r}{c!}$. Also assume that $\partial_{z^c}(f)$ is a unit of the form $\partial_{z^c}(f)=1+F$ for some element $F\in (y,z)$. Then the following hold:
 \begin{enumerate}[(1)]\setcounter{enumi}{5}
  \item There is an element $g\in R$ with $g\in(y,z)$ such that the coordinate change $z=\wt z+g$ has the following property: Set $\wt J_{-1}=\coeff_{(\x,y,\wt z)}^c(\wh J)$. Then either $\ord_{(y)}\wt J_{-1}>\ord_{(y)}J_{-1}$ or $\ord_{(y)}\wt J_{-1}=\ord_{(y)}J_{-1}$ and $f$ is $\ord_{(y)}$-clean with respect to $\wt J_{-1}$.
  \item If $\neg(1)_{\ord_{(y)}}$ and $\neg(2)_{\ord_{(y)}}$ hold, then $f$ is $\ord_{(y)}$-clean with respect to $J_{-1}$ if and only if there is a multi-index $\alpha\in\N^n$ with $0<|\alpha|<q$ such that 
 \[\partial_{\x^\alpha}\partial_{y^{qm}}\partial_{z^{c-q}}(f)\notin(y,z).\]
 \end{enumerate}
 For the last statement, assume that there exists a multi-index $\alpha\in\N^n$ with $0<|\alpha|<q$ such that $\partial_{\x^\alpha}\partial_{y^{qm}}\partial_{z^{c-q}}(f)\notin\mWa$. Also, assume that for all indices $c-q<k<c$ either $\frac{c-k}{c!}r\notin\N$ or $\partial_{y^j}\partial_{z^k}(f)\in(y,z)$ for $j=\frac{c-k}{c!}r$.
 \begin{enumerate}[(1)]\setcounter{enumi}{7}
   \item Consider an $\ord$-cleaning step $z=\wt z+g$ with respect to $f$ and $J_{-1}$. Set $\wt J_{-1}=\coeff_{(\x,y,\wt z)}^c(J)$. If $\ord\wt J_{-1}>\ord J_{-1}$, then $\wt J_{-1}$ has the form
   \[\wt J_{-1}=(y^{mc!})\cdot I\]
   for an ideal $I$ that fulfills $\ord_{(y)}I=0$ and $0<\ord I<c!$.
 \end{enumerate}

\end{proposition}
\begin{proof}
 We begin the proof by noting that
 \[\dijf=\sum_{k\geq i}\binom{k}{i}\partial_{y^j}(f_k)z^{k-i}.\]
 In particular, $\dijf\in (y,z)$ holds if and only if $\ord_{(y)}\partial_{y^j}(f_i)>0$.
 
 By induction on $j$, we conclude that $\ord_{(y)}f_i\geq j$ holds if and only if $\partial_{y^l}\partial_{z^i}(f)\in(y,z)$ holds for all indices $l<j$.
 
 (1): The assertion follows from Lemma \ref{w_coeff}.
 
 (2): This is immediate from the above.
 
 (3): We already know that $f_i=y^{\frac{c-i}{c!}r}g_i$ for some element $g_i\in K[[\x,y]]$. By Proposition \ref{basis_for_diff_ops} (3) we know that
 \[\partial_{y^j}(f_i)=g_i+yG_i\]
 for some element $G_i\in K[[\x,y]]$. Thus, $\dijf\notin\mWa$ is equivalent to $\ord g_i=0$. 
 
 (4): This is clear since $\partial_{z^c}(f)\notin\mWa$ is equivalent to $\ord f_c=0$ by the above.
 
 (5): By Proposition \ref{basis_for_diff_ops} (3) we can compute
 \[\partial_{z^c}(u^{-1}\cdot f)=\sum_{i=0}^c\partial_{z^{c-i}}(u^{-1})\cdot \partial_{z^i}(f)\]
 \[=\underbrace{u^{-1}\cdot \partial_{z^c}(f)}_{=1}+\sum_{i=0}^{c-1}\partial_{z^{c-i}}(u^{-1})\cdot \underbrace{\partial_{z^i}(f)}_{\in (y,z)}.\]
 
 (6): Assume that $f$ is not already $\ord_{(y)}$-clean with respect to $J_{-1}$. Thus, there exists an element $G\in K[[\x,y]]$ such that
 \[\init_{(y)}(f_{c-q})=\init_{(y)}(f_c)\cdot G^q.\]
 Since we can write $f_{c-q}=y^{qm}g_{c-q}$ and $G=y^mH$ for certain elements $g_{c-q},H\in K[[\x,y]]$, this can also be written as 
 \[\init_{(y)}(g_{c-q})=\init_{(y)}(f_c)\cdot H^q.\]
 Using differential operators, we can express this as 
 \[\partial_{y^{qm}}\partial_{z^{c-q}}(f)-\partial_{z^c}(f)\cdot H^q\in (y,z).\]
 Consider the ring $R'=R/(y,z)$ and denote for an element $x\in R$ its residue in $R'$ by $\ol x$. Set 
 \[h=(\partial_{z^c}(f))^{-1}\cdot \partial_{y^{qm}}\partial_{z^{c-q}}(f)\in R.\]
 Then we know that $\ol H^{q}-\ol h=0$. Consequently, $\ol H^q\in R'$. By Lemma 1.3.1.4 in \cite{Kawanoue_IF_1}, also $\ol H\in R'$ holds.
 
 Thus, there is an element $Q\in (y,z)$ such that $H+Q\in R$. Set $g=-\binom{c}{q}^{-1}y^m(H+Q)$. Then $g=-\binom{c}{q}^{-1}G+y^{m+1}H_1+\wt zy^mH_2$ for certain elements $H_1\in K[[\x,y]]$ and $H_2\in\hOWa$. Since the coefficient ideal $\wt J_{-1}$ is stable under multiplication of $\wt z$ with units by Proposition \ref{coeff_ideal_well_def}, we may assume without loss of generality that $H_2=0$. The assertion now follows from Proposition \ref{w_cleaning_improves_w}, Lemma \ref{m_under_coord_changes} and Lemma \ref{stabilizing_coordinate_changes}.
 
 (7): By what we have just shown, $f$ is $\ord_{(y)}$-clean with respect to $J_{-1}$ if and only if there is no element $\ol H\in R'=R/(y,z)$ such that $\ol H^q=\ol h$ where $h$ was defined as
 \[h=u^{-1}\cdot \partial_{y^{qm}}\partial_{z^{c-q}}(f)\]
 where $u=\partial_{z^c}(f)=1+F$ with $F\in (y,z)$.
 
 Since $\Omega_{R'/K}=(d \ol x_1,\ldots,d \ol x_n)$, we know by Proposition \ref{basis_for_diff_ops} and Proposition \ref{p_th_powers_via_diff_ops} that this is equivalent to
 \[(\ol h)\neq \Diff^{q-1}_{R'/K}(\ol h)=(\ol{\partial_{\x^\alpha}(h)}:|\alpha|<q).\]
 This is further equivalent to the fact that there exists a multi-index $\alpha\in\N^n$ with $0<|\alpha|<q$ such that $\ol{\partial_{\x^\alpha}(h)}\neq0$. This means that $\partial_{\x^\alpha}(h)\notin (y,z)$. By Proposition \ref{basis_for_diff_ops} (3) we can compute
 \[\partial_{\x^\alpha}(h)=u^{-1}\cdot\partial_{\x^\alpha}\partial_{y^{qm}}\partial_{z^{c-q}}(f)+\sum_{\beta<\alpha}\underbrace{\partial_{\x^{\alpha-\beta}}(u^{-1})}_{=\partial_{\x^{\alpha-\beta}}(\wt F)\in (y,z)}\partial_{\x^\beta}\partial_{y^{qm}}\partial_{z^{c-q}}(f)\]
 where $u^{-1}=1+\wt F$ with $\wt F\in (y,z)$.
 This proves the assertion.
 
 (8): By Lemma \ref{w_cleaning_preserves_v_cleaning} we know that $\ord_{(y)}\wt J_1\geq r$. Thus, we can write $\wt J_1=(y^r)\cdot I$ for some ideal $I$. Since $f$ is $\ord_{(y)}$-clean with respect to $J_{-1}$ by assertion (7), we know by Proposition \ref{w_cleaning_maximizes_w} that $\ord_{(y)}I=0$. Since $\ord\wt J_1>\ord J_1$, we also know that $\ord I>0$. It remains to verify that $\ord I<c!$.
 
 We can write $f_i=y^{(c-i)m}g_i$ for all indices $i<c$ and certain elements $g_i\in K[[\x,y]]$. From what we have already shown, we can conclude that $\ord_{(y)}g_i>0$ for $c-q<i<c$ and $\ord \partial_{\x^\alpha}(g_{c-q})=0$. Further, $f_c=1+H$ for some element $H\in K[[\x,y]]$ with $\ord_{(y)}H>0$. Also we can write $g=y^mh$ for some element $h\in K[[\x,y]]$.
 
 Now consider the expansion $f=\sum_{i\geq0}\wt f_i\wt z^i$ with $\wt f_i\in K[[\x,y]]$. By Lemma \ref{coordinate_change_g_z} we can compute
 \[\wt f_{c-q}=\sum_{i\geq c-q}\binom{i}{c-q}f_ig^{i-(c-q)}\]
 \[=y^{qm}\Big(g_{c-q}+\binom{c}{c-q}h^q+\underbrace{\sum_{c-q<i<c}\binom{i}{c-q}g_ih^{i-(c-q)}+\binom{c}{c-q}Hh^q+\sum_{i>c}\binom{i}{c-q}f_i y^{(i-c)m}h^{i-(c-q)}}_{=y\wt H}\Big)\]
 for some element $\wt H\in K[[\x,y]]$. By Lemma \ref{diffops_on_p_th-powers}, this implies that
 \[\partial_{\x^\alpha}(\wt f_{c-q})=y^{mq}(\partial_{\x^\alpha}(g_{c-q})+y\partial_{\x^\alpha}(\wt H))=y^{mq}\wt u\]
 for a unit $\wt u\in K[[\x,y]]^*$. Thus, $\ord \wt f_{c-q}\leq qm+|\alpha|<q(m+1)$. By Lemma \ref{w_coeff} this implies that $\ord\wt J_1<(n+1)c!$. Hence, $\ord I<c!$.
\end{proof}

\section[Coefficient ideal and directrix for $\tau>1$]{Coefficient ideal and directrix for $\tau\geq2$} \label{section_tau_geq_2} 

Consider the power series ring $R=K[[\x,z]]$ and let $J\subseteq R$ be an ideal of order $\ord J=c$. Let $J_{-1}$ be the coefficient ideal $J_{-1}=\coeff_{(\x,z)}^c(J)$. By Lemma \ref{coeff_ideal_and_directrix} we know that the order of $J_{-1}$ is related to the directrix of $J$ in the way that $\ord J_{-1}>c!$ holds if and only if the directrix of $J$ has the form $\Dir(J)=(\ol z)$. Clearly, this implies that $\tau(J)=\dim_K\Dir(J)=1$. This section contains several technical lemmas that concern themselves with the case $\tau(J)\geq2$.

The first of them, Lemma \ref{J_2_tau=2}, can be seen as an analogue of Lemma \ref{coeff_ideal_and_directrix} that characterizes the case $\tau(J)=2$ in terms of the order of the second coefficient ideal. 

\begin{lemma} \label{J_2_tau=2}
 Let $R=K[[\x,y,z]]$ and let $J\subseteq R$ be an ideal of order $\ord J=c$.
 
 Assume that $\tau(J)\geq 2$. Set 
 \[J_z=\coeff^c_{(\x,y,z)}(J),\ J_y=\coeff^c_{(\x,z,y)}(J) .\]
 Notice that $\ord J_z=\ord J_y=c!$ by Lemma \ref{coeff_ideal_and_directrix}. Further, set  
 \[J_{z,y}=\coeff^{c!}_{(\x,y)}(J_z),\ J_{y,z}=\coeff^{c!}_{(\x,z)}(J_y).\]
 
 Then the following hold:
 \begin{enumerate}[(1)]
  \item $\ord J_{z,y}=\ord J_{y,z}$.
  \item $\ord J_{z,y}\geq c!!$.
  \item $\ord J_{z,y}>c!!\iff \tau(J)=2$ and $\Dir(J)=(\ol z,\ol y)$.
 \end{enumerate}
\end{lemma}
\begin{proof}
 Let each element $f\in J$ have an expansion $f=\sum_{i,j\geq0}f_{i,j}y^jz^i$ with $f_{i,j}\in K[[\x]]$.
 
 (1): We can compute by Lemma \ref{ord_J_2} that
 \[\ord J_{z,y}=\min_{f\in J}\min_{\substack{i<c\\j<c-i}}\frac{c!!}{c!-\frac{c!}{c-i}j}\frac{c!}{c-i}\ord f_{i,j}\]
 \[=\min_{f\in J}\min_{i+j<c}\frac{c!!}{c-i-j}\ord f_{i,j}\]
 \[=\min_{f\in J}\min_{\substack{j<c\\i<c-j}}\frac{c!!}{c!-\frac{c!}{c-j}i}\frac{c!}{c-j}\ord f_{i,j}=\ord J_{y,z}.\]
 
 (2): This is immediate from Lemma \ref{ogeqc!}.
 
 (3): By above formula, $\ord J_{z,y}>c!!$ is equivalent to the fact that for all elements $f\in J$ and all indices $i,j$ with $i+j<c$ the inequality $\ord f_{i,j}>c-i-j$ holds. Notice that this is equivalent to the inequality $\ord f_{i,j}y^jz^i>c$. Hence, $\ord J_{y,z}>c!!$ is equivalent to the fact that for each element $f\in J$ with $\ord f=c$ the inclusion $\init(f)\in (y^{c-i}z^i,0\leq i\leq c)$ holds. Since $\tau(J)\geq2$ by assumption, this is equivalent to $\Dir(J)=(\ol z,\ol y)$. 
\end{proof}

Although $\ord J_{-1}=c!$ is equivalent to $\Dir(J)\neq(\ol z)$, it generally does not imply that $\tau(J)>1$ holds. But we will show in the following lemma that this implication holds under the assumption that there is an element $f\in J$ which is $\ord$-clean with respect to $J_{-1}$. (Hence, the order of $J_{-1}$ is maximal by Proposition \ref{w_cleaning_maximizes_w}.) 

\begin{lemma} \label{cleanness_and_tau}
 Let $R=K[[\x,z]]$ and $J\subseteq R$ an ideal of order $\ord J=c$. Set $J_{-1}=\coeff_{(\x,z)}^c(J)$.
 
 If $\ord J_{-1}=c!$ and there is an element $f\in J$ which is $\ord$-clean with respect to $J_{-1}$, then $\tau(J)\geq2$.
\end{lemma}
\begin{proof}
 Assume  that $\tau(J)=1$. Then by Lemma \ref{directrix_from_z_regularity} there is an element $g\in K[[\x]]$ with $\ord g\geq1$ such that $\Dir(J)=(\ol{z_1})$ for $z_1=z+g$. By Lemma \ref{coeff_ideal_and_directrix} we know that $\ord \coeff_{(\x,z_1)}^c(J)>c!$. But this is a contradiction to Proposition \ref{w_cleaning_maximizes_w}. Thus, we know that $\tau(J)\geq2$ holds.
\end{proof}

The next lemma is specific to the surface case. It will allow us to explicitly compute $\tau(J)$ in the case that $\ord J_{-1}=c!$ and $J_{-1}$ is a principal monomial ideal.

\begin{lemma} \label{another_tau=2_lemma}
 Let $R=K[[x,y,z]]$ and $J\subseteq R$ an ideal of order $c=\ord J$. Set $J_{-1}=\coeff_{(x,y,z)}^c(J)$. Assume that $J_{-1}$ is a principal monomial ideal of order $\ord J_{-1}=c!$. Further, let $f\in J$ be an element that is $\ord$-clean with respect to $J_{-1}$.
 
 \begin{enumerate}[(1)]
  \item If $J_{-1}=(y^{c!})$, then $\tau(J)=2$ and $\Dir(J)=(\ol y,\ol z)$.
  \item If $J_{-1}=(x^{r_x}y^{r_y})$ with $0<r_x,r_y<c!$, then $\tau(J)=3$.
 \end{enumerate}
\end{lemma}
\begin{proof}
 By Lemma \ref{cleanness_and_tau} we know that $\tau(J)\geq2$ holds.

 (1): By Corollary \ref{form_of_top_locus_in_monomial_case} we know that $\rad(\Diff_{R/K}^{c-1}(J))=(y,z)$. Hence, $\Dir(J)\subseteq(\ol y,\ol z)$ by Proposition \ref{ord_via_diff_ops} and \ref{directrix_and_top_locus}. Thus, $\tau(J)=2$ and $\Dir(J)=(\ol y,\ol z)$. 
 
 (2): Assume that $\tau(J)=2$. We know by Lemma \ref{directrix_from_z_regularity} that $\Dir(J)\neq(\ol x,\ol y)$. Without loss of generality we can assume that there are constants $s,t\in K$ such that $\Dir(I_3)=(\ol{z+sx},\ol{y+tx})$. Set $\wt J_{-1}=\coeff_{(x,y+tx,z+tx)}^c(J)=\coeff_{(x,y,z+tx)}^c(J)$. By Lemma \ref{m_under_coord_changes} we know that $\ord_{(x)}\wt J_{-1}=\ord_{(x)}J_{-1}=r_x$. Since $0<r_x<c!$, this implies that $\tau(\wt J_{-1})=2$. Set $\wt J_{-2}=\coeff_{(x,y+tx)}^{c!}(\wt J_{-1})$. By Lemma \ref{coeff_ideal_and_directrix} we know that $\ord\wt J_{-2}=c!!$. But this contradicts Lemma \ref{J_2_tau=2} (3).
\end{proof}

\section{Formal flags}  \label{section_flags_technical_stuff}

 As discussed in Section \ref{section_modifying_the_residual_order}, invariants associated to formal flags are an essential building block of our resolution invariant $\ivX$. In this section we will prove several simple technical results on formal flags in a $3$-dimensional ambient space.
 
 We will consider the following setting for this section:
 
 Let $R=K[[x,y,z]]$ be the power series ring in $3$ variables over an algebraically closed field $K$. A \emph{formal flag} $\F$ consists of a regular curve $\F_1$ and a regular hypersurface $\F_2$ in $\Spec(R)$ that fulfill $\F_1\subseteq\F_2$.
 
 Let $E\subseteq\Spec(R)$ be a simple normal crossings divisor. A formal flag $\F$ is said to be \emph{compatible} with $E$ if $\F_2\not\subseteq E$ and the union $\F_2\cup E$ has simple normal crossings.

\begin{lemma} \label{form_of_associated_components}
Let $\F$ be a formal flag of the form $\F_2=V(z)$, $\F_1=V(y,z)$ that is compatible with a simple normal crossings divisor $E\subseteq\Spec(R)$. Assume that the union $\F_1\cup(\F_2\cap E)$ is not simple normal crossings.

Let $D$ be a component of $E$. Denote by $n_{\F,D}$ the intersection multiplicity of the formal curves $\F_1$ and $\F_2\cap D$. The following hold: 
 \begin{enumerate}[(1)]
  \item $n_{\F,D}$ is finite.
  \item If $n_{\F,D}=1$, then $D=V(x+g(y,z))$ for some element $g\in K[[y,z]]$ with $\ord g\geq 1$.
  \item If $n_{\F,D}>1$, then $D=V(y+Q(x)+zG)$ for elements $Q\in K[[x]]$ of order $\ord Q=n_{\F,D}$ and $G\in R$.
  \item There is at most one component $D$ of $E$ for which $n_{\F,D}>1$ holds.
  \item If $n_{\F,D}=1$ holds for all components $D$ of $E$, then $E$ has exactly two components.
  \end{enumerate}
\end{lemma}
\begin{proof}
 (1),(2),(3): Let $h\in R$ be an element such that $D=V(h)$. Let $h$ have the form 
 \[h=ax+by+cz+Q_1\]
 where $a,b,c\in K$ and $Q_1\in R$ with $\ord Q_1\geq2$. Since $\F$ is compatible with $E$, it is easy to see that $(a,b)\neq(0,0)$. 
 
 If $a\neq0$, then we can assume by the Weierstrass preparation theorem that $h=x+g(y,z)$ with $g\in K[[y,z]]$. Then 
 \[n_{\F,D}=\dim_K(K[[x,y,z]]/(z,y,x+g(y,z)))=1.\]
 
 If $a=0$, then we can assume by the Weierstrass preparation theorem that $h=y+Q(x)+zG$ for elements $Q\in K[[x]]$ with $\ord Q\geq2$ and $G\in R$. If $Q=0$, then $\F_1=\F_2\cap D$. Consequently, $\F_1\cup(\F_2\cap E)=\F_2\cap E$ has simple normal crossings, which contradicts our assumption. So we know that $Q\neq0$ and $\ord Q<\infty$. Consequently, 
 
 \[n_{\F,D}=\dim_K([[x,y,z]]/(z,y,y+Q(x)+zG))=\dim_K(K[[x]]/Q(x))=\ord Q.\]
 
 (4): Now assume that there were two distinct components $D_1$, $D_2$ of $E$ such that $n_{\F,D_1}>1$ and $n_{\F,D_2}>1$. Then we can apply assertion (3) to see that the union $D_1\cup D_2$ is not simple normal crossings. This contradicts the assumption that $E$ is a simple normal crossings divisor.
 
 (5): Since $\F$ is compatible with $E$, we know that $D$ has at most two components. Now assume that $E$ has only one component. Then $E=V(x+g)$ for an element $g\in K[[y,z]]$ by assertion (2). Consequently, $\F_1\cup(\F_2\cap E)$ has simple normal crossings. This contradicts our assumption.
\end{proof}

\begin{lemma} \label{curve_determined_by_component}
 Let $\F$ be a formal flag of the form $\F_2=V(z)$, $\F_1=V(y,z)$ that is compatible with $E$.
 
 Let $D$ be a component of $E$ for which $\F_1\subseteq D$ holds. Then $\F_1=\F_2\cap D$.
\end{lemma}
\begin{proof}
 We know that $D=V(h)$ for an element $h\in R$ of the form $h=ay+bz+Q_1$ with $a,b\in K$ and $Q_1\in (y,z)$ with $\ord Q_1\geq2$. Since $\F_2\cup D$ has simple normal crossings, we know that $a\neq0$. By the Weierstrass preparation theorem we may assume that $h=y+zG$ for an element $G\in R$. Consequently, 
 \[\F_2\cap D=V(z,y+zG)=V(z,y)=\F_1.\]  
\end{proof}

\chapter{The resolution invariant for surfaces in arbitrary characteristic} \label{chapter_invariant}

 In the remaining three chapters of the thesis, the new proof for the embedded resolution of surface singularities in arbitrary characteristic will be presented. In accordance with the philosophy of proving embedded resolution of singularities via an upper semicontinuous invariant, the resolution invariant that will be defined in this chapter constitutes the main ingredient of the proof. A thorough explanation and motivation for the particular form of our resolution invariant was given in Chapter 3. The goal of this chapter is to give a rigorous definition of the resolution invariant and prove some basic properties.
 
 The resolution invariant that we are going to use has the form
 \[\iv_\X(a)=(o,c,d,n,s,r,l)\in\N^7\]
 where $\N^7$ is considered with the lexicographic order. The first two components $(o,c)$ are given by the order function and a modification of the order function which guarantees that the top locus of $\iv_\X$ always has simple normal crossings with the exceptional locus produced by previous blowups. The middle part $(d,n,s)$ is defined as the maximum of the flag invariant $\inv(\F)=(d_\F,n_\F,s_\F)$ over all valid flags $\F\in\FF$. We will show in Section \ref{section_inv_drops_for_d>0} that $(d,n,s)$ decreases under point-blowup whenever $(o,c)$ remains constant and $\X$ is not in a terminal case at $a$ already. If $\X$ is in a terminal case at $a$, we set $(d,n,s)=(0,0,0)$ and consider the combinatorial pair $(r,l)$ instead which will be shown to decrease during combinatorial resolution in Section \ref{section_inv_drops_for_d=0}.
 
 As a first step, we will define in Section \ref{section_resolution_settings} the exact setting which our resolution invariant is defined. We will then define in Section \ref{section_flags} the flag invariant that was introduced in Section \ref{section_modifying_the_residual_order}. After that, we will define in Section \ref{section_monom_s_r_case} the combinatorial pair $(r,l)$ for the terminal cases. In Section \ref{section_maximizing_flags} we will show that there exists a flag that maximizes the flag invariant over all valid flags. Finally, the resolution invariant $\ivX$ will be defined in Section \ref{section_def_of_resolution_invariant}.

\section{Resolution settings} \label{section_resolution_settings}

 In this section we will define resolution settings as triples $(W,X,E)$ consisting of a regular $3$-dimensional ambient variety $W$, a hypersurface $X\subseteq W$ whose singularities we want to resolve and a simple normal crossings divisor $E\subseteq W$ which will consist the exceptional components produced by previous blowups. It is a standard technique for the embedded resolution of singularities to consider the divisor $E$ with additional information. More specifically, we want to keep track of the \emph{age} of each component of $E$ and assign to each component of $E$ a unique \emph{label}. This information will be part of our resolution setting and will be used in the definition of the resolution invariant $\iv_\X$.

\subsection{Definition and transformation under blowups}

\begin{definition}
 Let $W$ be a regular variety. A \emph{labeled simple normal crossings divisor} on $W$ is a simple normal crossings divisor $E$ on $W$ together with two maps
 \[\age,\lab:\Comp(E)\to\N_{>0}.\] 
 where $\Comp(E)$ denotes the set of irreducible components of $E$. 
 
 We define for each positive integer $o>0$ the simple normal crossings divisors 
 \[\Ego=\bigcup_{\substack{D\in\Comp(E)\\ \age(D)=o}} D,\]
 \[\Ebo=\bigcup_{\substack{D\in\Comp(E)\\ \age(D)>o}} D.\]
\end{definition}

 The age of a component of $E$ can be thought of as the era of the resolution process during which it was created. To make this more precise, recall that our ultimate goal is to successively lower the maximal order $\max_{a\in X}\ord_aX$ by blowing up regular centers in the top locus $\topp(X)$. Thus, a way of measuring the age of a component $D$ of $E$ is to define $\age(D)$ as what the maximal order was at the iteration of the resolution process during which $D$ was created. Consequently, if we set $o=\max_{a\in X}\ord_aX$ the maximal order of $X$ in the current iteration of the resolution process, then $\Ebo$ consists of all components that were produced in previous eras, when the maximal order was still bigger than $o$, while $\Ego$ consists of all components that were created in the current era.
 
 The reason for this partition of $E$ into old and new components has to do with the choice of the center. A permissible center of blowup $Z$ for a triple $(W,X,E)$ is required to have simple normal crossings with $E$. While the components of $\Ego$ will automatically have simple normal crossings with the centers that we need to blow up to measure improvement of $X$, we can generally make no such statement about the components of $\Ebo$. This distinction will be reflected in the definition of our resolution invariant. As mentioned before, this is a common technique that is used in many proofs of embedded resolution of singularities and thus, we will not explain it in detail here. A more thorough discussion of this technique can be found in \cite{Ha_BAMS_1} p. 359.
 
 On the other hand, the label $\lab(D)$ is simply used as a unique identifier that enables us to tell all components of $E$ apart from each other. It will be used as a component of the combinatorial pair to measure improvement under blowup during combinatorial resolution.

\begin{definition}
 A \emph{$3$-dimensional resolution setting} is a triple $\X=(W,X,E)$ with the following properties:
 \begin{itemize}
  \item $W$ is a $3$-dimensional regular variety.
  \item $X\subseteq W$ is a hypersurface.
  \item $E$ is a labeled simple normal crossings divisor on $W$ with $X\not\subseteq E$ such that the map $\lab:\Comp(E)\to\N$ is injective.
  \item Let $a\in X$ be a point with order $o=\ord_aX$. Then $\age(D)\geq o$ for all components $D$ of $E$ that contain $a$. Denote by $\Eg$ the union of the components of $\Ego$ that contain $a$.
  \item For each closed point $a\in X$ there exists a regular system of parameters $\x=(x,y,z)$ for $\OWa$ such that $\Eg\subseteq V(xy)$ and there exists an element $f\in\IXa$ which is $z$-regular of order $o=\ord_aX$ with respect to the parameters $\x$.
 \end{itemize}
 A $3$-dimensional resolution setting $\X=(W,X,E)$ is said to be \emph{resolved} at a point $a\in X$ if $X$ is regular at $a$ and $X\cup E$ has simple normal crossings at $a$.
\end{definition}

\begin{remark}
 The last property in the definition of $3$-dimensional resolution settings is the most important one. It has several implications of both algebraic and geometric nature:
 \begin{itemize}
  \item $\Eg$ has at most two components.
  \item The element $f\in\IXa$ which is $z$-regular of order $o$ is essential for us to be able to apply the cleaning techniques that we developed in Chapter 5.
  \item For all regular hypersurfaces $H=V(z+g)$ with $g\in K[[x,y]]$, the union $H\cup\Eg$ has simple normal crossings.
 \end{itemize}
 All of the above will be essential for our proof of resolution of surface singularities.
\end{remark}

 To be able to use the resolution settings we just defined for proving the embedded resolution of surface singularities, we need to make sure of two things: First, the requirements for a resolution setting have to be fulfilled at the start of the resolution process. (Thus, when $E$ is empty.) This will be verified in Lemma \ref{start_is_a_res_setting}. Also, we have to check that the properties of a resolution setting are stable under a blowup with a permissible center. To this end, we will first define what a permissible center of blowup is and how a resolution setting transforms under blowup. We will then prove in Proposition \ref{resolution_setting_under_blowup} that the transform of a resolution setting $\X$ still fulfills all the properties in the previous definition.

\begin{lemma} \label{start_is_a_res_setting}
 Let $W$ be a $3$-dimensional regular variety and $X\subseteq W$ a hypersurface. Then $\X=(W,X,\emptyset)$ is a $3$-dimensional resolution setting.
\end{lemma}
\begin{proof}
 We only have to verify that at each closed point $a\in X$ there exists a regular system of parameters $\x=(x,y,z)$ for $\OWa$ and an element $f\in\IXa$ which is $z$-regular of order $o=\ord_aX$ with respect to $\x$. But this is immediate from Lemma \ref{z_regular_after_generic_coord_change}.
\end{proof}

\begin{definition}
 Let $\X=(W,X,E)$ be a $3$-dimensional resolution setting. A closed subset $Z\subseteq W$ is said to be a \emph{permissible center of blowup} for $\X$ if the following two properties hold:
 \begin{itemize}
  \item $Z$ is a permissible center of blowup for $X$ with respect to the order function.
  \item $Z\cup E$ has simple normal crossings.
 \end{itemize}
 
 Let $\pi:W'\to W$ be the blowup at a permissible center $Z$. Let $X'$ be the strict transform of $X$. Denote the exceptional divisor of the blowup by $\Dnew=\pi^{-1}(Z)$. Set $\age(\Dnew)=\max_{a\in X}\ord_aX$ and
 \[\lab(\Dnew)=\max(\{0\}\cup\{\lab(D):D\in\Comp(E)\})+1.\]
 Set $E'=E\st\cup \Dnew$ where $E\st$ denotes the strict transform of $E$. The labels on $E\st$ are defined via $\age(D)=\age(\pi(D))$ and $\lab(D)=\lab(\pi(D))$ for all components $D$ of $E\st$. 
 
 The triple $\X'=(W',X',E')$ is called the \emph{transform} of $\X$ under the blowup $\pi$.
\end{definition}

\begin{proposition} \label{resolution_setting_under_blowup}
 Let $\X=(W,X,E)$ be a $3$-dimensional resolution setting. Consider the blowup $\pi:W'\to W$ along a permissible center $Z$ for $\X$. Then the transform $\X'=(W',X',E')$ of $\X$ under $\pi$ is again a $3$-dimensional resolution setting.
\end{proposition}
\begin{proof}
 It follows from basic properties of blowups that $W'$ is again a $3$-dimensional regular variety, $X'\subseteq W'$ is a hypersurface and $E'$ is a simple normal crossings divisor on $W'$.
 
 By the definition of the labels on $E'$, it is clear that the map $\lab:\Comp(E')\to\N$ is again injective.
 
 Now let $a'\in X'$ be a closed point and set $a=\pi(a')$. Further, set $o=\ord_aX$ and $o'=\ord_{a'}X'$. By Proposition \ref{ord_under_blowup}, $o'\leq o$. It is clear that
 \[\age(\Dnew)=\max_{a\in X}\ord_aX\geq o\geq o'.\]
 Let $D'$ be a component of $E'$ with $a'\in D'$ and $D'\neq\Dnew$. Then $D'$ is the strict transform of a component $D$ of $E$ that contains $a$. Hence, 
 \[\age(D')=\age(D)\geq o\geq o'.\]
 This proves that $\age(D')\geq o'$ for all components $D'$ of $E'$ that contain $a'$.
 
 We now want to prove the existence of a regular system of parameters $\x'$ for $\OWap$ with the claimed properties. If $a'\notin \Dnew$, the blowup $\pi$ is locally at $a'$ an isomorphism. Thus, the existence of regular parameters $\x'$ with the claimed properties is clear. So we can assume that $a'\in\Dnew$. 
 
 Consider first the case $o'<o$. This implies $\Egp=\emptyset$. The statement then follows from Lemma \ref{z_regular_after_generic_coord_change}.
 
 Now assume that $o'=o$. Let $\x=(x,y,z)$ be a regular system of parameters for $\OWa$ such that $\Eg\subseteq V(xy)$ and let $f\in \IXa$ be an element that is $z$-regular of order $o$ with respect to $\x$. By Lemma \ref{z_regular_under_coord_change} and Lemma \ref{directrix_from_z_regularity} we can assume without loss of generality that the local hypersurface $H=V(z)$ is adjacent to $X$ at $a$. Since $Z$ is a permissible center for blowup, we know that $Z\cup\Eg$ has simple normal crossings at $a$. Hence, we may further assume without loss of generality that the ideal $I_{Z,a}$ is generated by some of the parameters $x,y,z$. By Lemma \ref{directrix_and_top_locus} this implies that  $z\in I_{Z,a}$.
 
 By Lemma \ref{directrix_under_blowup}, $a'$ is not contained in the $z$-chart. Assume without loss of generality that $a'$ is contained in the $x$-chart and let $\x'=(x',y',z')$ be the induced parameters for $\OWap$. It is easy to see that $\Egp\subseteq V(x'y')$ holds. Further, we know by Lemma \ref{z_regular_stable_under_blowup} that there is an element $f'\in I_{X',a'}$ which is $z'$-regular of order $o'=o$ with respect to $\x'$.
\end{proof}

\subsection{The basic invariants $o_\X$ and $c_\X$}

 As mentioned before, we do not know in general whether $\Ebo$ (where $o=\max_{a\in X}\ord_aX$) has simple normal crossings with the centers that we need to blow up to lower $\ord X$. The solution to this problem is to, instead of trying to lower $\ord X$, we try to lower $\ord (X\cup E_{>o})$. Hence, we choose centers which are contained in the top locus of $X\cup\Ebo$. Since such centers are always locally contained in $\Ebo$, they automatically have simple normal crossings with $\Ebo$.
 
 The obvious drawback of this technique is that we have no control over the maximal order of $X\cup\Ebo$ whenever $o$, the maximal order of $X$, decreases under blowup. To amend this, we will consider the pair $(\ord_aX,\ord_a(X\cup\Ebo))$ with the lexicographic order.

 This motivates the definition of the first two components $o$ and $c$ of the resolution invariant $\iv_\X$:

 Let $\X=(W,X,E)$ be a $3$-dimensional resolution setting. We define the maps $o_\X,c_\X:X\to\N$ via
 \[o\Xa=\ord_aX,\]
 \[c\Xa=\ord_a(X\cup\Ebo).\]
 
 Define for each pair $(o,c)\in\N^2$ the set
 \[X_{\geq(o,c)}=\{a\in X:(o_\X(a),c_\X(a))\geq (o,c)\}\subseteq X.\]
 It is immediate from Proposition \ref{order_is_usc} that $X_{\geq(o,c)}$ is closed. In other words, the map $(o_\X,c_\X):X\to\N^2$ is upper semicontinuous.
 
 Let $a\in X$ be a closed point and set $(o,c)=(o\Xa,c\Xa)$. Let
 \[I_3(a)=\hIXa\cdot\wh I_{\Ebo,a}\]
 be the ideal that defines $X\cup\Ebo$ in $\hOWa$. Clearly, $\ord I_3(a)=c$.
 
 Let $I_{\geq(o,c)}(a)=\wh I_{X_{\geq(o,c)},a}\subseteq\hOWa$ be the ideal that defines $X_{\geq(o,c)}$ in $\hOWa$. We know from Proposition \ref{ord_via_diff_ops} that the equality
 \[I_{\geq(o,c)}(a)=\rad(\Diff_{\hOWa/K}^{-1}(I_3(a))).\]
 holds.
 
 If it is clear from the context which point $a$ is considered, we will denote above ideals just by $I_3$ and $\Ioc$. Similarly, we will usually just denote $c=c_\X(a)$.

\begin{lemma} \label{c=1_measures_resolved}
Let $\X=(W,X,E)$ be a $3$-dimensional resolution setting and $a\in X$ a closed point. Set $(o,c)=(o\Xa,c\Xa)$. The following hold:
 \begin{enumerate}[(1)]
  \item $(o,c)\geq (1,1)$.
  \item If $(o,c)=(1,1)$ holds, then $\X$ is resolved at $a$.
 \end{enumerate}
\end{lemma}
\begin{proof}
 (1): This is clear.
 
 (2): Since $\ord_aX=o=1$ and $X$ is a hypersurface, we know by Lemma \ref{order_basic_properties} (3) that $X$ is regular at $a$. Since $o=c=1$, we also know that $a$ is not contained in $E_{>o}$.
 
 Further, there exists a regular system of parameters $\x=(x,y,z)$ for $\OWa$ such that $\Eg\subseteq V(xy)$ and an element $f\in \IXa$ that is $z$-regular of order $1$. Thus, $(x,y,f)$ is a regular system of parameters for $\OWa$. Consequently, $X\cup E$ has simple normal crossings at $a$.
\end{proof}

\begin{lemma} \label{no_z^c}
If $c>1$, then there exists no regular parameter $z\in\hOWa$ such that $I_3=(z^c)$.
\end{lemma}
\begin{proof}
 Assume that $I_3=(z^c)$ where $c>1$. It is clear that this implies that $a$ is not contained in $E_{>o}$. Hence, $I_3=\wh I_{X,a}$.
 
 The local ring $\OO_{X,a}$ is reduced. By a theorem of Chevalley (\cite{Zariski_Samuel} Thm. 32, p. 320) this implies that $\OO_{X,a}$ is analytically unramified, meaning that its completion $\wh\OO_{X,a}$ is reduced. Since $\wh \OO_{X,a}=\hOWa/\wh I_{X,a}=\hOWa/(z^c)$, this is a contradiction.
\end{proof}

\subsection{Apposite parameters}

 By considering the pair $(o_\X,c_\X)$ instead of the order function on $X$, our goal has changed from lowering the order of the ideal $\hIXa$ to lowering the order of the ideal $I_3(a)$. While the existence of an element $f\in\hIXa$ that is $z$-regular of order $o_\X(a)$ was guaranteed by the definition of the resolution setting, the existence of such an element in the ideal $I_3(a)$ has to be proven. This motivates the following definition.

\begin{definition}
 Let $\X=(W,X,E)$ be a $3$-dimensional resolution setting and $a\in X$ a closed point. A regular system of parameters $\x=(x,y,z)$ for $\hOWa$ is said to constitute \emph{apposite parameters} for $\X$ at $a$ if the following two properties hold:
\begin{itemize}
 \item $\Eg\subseteq V(xy)$.
 \item There exists an element $f\in I_3(a)$ which is $z$-regular of order $c=c\Xa$ with respect to $\x$.
\end{itemize}
 If $x,y,z$ are moreover elements of the local ring $\OWa$, we will call them \emph{local apposite parameters}.
\end{definition}

 Apposite parameters play an important role in our proof. The element $f\in I_3(a)$ that is $z$-regular of order $c$ will allow us to apply cleaning to $z$ with respect to coefficient ideals of $I_3(a)$. In particular, we will be able to show in Section \ref{section_maximizing_flags} that for given apposite parameters $\x=(x,y,z)$ there always exists a maximizing flag $\F\in\FF$ with $\F_2=V(z+g)$ for an element $g\in K[[x,y]]$. This in turn will be essential for showing the decrease of $\iv_\X$ under point-blowups in the non-terminal case in Section \ref{section_inv_drops_for_d>0}.

\begin{lemma} \label{good_parameters}
 Let $\X=(W,X,E)$ be a $3$-dimensional resolution setting and $a\in X$ a closed point. There exist local apposite parameters for $\X$ at $a$.
\end{lemma}
\begin{proof}
 Set $o=o\Xa$. Let $\Egold$ be the union of those components of $E_{>o}$ which contain $a$. Since $\Eg\cup\Egold$ has simple normal crossings, there exists a regular system of parameters $\x=(x,y,z)$ for $\OWa$ such that $\Eg\cup\Egold\subseteq V(xyz)$. By assumption on the resolution setting, we can further assume that $\Eg\subseteq V(xy)$ and there exists a parameter $w\in\OWa$ which is $z$-regular and an element $f\in\IXa$ which is $\w$-regular of order $o=o\Xa$ with respect to the parameters $(x,y,w)$. By Lemma \ref{z_regular_under_coord_change}, the element $f$ is also $z$-regular of order $o$ with respect to the parameters $\x=(x,y,z)$.
 
 The statement is trivial if $\Egold=\emptyset$. By Lemma \ref{z_regular_after_generic_coord_change} the statement follows for $\Eg=\emptyset$ . Also, it is clear that the statement holds if $\Egold=V(z)$. So we may assume without loss of generality that $\Eg=V(x)$ and either $\Egold=V(y)$ or $\Egold=V(yz)$. It is easy to see that there is a constant $t\in K^*$ such that $f$ is $z$-regular of order $o$ with respect to the parameters $\x_1=(x,y+tz,z)$. Also, it is then clear that $I_{E_{>o},a}=(g)$ for an element $g\in\OWa$ that is $z$-regular of order $c-o$ with respect to the parameters $\x_1=(x,y+tz,z)$. Hence, $f\cdot g\in I_3(a)$ is an element which is $z$-regular of order $c$ with respect to $\x_1$. Thus, $\x_1$ are apposite parameters for $\X$ at $a$.
\end{proof}

\section{The flag invariant} \label{section_flags}

 We will now rigorously define the flag invariant that was outlined in Section \ref{section_modifying_the_residual_order}. The well-definedness of the various invariants associated to flags follows from the fundamental results we proved in Chapter 4.
 
 For this whole section, let $\X=(W,X,E)$ be a $3$-dimensional resolution setting and $a\in X$ a closed point such that $c=c\Xa>1$.

\subsection{Compatible flags, associated multiplicity $n_\F$ and associated divisor $D_\F$}
 
 A \emph{formal flag} $\F$ at $a$ consists of a regular curve $\F_1$ and a regular hypersurface $\F_2$ in $\Spec(\hOWa)$ such that $\F_1\subseteq\F_2$. Since we only consider formal flags in the following, we will just call them \emph{flags}. 
 
 A flag $\F$ is said to be \emph{compatible with $E$} if the following two properties hold:
 \begin{itemize} 
  \item The union $\F_2\cup \Eg$ has simple normal crossings.
  \item $\F_2\not\subseteq\Eg$.
 \end{itemize}

 Denote by $\FF(a)$ the set
 \[\FF(a)=\{\F\text{ is a formal flag at $a$ that is compatible with $E$}\}.\]
 If it is clear from the context which point $a$ we consider, we will just write $\FF$ instead of $\FF(a)$.
 
 Let $\F\in\FF$ be a flag. We define its \emph{associated multiplicity} $n_\F$ the following way: 
 \begin{itemize}
  \item If $\F_1\cup(\Eg\cap\F_2)$ has simple normal crossings, set $n_\F=0$.
  \item If $\F_1\cup(\Eg\cap\F_2)$ does not have simple normal crossings, set
  \[n_\F=\max\{\mult_a(\F_1,D\cap\F_2):\text{$D$ is a component of $\Eg$}\}\]
  where $\mult_a(.,.)$ denotes the intersection multiplicity of two curves at $a$.
 \end{itemize}
 The number $n_\F$ is finite by Lemma \ref{form_of_associated_components} (1).
 
 For a flag $\F\in\FF$ with $n_\F>1$ we define its \emph{associated component} $D_\F$ as the unique component of $\Eg$ that fulfills
 \[\mult_a(\F_1,D_\F\cap\F_2)=n_\F.\]
 There can be only one such component by Lemma \ref{form_of_associated_components} (4). Further, notice that $n_\F=1$ implies that $\Eg$ has two components by Lemma \ref{form_of_associated_components} (5).
 
 Two flags $\F,\G\in\FF$ are said to be \emph{comparable} if their associated multiplicities and (if applicable) their associated components coincide.

\subsection{Subordinate parameters and the flag invariant $\inv(\F)$}
 
 Let $\F\in\FF$ be a flag. We will now define parameters which are subordinate to $\F$ and the associated invariants. The definition uses a case distinction between $n_\F=0$ and $n_\F>0$.\\

 If $n_\F=0$, we make the following definitions:

 A regular system of parameters $\x=(x,y,z)$ for $\hOWa$ is said to be \emph{subordinate} to $\F$ is $\F_2=V(z)$, $\F_1=V(y,z)$ and $\Eg\subseteq V(xy)$ hold. The existence of subordinate parameters is guaranteed since $\F_1\cup(\F_2\cap\Eg)$ has simple normal crossings. Notice that the ordering of the parameters $(x,y,z)$ is relevant.
 
 Let $\x=(x,y,z)$ be subordinate to $\F$. Define the coefficient ideal $J_{2,\x}$ as
 \[J_{2,\x}=\coeff_{(x,y,z)}^c(I_3).\]
 Let the ideal $J_{2,\x}$ have the factorization
 \[J_{2,\x}=M_{2,\x}\cdot I_{2,\x}\]
 where $M_{2,\x}=(x^{r_x}y^{r_y})$ is a principal monomial ideal and $r_x,r_y\in\N$ are defined as
 \[r_x=\begin{cases}
        \ord_{(x)}J_{2,\x} & \text{if $V(x)\subseteq\Eg$,}\\
        0 & \text{otherwise.}
       \end{cases}
\]
 \[r_y=\begin{cases}
        \ord_{(y)}J_{2,\x} & \text{if $V(y)\subseteq\Eg$,}\\
        0 & \text{otherwise.}
       \end{cases}
\]
 Set
 \[m_\F=\ord M_{2,\x},\]
 \[d_\F=\ord I_{2,\x}.\]
 By Proposition \ref{residual_order_coord_indep}, the invariants $m_\F$ and $d_\F$ are independent of the chosen subordinate parameters $\x$.
 
 Further, define the second coefficient ideal $J_{1,\x}$ as
\[J_{1,\x}=\begin{cases}
             \coeff_{(x,y)}^{d_\F}(I_{2,\x}) & \text{if $d_\F\geq c!$,}\\
             \coeff_{(x,y)}^{d_\F(c!-d_\F)}(I_{2,\x}^{c!-d_\F}+M_{2,\x}^{d_\F}) & \text{if $0<d_\F<c!$,}\\
             0 & \text{if $d_\F=0$.}
            \end{cases}
\]
 and set 
 \[s_\F=\ord J_{1,\x}.\]
 By Proposition \ref{slope_coord_indep}, the invariant $s_\F$ is independent of the chosen subordinate parameters $\x$.\\\

 If $n_\F>0$, we make the following definitions:
 
 A regular system of parameters $\x=(x,y,z)$ for $\hOWa$ is said to be \emph{subordinate} to $\F$ is $\F_2=V(z)$ and $\F_1=V(y,z)$ hold.
 
 Let $\x=(x,y,z)$ be subordinate to $\F$. Define the coefficient ideal $J_{2,\x}$ as
 \[J_{2,\x}=\coeff_{(x,y,z)}^c(I_3).\]
 Further, we define the weighted order functions $\w_{\F,\x}:K[[x,y]]\to\Ni$ and $\y_{\F,\x}:K[[x,y]]\to\Ni^2$ on the parameters $(x,y)$ via 
 \[\w_{\F,\x}(x)=1,\ \w_{\F,\x}(y)=n_\F\text,\]
 \[\y_{\F,\x}(x)=(1,0),\ \y_{\F,\x}(y)=(n_\F,1).\]
 If there is no risk of confusion, we will just denote these weighted order functions by $\w$ and $\y$.
 
 Set 
 \[m_{\F,\x}=\w(J_{2,\x}),\]
 \[d_{\F,\x}=\ord_{(y)}\minit_\w(J_{2,\x}).\]
 By Lemma \ref{double_weighted order function}, this can also be expressed as
 \[(m_{\F,\x},d_{\F,\x})=\y(J_{2,\x}).\]
 Notice that $d_{\F,\x}\leq \frac{1}{n_\F}\cdot m_{\F,\x}$ holds.
 
 Now define the invariants $m_\F$ and $d_\F$ as 
 \[m_\F=\begin{cases}
         m_{\F,\x} & \text{if $m_{\F,\x}\geq n_\F\cdot c!$,} \\
         n_\F\cdot c! & \text{if $m_{\F,\x}< n_\F\cdot c!$}
        \end{cases}
\]
and 
 \[d_\F=\begin{cases}
         d_{\F,\x} & \text{if $d_{\F,\x}\geq c!$,}\\
         d_{\F,\x} & \text{if $0<d_{\F,\x}<c!$ and $c!\nmid m_\F$,}\\
         -1 & \text{if $0<d_{\F,\x}<c!$ and $c!\mid m_\F$,}\\
         -1 & \text{if $d_{\F,\x}=0$.}
        \end{cases}
\]
By Proposition \ref{m_d_coord_indep} these invariants are independent of the chosen subordinate parameters $\x$.

Finally, assign to $s_\F$ the trivial value
\[s_\F=0.\]\\

The \emph{flag invariant} of a flag $\F\in\FF$ is defined as
\[\inv(\F)=(d_\F,n_\F,s_\F).\]

If we want to emphasize the dependence of the ideals $J_{2,\x}$, $I_{2,\x}$, $\ldots$ on the point $a$, we will denote them by $J_{2,\x}(a)$, $I_{2,\x}(a)$, $\ldots$

\subsection{Valid and maximizing flags}

A flag $\F\in\FF$ is said to be \emph{valid} if it fulfills 
\[m_\G\leq m_\F\]
for all flags $\G\in\FF$ that are comparable to $\F$.

A valid flag $\F\in\FF$ is said to be \emph{maximizing} if it fulfills
\[\inv(\G)\leq\inv(\F)\]
for all valid flags $\G\in\FF$.

We will show in Section \ref{section_maximizing_flags} that maximizing flags exist under the condition that $c>1$ holds.

\section{Invariants for the terminal cases} \label{section_monom_s_r_case}

 As mentioned before, the maximum $(d_\F,n_\F,s_\F)$ of the flag invariant $\inv(\F)$ over all valid flags $\F\in\FF$ will only be used as a resolution invariant until we reach a terminal case. If $\X$ is in a terminal case at $a$, we set $(d,n,s)=(0,0,0)$ and consider the combinatorial pair $(r,l)$ instead. This will enable us to measure improvement during combinatorial resolution. 
 
 In this section we will define the combinatorial pair for the monomial case and the small residual case. While the definition for the monomial case is a standard technique for the embedded resolution of singularities, the definition of the combinatorial pair for the small residual case is unique to this thesis. Further, we will show that if $\X$ is both in the monomial case and in the small residual case at a point $a$, the two respective combinatorial pairs coincide.
 
 Let in the following $\X=(W,X,E)$ be a $3$-dimensional resolution setting and $a\in X$ a closed point such that $c=c\Xa>1$.

\subsection{Associated labels $l_x,l_y$}

 Let $\x=(x,y,z)$ be apposite parameters for $\X$ at $a$. We define the \emph{associated labels} $l_x,l_y\in\N$ as
 \[l_x=\begin{cases}
        \lab(V(x)) & \text{if $V(x)\subseteq\Eg$,}\\
        0 & \text{otherwise.}
       \end{cases}
\] \[l_y=\begin{cases}
        \lab(V(y)) & \text{if $V(y)\subseteq\Eg$,}\\
        0 & \text{otherwise.}
       \end{cases}
\]

\begin{lemma} \label{same_curve_same_l_y_lemma}
 Let $\x=(x,y,z)$ and $\x_1=(x_1,y_1,z_1)$ be apposite parameters for $\X$ at $a$. Assume that the identity of ideals $(y,z)=(y_1,z_1)$ holds.
 
 Then $l_y=l_{y_1}$.
\end{lemma}
\begin{proof} Consider the case $l_y>0$. Thus, $V(y_1,z_1)=V(y,z)\subseteq V(y)\subseteq\Eg$. Hence, we know by Lemma \ref{curve_determined_by_component} that $V(y_1,z_1)=V(z_1)\cap V(y)$. Thus, $l_y=l_{y_1}$.
 
 From this it is also clear that $l_y=0$ implies also $l_{y_1}=0$.
\end{proof}

\subsection{The combinatorial pair for the monomial case}

 The \emph{monomial case} is defined as follows:

 \begin{definition}
  We say that $\X$ is in the monomial case at $a$ if there exist apposite parameters $\x=(x,y,z)$ for $\X$ at $a$ such that the following three properties hold:
  \begin{itemize}
   \item The coefficient ideal $J_{2,\x}=\coeff_{\x}^c(I_3)$ is a principal monomial ideal of the form
   \[J_{2,\x}=(x^{r_x}y^{r_y})\]
   for certain numbers $r_x,r_y\in\N$. (Notice that we do not require that $V(x^{r_x}y^{r_y})\subseteq\Eg$ holds.)
   \item There is an element $f\in I_3$ which is $\ord$-clean with respect to $J_{2,\x}$.
   \item If $\Eg=\emptyset$, then either $r_x=0$ or $r_y=0$.
  \end{itemize}
 In this case, we say that the parameters $\x$ are of \emph{monomial type}.
 \end{definition}

 To define the combinatorial pair for the monomial case, we use the following Lemma:

\begin{lemma} \label{monomial_case} 
 Let $\X$ be in the monomial case at $a$ and let $\x=(x,y,z)$ be apposite parameters of monomial type. Let the coefficient ideal $J_{2,\x}$ have the form $J_{2,\x}=(x^{r_x}y^{r_y})$. Then the following statements hold:
 \begin{enumerate}[(1)]
 \item The set $\{(r_x,l_x),(r_y,l_y)\}\subseteq\N^2$ is independent of the choice of the parameters $\x$ of monomial type.
 \item The following equality holds:
  \[I_{\geq(o,c)}(a)=\begin{cases}
                   (xy,z) & \text{if $r_x,r_y\geq c!$,}\\
                   (x,z) & \text{if $r_x\geq c!$ and $r_y<c!$,}\\
                   (y,z) & \text{if $r_y\geq c!$ and $r_x<c!$,}\\
                   (x,y,z) & \text{if $r_x,r_y<c!$.}
                  \end{cases}\]
 \end{enumerate}
\end{lemma}
\begin{proof}
 (1): Let $\x_1=(x_1,y_1,z_1)$ be another apposite system of parameters which is of monomial type. Hence, $J_{2,\x_1}=(x_1^{s_x}y_1^{s_y})$ for certain integers $s_x,s_y\in\N$ and there is an element $f_1\in I_3$ which is $\ord$-clean with respect to $J_{2,\x_1}$.
 
 By Lemma \ref{w_cleaning_preserves_v_cleaning} we may assume without loss of generality that the element $f$ is also $\ord_{(x)}$-clean and $\ord_{(y)}$-clean with respect to $J_{2,\x}$. For the same reason we may assume that $f_1$ is $\ord_{(x_1)}$-clean and $\ord_{(y_1)}$-clean with respect to $J_{2,\x_1}$.
 
 The remainder of the proof will use a case distinction according to the number of components of $\Eg$.
 
 If $\Eg=V(xy)$, we may assume that $x=x_1$ and $y=y_1$. Hence, we may also assume that $z_1=z+g$ for an element $g\in K[[x,y]]$. It then follows from Proposition \ref{w_cleaning_maximizes_w} that $s_x=r_x$ and $s_y=r_y$.
 
 If $\Eg=V(y)$, we may assume that $y=y_1$. If $z_1=z+g$ for an element $g\in K[[x,y]]$, it follows from Proposition \ref{w_cleaning_maximizes_w} that $s_y=r_y$ and $\ord J_{2,\x}=\ord J_{2,\x_1}$. Consequently, also $s_x=r_x$.
 
 Assume now that $\Eg=V(y)$ and $z_1$ is not $z$-regular. By Lemma \ref{ogeqc!} and Lemma \ref{z_regular_blocks_other_parameters} this implies that $\ord J_{2,\x}=\ord J_{2,\x_1}=c!$ and $s_y=r_y=0$. Consequently, $s_x=r_x=c!$.
 
 Finally, assume that $\Eg=\emptyset$. Hence, we may assume that $r_x=s_x=0$. If $z_1=z+g$ for some element $g\in K[[x,y]]$, then $\ord J_{2,\x}=\ord J_{2,\x_1}$ by Proposition \ref{w_cleaning_maximizes_w}. Hence, $s_y=r_y$. On the other hand, if $z_1$ is not $z$-regular, we know by Lemma \ref{ogeqc!} and Lemma \ref{z_regular_blocks_other_parameters} that $\ord J_{2,\x}=\ord J_{2,\x_1}=c!$. Hence, $r_y=s_y=c!$.
 
 (2): This was proven in Corollary \ref{form_of_top_locus_in_monomial_case}.
\end{proof}

 The \emph{combinatorial pair for the monomial case} is defined as follows:
 
 Let $\X$ be in the monomial case at $a$ and let $\x=(x,y,z)$ be apposite parameters of monomial type. If the coefficient ideal $J_{2,\x}$ has the form
 \[J_{2,\x}=(x^{r_x}y^{r_y}),\]
 we define the combinatorial pair for $\X$ at $a$ as
 \[(r\mon,l\mon)=\begin{cases}
                        \max\{(r_x,l_x),(r_y,l_y)\} & \text{if $r_x\geq c!$ or $r_y\geq c!$,}\\
                        (r_x,l_x)+(r_y,l_y) & \text{if $r_x,r_y<c!$.}
                       \end{cases}
\] 
This definition is independent of the choice of $\x$ by Lemma \ref{monomial_case} (1).

\subsection{The combinatorial pair for the small residual case}

The \emph{small residual case} is defined as follows:
 
 \begin{definition}
  We say that $\X$ is in the small residual case at $a$ if there exist apposite parameters $\x=(x,y,z)$ for $\X$ at $a$ such that the following two properties hold:
  \begin{itemize}
   \item The coefficient ideal $J_{2,\x}=\coeff_{\x}^c(I_3)$ has the form
   \[J_{2,\x}=(y^{mc!})\cdot I\]
   for a positive integer $m>0$ and an ideal $I\subseteq K[[x,y]]$ with $\ord_{(y)}I=0$ and $0<\ord I<c!$.
   \item There is an element $f\in I_3$ which is $\ord_{(y)}$-clean with respect to $J_{2,\x}$.
  \end{itemize}
 In this case, we say that the parameters $\x$ are of \emph{small residual type}.
 \end{definition}

 To define the combinatorial pair for the small residual case, we use the following lemma:

\begin{lemma} \label{small_residual_case}
Let $\X$ be in the small residual case at $a$ and let $\x=(x,y,z)$ be apposite parameters of small residual type. Let the coefficient ideal $J_{2,\x}$ have the form $J_{2,\x}=(y^{mc!})\cdot I$ with $\ord_{(y)}I=0$ and $0<\ord I<c!$. Then the following statements hold:
 \begin{enumerate}[(1)]
 \item The pair $(mc!,l_y)\in\N^2$ is independent of the choice of the parameters $\x$ of small residual type.
 \item The following equality holds:
  \[I_{\geq(o,c)}(a)=(y,z).\]
 \end{enumerate}
\end{lemma}
\begin{proof}
 (1): Let $\x_1=(x_1,y_1,z_1)$ be another apposite system of parameters which is of small residual type. Hence, $J_{2,\x_1}=(y_1^{m_1c!})\cdot I_1$ for a positive integer $m_1>0$ and an ideal $I_1$ with $\ord_{(y_1)}I_1=0$ and $0<\ord I_1<c!$.
 
 Since $\ord J_{2,\x}>c!$ and $\ord J_{2,\x_1}>c!$, we may assume by Lemma \ref{coeff_ideal_and_directrix} that $z_1=z+g$ for some element $g\in K[[x,y]]$ with $\ord g\geq2$. Since both $\ord J_{2,\x}$ and $\ord J_{2,\x_1}$ are not divisible by $c!$, it follows from Lemma \ref{m_under_coord_changes} that $\ord J_{2,\x}=\ord J_{2,\x_1}$. Consequently, $m_1=m$. 
 
 The identity $l_y=l_{y_1}$ follows from Lemma \ref{same_curve_same_l_y_lemma} in combination with assertion (2).

 (2): This was proven in Lemma \ref{form_of_top_locus_in_g_p_case}.
\end{proof}

 The \emph{combinatorial pair for the small residual case} is defined as follows:
 
 Let $\X$ be in the small residual case at $a$ and let $\x=(x,y,z)$ be apposite parameters of small residual type. If the coefficient ideal $J_{2,\x}$ has the form
 \[J_{2,\x}=(y^{mc!})\cdot I,\]
 for a positive integer $m>0$ and an ideal $I$ with $\ord_{(y)}I=0$ and $0<\ord I<c!$, then we define the combinatorial pair for $\X$ at $a$ as
 \[(r\gp,l\gp)=(mc!,l_y).\]
This definition is independent of $\x$ by Lemma \ref{small_residual_case} (1).

\subsection{Well-definedness of the combinatorial pair}

 Since it may happen that $\X$ is at a point $a\in X$ both in the monomial case and in the small residual case, we need to verify that the combinatorial pair is still well-defined in this case. This will be shown in the next proposition.

\begin{proposition} \label{combinatorial_tuple_well_def}
 Assume that $\X$ is at $a$ both in the monomial case with combinatorial pair $(r\mon,l\mon)$ and in the small residual case with combinatorial pair $(r\gp,l\gp)$. Then 
 \[(r\mon,l\mon)=(r\gp,l\gp).\]
\end{proposition}
\begin{proof}
 Let $\x=(x,y,z)$ be apposite parameters of monomial type. Let the coefficient ideal $J_{2,\x}$ have the form $J_{2,\x}=(x^{r_x}y^{r_y})$. Let $f\in I_3$ be an element that is $\ord$-clean with respect to $J_{2,\x}$. By Lemma \ref{w_cleaning_preserves_v_cleaning} we may assume without loss of generality that $f$ is also $\ord_{(x)}$-clean and $\ord_{(y)}$-clean with respect to $J_{2,\x}$.
 
 Further, let $\x_1=(x_1,y_1,z_1)$ be apposite parameters of small residual type. Let the coefficient ideal $J_{2,\x_1}$ have the form $J_{2,\x_1}=(y_1^{mc!})\cdot I$ for a positive integer $m>0$ and an ideal $I$ with $\ord_{(y_1)}I=0$ and $0<\ord I<c!$. Let $f_1\in I_3$ be an element that is $\ord_{(y_1)}$-clean with respect to $J_{2,\x_1}$.
 
 By Lemma \ref{monomial_case} (2) and Lemma \ref{small_residual_case} (2) we may assume that $(y,z)=(y_1,z_1)$, $r_y\geq c!$ and $r_x<c!$. The identity $l_y=l_{y_1}$ follows now from Lemma \ref{same_curve_same_l_y_lemma}.
 
 Since $\ord J_{2,\x_1}>c!$ and $f\in I_3$ is $z$-regular of order $c$, we may assume by Lemma \ref{z_regular_blocks_other_parameters} that $z_1=z+g$ for an element $g\in K[[x,y]]$. By Proposition \ref{w_cleaning_maximizes_w} we know that 
 \[\ord J_{2,\x}\geq\ord J_{2,\x_1}>c!.\]
 By Lemma \ref{coeff_ideal_and_directrix} this implies that $\ord g\geq2$. Due to the identity, $(y,z)=(y_1,z_1)$ we may hence assume that $y_1=y+zG$ for some element $G\in\hOWa$. Set $\wt\x_1=(x,y,z_1)$. Then
 \[mc!=\ord_{(y_1)} J_{2,\x_1}=\ord_{(y)} J_{2,\wt\x_1}\]
 by Proposition \ref{residual_order_coord_indep} (2). It then follows from Proposition \ref{w_cleaning_maximizes_w} that $r_y\geq mc!$. Using a symmetric argument, we conclude that $r_y=mc!$. Consequently,
 \[(r\mon,l\mon)=(r_y,l_y)=(mc!,l_{y_1})=(r\gp,l\gp).\]
\end{proof}

\section{Existence and construction of maximizing flags} \label{section_maximizing_flags}

 This section has two goals: First, we want to prove that maximizing flags exist. This ensures that the resolution invariant $\iv_\X$ is well-defined. Furthermore, we want to show how flags that maximize the flag invariant $\inv(\F)$ over certain subsets of $\FF$ can be constructed from apposite parameters. These explicit constructions will play an important role when proving that the invariant $\iv_\X$ decreases under point-blowup in the non-terminal case.
 
 The proofs in this section will make use of virtually all results we established in Chapter 5. 
 
 Again, let in the following $\X=(W,X,E)$ be a $3$-dimensional resolution setting and $a\in X$ a closed point such that $c=c\Xa>1$.

\subsection{Basic results on the associated invariants $m_\F$ and $d_\F$}

 We begin by proving several basic lemmas on the invariants $m_\F$ and $d_\F$. The first one guarantees that they are always finite.

\begin{lemma} \label{m_and_d_are_finite}
 For all flags $\F\in\FF$ the invariants $m_\F$ and $d_\F$ are finite.
\end{lemma}
\begin{proof}
 Let $\F\in\FF$ be a flag with subordinate parameters $\x=(x,y,z)$. If either $m_\F=\infty$ or $d_\F=\infty$, it is clear that $J_{2,\x}=0$ has to hold. By Lemma \ref{coeff_is_zero} this implies that $I_3=(z^c)$. This contradicts Lemma \ref{no_z^c}.
\end{proof}

 The next two lemmas tell us to which extent the surface $\F_2$ already determines the associated invariants of a flag $\F\in\FF$.

\begin{lemma} \label{F_2_determines_numerals}
 Let $\F,\G\in\FF$ be comparable flags with $\F_2=\G_2$. The following hold:
 \begin{enumerate}[(1)]
  \item If $n_\F=0$, then $m_\F=m_\G$ and $d_\F=d_\G$.
  \item If $n_\F>0$, then $m_\F=m_\G$.
 \end{enumerate}
\end{lemma}
\begin{proof}
 (1): This is immediate from Proposition \ref{residual_order_coord_indep}.
 
 (2): Since $\F$ and $\G$ are comparable, it is easy to see that there exists a regular system of parameters $\x=(x,y,z)$ for $\hOWa$ such that $\F_2=\G_2=V(z)$, $\F_1=V(z,y)$ and $\G_1=V(z,y+h)$ for an element $h\in K[[x]]$ with $\ord h\geq n_\F$. The statement then follows from Proposition \ref{invariance_of_w} (1).
\end{proof}

\begin{lemma} \label{d_only_depends_on_strict_type}
 Let $\F\in\FF$ be a flag with $n_\F>0$. Let $\x=(x,y,z)$ be subordinate to $\F$. Consider a flag $\G\in\FF$ of the form $\G_2=\F_2=V(z)$, $\G_1=V(z,y+h)$ where $h\in K[[x]]$ has order $\ord h>n_\F$.
 
 Then $\F$ and $\G$ are comparable and $\inv(\F)=\inv(\G)$ holds.
\end{lemma}
\begin{proof}
 Using Lemma \ref{form_of_associated_components}, it is easy to see that the flags $\F$ and $\G$ are comparable.
 
 Consider first the case that $d_\F\neq-1$. Thus, the lexicographic inequality 
 \[\y(J_{2,\x})=(m_{\F,\x},d_{\F,\x})\geq (n_\F\cdot c!,c!)=c!\cdot\y(y)\]
 has to hold. Since also
 \[\y(h)\geq (n_\F+1,0)>(n_\F,1)=\y(y)\]
 holds, the assertion follows from Proposition \ref{invariance_of_w} (1) applied to the weighted order function $\y$.
 
 As a consequence, the assertion automatically also holds in the case that $d_\F=-1$.
\end{proof}

 Recall that the cleaning techniques we established in Chapter 5 only tell us how invariants associated to coefficient ideals behave under coordinate changes $z\mapsto z+g(\x)$. Thus, we have good control over the associated invariants of flags $\F\in\FF$ with $\F_2=V(z+g(x,y))$ when $\x=(x,y,z)$ are apposite parameters. While we cannot use the techniques of Chapter 5 to treat the flags which are not of this form, we will show in the following lemma that the associated invariants $m_\F$ and $d_\F$ of these flags are always minimal. Thus, they can mostly be ignored for questions of maximality.

\begin{lemma} \label{trivial_flags}
 Let $\x=(x,y,z)$ be apposite parameters for $\X$ at $a$. Let $\F\in\FF$ be a flag with $\F_2=V(z_1)$ for a parameter $z_1\in\hOWa$ that is not $z$-regular with respect to $\x$. Then the following hold:
 \begin{enumerate}[(1)]
  \item If $n_\F=0$, then $m_\F=0$ and $d_\F=c!$.
  \item If $n_\F>0$, then $m_\F=n_\F\cdot c!$ and $d_\F=-1$.
 \end{enumerate}
\end{lemma}
\begin{proof}
 (1): Since $\F$ is compatible with $E$, we can assume without loss of generality that $\F_2=V(y)$ and $\Eg\subseteq V(x)$. Further, we can assume by Lemma \ref{F_2_determines_numerals} that $\F_1=V(y,x)$ since all the invariants in the statement only depend on the hypersurface $\F_2$. Thus, the parameters $\x_1=(z,x,y)$ are subordinate to $\F$. It follows from Lemma \ref{z_regular_blocks_other_parameters} that $\ord J_{2,\x_1}\leq c!$ and $\ord_{(x)} J_{2,\x_1}=0$. Since $\ord J_{2,\x_1}\geq c!$ by Lemma \ref{ogeqc!}, this implies that $m_\F=0$ and $d_\F=c!$.
 
 (2): Since $\F$ is compatible with $E$ and $\Eg\neq\emptyset$, we can assume without loss of generality that $\F_2=V(y)$ and $\Eg=V(x)$. Since $\Eg$ has only one component, we know that $n_\F\geq 2$. It is easy to see that $\F_1=V(y,x_1)$ where $x_1=x+Q(z)$ for some element $Q\in K[[z]]$ of order $\ord Q=n_\F$. Thus, $\x_1=(z,x_1,y)$ is a regular system of parameters for $\hOWa$ that is subordinate to $\F$. Since $\ord Q>1$, it is clear that $f$ is $z$-regular of order $c$ with respect to $\x_1$. The associated weighted order function $\w:K[[x_1,z]]\to\N_\infty$ is defined by $\w(z)=1$ and $\w(x_1)=n_\F$. By Lemma \ref{z_regular_blocks_other_parameters} we know that 
 \[m_{\F,\x_1}=\w(J_{2,\x_1})\leq c!<n_\F\cdot c!.\]
 Thus, also $d_{\F,\x}<c!$ holds. This proves the assertion.
\end{proof}

\subsection{Maximizing flags $\F$ with $n_\F=0$}

 Using cleanness with respect to weighted order functions, we can now give a sufficient criterion for a flag $\F\in\FF$ with $n_\F=0$ to be valid and to maximize $d_\F$ over all valid flags that are comparable to $\F$. Further, using the secondary $\ord$-cleanness property that was developed in Section \ref{section_s_cleaning}, we also give a sufficient criterion for a flag $\F\in\FF$ with $n_\F=0$ and subordinate parameters $\x=(x,y,z)$ to maximize the flag invariant $\inv(\F)$ over all coordinate changes $z\mapsto z+g$ with $g\in K[[x,y]]$.

\begin{proposition} \label{cleaning_transversal_flags}
 Let $\F\in\FF$ be a flag with $n_\F=0$ and $\x=(x,y,z)$ a regular system of parameters for $\hOWa$ that is subordinate to $\F$. Let $f\in I_3$ be an element with the following properties:
 \begin{itemize}
  \item $f$ is $\ord$-clean with respect to $J_{2,\x}$.
  \item $f$ is $\ord_{(x)}$-clean with respect to $J_{2,\x}$.
  \item $f$ is $\ord_{(y)}$-clean with respect to $J_{2,\x}$.
  \end{itemize}
 Then the following hold:
 \begin{enumerate}[(1)]
  \item The flag $\F$ is valid.
  \item Let $\G\in\FF$ be a valid flag which is comparable to $\F$. Then $d_\G\leq d_\F$.
  \item If $f$ is also secondary $\ord$-clean with respect to $\coeff_{(x,y)}^{d_\F}(I_{2,\x})$, then for all valid flags $\G$ of the form $\G_2=V(z+g)$, $\G_1=V(z+g,y)$ for some element $g\in K[[x,y]]$ the inequality $\inv(\G)\leq \inv(\F)$ holds.
 \end{enumerate}
\end{proposition}
\begin{proof}
 (1),(2): Let $\G\in\FF$ be another flag with $n_\G=0$ and subordinate parameters $\x_1=(x_1,y_1,z_1)$. We will show that the lexicographic inequality $(m_\G,d_\G)\leq (m_\F,d_\F)$ holds. 
 
 If $z_1$ is not $z$-regular with respect to $\x$, we know by Lemma \ref{trivial_flags} that $m_\G=0$ and $d_\G=c!$. Since $\ord J_{2,\x}\geq c!$ by Lemma \ref{ogeqc!}, this proves that $(m_\G,d_\G)\leq (m_\F,d_\F)$. 
 
 Now consider the case that $z_1$ is $z$-regular with respect to $\x$. By the Weierstrass preparation theorem, we may assume that $z_1=z+g$ for an element $g\in K[[x,y]]$. Since the invariants $m_\G$ and $d_\G$ only depend on the hypersurface $\G_2$ by Lemma \ref{F_2_determines_numerals}, we can assume without loss of generality that $y_1=y$ and $x_1=x$. The inequality $(m_\G,d_\G)\leq (m_\F,d_\F)$ now follows from Proposition \ref{w_cleaning_maximizes_w}. 
 
 (3): Let $\G$ be a valid flag of the form $\G_2=V(z+g)$, $\G_1=V(z+g,y)$. Set $\x_1=(x,y,z+g)$. Since $\G$ is valid, we know that $m_\G=m_\F$ and $d_\G\leq d_\F$. Assume from now on that $d_\G=d_\F$ holds. Set $d=d_\F$. Then we know by Proposition \ref{s_cleaning_maximizes_slope} that 
 \[\ord \coeff_{(x,y)}^{d}(I_{2,\x_1})\leq\ord \coeff_{(x,y)}^{d}(I_{2,\x})\] holds. If $d\geq c!$, this suffices to show that $s_\G\leq s_\F$.
 
 If $0<d<c!$, then we can compute with Lemma \ref{coeff_ideal_of_powers} that
 \[s_\G=\min\Big\{\frac{(d(c!-d))!}{d!}\ord\coeff_{(x,y)}^{d}(I_{2,\x_1}),\frac{(d(c!-d))!}{(c!-d)!}\ord\coeff_{(x,y)}^{c!-d}(M_{2,\x_1})\Big\}\]
 and an analogous formula holds for $s_\F$.
 
 Since $M_{2,\x}=M_{2,\x_1}$ holds, this proves that $s_\G\leq s_\F$. Hence, $\inv(\G)\leq\inv(\F)$ holds.
\end{proof}

 Using the previous proposition and the results that we devised in Section \ref{section_maximizing_over_y_and_z}, we can now show that there exists a valid flag $\F\in\FF$ with $n_\F=0$ that maximizes the flag invariant $\inv(\F)$ over all valid flags that are comparable to $\F$.

\begin{proposition} \label{transversal_flags_can_be_maximized_with_z}
 Let $\x=(x,y,z)$ be apposite parameters for $\X$ at $a$. Then there exists a valid flag $\F\in\FF$ with $n_\F=0$ such that $\F_2=V(z+g)$ for some element $g\in K[[x,y]]$ and for all valid flags $\G\in\FF$ which are comparable to $\F$, the inequality $\inv(\G)\leq\inv(\F)$ holds.
\end{proposition}
\begin{proof}
 By definition, there is an element $f\in I_3$ which is $z$-regular of order $c$. By Lemma \ref{w_cleaning_preserves_v_cleaning} and Lemma \ref{s_cleaning_step_preserves_setting} we can assume after a change of coordinates $z\mapsto z+g$ with $g\in K[[x,y]]$ that $f$ is $\ord$-clean, $\ord_{(x)}$-clean and $\ord_{(y)}$-clean with respect to $J_{2,\x}$.
 
 Set $\F_2=V(z)$ and $\F_1=V(z,y)$. By Lemma \ref{cleaning_transversal_flags}, $\F$ is valid and for all valid flags $\G$ which are comparable to $\G$, the inequality $d_\G\leq d_\F$ holds. It remains to find such a flag $\G$ which fulfills $d_\G=d_\F$ and maximizes $s_\G$.
 
 To this end, we will at first only consider the valid flags $\G$ which are comparable to $\F$ and are of the form $\G_2=V(z+g)$ for some element $g\in K[[x,y]]$. Then necessarily either $\G_1=V(z+g,y+h)$ with $h\in K[[x]]$ or $\G_1=V(z+g,x+k)$ with $k\in K[[y]]$. By Proposition \ref{maximum_over_y_and_z_exists} there is a valid flag $\HH$ of the form $\HH=V(z+g)$ which is comparable to $\F$, fulfills $d_\HH=d_\F$ and $s_\G\leq s_\HH$ for all such flags $\G$. We may assume after a change of coordinates that $\F=\HH$.
 
 Now assume that there exists a valid flag $\G$ which is comparable to $\F$ and fulfills $d_\G=d_\F$ and $s_\G>s_\F$. Let $\G$ be have the form $\G_2=V(z_1)$, $\G_1=V(z_1,y_1)$. By what we have just shown, we know that $z_1$ that is not $z$-regular with respect to $\x$. We know by Lemma \ref{trivial_flags} that $m_\F=m_\G=0$ and $d_\F=d_\G=c!$. By Lemma \ref{cleanness_and_tau} this implies that $\tau(I_3)\geq 2$. Using Lemma \ref{J_2_tau=2}, we see that $s_\G>s_\F\geq c!!$ and $\Dir(I_3)=(\ol z_1,\ol y_1)$. By Lemma \ref{directrix_from_z_regularity} there exists an element $g\in K[[x,y]]$ such that $\ol{z+g}\in\Dir(I_3)$. Since $z_1$ is not $z$-regular, this implies that $y_1$ is $z$-regular. We may assume that $y_1=z+g_1$ for some element $g_1\in K[[x,y]]$. Define the flag $\HH$ as $\HH_2=V(z+g_1)=V(y_1)$ and $\HH_1=\G_1=V(y_1,z_1)$. The flag $\HH$ is comparable to $\F$. Further, since $\F$ is valid, it is clear by Lemma \ref{ogeqc!} that $m_\HH=0$, $d_\HH=c!$ and $\HH$ is also valid. Further, we know already know that $s_\HH\leq s_\F$ holds. But by Lemma \ref{J_2_tau=2} (1) we also know that $s_\G=s_\HH$. This contradicts the assumption that $s_\G>s_\F$.
\end{proof}

\subsection{Maximizing flags $\F$ with $n_\F>0$}

 Similarly to the previous section we will now use cleanness with respect to weighted order functions to give a sufficient criterion for a flag $\F\in\FF$ with $n_\F>0$ to be valid and to maximize the flag invariant $\inv(\F)$ over all coordinate changes $z\mapsto z+g$ with $g\in K[[x,y]]$.

\begin{proposition} \label{cleaning_tangential_flags}
 Let $\F\in\FF$ be a flag with $n_\F>0$ and $\x=(x,y,z)$ a regular system of parameters for $\hOWa$ that is subordinate to $\F$. Let $f\in I_3$ be an element that is $\y_{\F,\x}$-clean with respect to $J_{2,\x}$. Then the following hold:
 \begin{enumerate}[(1)]
  \item The flag $\F$ is valid.
  \item For all valid flags $\G\in\FF$ of the form $\G_2=V(z+g)$, $\G_1=V(z+g,y)$ for some element $g\in K[[x,y]]$ the inequality $\inv(\G)\leq \inv(\F)$ holds.
 \end{enumerate}
\end{proposition}
\begin{proof}
 (1): Let $\G$ be a flag that is comparable to $\F$. By Lemma \ref{trivial_flags} (2) we can assume without loss of generality that $\G_2=V(z+g)$ for some element $g\in K[[x,y]]$. Further, we can assume by Lemma \ref{F_2_determines_numerals} that $\G_1=V(z+g,y)$. Set $\x_1=(x,y,z+g)$. Then $m_{\G,\x_1}\leq m_{\F,\x}$ follows from Proposition \ref{w_cleaning_maximizes_w}. Hence, $m_\G\leq m_\F$.
 
 (2): Let $\x_1=(x,y,z+g)$ be subordinate to $\G$. We know by Proposition \ref{w_cleaning_maximizes_w} that the lexicographic inequality 
 \[(m_{\G,\x_1},d_{\G,\x_1})\leq (m_{\F,\x},d_{\F,\x})\]
 holds. If $m_{\G,\x_1}<n_\F\cdot c!$, then $d_\G=-1$ holds and it is clear that $\inv(\G)\leq\inv(\F)$ holds. Otherwise, we conclude that $m_{\G,\x_1}=m_{\F,\x}$ since $\G$ is valid and thus, $m_\G=m_\F$ holds. Hence, $d_{\G,\x_1}\leq d_{\F,\x}$. Since $m_\G=m_\F$, this implies that $d_\G\leq d_\F$. Thus, $\inv(\G)\leq \inv(\F)$.
\end{proof}

 To guarantee the existence of a maximizing flag $\F\in\FF$, there are two things that we still have to show. While we already know by Lemma \ref{m_and_d_are_finite} that the invariant $d_\F$ is finite for each flag $\F\in\FF$, we have not yet established a bound for $d_\F$. Further, we have to show that there is a bound for the associated multiplicity $n_\F$ of valid flags $\F\in\FF$ with maximal $d_\F$. Both of these will be established in the next proposition, using the results of Section \ref{section_kangaroo_calc}.

\begin{proposition} \label{maximal_n_and_d}
 There are numbers $d_*,N\in\N$ such that the following hold:
 \begin{enumerate}[(1)]
  \item For all valid flags $\F\in\FF$ with $n_\F>0$, the inequality $d_\F\leq d_*$ holds.
  \item For all valid flags $\F\in\FF$ with $n_\F\geq N$, the equality $d_\F=-1$ holds.
 \end{enumerate}
\end{proposition}
\begin{proof}
 (1): Let $\x=(x,y,z)$ be a apposite parameters for $\hOWa$ and $f\in I_3$ an element that is $z$-regular of order $c$. Since the existence of flags $\F\in\FF$ with $n_\F>0$ implies that $\Eg\neq\emptyset$, we can assume without loss of generality that $V(y)\subseteq\Eg$. It is clear that is suffices to prove the assertion only for flags $\F\in\FF$ with $n_\F=1$ or $n_\F>1$ and $D_\F=V(y)$. Thus, we will only consider flags of this type in the following.
 
 By Lemma \ref{w_cleaning_preserves_v_cleaning} and Lemma \ref{automatic_w_n_cleaning_for_large_n} we can assume after a change of coordinates $z\mapsto z+g(x,y)$ that $f$ is $\w_n$-clean with respect to $J_{2,\x}$ for the weighted-order functions $\w_n:K[[x,y]]\to\Ni$ which are defined via $\w_n(x)=1$ and $\w_n(y)=n$.
 
 Define for each positive integer $n>0$ the number 
 \[d_n=d(\minit_{\w_n}(J_{2,\x}))\]
 where $d(I)=\ord I-\ord_{(x)}I-\ord_{(y)}I$ for an ideal $I\subseteq K[[x,y]]$. By Lemma \ref{automatic_w_n_cleaning_for_large_n} there exists a number $N\in\N$ such that $d_n=0$ for all $n\geq N$. Set 
 \[d_*=\max\{d_n:n\in\N\}+\eps\]
 where
 \[\eps=\begin{cases}
         0 & \text{if $\chara(K)=0$,}\\
         \frac{c!}{p} & \text{if $\chara(K)=p>0$.}
        \end{cases}
\]
 Now let $\F\in\FF$ be a valid flag with $n_\F=1$ or $n_\F>1$ and $D_\F=V(y)$. We want to show that $d_\F\leq d_*$. Set $n=n_\F$. By Lemma \ref{d_only_depends_on_strict_type} and Proposition \ref{cleaning_tangential_flags} we may assume without loss of generality that $\F_2=V(z_1)$ and $\F_1=V(z_1,y_1)$ where $y_1=y+tx^n$ for a constant $t\in K^*$ and the coordinate change $z_1=z+g$ is $\y$-cleaning with respect to $J_{2,\x}$ and $f$. Here, the weighted order function $\y:K[[x,y_1]]\to\Ni^2$ is defined via $\y(x)=(1,0)$ and $\y(y_1)=(n,1)$.
 
 It follows from Proposition \ref{kangaroo_prop} (1) and (2) that 
 \[d_\F\leq d_n+\eps\leq d_*.\]
 
 (2): Continuing the previous calculation, consider the case $n\geq N$. Thus, $d_n=0$. Set $\x_1=(x,y_1,z+g)$. Then we know by Proposition \ref{kangaroo_prop} (3) and (4) that either $d_{\F,\x_1}=0$ or $d_{\F,\x_1}<c!$ and $m_{\F,\x_1}\in c!\cdot\N$ hold. In either case, this implies that $d_\F=-1$.
\end{proof}

\begin{proposition} \label{tangential_flags_can_be_maximized_with_z}
 Let $\x=(x,y,z)$ be apposite parameters and $f\in I_3$ an element that is $z$-regular of order $c$. Assume that $V(y)\subseteq\Eg$ and let $n>0$ be a positive integer. If $n=1$, assume that $\Eg=V(xy)$.
 
 There is a valid flag $\F\in\FF$ with $n_\F=n$ and $D_\F=V(y)$ such that for all valid flags $\G\in\FF$ that are comparable to $\F$, the inequality $\inv(\G)\leq\inv(\F)$ holds. Further, there is a non-zero constant $t\in K^*$ such that $\F$ can be chosen in the following way: Set $y_1=y+tx^n$ and let $\y:K[[x,y_1]]\to\Ni^2$ be the weighted order function that is defined by $\y(x)=(1,0)$ and $\y(y_1)=(n,1)$. Then $\F$ can be chosen as $\F_2=V(z_1)$ and $\F_1=V(z_1,y_1)$ where the coordinate change $z=z_1+g$ is $\y$-cleaning with respect to $J_{2,\x}$ and $f$.
 
\end{proposition}
\begin{proof}
 This follows from Lemma \ref{d_only_depends_on_strict_type}, Proposition \ref{cleaning_tangential_flags} and Proposition \ref{maximal_n_and_d}.
\end{proof}

\begin{proposition} \label{maximizing_flags_exist}
 There exists a valid flag $\F\in\FF$ such that for all other valid flags $\G\in\FF$ the inequality $\inv(\G)\leq\inv(\F)$ holds. Hence, $\F$ is a maximizing flag.
\end{proposition}
\begin{proof}
 This follows from Proposition \ref{transversal_flags_can_be_maximized_with_z} and Proposition \ref{maximal_n_and_d}.
\end{proof}

\begin{lemma} \label{inv_finite}
 If $\X$ is not in a terminal case at $a$ and $\F\in\FF$ is a maximizing flag, then $(d_\F,n_\F,s_\F)\in\N^3$. Hence, all components of the flag invariant are finite. Also, $d_\F>0$ holds.
\end{lemma}
\begin{proof}
 We know by Lemma \ref{m_and_d_are_finite} that $d_\F<\infty$. Also, $n_\F<\infty$ holds by definition. So assume that $s_\F=\infty$. Then necessarily, $n_\F=0$ has to hold. Let $\x=(x,y,z)$ be apposite parameters. By Proposition \ref{transversal_flags_can_be_maximized_with_z} we may assume that $\x$ are subordinate parameters to $\F$. Let $f\in I_3$ be an element that is $z$-regular of order $c$. By Proposition \ref{cleaning_transversal_flags} we may assume that $f$ is $\ord$-clean with respect to $J_{2,\x}$.
 
 Since $s_\F=\infty$ and consequently, $J_{1,\x}=0$, we know by Lemma \ref{coeff_is_zero} that $I_{2,\x}=(y^{d_\F})$. Thus, the coefficient ideal $J_{2,\x}$ is of the form $J_{2,\x}=M_{2,\x}\cdot(y^{d_\F})=(x^{m_x}y^{m_y})$ for certain $m_x,m_y\in\N$. Hence, the parameters $\x$ are of monomial type. This contradicts the assumption that $\X$ is not in a terminal case at $a$.
 
 Now assume that $d_\F=0$ holds. This implies again that $n_\F=0$. But then $J_{1,\x}=0$ holds by definition. This contradicts the finiteness of $s_\F$.
\end{proof}

\section{Definition of the resolution invariant $\iv_\X$} \label{section_def_of_resolution_invariant}

 Using the definitions and results of the previous sections, we can now define the resolution invariant $\iv_\X$ for surfaces in arbitrary characteristic.
 
 Let $\X=(W,X,E)$ be a $3$-dimensional resolution setting. We will now define a map $\iv_\X:X\to\N^7$ of the form
 \[\iv_\X=(o,c,d,n,s,r,l)\]
 where $\N^7$ is considered with respect to the lexicographic order. If we want to emphasize the dependence of the components of $\iv_\X$ on $\X$ and a point $a\in X$, we will denote them by $o\Xa$, $c\Xa$, $\ldots$
 
 Let $a\in X$ be a closed point. Let $(o,c)=(o_\X(a),c_\X(a))$ be defined as in Section \ref{section_resolution_settings}.
 
 If $(o,c)=(1,1)$, then $\X$ is resolved at $a$ by Lemma \ref{c=1_measures_resolved} and we set 
 \[\iv_\X(a)=(1,1,0,0,0,0,0).\]
 
 If $(o,c)>(1,1)$, then we know by Proposition \ref{maximizing_flags_exist} that there exists a maximizing flag $\F\in\FF(a)$.
 
 If $\X$ is not in a terminal case at $a$, we set
 \[\iv_\X(a)=(o,c,d_\F,n_\F,s_\F,0,0)\]
 where $\inv(\F)=(d_\F,n_\F,s_\F)$ is the flag invariant of a maximizing flag $\F\in\FF(a)$. Notice that $(d_\F,n_\F,s_\F)\in\N^3$ by Lemma \ref{inv_finite}.
 
 If $\X$ is in a terminal case at $a$, we set
 \[\iv_\X(a)=(o,c,0,0,0,r,l)\]
 where $(r,l)$ is the combinatorial pair as it was defined in Section \ref{section_monom_s_r_case}.
 
 For non-closed points $\xi\in X$ we define
 \[\iv_\X(\xi)=\min\{\ivXa:\text{$a\in\ol{\{\xi\}}$ is a closed point}\}.\]

\begin{remark}
 The first property that we required from a resolution invariant in Section \ref{section_resolution_via_usc_inv} was that it was a local geometric invariant. Since $\ivX$ depends on the labels on $E$, it is not an entirely geometric invariant. Instead, $\ivX$ is a local invariant of the setting $\X$ in the following sense:
 
 Let $\X=(W,X,E)$ and $\X_1=(W_1,X_1,E_1)$ be two $3$-dimensional resolution settings. Further, let $a\in X$ and $a_1\in X_1$ be points such that the following holds: There is an isomorphism of local rings $\OWa\cong\OO_{W_1,a_1}$ (it also suffices to have an isomorphism between the completions) that locally maps $X$ and $X_1$, as well as $E$ and $E_1$, into each other. Further, the isomorphism maps each component $D$ of $E$ that contains $a$ onto a components of $D_1$ of $E_1$ that fulfills $\age(D_1)=\age(D)$ and $\lab(D_1)=\lab(D)$. Then the equality $\ivXa=\iv_{\X_1}(a_1)$ holds.
 
 In particular, this applies to the following situation: Consider a blowup $\pi:W'\to W$ along a center $Z$ that is permissible for the setting $\X=(W,X,E)$. Let $\X'=(W',X',E')$ be the induced resolution setting. Further, let $a\in X\setminus Z$ be a point outside the center and $a'\in X'$ the unique point lying over $a$. Then $\ivXap=\ivXa$.
\end{remark}

\chapter{Upper-semicontinuity of $\iv_\X$ and permissibility of its top locus} \label{chapter_usc}

 The goal of this chapter is to prove two key properties of the resolution invariant $\iv_\X$ that we defined in the previous chapter: Its upper semicontinuity and the fact that the set $\Xmax$ of points $a\in X$ at which $\iv_\X$ is maximal is a permissible center of blowup for $\X$. To be able to prove that the invariant decreases under blowup along $\Xmax$, we also need to make sure that the invariant $\iv_\X$ determines the correct centers. In particular, the invariant $(d,s,n)$ that is used whenever $\X$ is not in a terminal case is only expected to decrease under point-blowups. Thus, we need to prove that there are only finitely points at which $\X$ is not in a terminal case. Further, particular centers are required to measure improvement during combinatorial resolution.
 
 To be more explicit, consider a curve $C$ which is contained in the set $X_{\geq(o,c)}$ for a pair $(o,c)>(1,1)$. Let $\xi$ be the generic point of $C$ and assume that $(o_\X(\xi),c_\X(\xi))=(o,c)$. By upper semicontinuity of the order function, we know that for all but finitely many closed points $a\in X$ the equality $(o\Xa,c\Xa)=(o,c)$ holds. In Proposition \ref{key_proposition_for_usc}, the central result of this chapter, we will show that $\X$ is in a terminal case at all but finitely many of these points. Further, there is a generic combinatorial pair $(r_\xi,l_\xi)$ which prescribes the combinatorial pair at all points $a\in C$ at which $\X$ is in a terminal case and which are not contained in any other components of $X_{\geq(o,c)}$. If $\iv_\X$ is constant along $C$, then we will show that the curve $C$ is a permissible center of blowup for $\X$.
 
 To prove these results, we will develop in Section \ref{section_invariant_in_monom_case} techniques to compute the combinatorial pair $(r,l)$ at a point $a\in X$. These will be used in the proof of Proposition \ref{key_proposition_for_usc} to prove the existence of the generic combinatorial pair $(r_\xi,l_\xi)$.
 
 Unless explicitly mentioned otherwise, we will always consider the following setting in this chapter: $\X=(W,X,E)$ is a $3$-dimensional resolution setting and $a\in X$ a closed point with $c=c\Xa>1$.

\section{Determining the combinatorial pair in the terminal cases}\label{section_invariant_in_monom_case}

 The next two lemmas will give us a way to compute the combinatorial pair $(r,l)$ from a coefficient ideal $J_{2,\x}$ without necessarily showing that the parameters $\x$ are of monomial or small residual type. We will make use of an element $f\in I_3$ which is clean with respect to $J_{2,\x}$ for certain weighted order functions.

\begin{lemma} \label{correct_m_lemma}
 Let $\X$ be in a terminal case at $a$. Let $\x=(x,y,z)$ be apposite parameters for $\X$ at $a$ such that $I_{\geq(o,c)}=(y,z)$ holds.
  
 Set $J_{2,\x}=\coeff^c_{\x}(I_3)$. Assume that $r_y:=\ord_{(y)} J_{2,\x}\geq c!$ and that there is an element $f\in I_3$ which is $\ord_{(y)}$-clean with respect to $J_{2,\x}$.
 
 Then the combinatorial pair for $\X$ at $a$ is $(r,l)=(r_y,l_y)$. 
\end{lemma}
\begin{proof}
 
 By Lemma \ref{w_cleaning_preserves_v_cleaning} we may assume without loss of generality that $f$ is also $\ord$-clean with respect to $J_{2,\x}$.
 
 We first consider the case $r_y=c!$. Then $J_{2,\x}$ is of the form $J_{2,\x}=(y^{c!})$. Hence, the parameters $\x$ are of monomial type and the result follows from the definition of the combinatorial pair for the monomial case. Hence, we assume for the remainder of the proof that $r_y>c!$. 
 
 Assume first that $\X$ is in the monomial case at $a$. Let $\x_1=(x_1,y_1,z_1)$ be a apposite parameters of monomial type. By Lemma \ref{monomial_case} (2) we may assume that $(y_1,z_1)=(y,z)$ and that the coefficient ideal $J_{2,\x_1}$ has the form $J_{2,\x_1}=(x_1^{s_x}y_1^{s_y})$ with $s_x<c!$ and $s_y\geq c!$. Consequently, $(r\mon,l\mon)=(s_y,l_{y_1})$. The identity $l_{y_1}=l_y$ follows from Lemma \ref{same_curve_same_l_y_lemma}.
 
 Let $f_1\in I_3$ be an element that is $\ord$-clean with respect to $J_{2,\x_1}$. By Lemma \ref{z_regular_blocks_other_parameters} we may assume that $z_1=z+g$ for an element $g\in K[[x,y]]$. By Proposition \ref{w_cleaning_maximizes_w} we know that
 \[\ord J_{2,\x_1}=\ord J_{2,\x}\geq r_y>c!.\]
 By Lemma \ref{coeff_ideal_and_directrix} this implies that $\ord g\geq2$. Consequently, we may assume that $y_1=y+zG$ for some element $G\in\hOWa$. Set $\wt\x_1=(x,y,z_1)$. By Lemma \ref{residual_order_coord_indep} (2) we know that
 \[s_y=\ord_{(y_1)}J_{2,\x_1}=\ord_{(y)}J_{2,\wt\x_1}.\]
 It follows from Proposition \ref{w_cleaning_maximizes_w} that $s_y\leq r_y$. By a symmetric argument we conclude that $s_y=r_y$ holds.
 
 Now assume that $\X$ is in the small residual case at $a$. Let $\x_1=(x_1,y_1,z_1)$ be a apposite parameters of small residual type. By Lemma \ref{small_residual_case} (2) we know that $(y_1,z_1)=(y,z)$ and that the coefficient ideal $J_{2,\x_1}$ has the form $J_{2,\x_1}=(y_1^{mc!})\cdot I$ for a positive integer $m>0$ and an ideal $I$ with $\ord_{(y_1)}I=0$ and $0<\ord I<c!$. Consequently, $(r\gp,l\gp)=(mc!,l_{y_1})$. Again, the identity $l_{y_1}=l_y$ follows from Lemma \ref{same_curve_same_l_y_lemma}.
 
 Let $f_1\in I_3$ be an element which is $\ord_{(y)}$-clean with respect to $J_{2,\x_1}$. By Lemma \ref{c!_does_not_divide_m} we know that $f_1$ is also $\ord$-clean with respect to $J_{2,\x_1}$. Hence, we may use the same arguments as in the monomial case to show that the identity $mc!=r_y$ holds.
\end{proof}

\begin{lemma} \label{correct_m_lemma_for_two_curves}
 Let $\X$ be in the monomial case at $a$ and let $\x=(x,y,z)$ be apposite parameters with the property that the coefficient ideal $J_{2,\x}=\coeff_{\x}^c(I_3)$ fulfills $r_x:=\ord_{(x)}J_{2,\x}\geq c!$ and $r_y:=\ord_{(y)}J_{2,\x}\geq c!$. Assume further that there exists an element $f\in I_3$ which is both $\ord_{(x)}$-clean and $\ord_{(y)}$-clean with respect to $J_{2,\x}$.
 
 Then the combinatorial pair for $\X$ at $a$ is $(r,l)=\max\{(r_x,l_x),(r_y,l_y)\}$ and $l_x\neq l_y$.
\end{lemma}
\begin{proof}
 By Lemma \ref{w_cleaning_preserves_v_cleaning} we may assume without loss of generality that $f$ is also $\ord$-clean with respect to $J_{2,\x}$. 
 
 Let $\x_1=(x_1,y_1,z_1)$ be a apposite parameters of monomial type. Let the coefficient ideal $J_{2,\x_1}$ have the form $J_{2,\x_1}=(x_1^{s_x}y_1^{s_y})$ and let $f_1\in I_3$ be an element which is $\ord$-clean with respect to $J_{2,\x_1}$. By Lemma \ref{w_cleaning_preserves_v_cleaning} we may assume that $f_1$ is also $\ord_{(x_1)}$-clean and $\ord_{(y_1)}$-clean with respect to $J_{2,\x_1}$.
 
 By Proposition \ref{top_locus_defined_via_coeff_ideal} (1) we can compute that that $\Ioc\subseteq (z,xy)$. By Lemma \ref{monomial_case} (1) this implies that $\Ioc=(z,xy)=(z_1,x_1y_1)$ and $s_x,s_y\geq c!$.
 
 By the definition of the monomial case, this implies that $\Eg\neq\emptyset$. Let us assume without loss of generality that $V(y)\subseteq\Eg$. Hence, we may assume $y=y_1$. Since $(z,y)=(z_1,y_1)$, this implies that $z_1=z+g$ for some element $g\in K[[x,y]]$. By Proposition \ref{w_cleaning_maximizes_w} this implies that $r_y=s_y$.
 
 Further, we know by Lemma \ref{coeff_ideal_and_directrix} that $\ord g\geq2$. Since $(z,x)=(z_1,x_1)$, we may assume that $x_1=x+zG$ for an element $G\in\hOWa$. Set $\wt\x_1=(x,y,z_1)$. By Proposition \ref{residual_order_coord_indep} (2) we know that
 \[s_x=\ord_{(x_1)}J_{2,\x_1}=\ord_{(x)} J_{2,\wt\x_1}.\]
 By Proposition \ref{w_cleaning_maximizes_w} we know that $s_x\leq r_x$. By a symmetric argument we conclude that $s_x=r_x$.
 
 The claimed form of $(r,l)$ follows from the definition of the combinatorial pair in the monomial case.
 
 Further, $l_x\neq l_y$ since $l_x=l_y$ would imply that $\lab(V(x))=\lab(V(y))$. Since $\lab:\Comp(E)\to\N$ is injective, this would contradict the fact that $E$ is a simple normal crossings divisor.
\end{proof}

\section{Generic value of $\iv_\X$ along curves in $X_{\geq(o,c)}$} \label{section_generic_value_of_inv_along_curves}

 The following result will enable us to prove upper semicontinuity of $\iv_\X$ and the fact that $\Xmax$ is a permissible center of blowup for $\X$. Further, we will use it in Section \ref{section_inv_drops_for_d=0} to determine the center of blowup during combinatorial resolution. The foundation for the proof of the following proposition is provided by the results we developed in Section \ref{section_global_local_expansion}.

\begin{proposition} \label{key_proposition_for_usc}
 Let $\X=(W,X,E)$ be a $3$-dimensional resolution setting. Let $C\subseteq X$ be a curve with generic point $\xi$. Set $o=o_\X(\xi)$ and $c=c_\X(\xi)$ and assume that $(o,c)>(1,1)$. Let $a\in C$ be a closed point with $(o\Xa,c\Xa)=(o,c)$. The following hold:
 \begin{enumerate}[(1)]
  \item If $\X$ is in a terminal case at $a$, then $C$ is regular at $a$ and $C\cup E$ has simple normal crossings at $a$.
  \item There is a \emph{generic combinatorial pair} $(r_\xi,l_\xi)\in\N^2$ with the property that
 \[\iv_\X(b)=(o,c,0,0,0,r_\xi,l_\xi)\]
 for all but finitely many closed points $b\in C$.
 \item If $\X$ is in a terminal case at $a$, then any other curve $C'\subseteq X$ with generic point $\zeta$ passing through $a$ that fulfills $(o_\X(\zeta),c_\X(\zeta))=(o,c)$ has a generic combinatorial pair $(r_\zeta,l_\zeta)\neq(r_\xi,l_\xi)$. The combinatorial pair $(r,l)$ at $a$ equals the maximum over all these pairs.
 \end{enumerate}
\end{proposition}
 \begin{proof}
 (1): By Proposition \ref{monomial_case} (2) and Proposition \ref{small_residual_case} (3) there exists a regular system of parameters $\x=(x,y,z)$ for $\hOWa$ such that $\Eg\subseteq V(xy)$ and either $\Ioc(a)=(y,z)$ or $\Ioc(a)=(xy,z)$. If $\Ioc(a)=(y,z)$, then $C=V(y,z)$ locally at $a$ and it is clear that $C\cup\Eg$ has simple normal crossings at $a$.
 
 If $\Ioc(a)=(xy,z)$, we know that $\X$ is in the monomial case at $a$ and $\Eg\neq\emptyset$ by Lemma \ref{monomial_case}. Consequently, $V(x,z)$ and $V(y,z)$ cannot be two analytic branches of the same curve. Hence, either $C=V(y,z)$ or $C=V(x,z)$ locally at $a$. In either case, it is clear that $C\cup\Eg$ has simple normal crossings at $a$.
 
 The fact that $C\cup E$ has simple normal crossings at $a$ now follows from the fact that $C\cup \Eg$ and $\Eg\cup E_{>o}$ are simple normal crossings and $C\subseteq E_{>o}$ holds. 
 
 (2): Let $a\in C$ be a closed point such that $(o\Xa,c\Xa)=(o,c)$, the curve $C$ is regular at $a$ and $C\cup\Ego$ has simple normal crossings at $a$. (Such a point exists by Proposition \ref{order_is_usc}.) By Lemma \ref{good_parameters} there exists a regular system of parameters $\x=(x,y,z)$ for $\OWa$ which constitutes local apposite parameters for $\X$ at $a$.
 
 By Lemma \ref{directrix_from_z_regularity} there constants $\lambda_1,\lambda_2\in K$ such that $\ol{z+\lambda_1 x+\lambda_2 y}\in\Dir(I_3(a))$. This implies by Lemma \ref{directrix_and_top_locus} that there is an element $Q\in\OWa$ with $\ord Q\geq2$ such that $z+\lambda_1x+\lambda_2y+Q\in I_{C,a}$. By Lemma \ref{z_regular_under_coord_change} we may assume without loss of generality that $\lambda_1=\lambda_2=Q=0$. Hence, $C$ is contained in the regular local hypersurface $V(z)$ at $a$.
 
 Since $C\cup\Ego$ has simple normal crossings at $a$, we can choose an open affine neighborhood $U=\Spec(R)$ of $a$ and elements $x,y,z\in R$ such that the following hold:
 \begin{itemize}
  \item $\Ego\cap U\subseteq V(xy)$.
  \item $C\cap U=V(y,z)$.
  \item The module of differentials $\Omega_{R/K}$ is freely generated by $dx,dy,dz$. For each closed point $b\in C\cap U$ there is a constant $t_b\in K$ such that $\x_b=(x_b,y,z)$ with $x_b=x-t_b$ is a regular system of parameters for $\hOWb$. We can assume that $t_a=0$, $x_a=x$.
  \item $(X\cup\Ebo)\cap U=V(J)$ where $J\subseteq R$ is a principal ideal that is generated by an element $f$ which is $z$-regular of order $c$ in the completed local ring $\hOWa$ with respect to the regular system of parameters $\x=(x,y,z)$.
 \end{itemize}
 In the following, we will shrink $U$ several times until it is small enough to show that $\iv_\X$ is constant along $C\cap U\setminus\{a\}$. By Proposition \ref{order_is_usc} we can assume that $(o_\X(b),c_\X(b))=(o,c)$ for all closed points $b\in C\cap U$.
 
 By Proposition \ref{hochster} we know that $J\subseteq (y,z)^c$. This implies by Lemma \ref{w_coeff} that $\ord_{(y)} J_{2,\x}(a)\geq c!$.
 
 Since $f$ is $z$-regular of order $c$, we know by Proposition \ref{neighborhood_cleaning} (4) that $\partial_{z^c}(f)\notin m_{U,a}$. After shrinking $U$, we can assume that $\partial_{z^c}(f)\in R^*$ holds. Hence, we can assume by Proposition \ref{neighborhood_cleaning} (5) that $\partial_{z^c}(f)=1+F$ for an element $F\in (y,z)$. Using Proposition \ref{neighborhood_cleaning} (4) again, we see that $f$ is $z$-regular of order $c$ with respect to the parameters $\x_b$ as an element of the ring $\hOWb$ for all closed points $b\in C\cap U$.
 
 By Proposition \ref{neighborhood_cleaning} (6) we can assume without loss of generality that $f$ is $\ord_{(y)}$-clean with respect to $J_{2,\x}(a)$. Set $r_y=\ord_{(y)}J_{2,\x}(a)$. By Proposition \ref{neighborhood_cleaning} (1) we know that $\ord_{(y)}J_{2,\x_b}(b)=r_y$ holds for all closed points $b\in C\cap U$. 
 
 Set $(r_\xi,l_\xi)=(r_y,l_y)$.
 
 Let $f$ have the power series expansion $f=\sum_{i\geq0}f_iz^i$ with $f_i\in K[[x,y]]$. By Lemma \ref{w_coeff} there exists an index $i<c$ such that $\ord_{(y)} f_i=\frac{c-i}{c!}r_y$. Let $i<c$ be maximal with this property.
 
 Set 
 \[F_i=\partial_{y^j}\partial_{z^i}(f)\]
 where $j=\frac{c-i}{c!}r_y$. By Proposition \ref{neighborhood_cleaning} (2) we know that $F_i\notin (y,z)$. After shrinking $U$, we can assume that $F_i\notin m_{U,b}$ for all closed points $b\in C\cap U$ with $b\neq a$. Let $f$ have the expansion $f=\sum_{i\geq0}f_i^{(b)}z^i$ with $f_i^{(b)}\in K[[x_b,y]]$ for each closed point $b\in C\cap U$. Then we know by Proposition \ref{neighborhood_cleaning} (3) that $\ord f_i^{(b)}=\frac{c-i}{c!}r_y$. By Lemma \ref{w_coeff} this implies that the equality
 \[J_{2,\x_b}(b)=(y^{r_y})\]
 holds for all closed points $b\in C\cap U$ with $b\neq a$.
 
 If $c-q<i<c$, then $f$ is $\ord$-clean with respect to $J_{2,\x_b}(b)$ for all closed points $b\in C\cap U$ with $b\neq a$. Hence, the apposite parameters $\x_b$ are of monomial type. Consequently,
 \[\iv_\X(b)=(o,c,0,0,0,r_y,l_y).\]
 If $i<c-q$, then we know by maximality of $i$ that 
 \[\ord f_{c-q}^{(b)}\geq \ord_{(y)}f_{c-q}^{(b)}=\ord_{(y)}f_{c-q}>\frac{q}{c!}r_y\]
 for all closed points $b\in C\cap U$. Consequently, $f$ is $\ord$-clean with respect to $J_{2,\x_b}(b)$ for all closed points $b\in C\cap U$ with $b\neq a$ and the same argument as before applies. The same holds by Lemma \ref{c!_does_not_divide_m} if $r_y\notin c!\cdot\N$. 
 
 So assume from now on that $i=c-q$ and $r_y\in c!\cdot\N$ hold. Set $m=\frac{r_y}{c!}$. Since $f$ is $\ord_{(y)}$-clean with respect to $J_{2,\x}(a)$, but the properties $(1)_{\ord_{(y)}}$ and $(2)_{\ord_{(y)}}$ do not hold, we know by Proposition \ref{neighborhood_cleaning} (7) that there exists a number $0<k<q$ such that
 \[\partial_{x^k}(F_{c-q})=\partial_{x^k}\partial_{y^{qm}}\partial_{z^{c-q}}(f)\notin (y,z).\]
 Notice that $\partial_{x_b^k}(F_{c-q})=\partial_{x^k}(F_{c-q})$. Hence, $f$ is $\ord_{(y)}$-clean with respect to $J_{2,\x_b}(b)$ for all closed points $b\in C\cap U$ by Proposition \ref{neighborhood_cleaning} (7). Further, we can assume after shrinking $U$ further that $\partial_{x^k}(F_{c-q})\notin m_{U,b}$ for all closed points $b\in C\cap U$ with $b\neq a$. 
 
 Let $b\in C\cap U$ with $b\neq a$ be a closed point such that $f$ is not $\ord$-clean with respect to $J_{2,\x_b}(b)$. Let $z=z_b+g_b$ with $g_b\in K[[x_b,y]]$ be an $\ord$-cleaning step with respect to $f$ and $J_{2,\x_b}(b)$. Set $\wt\x_b=(x_b,y,z_b)$. By Lemma \ref{w_cleaning_improves_w} we know that either $\ord J_{2,\wt\x_b}(b)>\ord J_{2,\x_b}(b)$ or $\ord J_{2,\wt\x_b}(b)=\ord J_{2,\x_b}(b)$ and $f$ is $\ord$-clean with respect to $J_{2,\wt\x_b}(b)$. In the latter case, it follows from Lemma \ref{w_cleaning_preserves_v_cleaning} that $J_{2,\wt\x_b}(b)=(y^{r_y})$. Hence, the parameters $\wt\x_b$ are of monomial type and 
 \[\iv_\X(b)=(o,c,0,0,0,r_y,l_y).\]
 
 So assume now that $\ord J_{2,\wt\x_b}(b)>\ord J_{2,\x_b}(b)$. By Proposition \ref{neighborhood_cleaning} (8) we know that $J_{2,\wt\x_b}(b)$ has the form
 \[J_{2,\wt\x_b}(b)=(y^{mc!})\cdot I\]
 for an ideal $I$ with $\ord_{(y)}I=0$ and $0<\ord I<c!$. Since $f$ is $\ord_{(y)}$-clean with respect to $J_{2,\wt\x_b}(b)$ by Lemma \ref{w_cleaning_preserves_v_cleaning}, the parameters $\wt\x_b$ are of small residual type. Since $mc!=r_y$, this implies that
 \[\iv_\X(b)=(o,c,0,0,0,r_y,l_y).\]
 
 (3): Continuing the previous considerations, we now assume that $\X$ is in a terminal case at $a$. By what we have already shown, we know that $\ord_{(y)}J_{2,\x}(a)=r_\xi\geq c!$, $l_y=l_\xi$ and $f$ is $\ord_{(y)}$-clean with respect to $J_{2,\x}(a)$. Further, we know by Lemma \ref{monomial_case} (2) and Lemma \ref{small_residual_case} (2) that at most two components of $X_{\geq (o,c)}$ pass through $a$
 
 If only one component of $X_{\geq(o,c)}$ passes through $a$, this component coincides with $C$. Consequently, $I_{\geq(o,c)}(a)=(y,z)$. By Lemma \ref{correct_m_lemma} this implies that $(r,l)=(r_\xi,l_\xi)$.
 
 Now assume that another component $\wt C$ of $X_{\geq(o,c)}$ with generic point $\zeta$ passes through $a$. Using the same techniques as before, we can assume that $\wt C\cap U=V(x,z)$, $\ord_{(x)}J_{2,\x}(a)=r_\zeta$, $l_x=l_\zeta$ and $f$ is $\ord_{(x)}$-clean with respect to $J_{2,\x}(a)$. Hence, we know by Lemma \ref{correct_m_lemma_for_two_curves} that $(r,l)=\max\{(r_\xi,l_\xi),(r_\zeta,l_\zeta)\}$ and $(r_\xi,l_\xi)\neq (r_\zeta,l_\zeta)$.
\end{proof}

\begin{corollary} \label{only_finitely_many_points_with_d>0}
 Let $\X=(X,W,E)$ be a $3$-dimensional resolution setting and $o,c\in\N$ such that $(o,c)>(1,1)$. Then the set
 \[\{a\in X:(o\Xa,c\Xa)=(o,c),\text{$\X$ is not in a terminal case at $a$}\}\]
is discrete.
\end{corollary}
\begin{proof}
 Since $(o_\X,c_\X):X\to\N^2$ is upper semicontinuous, it suffices to verify the statement for a curve $C\subseteq X$ with generic point $\xi$ that fulfills $(o_\X(\xi),c_\X(\xi))=(o,c)$. The statement then follows from Proposition \ref{key_proposition_for_usc} (2).
\end{proof}

\section{Conclusions} \label{section_inv_is_usc} 

\begin{keytheorem} \label{inv_is_usc}
 Let $\X=(W,X,E)$ be a $3$-dimensional resolution setting and $\ii\in\N^7$. Then the set
 \[X_{\geq \ii}=\{a\in X:\iv_\X(a)\geq \ii\}\]
 is closed. In other words, the function $\iv_\X:X\to\N^7$ is upper semicontinuous.
\end{keytheorem}
\begin{proof}
 Set $\ii=(o,c,d,n,s,r,l)$. 
 Since $X_{\geq(o,c)}$ and $X$ is assumed to be irreducible, either $X_{\geq(o,c)}=X$ holds or $X_{\geq(o,c)}$ is at most $1$-dimensional. By Lemma \ref{c=1_measures_resolved}, $X_{\geq\ii}=X$ is equivalent to $(o,c)\leq(1,1)$ and $X_{\geq(o,c)}=X$.
 
 So assume that $(o,c)>(1,1)$ and $X_{\geq(o,c)}$ is at most $1$-dimensional. If $(d,n,s)>(0,0,0)$, we know by Corollary \ref{only_finitely_many_points_with_d>0} that $X_{\geq\ii}=X_{\geq(o,c+1)}\cup Y$ for a discrete set $Y$. Hence, $X_{\geq\ii}$ is closed.
 
 Now assume that $(d,n,s)=(0,0,0)$. Let $C$ be a $1$-dimensional component of $X_{\geq(o,c)}$ that is not entirely contained in $X_{\geq(o,c+1)}$. By Theorem \ref{order_is_usc} this implies for the generic point $\xi$ of $C$ that $(o_\X(\xi),c_\X(\xi))=(o,c)$. By Proposition \ref{key_proposition_for_usc} there is a pair $(r_\xi,l_\xi)\in\N^2$ such that $\iv_\X(a)=(o,c,0,0,0,r_\xi,l_\xi)$ for all but finitely many points $a\in C$ and $\iv_\X(a)\geq (o,c,0,0,0,r_\xi,l_\xi)$ at all points $a\in C$. Thus, $X_{\geq\ii}$ is the union of $X_{\geq(o,c+1)}$, all curves $C$ as above with $(r_\xi,l_\xi)\geq (r,l)$ and a discrete set $Y$. Hence, $X_{\geq\ii}$ is closed.
\end{proof}

\begin{keytheorem} \label{top_inv_is_permissible}
Let $\X=(W,X,E)$ be a $3$-dimensional resolution setting that is not already resolved. Set $\iv_{\max}=\max_{a\in X}\iv_\X(a)$ and 
\[X_{\max}=\{a\in X:\iv_\X(a)=\iv_{\max}\}.\]
Then $X_{\max}$ is a permissible center of blowup for $\X$.
\end{keytheorem}
\begin{proof}
 Set $\iv_{\max}=(o,c,d,n,s,m,l)$. By Lemma \ref{c=1_measures_resolved}, we can assume that $(o,c)>(1,1)$. By Theorem \ref{inv_is_usc} the set $X_{\max}$ is closed. If $(d,n,s)>(0,0,0)$, we know by Corollary \ref{only_finitely_many_points_with_d>0} that $X_{\max}$ is a discrete set. Hence, the assertion is trivial in this case.
 
 Now consider the case that $(d,n,s)=(0,0,0)$. By Proposition \ref{key_proposition_for_usc}, the set $X_{\max}$ is the union of a discrete set $Y$ and finitely many mutually disjoint regular curves $C$ with the property that $C\cup E$ has simple normal crossings. This proves the assertion.
\end{proof}

\chapter{Lexicographical decrease of $\iv_\X$ under blowup of its top locus} \label{chapter_decrease}

 In the previous chapter we established that set $\Xmax$ of points $a\in X$ at which the invariant $\iv_\X$ assumes its maximal value constitutes a permissible center of blowup. In this chapter we will show that blowing up this center always makes $\iv_\X$ decrease until $\X$ is resolved. This will complete the proof that $\iv_\X$ fulfills all the properties that were required for a resolution invariant in Section \ref{section_resolution_via_usc_inv}.
 
 The proof for showing that $\iv_\X$ decreases under blowup of its top locus is fundamentally different according to whether $\X$ is already in a terminal case along $\Xmax$ or not. For the terminal cases, it is only necessary to prove that the combinatorial pair decreases, which is quite straightforward to verify. In the non-terminal case, the proof is much more involved and in fact, most of the techniques that were developed in the thesis will be used for this proof. The arguments for the proof of Proposition \ref{inv_drops_in_non_terminal_case} were sketched in Section \ref{section_modifying_the_residual_order}. The techniques that will be used are mainly the cleaning techniques from Chapter 5, the results of Section \ref{section_weighted_orders_under_blowup} and Section \ref{section_kangaroo_calc} and the methods to construct maximizing flags we developed in Section \ref{section_maximizing_flags}.

\section{Decrease of $\iv_\X$ under point-blowup in the non-terminal case} \label{section_inv_drops_for_d>0}

 We consider the following setting in this section: Let $\X=(W,X,E)$ be a $3$-dimensional resolution setting and $a\in X$ a closed point with $c\Xa>1$. Consider the blowup $\pi:W'\to W$ at the point $a$ and let $\X'=(W',X',E')$ be the induced resolution setting. Set $\Dnew=\pi^{-1}(a)$. Further, let $a'\in\pi^{-1}(a)\cap X'$ be a closed point lying over $a$. We will denote the components of $\ivXa$ by $o,c,\ldots $ and the components of $\ivXap$ by $o',c',\ldots$
 
 In the following, we will always assume that $(o',c')=(o,c)$, $d>0$ and $d'>0$ hold. Thus, neither is $\X$ in a terminal case at $a$ nor is $\X'$ in a terminal case at $a'$ by Lemma \ref{inv_finite}.
 
 The goal of this section is to prove that $\ivXap<\ivXa$ holds. To this end, we need to show that for each valid flag $\G\in\FF(a')$ there exists a valid flag $\F\in\FF(a)$ such that $\inv(\G)<\inv(\F)$ holds.

\begin{definition}
 Let $\F\in\FF(a)$ be a formal flag at $a$ with the property that the point $a'$ lies on the strict transform of $\F_1$. Then we define the \emph{induced flag $\F'$ at $a'$} as the formal flag that consists of the strict transform $\F_2'$ of $\F_2$ and of the strict transform $\F_1'$ of $\F_1$.  
\end{definition}

 The following result guarantees that in almost all cases $\inv(\F')<\inv(\F)$ holds for the induced flag $\F'$.

\begin{proposition} \label{flag_invs_under_blowup}
Let $\F\in\FF(a)$ be a flag with subordinate parameters $\x=(x,y,z)$ such that $a'$ lies on $\F_1'$.
\begin{enumerate}[(1)]
 \item The induced flag $\F'$ is compatible with $E'$. In other words, $\F'\in\FF(a')$.
 \item The point $a'$ is the origin of the $x$-chart and the induced parameters $\x'=(x',y',z')$ for $\hOWap$ are subordinate to the induced flag $\F'$.
 \item If $n_\F=0$, the following hold:
 \begin{itemize}
  \item $n_{\F'}=0$.
  \item $d_{\F'}\leq d_{\F}$.
  \item If $d_{\F'}=d_\F$ and $s_\F<\infty$, then $s_{\F'}<s_\F$.
 \end{itemize}
\item If $n_\F=1$, the following hold:
\begin{itemize}
 \item $n_{\F'}=0$.
 \item If $d_\F\neq-1$, then $d_{\F'}\leq d_\F$.
\end{itemize}
 \item If $n_\F>1$, the following hold:
 \begin{itemize}
  \item $n_{\F'}=n_\F-1$.
  \item If $n_{\F'}>1$, then $D_{\F'}=(D_\F)'$ where $(D_\F)'$ denotes the strict transform of $D_\F$.
  \item $d_{\F'}=d_\F$.
 \end{itemize}
\end{enumerate}
\end{proposition}
\begin{proof}
 Since $o'=o$, we know that 
 \[\Egp=(\Eg)\st\cup\Dnew\]
 where $(\Eg)\st$ denotes the strict transform of $\Eg$. Since $\F_2\cup\Eg$ has simple normal crossings, it is clear that $\F_2'\cup\Egp$ again has simple normal crossings. Since $\F_2'$ is defined as the strict transform of $\F_2$, it is also clear that $\F_2'\not\subseteq\Egp$. Thus, $\F'$ is compatible with $E'$.
 
 For simplicity, we will denote by abuse of notation the induced parameters as $\x'=(x,y,z)$. Since $\F_1=V(y,z)$ and $a'\in\F_1'$, it is clear that $a'$ is the origin of the $x$-chart and $\Dnew=V(x)$. Further, $\F_2'=V(z)$ and $\F_1'=V(z,y)$.
 
 Consider first the case that $n_\F=0$. Then $\Eg\subseteq V(xy)$ and consequently, $\Egp\subseteq V(xy)$. Thus, $n_{\F'}=0$ and the parameters $\x'$ are subordinate to $\F'$. Moreover, it follows from Proposition \ref{w_clean_stable_under_blowup} (1) that the ideal $I_{2,\x'}(a')$ has the same order as the weak transform of $I_{2,\x}(a)$. This proves that $d_{\F'}\leq d_\F$ by Proposition \ref{ord_under_blowup}.
 
 Now consider the case that $d_{\F'}=d_\F$ and $s_\F<\infty$. Set $d=d_\F$. If $d\geq c!$, then $s_{\F'}=s_\F-d!<s_\F$ by Proposition \ref{s_clean_stable_under_blowup} (1). Now consider the case $0<d<c!$. Let $M_{2,\x}(a)=(x^{r_x}y^{r_y})$. By Proposition \ref{w_clean_stable_under_blowup} (1) we know that 
 \[M_{2,\x'}(a')=(x^{r_x+r_y+d-c!}y^{r_y}).\]
 Set 
 \[P_{2,\x}(a)=I_{2,\x}(a)^{c!-d}+M_{2,\x}(a)^d,\]
 \[P_{2,\x'}(a')=I_{2,\x'}(a')^{c!-d}+M_{2,\x'}(a')^d.\]
 Notice that $\ord P_{2,\x}(a)=d(c-d!)$ since $\ord M_{2,\x}(a)\geq c-d!$ by Lemma \ref{ogeqc!}. Denote for an ideal $I\subseteq K[[x,y]]$ its total transform by $I^*$. Then we can compute with Lemma \ref{coeff_ideal_under_blowup} that
 \[x^{-d(c!-d)}\cdot P_{2,\x}(a)^*=x^{-d(c!-d)}\cdot(I_{2,\x}(a)^*)^{c!-d}+x^{-d(c!-d)}\cdot(x^{r_x+r_y}y^{r_y})^d\]
 \[=(x^{-d}\cdot I_{2,\x}(a)^*)^{c!-d}+M_{2,\x'}(a')^{d}\subseteq P_{2,\x'}(a').\]
 Hence, $P_{2,\x'}(a')$ contains the weak transform of $P_{2,\x}(a)$. By Proposition \ref{w_clean_stable_under_blowup} (1) this proves that
 \[s_{\F'}=\ord\coeff_{(x,y)}^{d(c!-d)}(P_{2,\x'}(a'))\]
 \[\leq \ord\coeff_{(x,y)}^{d(c!-d)}(P_{2,\x}(a))-(d(c!-d))!=s_\F-(d(c!-d))!<s_\F.\]
 
 Now consider the case $n_\F=1$. Let $D$ be a component of $\Eg$. By Lemma \ref{form_of_associated_components} (1), $D=V(x+g)$ for an element $g\in K[[y,z]]$ with $\ord g\geq1$. Thus, the strict transform $D'$ of $D$ has the form
 \[D=V(x^{-1}(x+g(xy,yz))=V(1+g')\]
 for some element $g'\in\hOWap$ with $\ord g'\geq1$. Consequently, $a'\notin D'$. Thus, $\Egp=V(x)$. This proves that $n_{\F'}=0$ and the parameters $\x'$ are subordinate to $\F'$. 
 
 Let $\eta:K[[x,y]]\to\N_\infty$ be the weighted order function defined by $\eta(x)=1$ and $\eta(y)=2$. Then we can compute with Proposition \ref{w_clean_stable_under_blowup} (1) that
 \[d_{\F'}=\ord J_{2,\x'}(a')-\ord_{x}J_{2,\x'}(a')\]
 \[=(\eta(J_{2,\x}(a))-c!)-(\ord J_{2,\x}(a)-c!)=\eta(J_{2,\x}(a))-\ord J_{2,\x}(a).\]
 It remains to show that $\eta(J_{2,\x}(a))-\ord J_{2,\x}(a)\leq \ord_{(y)}\minit(J_{2,\x}(a))$. It is clear that there appears a term $x^iy^j$ with non-zero coefficient in the expansion of an element of $J_{2,\x}(a)$ such that $i+j=\ord J_{2,\x}(a)$ and $j=\ord_{(y)}\minit(J_{2,\x}(a))$. Trivially, also $i+2j\geq \eta(J_{2,\x}(a))$ holds. Thus,
 \[\ord_{(y)}\minit(J_{2,\x}(a))=j=\underbrace{(i+2j)}_{\geq \eta(J_{2,\x}(a))}-\underbrace{(i+j)}_{=\ord J_{2,\x}(a)}\geq\eta(J_{2,\x}(a))-\ord J_{2,\x}(a).\]
 If $d_\F\neq-1$, this proves that $d_{\F'}\leq d_\F$.
 
 Finally, consider the case $n_\F>1$. By Lemma \ref{form_of_associated_components} (2) the associated component $D_\F$ of $\F$ has the form 
 \[D_\F=V(y+Q(x)+zG)\]
 for elements $Q\in K[[x]]$ with $\ord Q=n_\F$ and $G\in\hOWa$. Hence, the strict transform $(D_\F)'$ contains $a'$ and has the form 
 \[(D_\F)'=V(y+Q'(x)+z\wt G)\]
 where $Q'(x)=x^{-1}Q(x)$ and $\wt G\in\hOWap$. Since $\Dnew=V(x)$ and $\ord Q'\geq1$, it is clear that the union of 
 \[\F_1'=V(y,z)\text{ and }\F_2'\cap \Egp=V(x\cdot(y+Q'(x)),z)\]
 is not simple normal crossings. Hence, $n_{\F'}>0$. Since
 \[\mult_{a'}(\F_1',\F_2'\cap (D_\F)')=\ord Q'=\ord Q-1=n_\F-1\]
 and $\mult_{a'}(\F_1',\F_2'\cap\Dnew)=1$, it is clear that $n_{\F'}=n_\F-1$ and $D_{\F'}=(D_\F)'$.
 
 Further, we can compute by Proposition \ref{w_clean_stable_under_blowup} (1) that
 \[(m_{\F',\x'},d_{\F',\x'})=\y_{\F',\x'}(J_{2,\x'}(a'))=\y_{\F,\x}(J_{2,\x}(a))-(c!,0)=(m_{\F,\x}-c!,d_{\F,\x}).\]
 This proves that $d_{\F'}=d_\F$.
\end{proof}

 There are two kinds of flags $\G\in\FF(a')$ which cannot be induced by flags $\F\in\FF(a)$. These are on the one hand, flags $\G\in\FF(a')$ with $n_\G>1$ and $D_\G=\Dnew$ and on the other hand, flags $\G\in\FF(a')$ with $n_\G=0$ and $\G_1=\G_2\cap\Dnew$. To bound $\inv(\G)$ for flags with $n_\G>1$ and $D_\G=\Dnew$, we will use the special version of Moh's bound that we developed in Section \ref{section_kangaroo_calc}. On the other hand, flags with $n_\G=0$ and $\G_1=\G_2\cap\Dnew$ can be ignored due to the following lemma:

\begin{lemma} \label{forget_the_x_flag}
 In the setting of the previous proposition, let $\F\in\FF(a)$ be such that $n_\F=0$ and assume that $d_{\F'}=d_\F>0$. Call the induced parameters in the $x$-chart again $\x'=(x,y,z)$. Let $f'\in I_3(a')$ be an element that is $\ord$-clean, $\ord_{(x)}$-clean, $\ord_{(y)}$-clean and $\rho$-clean with respect to $J_{2,\x}(a')$, where the weighted order function $\rho:K[[x,y]]\to\Ni^2$ is defined via $\rho(x)=(1,1)$ and $\rho(y)=(1,0)$. 
 
 Let $\G\in\FF(a')$ be a flag of the form $\G_2=V(z+g)$, $\G_1=V(z+g,x)$ for some element $g\in K[[x,y]]$. Hence, $\G_2\cap\Dnew=\G_1$. If $\G$ is valid, then $\inv(\G)\leq\inv(\F')$.
\end{lemma}
\begin{proof}
 Let $\G$ be a flag of above form. Clearly, $n_\G=0$. By Proposition \ref{flag_invs_under_blowup} we know $n_{\F'}=0$. Thus, we know by Proposition \ref{cleaning_transversal_flags} that the flag $\F'$ is valid and $d_\G\leq d_{\F'}$ holds. Thus, assume that $d_\G=d_{\F'}=d_\F$ holds. Set $d=d_\F$. It then remains to verify that $s_\G\leq s_{\F'}$ holds. By Lemma \ref{coeff_ideal_of_powers} we know that $s_\F\geq d!$ if $d\geq c!$ and $s_\F\geq (d(c!-d))!$ if $0<d<c!$. Thus, it suffices to show that 
 \[\coeff_{(x,y)}^d(I_{2,\x_1}(a'))=d!\]
 where the parameters $\x_1=(y,x,z+g)$ are subordinate to $\G$.
 
 Assume that $\ord \coeff_{(x,y)}^d(I_{2,\x_1}(a'))>d!$. By Lemma \ref{coeff_ideal_and_directrix}, this is equivalent to $\Dir(I_{2,\x_1}(a'))=(\ol x)$. In other words, $\minit(I_{2,\x_1}(a'))=(x^d)$. Let $J_{2,\x_1}(a')$ have the factorization 
 \[J_{2,\x_1}(a')=(x^{r_x}y^{r_y})\cdot I_{2,\x_1}(a').\]
 Then $\minit(I_{2,\x_1}(a'))=(x^d)$ is equivalent to 
 \[\ord_{(x)}\minit(J_{2,\x_1}(a'))=r_x+d.\]
 By Lemma \ref{double_weighted order function} we know that
 \[\rho(J_{2,\x_1}(a'))=(\ord J_{2,\x_1}(a'),\ord_{(x)}\minit(J_{2,\x_1}(a'))).\]
 Since $\G$ is valid and $d_\G=d$, we know by Proposition \ref{cleaning_transversal_flags} that the coefficient ideal $J_{2,\x'}(a')$ also has a factorization 
 \[J_{2,\x'}(a')=(x^{r_x}y^{r_y})\cdot I_{2,\x'}(a').\]
 Since $f'$ is $\rho$-clean with respect to $J_{2,\x'}(a')$ and $\ord J_{2,\x_1}(a')=\ord J_{2,\x'}(a')$, we know by Proposition \ref{w_cleaning_maximizes_w} that
 \[\ord_{(x)}\minit(J_{2,\x_1}(a'))\leq \ord_{(x)}\minit(J_{2,\x'}(a')).\]
 Hence, it suffices to show that $\ord_{(x)}\minit(J_{2,\x'}(a'))<r_x+d$ holds. 
 
 But since $I_{2,\x'}(a')$ has the same order as the weak transform of $I_{2,\x}(a)$ by Proposition \ref{w_clean_stable_under_blowup} (1) and the equality $\ord I_{2,\x'}(a')=\ord I_{2,\x}(a)$ holds, we know by Lemma \ref{directrix_under_blowup} that $\minit(I_{2,\x}(a))=(y^d)$. Thus, it is immediate to see that
 \[\minit(I_{2,\x'}(a'))\neq (x^d).\]
 In particular,
 \[\ord_{(x)}\minit(J_{2,\x_1}(a'))<r_x+d.\]
 This proves the assertion.
\end{proof}

 Notice that we have no estimate for $\inv(\F')$ in the case that $n_\F=1$ and $d_\F=-1$. To deal with this case, we will use the following lemma. It tells us that under an additional cleanness assumption, $\X'$ is actually in a terminal case at $a'$. Thus, $d'=0$ holds and $\ivXap<\ivXa$ is clear.

\begin{lemma} \label{super_s_r_lemma}
 In the setting of the previous proposition, let $\F\in\FF(a)$ be such that $n_\F=1$ and $d_\F=-1$. Call the induced parameters in the $x$-chart again $\x'=(x,y,z)$. Let $f'\in I_3(a')$ be an element that is $\ord$-clean, $\ord_{(x)}$-clean and $\ord_{(y)}$-clean with respect to $J_{2,\x'}(a')$.
 
 Then $\X'$ is in a terminal case at $a'$.
\end{lemma}
\begin{proof}
 It was shown in the proof of Proposition \ref{flag_invs_under_blowup} that $n_{\F'}=0$ and $d_{\F'}\leq d_{\F,\x}$ hold. In particular, $\x'$ are apposite parameters for $\X'$ at $a'$.
 
 If $0<d_{\F,\x}<c!$, we know that $m_\F=\ord J_{2,\x}(a)$ is divisible by $c!$. Consequently, we know by Proposition \ref{w_clean_stable_under_blowup} that $\ord_{(x)} J_{2,\x'}(a')$ is also divisible by $c!$. Since $d_{\F'}\leq d_{\F,\x}$ and $\Egp=V(x)$, this implies that $J_{2,\x'}(a')$ has the form
 \[J_{2,\x'}(a')=(x^{nc!})\cdot I\]
 for a positive integer $n>0$ and an ideal $I$ with $\ord_{(x)}I=0$ and $\ord I<c!$. If $\ord I=0$, then the parameters $\x'$ are of monomial type. If $\ord I>0$, then $\x'$ are parameters of small residual type.
 
 If $d_{\F,\x}=0$, we know that $d_{\F'}=0$. Thus, the coefficient ideal $J_{2,\x'}(a')$ is a principal monomial ideal. Hence, $\x'$ are parameters of monomial type.
\end{proof}

\begin{proposition} \label{inv_drops_in_non_terminal_case}
 Let $\G\in\FF(a')$ be a valid flag. Then there exists a valid flag $\F\in\FF(a)$ such that
 \[\inv(\G)<\inv(\F)\]
 holds. Consequently, $\ivXap<\ivXa$.
\end{proposition}
\begin{proof}
 Let $\x=(x,y,z)$ be apposite parameters for $\hOWa$ and $f\in I_3(a)$ an element that is $z$-regular of order $c$. By Lemma \ref{directrix_from_z_regularity} and Lemma \ref{z_regular_under_coord_change} we can assume without loss of generality that $\ol z\in\Dir(I_3(a))$. This implies by Lemma \ref{directrix_under_blowup} that $a'$ is not contained in the $z$-chart. Assume that $a'$ has the affine coordinates $(t,0)$ with $t\in K$ in the $x$-chart. By Proposition \ref{directrix_under_blowup} this implies that either $\Dir(I_3(a))=(\ol z)$ or $\Dir(I_3(a))=(\ol z,\ol{y-tx})$. If $t\neq0$ and $\Eg\neq V(xy)$, we can (after possibly swapping $x$ and $y$) replace $y$ by $y-tx$ since the parameters $(x,y-tx,z)$ are still apposite. Thus, we assume from now on that either $t=0$ or $\Eg=V(xy)$ holds. Let $\x'=(x',y',z')$ be the induced parameters for $\hOWap$ in the $x$-chart. Hence, the blowup map $\pi:\hOWa\to\hOWap$ is of the form
 \[\begin{array}{ll}
 \pi(x)=x', &\\
 \pi(y)=x'(y'+t), &\\ 
 \pi(z)=x'z'. &
\end{array}\]
 By Lemma \ref{z_regular_stable_under_blowup}, the parameters $\x'$ are again apposite and $f'=x'^{-c}\pi(f)\in I_3(a')$ is $z'$-regular of order $c$. Denote by $\Dnew=V(x')\subseteq\Egp$ the exceptional divisor of the blowup.
 
 The proof will be divided into four cases:
 \begin{enumerate}[(1)]
  \item $n_\G=0$ and $t=0$.
  \item $n_\G=0$, $t\neq0$ and $\Eg=V(xy)$.
  \item $n_\G=1$ or $n_\G>1$ and $D_\G\neq\Dnew$.
  \item $n_\G>1$ and $D_\G=\Dnew$.
 \end{enumerate}

 (1): Assume that $n_\G=0$ and $t=0$. Notice that in this case $V(y')\subseteq\Egp$ holds if and only if $V(y)\subseteq\Eg$. By Proposition \ref{transversal_flags_can_be_maximized_with_z} we can assume without loss of generality that $\G$ has the form $\G_2=V(z'+g)$ and either $\G_1=V(z'+g,y'+h)$ or $\G_1=V(z'+g,x')$ for elements $g\in K[[x',y']]$, $h\in K[[x']]$. Since $h\neq0$ implies that $V(y)\not\subseteq\Eg$, we can replace $y$ by $y+xh(x)$ without losing generality. Thus, we may assume that $h=0$.
 
 Let the flag $\F\in\FF(a)$ be defined by $\F_2=V(z)$ and $\F_1=V(z,y)$. Thus, $n_\F=0$. By Lemma \ref{w_cleaning_preserves_v_cleaning} and Lemma \ref{cleaning_preserves_directrix} we can assume that $f$ is clean with respect $J_{2,\x}(a)$ for finitely many given weighted order function defined on $(x,y)$. Since only finitely many different weighted order functions will be considered in each part of this proof, we will express this in the following as saying that $f$ is clean with respect to \emph{any} weighted order function defined on $(x,y)$. By Lemma \ref{s_cleaning_preserves_v_clean} we may also assume that $f$ is secondary $\ord$-clean with respect to $\coeff_{(x,y)}^{d_\F}(I_{2,\x}(a))$. By Proposition \ref{cleaning_transversal_flags} (1) the flag $\F$ is valid. By Proposition \ref{w_clean_stable_under_blowup} we can also assume that $f'$ is clean with respect to $J_{2,\x'}(a')$ for any weighted order function defined on $(x',y')$. Since $n_{\F'}=0$ by Proposition \ref{flag_invs_under_blowup}, this implies by Proposition \ref{cleaning_transversal_flags} that $\F'$ is valid and $d_\G\leq d_{\F'}$. By Proposition \ref{flag_invs_under_blowup} we know that $d_{\F'}\leq d_\F$. If $d_{\F'}<d_\F$, then it follows that $\inv(\G)<\inv(\F)$. So assume from now on that $d_{\F'}=d_\F$ holds. By Proposition \ref{s_clean_stable_under_blowup} the element $f'$ is secondary $\ord$-clean with respect to $\coeff_{(x',y')}^{d_{\F'}}(I_{2,\x'}(a'))$. Using Proposition \ref{cleaning_transversal_flags} (3) and Lemma \ref{forget_the_x_flag}, this implies that $\inv(\G)\leq\inv(\F')$ holds. Further, $\inv(\F')<\inv(\F)$ holds by Proposition \ref{flag_invs_under_blowup} if $s_\F<\infty$. So assume that $s_\F=\infty$. In this case, $\inv(\F')\leq\inv(\F)$ holds. Since $\X$ is not in a terminal case at $a$, the flag $\F$ is not maximizing by Lemma \ref{inv_finite}. Hence, there exists a valid flag $\HH\in\FF(a)$ such that $\inv(\G)\leq\inv(\F)<\inv(\HH)$. 
 
 (2): Assume that $n_\G=0$, $t\neq0$ and $\Eg=V(xy)$. Set $y_1=y-tx$. Then $a'$ is the origin of the $x$-chart with respect to the parameters $\x_1=(x,y_1,z)$ and the induced parameters are $\x'=(x',y',z')$. By Proposition \ref{transversal_flags_can_be_maximized_with_z} we can assume without loss of generality that $\G$ has the form $\G_2=V(z'+g)$ and either $\G_1=V(z'+g,y'+h)$ or $\G_1=V(z'+g,x')$ for elements $g\in K[[x',y']]$, $h\in K[[x']]$. Replacing $y_1$ by $y_1+xh(x)$, we can assume without loss of generality that $h=0$.
 
 Let the flag $\F\in\FF(a)$ be defined by $\F_2=V(z)$ and $\F_1=V(z,y_1)$. Thus, $n_\F=1$. As before, we can assume that $f$ is clean with respect to $J_{2,\x_1}(a)$ for any weighted order function defined on $(x,y_1)$ and $f'$ is clean with respect to $J_{2,\x'}(a')$ for any weighted order function defined on $(x',y')$. By Proposition \ref{cleaning_tangential_flags} (1) we know that $\F$ is valid. If $d_\F\neq-1$, we know by Proposition \ref{flag_invs_under_blowup} that $d_{\F'}\leq d_\F$. By Proposition \ref{cleaning_transversal_flags} (2) we also know that $d_{\G}\leq d_{\F'}$ holds. Since $n_\G<n_\F$, this proves that $\inv(\G)<\inv(\F)$. So assume now that $d_\F=-1$. Then $\X'$ is in a terminal case at $a'$ by Lemma \ref{super_s_r_lemma}. This contradicts our assumption.
 
 (3): Assume that either $n_\G=1$ or $n_\G>1$ and $D_\G\neq\Dnew$ hold. Hence, we know that $\Egp$ has a second component other than $\Dnew=V(x')$. Consequently, we know that $t=0$ and $V(y)\subseteq\Eg$. Thus, $D_\G=V(y')$ if $n_\G>1$. By Proposition \ref{tangential_flags_can_be_maximized_with_z} we can assume without loss of generality that $\G$ has the form $\G_2=V(z'+g)$ and $\G_1=V(z'+g,y_1')$ where $g\in K[[x',y']]$, $y_1'=y'+\lambda x'^n$ for a constant $\lambda\in K^*$ and $n=n_\G$. Set $y_1=y+\lambda x^{n+1}$. Then $a'$ is the origin of the $x$-chart with respect to the parameters $\x_1=(x,y_1,z)$ and the induced parameters for $\hOWap$ are $\x_1'=(x',y_1',z')$. Notice that $\ol y_1=\ol y$.
 
 Let the flag $\F\in\FF(a)$ be defined by $\F_2=V(z)$ and $\F_1=V(z,y_1)$. Thus, $n_\F=n_\G+1$. As before, we can assume that $f$ is $\y_{\F,\x_1}$-clean with respect to $J_{2,\x_1}(a)$ and $f'$ is $\y_{\F',\x_1'}$-clean with respect to $J_{2,\x_1'}(a')$. By Proposition \ref{flag_invs_under_blowup} we know that $n_{\F'}=n_\G$ and $D_{\F'}=V(y')$ if $n_{\F'}>1$. Further, we know by Proposition \ref{cleaning_tangential_flags} that $\F$ is valid and $\inv(\G)\leq\inv(\F')$. Further, $\inv(\F')<\inv(\F)$ holds by Proposition \ref{flag_invs_under_blowup}.
 
 (4): Assume that $n_\G>1$ and $D_\G=\Dnew=V(x')$. We can assume without loss of generality that $d_\G\neq -1$. By Proposition \ref{tangential_flags_can_be_maximized_with_z} we can further assume that $\G$ has the form $\G_2=V(z_1')$ and $\G_1=V(z_1',x_1')$ where $x_1'=x'+\lambda y'^n$ for a constant $\lambda\in K^*$, $n=n_\G$ and the coordinate change $z'=z_1'+g$ is $\y$-cleaning with respect to $f'$ and $J_{2,\x'}(a')$ for the weighted order function $\y:K[[x_1',y']]\to\Ni^2$ that is defined by $\y(x_1')=n_\G$ and $\y(y')=1$. Set $y_1=y-tx$. Then $a'$ is the origin of the $x$-chart with respect to the parameters $\x_1=(x,y_1,z)$ and the induced parameters for $\hOWap$ are $\x'=(x',y',z')$. Notice that the weighted order function $\w:K[[x',y']]\to\Ni^2$, $\w(x')=\w(x_1')=n$, $\w(y')=1$ is defined both on the parameters $(x',y')$ and $(x_1',y')$.
 
 Let the flag $\F\in\FF(a)$ be defined by $\F_2=V(z)$ and $\F_1=V(z,y_1)$. Hence, $n_\F=0$ if $t=0$ and $n_\F=1$ if $t\neq0$ and $\Eg=V(xy)$. As before, we can assume that $f$ is clean with respect to $J_{2,\x_1}(a)$ for any weighted order function defined on $(x,y_1)$ and $f'$ is clean with respect to $J_{2,\x'}(a')$ for any weighted order function defined on $(x',y')$. By Proposition \ref{cleaning_transversal_flags} and Proposition \ref{cleaning_tangential_flags}, this implies that $\F$ is valid. Further, we can assume by Lemma \ref{super_s_r_lemma} that $d_\F\neq -1$.
 
 Define the number $d_1$ as 
 \[d_1=\ord \minit_{\w}(J_{2,\x'}(a'))-\ord_{(x)} \minit_{\w}(J_{2,\x'}(a'))-\ord_{(y)} \minit_{\w}(J_{2,\x'}(a')).\]
 
 It follows from Proposition \ref{kangaroo_prop} and Proposition \ref{kangaroo_blowup_prop} that
 \[d_\G\leq \frac{d_1}{n_\G}\leq \frac{d_\F}{n_\G}+\eps\]
 holds, where
 \[\eps=\begin{cases}
         0 & \text{if $\chara(K)=0$ or $c!\nmid m_\G$,}\\
         \frac{c!}{p} & \text{if $\chara(K)=p>0$ and $c!\mid m_\G$.}
        \end{cases}
\]
 If $\eps=0$, this already proves that $\inv(\G)<\inv(\F)$ since $n_\G\geq2$. So assume now that $\eps>0$. Thus, $m_\G$ is divisible by $c!$. Hence, we know that $d_\G\geq c!$ holds. Since $n_\G\geq2$ and $p\geq2$, above inequality implies that
 \[d_\G\leq \frac{1}{2}(d_\F+c!).\]
 Since $d_\G\geq c!$, this implies that $d_\G\leq d_\F$ holds. Now assume that the equality $d_\G=d_\F$ holds. This implies that $d_\F=d_\G=c!$ and $d_1\leq \frac{c!}{2}$. But this is a contradiction to Proposition \ref{kangaroo_prop} (4). Consequently, $\inv(\G)<\inv(\F)$ holds.
\end{proof}

\section{Decrease of $\ivX$ in terminal cases} \label{section_inv_drops_for_d=0}

 In this section we will always consider the following setting: Let $\X=(W,X,E)$ be a $3$-dimensional resolution setting that is not already resolved. Set $\iv_{\max}=\max\{\ivXa:a\in X\}$ and
\[\Xmax=\{a\in X:\iv_\X(a)=\iv_{\max}\}.\]
 By Theorem \ref{top_inv_is_permissible}, the set $\Xmax$ is a permissible center of blowup for $\X$.  Let $\pi:W'\to W$ be the blowup of $W$ with center $\Xmax$ and $\X'=(W',X',E')$ the induced setting.
 
 Further, let $a\in\Xmax$ be a closed point and $a'\in\pi^{-1}(a)\cap X'$ a closed point lying over $a$. We will denote the components of $\ivXa$ by $o,\ldots$ and the components of $\ivXap$ by $o',\ldots$.
 
 In the following, we will always assume that $(o',c')=(o,c)$ and that $\X$ is in a terminal case at $a$. Our goal is to show that $\ivXap<\ivXa$ holds.

\begin{proposition} \label{inv_drops_in_g_p_case}
 If $\X$ is in the small residual case at $a$, then $\X'$ is again in the small residual case at $a'$ and $\ivXap<\ivXa$ holds.
\end{proposition}
\begin{proof}
 Let $\x=(x,y,z)$ be apposite parameters of small residual type. Hence, $J_{2,\x}(a)$ is of the form
 \[J_{2,\x}(a)=(y^{mc!})\cdot I\]
 for a positive integer $m>0$ and an ideal $I$ with $\ord_{(y)}I=0$ and $0<\ord I<c!$. Further, there is an element $f\in I_3(a)$ which is $\ord_{(y)}$-clean with respect to $J_{2,\x}(a)$.
 
 By Lemma \ref{small_residual_case} (2) we know that $\Ioc(a)=(y,z)$. By Proposition \ref{key_proposition_for_usc} this implies that $\Xmax$ has the form $\Xmax=V(y,z)$ locally at $a$.
 
 Since $\ord J_{2,\x}(a)>c!$, we know by Lemma \ref{coeff_ideal_and_directrix} that $\Dir(I_3(a))=(\ol z)$. Hence, $a'$ is the origin of the $y$-chart by Lemma \ref{directrix_under_blowup}. Call the induced parameters for $\hOWap$ again $\x'=(x,y,z)$. By Proposition \ref{w_clean_stable_under_blowup} (1) we know that
 \[J_{2,\x'}(a')=(y^{(m-1)c!})\cdot I'\]
 for an ideal $I'$ with $\ord_{(y)}I'=0$ and $\ord I'=\ord I$. By Proposition \ref{w_clean_stable_under_blowup} (2), there is an element $f'\in I_3(a')$ which is $\ord_{(y)}$-clean with respect to $J_{2,\x'}(a')$. Thus, the parameters $\x'$ are again of small residual type and
 \[r'=(m-1)c!<mc!=r.\]
 Thus, $\ivXap<\ivXa$ holds.
\end{proof}

\begin{proposition} \label{inv_drops_in_monom_case}
 If $\X$ is in the monomial case at $a$, then $\X'$ is again in the monomial case at $a'$ and $\ivXap<\ivXa$ holds.
\end{proposition}
\begin{proof}
 Let $\x=(x,y,z)$ be apposite parameters of monomial type. Hence, $J_{2,\x}(a)$ is of the form
 \[J_{2,\x}(a)=(x^{r_x}y^{r_y})\]
 and there is an element $f\in I_3(a)$ which is $\ord$-clean with respect to $J_{2,\x}(a)$. By Lemma \ref{w_cleaning_preserves_v_cleaning} we can even assume that $f$ is clean with respect to $J_{2,\x}(a)$ for any weighted order function defined on $(x,y)$.
 
 We will now consider two cases:
 \begin{enumerate}[(1)]
  \item $r_x\geq c!$ or $r_y\geq c!$.
  \item $r_x,r_y<c!$.
 \end{enumerate}

 (1): Assume without loss of generality that $r_y\geq c!$. If $r_x<c!$, we know by Lemma \ref{monomial_case} (2) that $\Ioc(a)=(y,z)$. By Proposition \ref{key_proposition_for_usc} this implies $\Xmax=V(y,z)$ locally at $a$.
 
 If $r_x\geq c!$, we know by Lemma \ref{monomial_case} (2) that $\Ioc(a)=(xy,z)$. By Proposition \ref{key_proposition_for_usc}, this implies that there are two components of $X_{\geq(o,c)}$ passing through $a$ which are both regular at $a$. The pairs $(r_x,l_x)$, $(r_y,l_y)$ are the generic combinatorial pairs along these two components and $(r_x,l_x)\neq(r_y,l_y)$. Assume without loss of generality that $(r_y,l_y)>(r_x,l_x)$ holds. Thus, $\Xmax=V(y,z)$ locally at $a$.
 
 If $\ord J_{2,\x}=r_x+r_y>c!$, then we know by Lemma \ref{coeff_ideal_and_directrix} that $\Dir(I_3(a))=(\ol z)$. Thus, $a'$ is the origin of the $y$-chart by Proposition \ref{directrix_under_blowup}. On the other hand, if $\ord J_{2,\x}(a)=c!$, then necessarily $\ord J_{2,\x}(a)=(y^{c!})$. By Lemma \ref{another_tau=2_lemma} (1) this implies $\tau(I_3(a))=2$ and $\Dir(I_3(a))=(\ol y,\ol z)$. But this implies by Proposition \ref{directrix_under_blowup} that $\ord I_3(a')<\ord I_3(a)$ which contradicts our assumption.
 
 For simplicity call the induced parameters for $\hOWap$ again $\x'=(x,y,z)$. Then by Proposition \ref{w_clean_stable_under_blowup} (1) we know that 
 \[J_{2,\x'}(a')=y^{-c!}J_{2,\x}(a)=(x^{r_x}y^{r_y-c!}).\]
 
 By Proposition \ref{w_clean_stable_under_blowup} (2) we know that there is an element $f'\in I_3(a')$ which is $\ord$-clean with respect to $J_{2,\x'}(a')$. Hence, the induced parameters $\x'$ are of monomial type and $\X'$ is in the monomial case at $a'$. It remains to show that $(r',l')<(r,l)$. For this, we have to consider several cases.
 
 If $r_y-c!\geq c!$ or $r_x\geq c!$, then $(r',l')=\max\{(r_x,l_x),(r_y-c!,\lab(\Dnew))\}$. If $r_y-c!\geq r_x$ holds, then $r'=r_y-c!<r_y=r$. On the other hand, if $r_y-c!<r_x$ and $r_y>r_x$, then $r'=r_x<r_y=r$. Lastly, if $r_y=r_x$, then $r'=r_y=r_x=r$ and $l'=l_x<l_y=l$.
 
 If $r_y-c!<c!$ and $r_x<c!$, then $(r',l')=(r_x+r_y-c!,l_x+\lab(\Dnew))$. Thus, $r'=r_y+(r_x-c!)<r_y=r$.
 
 (2): By Lemma \ref{monomial_case} (2) we know that $a$ is an isolated point of $X_{\geq(o,c)}$. Hence, it also an isolated point of $\Xmax$.
 
 If $\ord J_{2,\x}(a)=r_x+r_y>c!$, then $\Dir(I_3(a))=(\ol z)$ by Lemma \ref{coeff_ideal_and_directrix}. Thus, $a'$ is not contained in the $z$-chart by Proposition \ref{directrix_under_blowup}. On the other hand, if $\ord J_{2,\x}(a)=c!$, then $\tau(I_3(a))=3$ by Lemma \ref{another_tau=2_lemma} (2). By Proposition \ref{directrix_under_blowup} this would imply that $\ord I_3(a')<\ord I_3(a)$, which is a contradiction to our assumption.
 
 Consider first the case that $a'$ has affine coordinates $(t,0)$ with $t\in K^*$ in the $x$-chart. For simplicity call the induced parameters for $\hOWap$ again $\x'=(x,y,z)$. By Proposition \ref{w_clean_stable_under_blowup} (1) we know that 
 \[J_{2,\x'}(a')=(x^{r_x+r_y-c!}(y+t)^{r_y}).\]
 Hence, 
 \[\ord J_{2,\x'}(a')=r_x+r_y-c!<r_x<c!.\]
 This contradicts Lemma \ref{ogeqc!}.
 
 So we know that $a'$ is either the origin of the $x$-chart or of the $y$-chart. Assume without loss of generality that $a'$ is the origin of the $x$-chart. Then
 \[J_{2,\x}(a')=(x^{r_x+r_y-c!}y^{r_y}).\]
 By Proposition \ref{w_clean_stable_under_blowup} we can assume that there is an element $f'\in I_3(a')$ which is $\ord$-clean with respect to $J_{2,\x'}(a')$. Hence, the parameters $\x'$ are of monomial type. Further,
 \[r'=r_x+2r_y-c!=r_x+r_y+\underbrace{(r_y-c!)}_{<0}<r_x+r_y=r.\]
 Consequently, $\ivXap<\ivXa$ holds.
\end{proof}

\section{Conclusion}

\begin{keytheorem} \label{inv_drops}
 Let $\X=(W,X,E)$ be a $3$-dimensional resolution setting that is not resolved already. Let $\pi:W'\to W$ be the blowup of $W$ along the center $\Xmax$ induced by $\iv_\X$. Denote by $\X'=(W',X',E')$ the induced $3$-dimensional resolution setting. Let $a\in\Xmax$ and $a'\in\pi^{-1}(a)\cap X'$ be points. Then
 \[\ivXap<\ivXa.\]
\end{keytheorem}
\begin{proof}
 By Theorem \ref{inv_is_usc} we can assume without loss of generality that $a$ and $a'$ are closed points. Denote the components $o\Xa,\ldots$ of $\ivXa$ by $o,\ldots$ and the components $o\Xap,\ldots$ of $\ivXap$ by $o',\ldots$.
 
 Since $\X$ is not resolved, we know by Lemma \ref{c=1_measures_resolved} that $(o,c)>(1,1)$. Since $\Xmax$ is permissible with respect to the order function on $X$, we know that $o'\leq o$. If $o'=o$, then it is clear that also $c'\leq c$ holds. So assume from now on that $(o',c')=(o,c)$.
 
 If $\X$ is not in a terminal case at $a$, we know by Corollary \ref{only_finitely_many_points_with_d>0} that $a$ is an isolated point of $\Xmax$. If $\X'$ is in a terminal case at $a'$, then $\ivXap<\ivXa$ holds by definition. If $\X'$ is not in a terminal case at $a'$, the assertion is proved in Proposition \ref{inv_drops_in_non_terminal_case}. If $\X$ is in a terminal case at $a$, the assertion follows from Proposition \ref{inv_drops_in_g_p_case} and Proposition \ref{inv_drops_in_monom_case}.
\end{proof}

\bibliography{bibliography}

\providecommand{\bysame}{\leavevmode\hbox to3em{\hrulefill}\thinspace}
\providecommand{\MR}{\relax\ifhmode\unskip\space\fi MR }
\providecommand{\MRhref}[2]{%
  \href{http://www.ams.org/mathscinet-getitem?mr=#1}{#2}
}
\providecommand{\href}[2]{#2}
\begin{thebibliography}{AHV75}

\bibitem[Abh56]{Abhyankar_56}
S.~Abhyankar, \emph{{Local uniformization on algebraic surfaces over ground
  fields of characteristic {$p\ne 0$}}}, Ann. of Math. (2) \textbf{63} (1956),
  491--526.

\bibitem[Abh66]{Abhyankar_3_Folds}
\bysame, \emph{{Resolution of singularities of embedded algebraic surfaces}},
  {Pure and Applied Mathematics, Vol. 24}, Academic Press, New York-London,
  1966.

\bibitem[Abh67]{Abhyankar_67}
\bysame, \emph{{Nonsplitting of valuations in extensions of two dimensional
  regular local domains}}, Math. Ann. \textbf{170} (1967), 87--144.

\bibitem[Abh83]{Abhyankar_Plane_Curves}
\bysame, \emph{{Desingularization of plane curves}}, {Singularities, {P}art 1
  ({A}rcata, {C}alif., 1981)}, {Proc. Sympos. Pure Math.}, vol.~40, Amer. Math.
  Soc., Providence, RI, 1983, pp.~1--45.

\bibitem[Abh88]{Abhyankar_Good_Points}
\bysame, \emph{{Good points of a hypersurface}}, Adv. in Math. \textbf{68}
  (1988), no.~2, 87--256.

\bibitem[AHV75]{AHV_Maximal_Contact}
J.-M. Aroca, H.~Hironaka, and J.-L. Vicente, \emph{{The theory of the maximal
  contact}}, Instituto ``Jorge Juan'' de Matem{\'a}ticas, Consejo Superior de
  Investigaciones Cientificas, Madrid, 1975, Memorias de Matem{\'a}tica del
  Instituto ``Jorge Juan'', No. 29.

\bibitem[AZ55]{Zariski_Abhyankar_55}
S.~Abhyankar and O.~Zariski, \emph{{Splitting of valuations in extensions of
  local domains}}, Proc. Nat. Acad. Sci. U. S. A. \textbf{41} (1955), 84--90.

\bibitem[Ben70]{Bennett}
B.~Bennett, \emph{{On the characteristic functions of a local ring}}, Ann. of
  Math. (2) \textbf{91} (1970), 25--87.

\bibitem[BM91]{BM_A_Simple_Proof}
E.~Bierstone and P.~Milman, \emph{{A simple constructive proof of canonical
  resolution of singularities}}, {Effective methods in algebraic geometry
  ({C}astiglioncello, 1990)}, {Progr. Math.}, vol.~94, Birkh{\"a}user Boston,
  Boston, MA, 1991, pp.~11--30.

\bibitem[BM97]{BM_Canonical_Desing}
\bysame, \emph{{Canonical desingularization in characteristic zero by blowing
  up the maximum strata of a local invariant}}, Invent. Math. \textbf{128}
  (1997), no.~2, 207--302.

\bibitem[BV01]{BV_Strengthening}
A.~Bravo and O.~Villamayor, \emph{{Strengthening the theorem of embedded
  desingularization}}, Math. Res. Lett. \textbf{8} (2001), no.~1-2, 79--89.

\bibitem[BV10]{BV_Simplification}
\bysame, \emph{{Singularities in positive characteristic, stratification and
  simplification of the singular locus}}, Adv. Math. \textbf{224} (2010),
  no.~4, 1349--1418.

\bibitem[BV11]{BV_Elimination}
\bysame, \emph{{Elimination algebras and inductive arguments in resolution of
  singularities}}, Asian J. Math. \textbf{15} (2011), no.~3, 321--355.

\bibitem[BV12]{Benito_Villamayor}
A.~Benito and O.~Villamayor, \emph{{Techniques for the study of singularities
  with applications to resolution of 2-dimensional schemes}}, Math. Ann.
  \textbf{353} (2012), no.~3, 1037--1068.

\bibitem[BV13]{BV_Monoidal}
\bysame, \emph{{Monoidal transforms and invariants of singularities in positive
  characteristic}}, Compos. Math. \textbf{149} (2013), no.~8, 1267--1311.

\bibitem[BV14]{BV_Multiplicity}
A.~Bravo and O.~Villamayor, \emph{{On the behavior of the multiplicity on
  schemes: stratification and blow ups}}, {The resolution of singular algebraic
  varieties}, Amer. Math. Soc., Providence, RI, 2014, pp.~81--207.

\bibitem[CJS09]{CJS}
V.~Cossart, U.~Jannsen, and S.~Saito, \emph{{Canonical embedded and
  non-embedded resolution of singularities for excellent two-dimensional
  schemes}}, 2009, preprint, arXiv:0905.2191.

\bibitem[CP08]{Cossart_Piltant_1}
V.~Cossart and O.~Piltant, \emph{{Resolution of singularities of threefolds in
  positive characteristic. {I}. {R}eduction to local uniformization on
  {A}rtin-{S}chreier and purely inseparable coverings}}, J. Algebra
  \textbf{320} (2008), no.~3, 1051--1082.

\bibitem[CP09]{Cossart_Piltant_2}
\bysame, \emph{{Resolution of singularities of threefolds in positive
  characteristic. {II}}}, J. Algebra \textbf{321} (2009), no.~7, 1836--1976.

\bibitem[Cut04]{Cutkosky_Book}
S.~D. Cutkosky, \emph{{Resolution of singularities}}, American Mathematical
  Society, Providence, R.I, 2004.

\bibitem[Cut11]{Cutkosky_Skeleton}
\bysame, \emph{{A skeleton key to {A}bhyankar's proof of embedded resolution of
  characteristic p surfaces}}, Asian J. Math. \textbf{15} (2011), no.~3,
  369--416.

\bibitem[EH02]{EH}
S.~Encinas and H.~Hauser, \emph{{Strong resolution of singularities in
  characteristic zero}}, Comment. Math. Helv. \textbf{77} (2002), no.~4,
  821--845.

\bibitem[Eis99]{Eisenbud_Com_Alg}
D.~Eisenbud, \emph{{Commutative Algebra: with a View Toward Algebraic Geometry
  (Graduate Texts in Mathematics)}}, Springer, 2 1999.

\bibitem[EV98]{EV_Good_Points}
S.~Encinas and O.~Villamayor, \emph{{Good points and constructive resolution of
  singularities}}, Acta Math. \textbf{181} (1998), no.~1, 109--158.

\bibitem[EV03]{EV_Char_Zero}
\bysame, \emph{{A new proof of desingularization over fields of characteristic
  zero}}, {Proceedings of the {I}nternational {C}onference on {A}lgebraic
  {G}eometry and {S}ingularities ({S}panish) ({S}evilla, 2001)}, vol.~19, 2003,
  pp.~339--353.

\bibitem[FK11]{AFK_Modified_Coeff_Ideal}
A.~Fr{\"u}hbis-Kr{\"u}ger, \emph{{A modified coefficient ideal for use with the
  strict transform}}, J. Symbolic Comput. \textbf{46} (2011), no.~5, 550--560.

\bibitem[Gir74]{Giraud_Maximal_Contact}
J.~Giraud, \emph{{Sur la th{\'e}orie du contact maximal}}, Math. Z.
  \textbf{137} (1974), 285--310.

\bibitem[Har77]{Hartshorne}
R.~Hartshorne, \emph{{Algebraic geometry}}, Springer-Verlag, New York, 1977.

\bibitem[Hau03]{Ha_BAMS_1}
H.~Hauser, \emph{{The {H}ironaka theorem on resolution of singularities (or:
  {A} proof we always wanted to understand)}}, Bull. Amer. Math. Soc. (N.S.)
  \textbf{40} (2003), no.~3, 323--403 (electronic).

\bibitem[Hau10]{Ha_BAMS_2}
\bysame, \emph{{On the problem of resolution of singularities in positive
  characteristic (or: a proof we are still waiting for)}}, Bull. Amer. Math.
  Soc. (N.S.) \textbf{47} (2010), no.~1, 1--30.

\bibitem[Hau14]{Ha_Obergurgl_2014}
\bysame, \emph{{Blowups and resolution}}, {The resolution of singular algebraic
  varieties}, Amer. Math. Soc., Providence, RI, 2014, pp.~1--80.

\bibitem[Hir64]{Hironaka_Annals}
H.~Hironaka, \emph{{Resolution of singularities of an algebraic variety over a
  field of characteristic zero. {I}, {II}}}, Ann. of Math. (2) 79 (1964),
  109--203; ibid. (2) \textbf{79} (1964), 205--326.

\bibitem[Hir84]{Hironaka_Bowdoin}
\bysame, \emph{{ Desingularization of excellent surfaces, Bowdoin 1967}},
  {Resolution of surface singularities, V. Cossart, J. Giraud, and U. Orbanz},
  {Lecture Notes in Mathematics}, vol. 1101, Springer-Verlag, 1984.

\bibitem[Hir12]{Hironaka_CMI}
\bysame, \emph{{Resolution of singularities}}, Manuscript distributed at the
  CMI Summer School 2012, 138 pp.

\bibitem[Hoc73]{Hochster_73}
M.~Hochster, \emph{{Criteria for equality of ordinary and symbolic powers of
  primes}}, Math. Z. \textbf{133} (1973), 53--65.

\bibitem[HW14]{HW}
H.~Hauser and D.~Wagner, \emph{{Alternative invariants for the embedded
  resolution of purely inseparable surface singularities}}, Enseign. Math.
  \textbf{60} (2014), no.~1-2, 177--224.

\bibitem[Kaw07]{Kawanoue_IF_1}
H.~Kawanoue, \emph{{Toward resolution of singularities over a field of positive
  characteristic. {I}. {F}oundation; the language of the idealistic
  filtration}}, Publ. Res. Inst. Math. Sci. \textbf{43} (2007), no.~3,
  819--909.

\bibitem[Kaw14]{Kawanoue_Obergurgl}
\bysame, \emph{{Introduction to the idealistic filtration program with emphasis
  on the radical saturation}}, {The resolution of singular algebraic
  varieties}, Amer. Math. Soc., Providence, RI, 2014, pp.~285--317.

\bibitem[KM]{Kawanoue_Matsuki_Surfaces}
H.~Kawanoue and K.~Matsuki, \emph{{Resolution of singularities of an idealistic
  filtration in dimension 3 after Benito-Villamayor}}, to appear in Adv. Stud.
  Pure Math.

\bibitem[KM10]{Kawanoue_Matsuki}
\bysame, \emph{{Toward resolution of singularities over a field of positive
  characteristic (the idealistic filtration program) {P}art {II}. {B}asic
  invariants associated to the idealistic filtration and their properties}},
  Publ. Res. Inst. Math. Sci. \textbf{46} (2010), no.~2, 359--422.

\bibitem[Kol07]{Kollar_Book}
J.~Koll{\'a}r, \emph{{Lectures on resolution of singularities}}, Princeton
  University Press, Princeton, 2007.

\bibitem[Lip75]{Lipman_Introduction}
J.~Lipman, \emph{{Introduction to resolution of singularities}}, {Algebraic
  geometry, Arcata 1974. {P}roc. {S}ympos. {P}ure {M}ath., {V}ol. 29}, Amer.
  Math. Soc., 1975, pp.~187--230.

\bibitem[Lip78]{Lipman_Surfaces}
\bysame, \emph{{Desingularization of two-dimensional schemes}}, Ann. Math. (2)
  \textbf{107} (1978), no.~1, 151--207.

\bibitem[Moh87]{Moh}
T.-T. Moh, \emph{{On a stability theorem for local uniformization in
  characteristic {$p$}}}, Publ. Res. Inst. Math. Sci. \textbf{23} (1987),
  no.~6, 965--973.

\bibitem[Nar83]{Narasimhan}
R.~Narasimhan, \emph{{Hyperplanarity of the equimultiple locus}}, Proc. Amer.
  Math. Soc. \textbf{87} (1983), no.~3, 403--408.

\bibitem[Vil89]{Villamayor_Constructiveness}
O.~Villamayor, \emph{{Constructiveness of {H}ironaka's resolution}}, Ann. Sci.
  {\'E}cole Norm. Sup. (4) \textbf{22} (1989), no.~1, 1--32.

\bibitem[Vil92]{Villamayor_Patching}
\bysame, \emph{{Patching local uniformizations}}, Ann. Sci. {\'E}cole Norm.
  Sup. (4) \textbf{25} (1992), no.~6, 629--677.

\bibitem[Vil07]{Villamayor_Hypersurface}
\bysame, \emph{{Hypersurface singularities in positive characteristic}}, Adv.
  Math. \textbf{213} (2007), no.~2, 687--733.

\bibitem[W{\l}o05]{Wlodarczyk}
J.~W{\l}odarczyk, \emph{{Simple {H}ironaka resolution in characteristic zero}},
  J. Amer. Math. Soc. \textbf{18} (2005), no.~4, 779--822 (electronic).

\bibitem[Zar44]{Zariski_1944}
O.~Zariski, \emph{{Reduction of the singularities of algebraic three
  dimensional varieties}}, Ann. of Math. (2) \textbf{45} (1944), 472--542.

\bibitem[ZS75]{Zariski_Samuel}
O.~Zariski and P.~Samuel, \emph{{Commutative algebra}}, Springer-Verlag, New
  York, 1975.

\end{thebibliography}
\bibliographystyle{amsalpha} %

\end{document}